\documentclass[a4paper]{article}
\usepackage [a4paper,left=2.5cm,bottom=1.7cm,right=2.5cm,top=1.7cm]{geometry}
\usepackage[english]{babel}
\usepackage{amssymb}
\usepackage{amsmath,amsthm}
\usepackage{subcaption}
\usepackage{tabularx,array}
\usepackage{graphicx}
\usepackage{enumitem} 

\newtheorem{theorem}{Theorem}
\newtheorem{lemma}{Lemma}

\newtheorem{proposition}{Proposition}
\newtheorem{definition}{Definition}

\newcommand{\keywords}[1]{\par\addvspace\baselineskip
\noindent\enspace\ignorespaces#1}
\usepackage{color}
\usepackage{xcolor}
\newcommand{\modar}{\color{black}}
\newcommand{\modch}{\color{black}}

\begin{document}

\title{Joint estimation for volatility and drift parameters of ergodic jump diffusion processes via contrast function.}
\author{Chiara Amorino$^{*}$, Arnaud Gloter\thanks{Laboratoire de Math\'ematiques et Mod\'elisation d'Evry, CNRS, Univ Evry, Universit\'e Paris-Saclay, 91037, Evry, France.}} 

\maketitle

\begin{abstract}
In this paper we consider an ergodic diffusion process with jumps whose drift coefficient depends on $\mu$ and volatility coefficient depends on $\sigma$, two unknown parameters. We suppose that the process is discretely observed at the instants $(t^n_i)_{i=0,\dots,n}$ with $\Delta_n=\sup_{i=0,\dots,n-1} (t^n_{i+1}-t^n_i) \to 0$. We introduce an estimator of $\theta:=(\mu, \sigma)$, based on a contrast function, which is asymptotically gaussian without requiring any conditions on the rate at which $\Delta_n \to 0$, assuming a finite jump activity.
This extends earlier results where a condition on the step discretization was needed (see \cite{GLM},\cite{Shimizu}) or where only the estimation of the drift parameter was considered (see \cite{Chapitre 1}).
In general situations, our contrast function is not explicit and in practise one has to resort to some approximation. We propose explicit approximations of the
 contrast function, such that the estimation of $\theta$ is
 feasible under the condition that 
 $n\Delta_n^k \to 0$ where $k>0$ can be arbitrarily large.
 This extends the results obtained by Kessler \cite{Kessler} in the case of continuous processes. 

\end{abstract}

\section{Introduction}
Recently, diffusion processes with jumps are becoming 
powerful tools to model various stochastic phenomena in many areas, for
example, physics, biology, medical sciences, social sciences, economics, and so
on. In finance, jump-processes were introduced to model the dynamic of exchange rates (\cite{Bates96}), asset prices (\cite{Merton76},\cite{Kou02}), or volatility processes (\cite{BarShe01},\cite{EraJohPol03}). Utilization of jump-processes in neuroscience, instead, can be found for instance in \cite{DitGre13}. Therefore, inference problems for such models from various types of data should be studied, in particular, inference from discrete observation should be desired since the actual data may be obtained discretely. \\
In this work, our aim is to estimate jointly the drift and the volatility parameter $(\mu, \sigma)= : \theta$ from a discrete sampling of the process $X^\theta$ solution to
\begin{equation*}
X_t^\theta=X_0^\theta+ \int_0^t b(\mu,X_s^\theta) ds + \int_0^t a(\sigma, X_s^\theta) dW_s + 
\int_0^t \int_{\mathbb{R}\backslash \left \{0 \right \}} \gamma(X_{s-}^\theta) z \tilde{\mu}(ds,dz),
\end{equation*}
where $W$ is a one dimensional Brownian motion and $\tilde{\mu}$ a compensated Poisson random measure, with a finite jump activity. 
We assume that the process is sampled at the times
$(t^n_i)_{i=0,\dots,n}$ where the sampling step $\Delta_n:=\sup_{i=0,\dots,n-1} t^n_{i+1}-t^n_i$ goes to zero. A crucial point for applications in the high frequency setting is to impose minimal
conditions on the sampling step size. This will be one of our main objectives in this
paper, for the joint estimation of $\mu$ and $\sigma$.\\
It is known that, as a consequence of the presence of a Gaussian component, it is impossible to estimate the drift parameter on a finite horizon time; we therefore assume that $t_n^n \to \infty$ and we suppose to have an ergodic process $X^\theta$.

 The topic of high frequency estimation for discretely observed diffusions in the case without jumps is well developed, by now. Florens-Zmirou has introduced, in \cite{Florens Zmirou}, an estimator for both the drift and the diffusion parameters under the fast sampling assumption $n\Delta_n^2 \to 0$. Yoshida \cite{Yoshida92} has then suggested a correction of the contrast function of \cite{Florens Zmirou} that releases the condition on the step discretization to $n \Delta_n^3 \to 0$.
In Kessler \cite{Kessler}, the author proposes an explicit modification of the Euler scheme contrast that allows him to build an estimator which is asymptotically normal under the condition $n\Delta_n^k \to 0$ where $k \ge 2$ is arbitrarily large. The result found by Kessler, therefore, holds for any arbitrarily slow polynomial decay to zero of the sampling step. 

When a jump component is added, less results are known. Shimizu \cite{Shimizu_SISP06} proposes parametric estimation of drift, diffusion and jump coefficients showing the asymptotic normality of the estimators under some explicit conditions relating the sampling step and the jump intensity of the process; such conditions on $\Delta_n$ are more restrictive as the intensity of jumps near zero is high. In the situation where the jump intensity is finite, the conditions of \cite{Shimizu_SISP06} reduces to $n \Delta_n^2 \to 0$. In \cite{GLM}, the condition on the sampling step is relaxed to $n \Delta_n^3 \to 0$, when one estimates the drift parameter only. In \cite{Mies} a jump-filtering technique similar to one used in \cite{GLM} is employed in order to derive a nonparametric estimator for the drift which is robust to symmetric jumps of infinite variance and infinite variation, and which attains the same asymptotic variance as for a continuous diffusion process.
Also in \cite{Chapitre 1} only the estimation of the drift parameter is studied. In such a case, the sampling step $(t_i^n)_{i=0,\dots,n}$ can be irregular, no condition on the rate at which $\Delta_n \to 0$ is needed and the assumption that the jumps of the process are summable, present in \cite{GLM}, is suppressed.

In this paper, we consider the joint estimation of the drift and the diffusion parameters with a jump intensity which is finite. Since for the applications it is important that assumptions on the rate at which $\Delta_n$ should tend to zero are less stringent as possible, our aim is to weaken the conditions on the decay of the sampling step in a way comparable to Kessler's work \cite{Kessler}, but in the framework of jump-diffusion processes. We therefore want to extend \cite{Chapitre 1} looking for the same results, but for the joint estimation of the drift and the diffusion parameters instead of focusing on the drift parameter only.\\
The joint estimation of the two parameters introduces some significant difficulties: since the drift and the volatility parameters are not estimated at the same rate, we have to deal with asymptotic properties in two different regimes. \\
Compared to previous results in which the parameters are estimated jointly (see \cite{Shimizu_SISP06}), we show that it is possible to remove any condition on the rate at which $\Delta_n$ has to go to zero. \\
Moreover, we consider a discretization step which is not uniform. This case, to our knowledge, has never been studied before for the joint estimation of the drift and the volatility of a diffusion with jumps.

A natural approach to estimate the unknown parameters would be to use a maximum
likelihood estimation, but the likelihood function based on the discrete sample is
not tractable in this setting, since it depends on the transition densities of X which are
not explicitly known. 
 \\
To overcome this difficulty several methods have been developed. For instance, in \cite{AitYu06} and \cite{LiChen16} closed form expansions of the transition density of jump-diffusions are studied while in \cite{JakSor17} the asymptotic behaviour of estimating functions is considered in the high frequency observation framework. They give condition to ensure the rate optimality and the efficiency.

Considering again the case of high frequency observation, a widely-used method is to consider pseudo-likelihood function, for instance based on the high frequency approximation of the dynamic of the process by the dynamic of the Euler scheme. This leads to explicit contrast functions with Gaussian structures (see e.g. \cite{Shimizu},\cite{Shimizu_SISP06},\cite{Masuda_AOS13}). 

In Kessler's paper the idea is to replace, in the Euler scheme contrast function, the contribution of the drift and the diffusion by two quantities $m$ and $m_2$ (or their explicit approximations with arbitrarily high order when $\Delta_n \to 0$); with
\begin{equation}
m (\mu, \sigma, x) := \mathbb{E}[X^\theta_{t_{i+1}}|X^\theta_{t_{i}}=x] \quad \mbox{and}
\label{eq: m kessler}
\end{equation}
$$m_2(\mu, \sigma, x) := \mathbb{E}[(X^\theta_{t_{i+1}}- m (\mu, \sigma, X^\theta_{t_{i}}))^2 |X^\theta_{t_{i}}=x].$$

In presence of jumps, the contrasts functions in \cite{Shimizu} (see also \cite{Shimizu_SISP06}, \cite{GLM}) resort to a filtering procedure in order to suppress the contribution of jumps and recover the continuous part of the process.

The contrast function we introduce is based on both the ideas described here above. Indeed, we define it as 
\begin{equation}
U_n(\mu, \sigma):= \sum_{i =0}^{n -1} [\frac{(X_{t_{i+1}} - m (\mu, \sigma, X_{t_i}))^2}{m_2(\mu, \sigma, X_{t_i})} + \log(\frac{m_2(\mu, \sigma, X_{t_i})}{\Delta_{n,i}})] \varphi_{\Delta_{n,i}^\beta}(X_{t_{i+1}} - X_{t_i})1_{\left \{|X_{t_i}| \le \Delta_{n,i}^{-k} \right \} },
\label{eq: contrast intro}
\end{equation}
where the function $\varphi$ is a smooth version of the indicator function that vanishes when the 
increments of the data are too large compared to the typical increments of a continuous diffusion process, and thus can be used to filter the contribution of the jumps. The idea is to use the size of $X_{t_{i+1}} - X_{t_i}$ in order to judge if a jump occurred or not in the interval $[t_i, t_{i + 1})$. The increment of $X$ with continuous transition could hardly exceed the threshold $\Delta_{n,i}^\beta$, therefore we can judge the existence of a jump in the interval if $|X_{t_{i+1}} - X_{t_i}| > \Delta_{n,i}^\beta $, for $\beta \in (\frac{1}{4}, \frac{1}{2})$. \\ 
The last indicator in \eqref{eq: contrast intro} avoids the possibility that $| X_{t_i}|$ is too big, the constant $k$ is positive and will be chosen later, related to the developments of $m$ and $m_2$ that are the natural extension to the case with jumps of the quantities proposed in \cite{Kessler}. Indeed, we have defined them as
$$m (\mu, \sigma,x): =\frac{\mathbb{E}[X_{t_{i+1}}^\theta \varphi_{\Delta_{n,i}^\beta}(X_{t_{i+1}}^\theta - X_{t_i}^\theta)|X_{t_i}^\theta = x]}{\mathbb{E}
[\varphi_{\Delta_{n,i}^\beta}(X_{t_{i+1}}^\theta - X_{t_i}^\theta)|X_{t_i}^\theta = x]} \quad \mbox{and} $$
$$ m_2 (\mu, \sigma,x): =\frac{\mathbb{E}[(X_{t_{i+1}}^\theta - m (\mu, \sigma,X_{t_i}))^2 \varphi_{\Delta_{n,i}^\beta}(X_{t_{i+1}}^\theta - X_{t_i}^\theta)|X_{t_i}^\theta = x]}{\mathbb{E}
[\varphi_{\Delta_{n,i}^\beta}(X_{t_{i+1}}^\theta - X_{t_i}^\theta)|X_{t_i}^\theta = x]}.$$

The rates for the estimation of the two parameters is not the same, which implies we have to deal with two different scaling of the contrast function, which lead us to the study of the asymptotic properties of the contrast function in two different asymptotic schemes.

The main result of our paper is that the estimator $\hat{\theta}_n := (\hat{\mu}_n, \hat{\sigma}_n)$ associated to the proposed contrast function converges with some explicit asymptotic variances.
Comparing to earlier results (\cite{Shimizu}, \cite{Shimizu_SISP06}, \cite{GLM}, \cite{Chapitre 1}), the sampling step $(t_i^n)_{i=0,\dots,n}$ can be irregular, no condition is needed on the rate at which $\Delta_n \to 0$ and the parameters of drift and diffusion are jointly estimated.

Moreover, we provide explicit approximations of $m_2$ that allows us to circumvent the fact that the contrast function is non explicit (explicit approximations of $m$ are given in \cite{Chapitre 1}). We give an expansion of $m_2$ exact up to order $\Delta^2_n$, which involves the jump intensity near zero, and is valid for any smooth truncation function $\varphi$.
 With the specific choices of $\varphi$ being
 oscillating functions,  
in particular, we remove the contribution of the jumps and we  are able to prove  explicit developments of the function $m_2$  valid up to any order. Together with the approximation of the function $m$ showed in Proposition 2 of \cite{Chapitre 1},  this allows us to approximate our contrast function, at arbitrary high order, by a completely explicit one, as it was in the paper by Kessler \cite{Kessler} in the continuous case. \\
This yields to a consistent and asymptotic normal estimator under the condition
$n \Delta_n^{k} \to 0$, where $k$ is related to the oscillating properties of the function $\varphi$. 
As $k$ can be chosen arbitrarily high, up to a proper choice of $\varphi$, our method allows to estimate the drift and the diffusion parameters, under the assumption that the sampling step tends to zero at some polynomial rate. 

Furthermore, we implement numerically our main results building two approximations of $m$ and $m_2$ from which we deduce two approximations of the contrast that we minimize in order to get the joint estimator of the parameters. We compare the estimators we find with the estimator that would result from the choice of an Euler scheme approximation for $m$ and $m_2$. From our simulations it appears that our joint estimator performs better than the Euler one, especially for the estimation of the parameter $\sigma$. 

The outline of the paper is the following. In Section \ref{S:Model} we introduce the model and we state the assumptions we need. The Section \ref{S:Construction_and_main} contains the construction of the estimator and our main results while in Section \ref{S:practical}
 we explain how to use in practical the contrast function for the joint estimation of the drift and the diffusion parameters, dealing with its approximation. We provide numerical results in Section \ref{S: Applications}. {\modch In Section \ref{S: Perspectives} we discuss about possible forward generalization of the obtained results, while in Section \ref{S: Limit theorems} we state useful propositions that we will use repeatedly in the following sections.} Section \ref{S:Proof_Main} is devoted to the proof of our main results while in Section \ref{S: Proof_limit} of the Appendix we prove the propositions stated in the sixth section. We conclude giving the proofs of some technical results in the Sections \ref{Ss:proof_der_m_m_2}--\ref{Ss:dev_m_2} of the Appendix.  
 
\section{Model, assumptions}\label{S:Model} 
We want to estimate the unknown parameter $\theta = (\mu, \sigma)$ in the stochastic differential equation with jumps 
\begin{equation}
X_t^\theta= X_0^\theta + \int_0^t b(\mu, X_s^\theta)ds + \int_0^t a(\sigma,X_s^\theta)dW_s + \int_0^t \int_{\mathbb{R} \backslash \left \{0 \right \} }
\gamma(X_{s^-}^\theta)z \tilde{\mu}(ds,dz), \quad t \in \mathbb{R}_+, 
\label{eq: model}
\end{equation}
where $\theta$ belongs to $\Theta := \Pi \times \Sigma$, a compact set of $\mathbb{R}^2$; $W=(W_t)_{t \ge 0}$ is a one dimensional Brownian motion and $\mu$ is a Poisson random measure associated to the L\'evy process $L=(L_t)_{t \ge 0}$ such that $L_t:= \int_0^t \int_\mathbb{R} z \tilde{\mu} (ds, dz)$. The compensated measure $\tilde{\mu}= \mu - \bar{\mu}$ is defined on $[0, \infty) \times \mathbb{R}$, the compensator is $\bar{\mu}(dt,dz): = F(dz) dt $, where conditions on the Levy measure $F$ will be given later. \\
We denote by $(\Omega, \mathcal{F}, \mathbb{P})$ the probability space on which $W$ and $\mu$ are defined and we assume that the initial condition $X_0^\theta$, $W$ and $L$ are independent. 

\subsection{Assumptions}
We suppose that the functions $b: \Pi \times \mathbb{R} \rightarrow \mathbb{R}$, $a : \Sigma \times \mathbb{R} \rightarrow \mathbb{R}$ and $\gamma: \mathbb{R} \rightarrow \mathbb{R}$ satisfy the following assumptions: \\
\\
\textbf{A1}: \textit{The functions $\gamma(x)$, $b(\mu, x)$ for all $\mu \in \Pi$ and $a(\sigma, x)$ for all $\sigma \in \Sigma$ are globally Lipschitz. Moreover, the Lipschitz constants of $b$ and $a$ are uniformly bounded on $\Pi$ and $\Sigma$, respectively.} \\
\\
Under Assumption 1 the equation (\ref{eq: model}) admits a unique non-explosive c\`adl\`ag adapted solution possessing the strong Markov property, cf \cite{Applebaum} (Theorems 6.2.9. and 6.4.6.). \\
The next assumption was used in \cite{18 GLM} to prove the irreducibility of the process $X^\theta$. \\
\\
\textbf{A2}: \textit{For all $\theta \in \Theta$ there exists a constant $t > 0$ such that $X_t^\theta$ admits a density $p_t^\theta(x,y)$ with respect to the Lebesgue measure on $\mathbb{R}$; bounded in $y \in \mathbb{R}$ and in $x \in K$ for every compact $K \subset \mathbb{R}$. Moreover, for every $x \in \mathbb{R}$ and every open ball $U \in \mathbb{R}$, there exists a point $z = z(x, U) \in supp(F) $ such that $\gamma(x)z \in U$.} \\
\\
Assumption 2 ensures, together with the Assumption 3 below, the existence of unique invariant distribution $\pi^\theta$, as well as the ergodicity of the process $X^\theta$. \\
\\
\textbf{A3 (Ergodicity)}: \textit{ (i) For all $q >0$, $\int_{|z|> 1} |z|^q F(z) dz < \infty$. \\
(ii) For all $\mu \in \Pi$ there exists $C > 0$ such that $xb(\mu,x)\le -C|x|^2$, if $|x| \rightarrow \infty$. \\
(iii)  $|\gamma(x)| / |x| \rightarrow 0$ as $|x|\rightarrow \infty$. \\
(iv) For all $\sigma \in \Sigma$ we have $|a(\sigma,x)| / |x| \rightarrow 0$ as $|x|\rightarrow \infty$. \\
(v) $\forall \theta \in \Theta$, $\forall q > 0$ we have $\mathbb{E}|X_0^\theta|^q < \infty$. } \\
\\
\textbf{A4 (Jumps)}: \textit{ 1. The jump coefficient $\gamma$ is bounded from below, that is $\inf_{x \in \mathbb{R}}|\gamma(x)|:= \gamma_{min} >0$. \\
2. The L\'evy measure $F$ is absolutely continuous with respect to the Lebesgue measure and we denote $F(z) = \frac{F(dz)}{dz}$. \\
3. $F$ is such that  $F(z) = \lambda F_0(z)$ and $\int_\mathbb{R} F_0(z)dz= 1$. } \\
\\
Assumption 4.1 is useful to compare size of jumps of $X$ and $L$. The Assumption $5$ ensures the existence of the contrast function we will define in next section.\\
\\
\textbf{A5 (Non-degeneracy)}: \textit{There exists some $c > 0$, such that $\inf_{x, \sigma}a^2(\sigma, x) \ge  c > 0$.}\\
\\
From now on we denote the true parameter value by $\theta_0$, an interior point of the parameter space $\Theta$ that we want to estimate. We shorten $X$ for $X^{\theta_0}$. \\
\\
We will use some moment inequalities for jump diffusions, gathered in the following lemma that follows from Theorem 66 of \cite{Protter GLM} and Proposition 3.1 in \cite{Shimizu}. 
\begin{lemma}
Let $X$ satisfies Assumptions 1-4. Let $L_t:= \int_0^t \int_\mathbb{R} z \tilde{\mu}(ds, dz)$ and let $\mathcal{F}_s := \sigma \left \{ (W_u)_{0 < u \le s}, (L_u)_{0 < u \le s}, X_0 \right \}$. \\
Then, for all $t > s > 0$, \\
1) for all $p \ge 2$, $\mathbb{E}[|X_t - X_s|^p]^\frac{1}{p} \le c |t-s|^\frac{1}{p}$, \\
2) for all $p \ge 2$, $p \in \mathbb{N}$, $\mathbb{E}[|X_t - X_s|^p|\mathcal{F}_s] \le c|t-s|(1 + |X_s|^p)$, \\
3) for all $p \ge 2$, $p \in \mathbb{N}$, $\sup_{h \in [0,1]} \mathbb{E}[|X_{s+h}|^p|\mathcal{F}_s] \le c(1 + |X_s|^p)$.\\
\label{lemma: Moment inequalities}
\end{lemma}
An important role is played by ergodic properties of solution of equation (\ref{eq: model}) \\
The following Lemma states that Assumptions $1 - 4$ are sufficient for the existence of an invariant measure $\pi^\theta$ such that an ergodic theorem holds and moments of all order exist. 
\begin{lemma}
Under assumptions 1 to 4,
for all $\theta \in \Theta$, $X^\theta$ admits a unique invariant distribution $\pi^\theta$ and the ergodic theorem holds: \\
1)For every measurable function $g: \mathbb{R} \rightarrow \mathbb{R}$ s. t. $\pi^\theta(g) < \infty$, we have a.s. $\lim_{t \rightarrow \infty} \frac{1}{t} \int_0^t g(X_s^\theta) ds = \pi^\theta (g).$ \\
2)For all $q > 0$, $\pi^\theta(|x|^q) < \infty $. \\
3) For all $q >0$, $\sup_{t \ge 0} \mathbb{E}[|X_t^\theta|^q] < \infty$.
\label{lemma: 2.1 GLM}
\end{lemma}
A proof is in \cite{GLM} (Section 8 of Supplement), it relies mainly on results of \cite{18 GLM}. \\ \\
\textbf{A6 (Identifiability)}: \textit{For all $\mu_1, \mu_2$ in $\Pi$, $\mu_1 = \mu_2$ if and only if $b(\mu_1, x) = b(\mu_2, x)$ for almost all $x$. Moreover, $\forall \sigma_1, \sigma_2$ in $\Sigma$, $\sigma_1=\sigma_2$ if and only if $a(\sigma_1, x) =a(\sigma_2, x)$ for almost all $x$.} \\
\\
\textbf{A7}: \textit{1. The derivatives $\frac{\partial^{k_1 + k_2} b}{ \partial x^{k_1} \partial \theta^{k_2}}$, with $k_1 +k_2 \le 4$ and $k_2 \le 3$, exist and they are bounded if $k_1 \ge 1$. If $k_1 = 0$, for each $k_2 \le 3$ they have polynomial growth. \\
2. The derivatives $\frac{\partial^{k_1 + k_2} a}{ \partial x^{k_1} \partial \theta^{k_2}}$, with $k_1 +k_2 \le 4$ and $k_2 \le 3$, exist and they are bounded if $k_1 \ge 1$. If $k_1 = 0$, for each $k_2 \le 3$ they have polynomial growth. \\
3. The derivatives $\gamma^{(k)}(x)$ exist and they are bounded for each $ 1 \le k \le 4$.} \\
\\
\textbf{A8}: \textit{Let $B$ be $\begin{pmatrix} 
-2 \int_\mathbb{R} (\frac{\partial_\mu b(x, \mu_0)}{a(x, \sigma_0)})^2 \pi(dx) & 0 \\
0 & 4 \int_\mathbb{R} (\frac{\partial_\sigma a(x, \sigma_0)}{a(x, \sigma_0)})^2 \pi(dx) 
\end{pmatrix}$, then $det(B) \neq 0$.} \\
\\

\section{Construction of the estimator and main results}\label{S:Construction_and_main}
Now we present a contrast function for estimating parameters.
\subsection{Construction of contrast function.}
Suppose that we observe a finite sample 
$X_{t_0},  ... , X_{t_n}$ with $0=t_0 \le t_1 \le ... \le t_n= : T_n$,
where $X$ is the solution to \eqref{eq: model} with $\theta = \theta_0$.  
Every observation time point depends also on $n$, but to simplify the notation we suppress this index. We will be working in a high-frequency setting, i.e.
$\Delta_n := \sup_{i =0, ... , n-1} \Delta_{n,i}\longrightarrow 0$ for $n \rightarrow \infty$, with $\Delta_{n,i}: = (t_{i+1} - t_i) $.
We assume that $\lim_{n \rightarrow \infty} T_n = \infty$. \\
In the sequel we will always suppose that the following assumption on the step discretization holds true. \\
\textbf{$A_{\mbox{Step}}$:} there exist two constants $c_1$, $c_2$ such that $c_2 < \frac{\Delta_n}{\Delta_{min}} < c_1$, where we have denoted $\Delta_{min}$ as $\min_{i =0, ... , n-1} \Delta_{n,i} $. \\
{\modch We introduce a version of the gaussian quasi -likelihood in which we have filtered the contribution of the jumps. As in our framework the data is observed discretely, to formulate a criterion to decide if a jump occurred in a particular interval or not we can only use the increment of our process. Such a criterion should depends on $n$ and should be more and more accurate as $n$ tends to infinity. This leads to the following contrast function: }
\begin{definition} \label{D:definition_contrast}
For $\beta \in (\frac{1}{4}, \frac{1}{2})$ we define the contrast function $U_n(\mu, \sigma)$ as
\begin{equation}
U_n(\mu, \sigma):= \sum_{i =0}^{n -1} [\frac{(X_{t_{i+1}} - m (\mu, \sigma, X_{t_i}))^2}{m_2(\mu, \sigma, X_{t_i})} + \log(\frac{m_2(\mu, \sigma, X_{t_i})}{\Delta_{n,i}})] \varphi_{\Delta_{n,i}^\beta}(X_{t_{i+1}} - X_{t_i})1_{\left \{|X_{t_i}| \le \Delta_{n,i}^{-k} \right \} },
\label{eq: contrast function}
\end{equation}
\end{definition}
with
\begin{equation}
m (\mu, \sigma,x): =\frac{\mathbb{E}[X_{t_{i+1}}^\theta \varphi_{\Delta_{n,i}^\beta}(X_{t_{i+1}}^\theta - X_{t_i}^\theta)|X_{t_i}^\theta = x]}{\mathbb{E}
[\varphi_{\Delta_{n,i}^\beta}(X_{t_{i+1}}^\theta - X_{t_i}^\theta)|X_{t_i}^\theta = x]};
\label{eq: definition m}
\end{equation}
\begin{equation}
m_2 (\mu, \sigma,x): =\frac{\mathbb{E}[(X_{t_{i+1}}^\theta - m (\mu, \sigma,X_{t_i}))^2 \varphi_{\Delta_{n,i}^\beta}(X_{t_{i+1}}^\theta - X_{t_i}^\theta)|X_{t_i}^\theta = x]}{\mathbb{E}
[\varphi_{\Delta_{n,i}^\beta}(X_{t_{i+1}}^\theta - X_{t_i}^\theta)|X_{t_i}^\theta = x]}
\label{eq: definition m2}
\end{equation}
and 
$$\varphi_{\Delta_{n,i}^\beta}(X_{t_{i+1}} - X_{t_i}) = \varphi( \frac{X_{t_{i+1}} - X_{t_i}}{\Delta_{n,i}^\beta}).$$ 
The function $\varphi$ is a smooth version of the indicator function, such that
$\varphi(\zeta) = 0$ for each $ \zeta$, with $|\zeta| \ge 2$ and $\varphi(\zeta) = 1$ for each $ \zeta $, with $ |\zeta| \le 1$. \\
The last indicator aims to avoid the possibility that $|X_{t_i}|$ is big. The constant $k$ is positive and it will be chosen later, related to the development of both $m$ and $m_2$. \\
Moreover we remark that $m$ and $m_2$ depend also on $t_i$ and $t_{i+1}$. By the homogeneity of the equation they actually depend on the difference $t_{i+1} - t_i$ but we omit such a dependence in the notation of the two functions here above to make the reading easier. \\
We define an estimator $\hat{\theta}_n$ of $\theta_0$ as 
\begin{equation}
\hat{\theta}_n = (\hat{\mu}_n, \hat{\sigma}_n) \in \mbox{arg}\min_{(\mu, \sigma) \in \Theta} U_n(\mu, \sigma).
\label{eq: def thetan}
\end{equation}
{\modch The idea is to use the size of the increment of the process $X_{t_{i+1}} - X_{t_i}$ in order to judge if a jump occurred or not in the interval $[t_i, t_{i + 1})$.
As it is hard for the increment of $X$ with continuous transition to overcome the threshold $\Delta_{n,i}^\beta$ for $\beta \le \frac{1}{2}$, we can assert the presence of a jump in $[t_i, t_{i + 1})$ if $|X_{t_{i+1}} - X_{t_i}| > \Delta_{n,i}^\beta $. \\
It is worth noting that if $\beta$ is too large, and therefore $\Delta_{n,i}^\beta$ is too small, we can't ignore the probability to have continuous diffusion reaching the threshold $\Delta_{n,i}^\beta$. On the other side, if $\beta$ is too small and therefore $\Delta_{n,i}^\beta$ is too big, it is possible to have an increment smaller than $\Delta_{n,i}^\beta$ even if a jump occurs in the considered interval. That's the reason why the threshold $\beta$ has to be chosen with great care. \\}
In the definition of the contrast function we have taken $\beta > \frac{1}{4}$ because, in Lemma \ref{lemma: conditional expected value} and Proposition \ref{prop: LT} below (and so, as a consequence, in the majority of the theorems of this work), such a technical condition on $\beta$ is required.  \\
We observe that, in general, there is no closed expression for $m$ and $m_2$, hence the contrast is not explicit.  However, it is proved in \cite{Chapitre 1} an explicit development of $m$ in the case where the intensity is finite and in this work we provide as well an explicit development of $m_2$ that lead us to an explicit version of our contrast function. \\
\subsection{Main results}
Before stating our main results, we need some further notations and assumptions.
We introduce the function $R$ defined as follows: for $\delta \ge 0$, we will denote $R(\theta, \Delta_{n,i}^\delta, x)$ for any function $R(\theta, \Delta_{n,i}^\delta, x)= R_{i,n}(\theta, x)$, where
$R_{i,n}: \Theta \times \mathbb{R} \longrightarrow \mathbb{R}$, $(\theta, x) \mapsto R_{i,n}(\theta, x) $ is such that 
\begin{equation}
\exists c > 0 \qquad |R_{i,n}(\theta,x)| \le c(1 + |x|^c)\Delta_{n,i}^\delta
\label{eq: definition R}
\end{equation}
uniformly in $\theta$ and with $c$ independent of $i,n$.  \\
The functions $R$ represent the term of rest and have the following useful property, consequence of the just given definition:
\begin{equation}
R(\theta, \Delta_{n,i}^\delta, x)= \Delta_{n,i}^\delta R(\theta, \Delta_{n,i}^0, x).
\label{propriety power R}
\end{equation}
We point out that it does not involve the linearity of $R$, since the functions $R$ on the left and on the right side are not necessarily the same but only two functions on which the control (\ref{eq: definition R}) holds with $\Delta_{n,i}^\delta $ and $\Delta_{n,i}^0$, respectively. \\ \\
In the sequel, we will need a development for the function $m_2$. We will assume that such a development exists, as stated in the next assumption: \\
\\
\textbf{Ad:} There exist three functions $r(\mu, \sigma, x)$, $r(x)$,  $R(\theta, 1, x)$ and $\delta_1, \delta_2 > 0$ and $k_0 > 0$ such that, for $|x| \le \Delta_{n,i}^{- k_0} $,
\begin{equation}
m_2(\mu, \sigma, x) = \Delta_{n,i} a^2(x, \sigma)(1 + \Delta_{n,i} r(\mu, \sigma, x))+ \Delta_{n,i}^{1 + \delta_1} r(x) + \Delta_{n,i}^{2 + \delta_2} R(\theta, 1, x), 
\label{eq: hp dl m2}
\end{equation}
where $r(\mu, \sigma, x)$ and $r(x)$ are particular functions $R(\theta, 1, x)$, that turns out from the development of $m_2$,
 and the function $r(x)$ does not depend on $\theta$. Moreover, the order of such functions does not change by deriving them with respect to both the parameters, that is for $\vartheta = \mu$ and $\vartheta = \sigma$, $ | {\modch \frac{\partial}{\partial \vartheta}} r(\mu, \sigma, x)| \le c(1 + |x|^c)$ and $|{\modch \frac{\partial}{\partial \vartheta}} R(\theta, 1, x)|\le c(1 + |x|^c)$. \\
\\
Assumption Ad is not restrictive. Examples of frameworks in which Ad holds are introduced in Propositions \ref{prop: dl m2 intensita finita} and \ref{prop: dl alla capitolo 1 m2}, that will be stated in the next section and proven in the appendix.
Let us stress that it is crucial for the proof of the consistency of the estimator that  the second main term of the expansion \eqref{eq: hp dl m2}, $\Delta_{n,i}^{1+\delta_1} r(x)$, does not depend on the parameter $\theta$.

The following theorems give a general consistency result and the asymptotic normality of the estimator $\hat{\theta}_n$.

\begin{theorem}{(Consistency)}
Suppose that Assumptions 1 to 7, $A_{\mbox{Step}}$ and Ad hold. Then the estimator $\hat{\theta}_n$ is consistent in probability:
$$\hat{\theta}_n \xrightarrow{\mathbb{P}} \theta_0, \qquad n \rightarrow \infty.$$
\label{th: consistency}
\end{theorem}

\begin{theorem}{(Asymptotic normality)}
Suppose that Assumptions 1 to 8, $A_{\mbox{Step}}$ and Ad hold. Then
$$(\sqrt{T_n}(\hat{\mu}_n - \mu_0), \sqrt{n}(\hat{\sigma}_n - \sigma_0) ) \xrightarrow{\mathcal{L}} N(0, K) \quad \mbox{for } n\rightarrow\infty,$$
where $K = \begin{pmatrix} 
( \int_\mathbb{R} (\frac{\partial_\mu b(x, \mu_0)}{a(x, \sigma_0)})^2 \pi(dx))^{-1} & 0 \\
0 & 2( \int_\mathbb{R} (\frac{\partial_\sigma a(x, \sigma_0)}{a(x, \sigma_0)})^2 \pi(dx))^{-1} 
\end{pmatrix}$.
\label{th: normality}
\end{theorem}
The proof of our main results will be presented in Section \ref{S:Proof_Main}. \\
\\
{\modch It is worth remarking here that, when $\sigma$ is known and only the parametric estimation of the drift is considered, the model \eqref{eq: model} is LAN with Fisher information $I(\mu)= \int_\mathbb{R} \frac{(\dot{b}(\theta,x))^2}{a^2(x)} \pi^\mu (dx)$ (see \cite{GLM}). \\
The H\'ajek$-$LeCam convolution theorem states that any regular estimator in a parametric model which satisfies LAN property is asymptotically equivalent to a sum of two independent random variables, one of which is normal with asymptotic variance equal to the inverse of Fisher information, and the other having arbitrary distribution. The efficient estimators are those with the second component identically equal to zero. Therefore, the estimator $\hat{\mu}_n$ is asymptotically efficient in the sense of the H\'ajek-Le Cam convolution theorem when the volatility parameter $\sigma$ is supposed to be known and also without its knowledge. \\
In the general case we are not aware of any result which proves that our model \eqref{eq: model} is LAN and therefore nothing can be said about the efficiency of the estimator $\hat{\theta}_n$. However, Gobet has proved in \cite{Gobet 2002} that, in absence of jumps, the LAN property of the model is satisfied with a Fisher information which matches with the variance matrix we found, in presence of jumps, for the joint estimation of the drift and the volatility parameters. \\
Such a result gives hope for an eventual efficiency of the estimator $\hat{\theta}_n$ here proposed.}

\section{Practical implementation of the contrast method}\label{S:practical}
In order to use in practice the contrast function \eqref{eq: contrast function}, one need to know the values of the quantities $m(\mu, \sigma, X_{t_i})$ and $m_2(\mu, \sigma, X_{t_i})$.
Even if in most cases, it seems impossible to find an explicit expression for them, explicit or numerical approximations of this functions seem available in many situations.

\subsection{Approximation of the contrast function}
Let us assume that one has at disposal an approximation of the functions $m(\mu, \sigma, x)$ and $m_2(\mu, \sigma, x)$, denoted by $\tilde{m}(\mu, \sigma, x)$ and $\tilde{m}_2(\mu, \sigma, x)$ which satisfy, for $|x|\le {\modch \Delta_{n,i} }^{-k_0}$, the following assumptions. \\
\textbf{$A\rho$ :}
\begin{enumerate}
    \item $| \tilde{m}(\mu, \sigma, x) -m(\mu, \sigma, x) | \leq R(\theta,{\modch \Delta_{n,i} }^{\rho_1},x), \quad | \tilde{m}_2(\mu, \sigma, x) -m_2(\mu, \sigma, x) | \leq R(\theta,{\modch \Delta_{n,i} }^{\rho_2},x),$ \\
where the constants $\rho_1>1$ and $\rho_2 >1$  assess the quality of the approximation.  
\item $| {\modch \frac{\partial^i}{\partial \mu^i}} \tilde{m}(\mu, \sigma, x)   - {\modch \frac{\partial^i}{\partial \mu^i}} m(\mu, \sigma, x) | + | {\modch \frac{\partial^i}{\partial \sigma^i}} \tilde{m}_2(\mu, \sigma, x)   - {\modch \frac{\partial^i}{\partial \sigma^i}} m_2(\mu, \sigma, x) | \le R(\theta,{\modch \Delta_{n,i} }^{1+\epsilon},x), \quad \text{for $i=1,2$},$
for all $|x|\le {\modch \Delta_{n,i} }^{-k_0}$ and where $\epsilon>0$.
\item The bounds on the derivatives of $m$ and $m_2$ gathered in Propositions \ref{prop: dl derivate prime}, \ref{prop: derivate seconde} and \ref{prop: dervate terze} hold true for $\tilde{m}$ and $\tilde{m}_2$ replacing $m$ and $m_2$.
\end{enumerate}
We have to act on the derivatives of the two approximated functions $\tilde{m}$ and $\tilde{m}_2$ as we do on $m$ and $m_2$. That is the reason why we need to add the third technical assumption here above, which assure we can move from the derivatives of the real functions to the approximated ones committing an error which is negligible.
Now, we consider $\widetilde{\theta}_n$ the estimator obtained from minimization of the contrast function \eqref{eq: contrast function} where one has replaced the functions $m(\mu, \sigma, X_{t_i})$ and $m_2(\mu, \sigma, X_{t_i})$ by their approximations $\tilde{m}(\mu, \sigma, x)$ and $\tilde{m}_2(\mu, \sigma, x)$: 
{\modch 
$$\tilde{U}_n(\mu, \sigma):= \sum_{i =0}^{n -1} [\frac{(X_{t_{i+1}} - \tilde{m} (\mu, \sigma, X_{t_i}))^2}{\tilde{m}_2(\mu, \sigma, X_{t_i})} + \log(\frac{\tilde{m}_2(\mu, \sigma, X_{t_i})}{\Delta_{n,i}})] \varphi_{\Delta_{n,i}^\beta}(X_{t_{i+1}} - X_{t_i})1_{\left \{|X_{t_i}| \le \Delta_{n,i}^{-k} \right \} },$$
$$\widetilde{\theta}_n = (\tilde{\mu}_n, \tilde{\sigma}_n) \in \mbox{arg}\min_{(\mu, \sigma) \in \Theta} \tilde{U}_n(\mu, \sigma).$$
}

Then, the result of Theorem \ref{th: normality} can be extended as follows. \\

\begin{proposition}
\label{Prop:normality_approximate_contrast} 
Suppose that Assumptions 1 to 8, Ad, $A_{\mbox{Step}}$ and $A \rho$ hold, with  $0 < k < k_0$, and that $\sqrt{n}\Delta_n^{\rho_1 -1/2} \rightarrow  0$ and $\sqrt{n}\Delta_n^{\rho_2 -1} \rightarrow  0$ as $n \rightarrow \infty$. \\
Then, the estimator $\widetilde{\theta}_n := (\tilde{\mu}_n, \tilde{\sigma}_n)$ is asymptotically normal:
$$(\sqrt{T_n}(\tilde{\mu}_n - \mu_0), \sqrt{n}(\tilde{\sigma}_n - \sigma_0) ) \xrightarrow{\mathcal{L}} N(0, K) \quad \mbox{for } n\rightarrow\infty,$$
where $K$ is the matrix defined in Theorem \ref{th: normality}.
\end{proposition}
Proposition \ref{Prop:normality_approximate_contrast} will be proven in Section \ref{S: proof prop approx}. \\
\\
We give below several examples of approximations of $m_2(\mu, \sigma, X_{t_i})$ which can be used, together with the approximations of $m(\mu, \sigma, X_{t_i})$ given in Proposition 2 of \cite{Chapitre 1}, to construct an explicit contrast function.

\subsection{Development of $m_2(\mu, \sigma, x)$.}\label{S: dl m2}
We provide two kinds of expansion for the function $m_2$. First, we prove high order expansions that involve only the continuous part of the generator of the process and necessitate the choice of oscillating functions $\varphi$. Second, we find an expansion up to order $\Delta_n^2$ for any function $\varphi$, and, in particular, show the validity of the condition $Ad$ in a general setting.
 For completeness, we recall also the expansions of the function $m$ found in \cite{Chapitre 1}.
\subsubsection{Arbitrarily high expansion with oscillating truncation functions.}
 We show we can write an explicit development for the function $m_2$, as we did for the function $m$ in Proposition 2 of \cite{Chapitre 1}, taking a particular oscillating function $\varphi$. In this way, it is therefore possible to make the contrast explicit with approximation at any order. We define $A_{K_1}^{(k)}(x) := \bar{A}_c^k(h_1)(x)$ and $A_{K_2}^{(k)}(x) := \bar{A}_c^k(h_2)(x)$, where $\bar{A}_c(f) := \bar{b} f' + \frac{1}{2} a^2 f''$, with $\bar{b}(\mu, y) = b( \mu, y) - \int_{\mathbb{R}} \gamma (y) z F(z) dz$; $K_1$ and $K_2$ we have written here above stand for "Kessler", based on the fact that the development we find is the same obtained in \cite{Kessler} in the case without jumps by the iteration of the continuous generator $\bar{A}_c$. The functions who appear in the definition of $A_{K_1}^{(k)}$ and $A_{K_2}^{(k)}$ are the following: $h_1(y): = (y - x)$, $h_2(y) = y^2$. \\
To get Proposition \ref{prop: dl m2 intensita finita} we need to add the following assumption: \\ \\
\textbf{Af:} We assume that $x \mapsto a(x, \sigma)$, $x \mapsto b(x, \mu)$ and $x \mapsto \gamma(x)$ are $\mathcal{C}^\infty$ functions, they have at most uniform in $\mu$ and $\sigma$ polynomial growth as well as their derivatives.
\begin{proposition}
Assume that Assumptions 1-4 and Af hold and let $\varphi$ be a $\mathcal{C}^\infty$ function that has compact support and such that $\varphi \equiv 1$ on $[-1, 1]$ and $\forall k \in \left \{ 0, ... , M \right \}$, $\int_\mathbb{R} x^k \varphi(x) dx = 0$ for $M \ge 0$. Moreover we suppose that the L\'evy {\modch density} $F$ is $\mathcal{C}^\infty$. Then, for  $|x| \le \Delta_{n,i}^{-k_0} $ with some $k_0 > 0$,
\begin{equation}
m_2(\mu, \sigma,x) = \sum_{k = 1}^{\lfloor \beta (M + 2) \rfloor} A_{K_2}^{(k)}(x) \frac{\Delta_{n,i}^k}{k!} - (x + \sum_{k = 1}^{\lfloor \beta (M + 2) \rfloor} A_{K_1}^{(k)}(x) \frac{\Delta_{n,i}^k}{k!})^2 + R(\theta, \Delta_{n,i}^{\beta  (M+2) }, x).
\label{eq: dl m2 intensita finita}
\end{equation}
Moreover, for $\vartheta = \mu$ or $\vartheta = \sigma$
\begin{equation}
|{\modch \frac{\partial}{\partial \vartheta}} R(\theta,\Delta_{n,i}^{\beta  (M+2) }, x)| \le R(\theta,\Delta_{n,i}^{\beta  (M+2) }, x). 
\label{eq: tesi derivate int finita}
\end{equation}
\label{prop: dl m2 intensita finita}
\end{proposition}

It is proved below Proposition 2 in \cite{Chapitre 1} that a function $\varphi$ which satisfies the assumptions here above exists: it is possible to build it through $\psi$, a function with compact support, $\mathcal{C}^\infty$ and such that
$\psi|_{[-1, 1]}(x)= \frac{x^M}{M !}$. It is enough to define $\varphi(x) : = \frac{\partial^M}{\partial x^M} \psi (x)$ to get $\varphi \equiv 1$ on $[-1, 1]$; $\varphi$ is $\mathcal{C}^\infty$, with compact support and such that for each $l \in \left \{ 0, ... M \right \}$, using the integration by parts,
$\int_\mathbb{R} x^l \varphi(x) dx = 0$. \\
It is thanks to such a choice of an oscillating function $\varphi$ that the contribution of the discontinuous part of the generator disappears and we get the same development found in the continuous case, in Kessler \cite{Kessler}, due only to the continuous generator. \\ 
In the situation where \eqref{eq: dl m2 intensita finita} holds true with $\lfloor \beta (M + 2) \rfloor > 2$, we get a development for $m_2$ as in $Ad$ for $r(x)$ identically $0$ and $r(\mu, \sigma,x)$ being an explicit function.

For completeness, let us recall that under the same assumptions as in Proposition \ref{prop: dl m2 intensita finita},  we have the following expansion for $m$ (see Proposition 2 in \cite{Chapitre 1}):
\begin{equation}\label{E:dev_Kessler_chap1}
m(\mu,\sigma,x)= x + \sum_{k = 1}^{\lfloor \beta (M + 2) \rfloor} A_{K_1}^{(k)}(x) \frac{\Delta_{n,i}^k}{k!} + R(\theta, \Delta_{n,i}^{\beta  (M+2) }, x), \text{ for $|x| \le \Delta_{n,i}^{-k_0} $, with $k_0>0$.}
\end{equation}
{\modch In the definition of the contrast function we can replace the two functions $m$ and $m_2$ with their approximations obtained from \eqref{eq: dl m2 intensita finita} and \eqref{E:dev_Kessler_chap1} :
$$\tilde{m}^{\rho_1} (\mu, \sigma, x) := x + \sum_{k = 1}^{\rho_1} A_{K_1}^{(k)}(x) \frac{\Delta_{n,i}^k}{k!} \quad \mbox{and}$$
$$\tilde{m_2}^{\rho_2} (\mu, \sigma, x) := \sum_{k = 1}^{\rho_2} A_{K_2}^{(k)}(x) \frac{\Delta_{n,i}^k}{k!} - (x + \sum_{k = 1}^{\rho_2} A_{K_1}^{(k)}(x) \frac{\Delta_{n,i}^k}{k!})^2.$$
The errors we commit are, respectively, $R(\theta, \Delta_{n,i}^{\rho_1 }, x)$ and $R(\theta, \Delta_{n,i}^{\rho_2 }, x)$; for $\rho_i \le \lfloor \beta (M + 2) \rfloor$ and $\rho_i \ge 1$, $i = 1, 2$. \\
By application of the Proposition \ref{Prop:normality_approximate_contrast} we can see that the associated estimator is asymptotically gaussian under the assumptions $\sqrt{n}\Delta_n^{\rho_1 - \frac{1}{2}} \rightarrow 0 $ and $\sqrt{n}\Delta_n^{\rho_2 - 1} \rightarrow 0 $, for $n$ going to $\infty$. Since $M$ and thus $\rho_1$ and $\rho_2$ can be chosen arbitrarily large, we see that the sampling step $\Delta_n$ is allowed to go to zero in a arbitrary slowly polynomial rate as function of $n$. That means that a slow sampling step needs to choose a truncation function with more vanishing moments.
}

\subsubsection{Second order expansion with general truncation functions.}
Another situation in which $Ad$ holds is gathered in the following proposition, that will be still proven in the appendix:
\begin{proposition}
Suppose that Assumptions A1 -A5 and A7 hold, that $\beta \in ( \frac{1}{4}, \frac{1}{2})$ and that the L\'evy {\modch density} $F$ is $\mathcal{C}^1$. Then there exists $k_0 > 0$ such that, for $|x| \le \Delta_{n,i}^{- k_0} $, 
\begin{equation}
m_2(\mu, \sigma,x) = \Delta_{n,i} a^2(x, \sigma) + \frac{ \Delta_{n,i}^{1 + 3 \beta}}{\gamma(x)}F(0) \int_\mathbb{R} v^2 \varphi(v) dv + \Delta_{n,i}^2 (3 \bar{b}^2(x, \mu) + h_2(x, \theta)) + \Delta_{n,i}^{(1 + 4 \beta) \land (2 + \beta) \land (3 - 2 \beta)} R(\theta, 1,x);
\label{eq: dl particolare m2}
\end{equation}
where $h_2 = \frac{1}{2} a^2(a')^2 + \frac{1}{2} a^3a'' + a^2\bar{b}' + aa'\bar{b} + \bar{b}^2$.  
Moreover, for both $\vartheta = \mu$ and $\vartheta = \sigma$, ${\modch \frac{\partial}{\partial \vartheta}} R(\theta, 1, x)$ is still a $R(\theta, 1, x)$ function.
\label{prop: dl m2 particolare per appl}
\end{proposition}
We observe that, defining $r(x) := F(0) \int_\mathbb{R} v^2 \varphi(v)dv $ and  $r(\mu, \sigma,x) := \frac{3 \bar{b}^2(x, \mu) + h_2(x, \theta)}{a^2(x,\sigma)} $, we get a development as in $Ad$ for $\delta_1 := 3 \beta$. \\
We observe that if $\int_\mathbb{R} v^2 \varphi(v) dv = 0$, we fall back in development of Proposition \ref{prop: dl m2 intensita finita} up to order 2. We therefore see that the choice of an oscillating truncated function $\varphi$ is necessary in order to remove the jump contribution. 

	It is worth noting here that biggest term after the main one is due to the jump part and do not depend on the parameters $\mu$ and $\sigma$. We will see in the sequel that is necessary, in order to prove the consistency of $\hat{\mu}_n$, that this contribution does not depend on the drift parameter. Considering indeed the difference of the contrast computed for two different values of the drift parameter, its presence results irrelevant. 

 We remark that the term with order $1 + 4 \beta$ is negligible compared to the order $2$ terms since in our setting $\beta$ is assumed to be bigger than $\frac{1}{4}$.
	
In Proposition \ref{prop: dl m2 particolare per appl} before $F$ is required to be $\mathcal{C}^1$, such assumption is no longer needed in following more general proposition.	


\begin{proposition}
Suppose that Assumptions A1 -A5 and A7 hold. Then there exists $k_0 > 0$ such that, for $|x| \le \Delta_{n,i}^{- k_0} $, 
$$m_2(\mu, \sigma,x) = \Delta_{n,i} a^2(x, \sigma) +  \frac{\Delta_{n,i}^{1 + 3 \beta}}{\gamma(x)}
\int_\mathbb{R} u^2 \varphi(u) F(u \frac{\Delta_{n,i}^\beta}{\gamma(x)}) du + \Delta_{n,i}^2 (3 \bar{b}^2(x, \mu) + h_2(x, \theta)) +$$
\begin{equation}
+ \frac{\Delta_{n,i}^{2 + \beta}a^2(x, \sigma)}{2 \gamma(x)}\int_{\mathbb{R}} (u \varphi'(u) + u^2 \varphi''(u)) F(\frac{u\Delta_{n,i}^\beta}{\gamma(x)}) du 
+ \Delta_{n,i}^{(3 - 2\beta) \land (2 + \beta)}R(\theta,1 ,x),
\label{eq: dl alla capitolo 1 m2}
\end{equation}
where $h_2 = \frac{1}{2} a^2(a')^2 + \frac{1}{2} a^3a'' + a^2\bar{b}' + aa'\bar{b} + \bar{b}^2$. \\
Moreover, for both $\vartheta = \mu$ and $\vartheta = \sigma$, ${\modch \frac{\partial}{\partial \vartheta}} R(\theta, 1, x)$ is still a $R(\theta, 1, x)$ function.
\label{prop: dl alla capitolo 1 m2}
\end{proposition}
We see that the contributions of the jumps depend on the density $F$ which argument in the integral depend on $\Delta_{n,i}$. If we choose a particular density function $F$ which is null in the neighborhood of $0$ the contribution of the jumps disappears and, in this case, we fall on the development for $m_2$ found by Kessler in the case without jumps (\cite{Kessler}), up to order $\Delta_{n,i}^2$. 

 The expansion \eqref{eq: dl alla capitolo 1 m2} looks cumbersome, however all terms are necessary to get an expansion with a remainder term of explicit order strictly greater than $2$, and valid for any finite intensity $F$. In the particular case where $F$ is $\mathcal{C}^1$, the first three terms in the expansion give the the main terms of the expansions \eqref{eq: dl particolare m2}, while the last integral term,  with order $\Delta_{n,i}^{2+\beta}$, is clearly a rest term. However, in the situation where $F$ may be unbounded near $0$, with $\int F(z) d z <\infty$, the last integral term is
	 only seen to be negligible versus $\Delta_{n,i}^2$.
	 Hence, this last integral term may be non negligible compared to the rest term and is needed in the expansion.

{\modch We remark that the construction of the contrast function requires the knowledge of the law of the jumps, which is typically unknown, since it affects $m$ and $m_2$. However, Propositions \ref{prop: dl m2 intensita finita} and \ref{prop: dl m2 particolare per appl} suggest us some suitable approximation of $m$ ad $m_2$ for which the corrections become explicit. For example, requiring also that $\int_{\mathbb{R}} z F(z) = 0$, $\bar{b}$ turns out being simply $b$ and therefore $\bar{A}_c (f)$ no longer depends on the jump law. It implies that $A_{k_1}$ in Proposition \ref{prop: dl m2 intensita finita} (and thus the expansions \eqref{eq: dl m2 intensita finita} and \eqref{E:dev_Kessler_chap1}) no longer depend on the jump law either. Moreover, Proposition \ref{prop: dl m2 particolare per appl} suggests us to estimate in a non -parametric way the quantity $F(0)$, which would provide an explicit correction as well. Even if we are not aware of any method for the estimation of $F(0)$, there are many papers related to such a topic. The problem of the estimation of $F(x)$ with $x \neq 0$ is connected for example to the works of \cite{Genon Comte} in high frequency and of \cite{Gugu} in low frequency, considering in both cases a L\'evy process. In \cite{Frank} the authors deal instead with a L\'evy driven Ornstein Uhlenbeck (OU) process. Figueroa - L\'opez proposes nonparametric estimators for the L\'evy density and obtains in \cite{Fig} both pointwise and uniform central limit theorems on an interval away from the origin for these estimators.}

 Finally, we recall the expansion of $m$ valid under the Assumptions A1-A4 (see Theorem 2 in \cite{Chapitre 1}),
\begin{equation}\label{E:dev_m_chap1}
m(\mu,\sigma,x)=x+ \Delta_{n, i} \bar{b}(x, \mu)  + \frac{\Delta_{n, i}^{1+2\beta}}{\gamma(x)} \int_{\mathbb{R}}  u\varphi(u) F(\frac{u \Delta_{n, i}^\beta}{\gamma(x)}) du + R(\theta,\Delta_{n, i}^{2-2\beta},x).
\end{equation}	


 Developments of $m_2$ given in Propositions \ref{prop: dl m2 intensita finita} - \ref{prop: dl alla capitolo 1 m2}, together with the developments of $m$ given in \eqref{E:dev_Kessler_chap1}, \eqref{E:dev_m_chap1} will be useful for the applications,
as illustrated in the following section. 

{\modch
\section{Simulation study}
\label{S: Applications}
In this section we provide some numerical results. First of all, in Section 5.1, we consider a second order expansion for $m$ and $m_2$ which holds true for any truncation function $\varphi$. After that, in Section 5.2, some oscillating truncation functions will be employed with the purpose of using an arbitrarily high expansion for $m$ and $m_2$, as in Kessler. As we will see, it will lead us to some estimators very well-performed.
\subsection{Second order expansion}
Let us consider the model 
\begin{equation}
X_t=X_0+ \int_0^t (\theta_1 X_s + \theta_2) ds + \sigma W_t + \gamma 
\int_0^t \int_{\mathbb{R}\backslash \left \{0 \right \}}  z \tilde{\mu}(ds,dz),
\label{eq:OU_in_simulation}
\end{equation}
where the compensator of the jump measure is $\overline{\mu}(ds,dz)=\lambda F_0(z)ds dz$ for $F_0$ the probability density of 
the law $\mathcal{N}(\mu_J,\sigma_J^2)$ with $\mu_J \in \mathbb{R}$, $\sigma_J >0$, $\sigma>0$,  $\theta_1<0$, $\theta_2 \in \mathbb{R}$, $\gamma \ge 0$, $\lambda \ge 0$. \\
We want to explore approximations for $m$ and $m_2$ which make us able to find en explicit version of the contrast function.
According to Theorem 2 of \cite{Chapitre 1}  (see also \eqref{E:dev_m_chap1}) we have, considering a threshold level which is $c\Delta_{n,i}^\beta$, the following development for $m$:
 
$$m (\theta_1, \theta_2, \sigma,x) = x + \Delta_{n,i} \bar{b}(x, \theta_1, \theta_2) + \frac{c^2 \Delta_{n,i}^{1 + 2 \beta}}{\gamma} \int_\mathbb{R} u \varphi(u) F (\frac{u c \Delta_{n,i}^\beta}{\gamma}) du + R(\theta, \Delta_{n,i}^{2 - 2 \beta}, x);$$
which leads us to the approximation 
$$\widetilde{m} (\theta_1, \theta_2,x) = x + \Delta_{n,i} \bar{b}(x, \theta_1, \theta_2) +
 \frac{c^2 \Delta_{n,i}^{1 + 2 \beta}}{\gamma} \int_\mathbb{R} u \varphi(u) F (\frac{u c \Delta_{n,i}^\beta}{\gamma}) du,$$
for $\bar{b}(x, \theta_1, \theta_2) = (\theta_1 x + \theta_2) - \gamma \lambda \mu_j$. The sampling scheme is uniform and therefore $\Delta_{n,i} = \frac{T}{n}$, for each $ i \in \left \{ 0, ... n - 1 \right \}$. It follows $|m (\theta_1, \theta_2, \sigma,x)  - \widetilde{m} (\theta_1, \theta_2,x)| \le R(\theta, \Delta_{n,i}^{2 - 2 \beta}, x) {\modch = R(\theta, \Delta_{n,i}^{\rho_1}, x)}.$ \\
Concerning the approximation of $m_2$, from its development gathered in Proposition \ref{prop: dl alla capitolo 1 m2} we define
$$\widetilde{m}_2(\sigma) := \Delta_{n,i} \sigma^2 + \frac{\Delta_{n,i}^{1 + 3 \beta} c^3}{\gamma} \int_{\mathbb{R}} u^2 \varphi(u ) F(\frac{u c \Delta_{n,i}^\beta}{\gamma}) du,$$
which is such that $|m_2 (\theta_1, \theta_2, \sigma,x)  - \widetilde{m}_2 (\sigma)| \le R(\theta, \Delta_{n,i}^{2}, x) {\modch = R(\theta, \Delta_{n,i}^{\rho_2}, x)}.$ \\
{\modch By application of Proposition \ref{Prop:normality_approximate_contrast} we see that the associated estimator is asymptotically gaussian under the strong assumption $\sqrt{n} \Delta_n^{\frac{3}{2} -2 \beta} \rightarrow 0$, for $n \rightarrow\infty$, with $\beta \in (\frac{1}{4}, \frac{1}{2})$. In the best case $\beta$ is close to $\frac{1}{4}$ and so the condition here above becomes equivalent to $n \Delta_n^2 \rightarrow 0$. Despite such a strong requirement on the discretization step, we expect the proposed method performing better than the Euler schema since the main contribution of the jumps has there already been removed. \\}
We want to compare the estimator $\tilde{\theta}_n$ we get by the minimization of the contrast function obtained by the expansions of $m$ and $m_2$ provided here above with the estimator based on the Euler scheme approximation:
\begin{equation}
\widetilde{m}^{E}(\theta_1, \theta_2, x)= x + (\theta_1 x + \theta_2 - \lambda \gamma \mu_j) \Delta_{n,i}, \qquad  \widetilde{m}_2^{E}( \sigma, x)= \sigma^2 \Delta_{n,i}.
\label{eq: Euler}
\end{equation}
We estimate jointly the parameter $\theta=(\theta_1,\theta_2, \sigma)$ by minimization of the contrast function 
\begin{equation}
U_n(\theta)=\sum_{i=0}^{n-1} [\frac{(X_{t_{i+1}} - \widetilde{m}(\theta_1, \theta_2, X_{t_i}))^2}{\widetilde{m}_2(\sigma, X_{t_i})} + \log(\frac{\widetilde{m}_2(\theta_1, \sigma, X_{t_i})}{\Delta_{n,i}})] \varphi_{c \Delta_{n,i}^\beta}(X_{t_{i+1}}-X_{t_i}),
\label{eq:contrast_OU}
\end{equation}
where $c>0$ will be specified later. \\
We compute the derivatives of the contrast function with respect to the three parameters:
$$ {\modch \frac{\partial}{\partial \theta_1}} U_n(\theta) = \sum_{i = 0}^{n - 1} \frac{2 (X_{t_{i + 1}}- \widetilde{m}(\theta_1, \theta_2, X_{t_i}))X_{t_i}}{\widetilde{m}_2(\sigma)} \varphi_{c \Delta_{n,i}^\beta}(X_{t_{i + 1}} - X_{t_i})= $$
$$ = \frac{2}{\widetilde{m}_2(\sigma)}\sum_{i = 0}^{n - 1} (X_{t_{i + 1}}- X_{t_i}- \Delta_{n,i} \theta_1 X_{t_i} - \Delta_{n,i} \theta_2 + \Delta_{n,i}\gamma \lambda \mu_j - J_1^i) X_{t_i}\varphi_{c \Delta_{n,i}^\beta}(X_{t_{i + 1}} - X_{t_i}),$$
where we have denoted as $J_1^i$ the term in the development of $\widetilde{m}$ turning up from the presence of jumps, which is 
 $\frac{c^2 \Delta_{n,i}^{1 + 2\beta}}{\gamma} \int_\mathbb{R} u \varphi(u) F (\frac{u c\Delta_{n,i}^\beta}{\gamma}) du$. \\
We want ${\modch \frac{\partial}{\partial \theta_1}} U_n(\theta) = 0$, it leads us to the definition of the following estimator:
$$\tilde{\theta}_{1,n} :=  \frac{\sum_{i = 0}^{n - 1} (X_{t_{i + 1}}- X_{t_i}- \Delta_{n,i} \theta_2 + \Delta_{n,i} \gamma \lambda \mu_j - J_1^i)X_{t_i} \varphi_{c \Delta_{n,i}^\beta}(X_{t_{i + 1}} - X_{t_i})}{\sum_{i = 0}^{n - 1} \Delta_{n,i} X_{t_i}^2 \varphi_{c \Delta_{n,i}^\beta}(X_{t_{i + 1}} - X_{t_i})}.$$
In the same way
$${\modch \frac{\partial}{\partial \theta_2}} U_n(\theta) = \sum_{i = 0}^{n - 1} \frac{2 (X_{t_{i + 1}}- \widetilde{m}(\theta_1, \theta_2, X_{t_i}))}{\tilde{m}_2(\sigma)} \varphi_{c \Delta_{n,i}^\beta}(X_{t_{i + 1}} - X_{t_i})= $$
$$ = \frac{2}{\widetilde{m}_2(\sigma)}\sum_{i = 0}^{n - 1} (X_{t_{i + 1}}- X_{t_i}- \Delta_{n,i} \theta_1 X_{t_i} - \Delta_{n,i} \theta_2 + \Delta_{n,i}\gamma \lambda \mu_j - J_1^i)\varphi_{c \Delta_{n,i}^\beta}(X_{t_{i + 1}} - X_{t_i}).$$
Since we want ${\modch \frac{\partial}{\partial \theta_2}} U_n(\theta) = 0$, we define $\tilde{\theta}_{2,n}$ in the following way:
\begin{equation}
\tilde{\theta}_{2,n} :=  \frac{\sum_{i = 0}^{n - 1} (X_{t_{i + 1}}- X_{t_i}- \Delta_{n,i} \theta_1 X_{t_i}) \varphi_{c \Delta_{n,i}^\beta}(X_{t_{i + 1}} - X_{t_i})}{\sum_{i = 0}^{n - 1} \Delta_{n,i} \varphi_{c \Delta_{n,i}^\beta}(X_{t_{i + 1}} - X_{t_i})} + \gamma \lambda \mu_j - \frac{\sum_{i = 0}^{n - 1} J_1^i \varphi_{c \Delta_{n,i}^\beta}(X_{t_{i + 1}} - X_{t_i})}{\sum_{i = 0}^{n - 1} \Delta_{n,i} \varphi_{c \Delta_{n,i}^\beta}(X_{t_{i + 1}} - X_{t_i})}=
\label{eq: deriv rispetto theta2 appli}
\end{equation}
$$= \tilde{\theta}^{\text{Euler}}_n  - \frac{\sum_{i = 0}^{n - 1} J_1^i \varphi_{c \Delta_{n,i}^\beta}(X_{t_{i + 1}} - X_{t_i})}{\sum_{i = 0}^{n - 1} \Delta_{n,i} \varphi_{c \Delta_{n,i}^\beta}(X_{t_{i + 1}} - X_{t_i})};$$
we can see $\tilde{\theta}_{2,n}$ as a corrected version of the estimator $\tilde{\theta}^{\text{Euler}}_n$ that would result considering the Euler scheme approximation for the function $m$, as in \eqref{eq: Euler}.
We observe moreover that, considering a uniform discretization step, the last term in \eqref{eq: deriv rispetto theta2 appli} becomes simply
 $\frac{J_1}{\Delta_n} := \frac{c^2 \Delta_{n}^{ 2\beta}}{\gamma} \int_\mathbb{R} u \varphi(u) F (\frac{u c\Delta_{n}^\beta}{\gamma}) du$. \\
Computing also the derivative of the contrast function with respect to $\sigma^2$, we have 
$$ {\modch \frac{\partial}{\partial \sigma^2}}  U_n(\theta) = \sum_{i = 0}^{n - 1} [\frac{-(X_{t_{i + 1}}- \widetilde{m}(\theta_1, \theta_2, X_{t_i}))^2 \partial_{\sigma^2}\widetilde{m}_2 (\sigma) + \widetilde{m}_2 (\sigma) \partial_{\sigma^2}\widetilde{m}_2 (\sigma)}{\widetilde{m}^2_2 (\sigma)}] \varphi_{c \Delta_{n,i}^\beta}(X_{t_{i + 1}} - X_{t_i}),$$
which is equal to zero if and only if 
$$\sum_{i = 0}^{n - 1} [ -(X_{t_{i + 1}}- \widetilde{m}(\theta_1, \theta_2, X_{t_i}))^2 \Delta_{n,i} + \Delta_{n,i}(\Delta_{n,i} \sigma^2 + J_2^i)] \varphi_{c \Delta_{n,i}^\beta}(X_{t_{i + 1}} - X_{t_i}) = 0,$$
for  $J_2^i := \frac{\Delta_{n,i}^{1 + 3 \beta} c^3}{\gamma} \int_{\mathbb{R}} u^2 \varphi(u ) F(\frac{u c \Delta_{n,i}^\beta}{\gamma}) du,$ which is the part deriving from the jumps in the development of $\widetilde{m}_2 (\sigma)$. It drives us to the estimator
$$\tilde{\sigma}^2_n := \frac{ \sum_{i = 0}^{n - 1} (X_{t_{i + 1}}- \tilde{m}(\theta_1, \theta_2, X_{t_i}))^2 \Delta_{n,i} \varphi_{c \Delta_{n,i}^\beta}(X_{t_{i + 1}} - X_{t_i}) }{\sum_{i = 0}^{n - 1} \Delta_{n,i}^2 \varphi_{c \Delta_{n,i}^\beta}(X_{t_{i + 1}} - X_{t_i})} - \frac{\sum_{i = 0}^{n - 1} \Delta_{n,i} J_2^i \varphi_{c \Delta_{n,i}^\beta}(X_{t_{i + 1}} - X_{t_i})}{\sum_{i = 0}^{n - 1} \Delta_{n,i}^2 \varphi_{c \Delta_{n,i}^\beta}(X_{t_{i + 1}} - X_{t_i})}.$$
Considering an uniform discretization step it is
$$\tilde{\sigma}^2_n := \frac{ \sum_{i = 0}^{n - 1} (X_{t_{i + 1}}- \tilde{m}(\theta_1, \theta_2, X_{t_i}))^2 \varphi_{c \Delta_{n,i}^\beta}(X_{t_{i + 1}} - X_{t_i}) }{ \Delta_{n} \sum_{i = 0}^{n - 1} \varphi_{c \Delta_{n,i}^\beta}(X_{t_{i + 1}} - X_{t_i})} - \frac{J_2}{\Delta_n}.$$
Again, it can be seen as a corrected version of $\tilde{\sigma}^{2, \text{Euler}}_n$, the estimator that would have resulted considering the approximation of the functions $m$ and $m_2$ as defined in \eqref{eq: Euler}. In such a case, not only we would not have seen the contribution of the jumps appearing in the second term here above, but also in the first term we should have replaced $\tilde{m}$ with its Euler approximation. \\
To illustrate the estimation method, we focus on the estimation of the parameters $\theta_2$ and $\sigma^2$ only. \\
For numerical simulations we choose $T = 100$, $n = 50\,000$, $\lambda = 1$,  $\gamma =1 $, $\theta_1 = -1$, $\theta_2 =2$, $\sigma = 0.5$, $d = 2$ and $c= 1.25$. We estimate the bias of the estimators using a Monte Carlo method based on 1000 replications. \\
First, we consider $\beta$ as big as possible, fixing it equal to $0.49$; then we will take $\beta = 0.3$; in both cases the jumps size has common law $\mathcal{N}(4, 0.25)$. \\
\begin{table}[h]
        \centering
        \begin{tabular}{|c|c|c|}
        \hline 
          & Mean  {\modch (dev std)} for $\theta_2=2$ &  Mean  {\modch (dev std)} for $\sigma^2=0.25$ \\ 
          \hline 
          \rule{0pt}{1.3em}
         $\tilde{\theta}^{\text{Euler}}_n$ & 2.00284 {\modch (0.05020)} & 0.26558 {\modch (0.00161)}
            \\ 
            \hline 
            \rule{0pt}{1.3em} $\displaystyle \widetilde{\theta}_n$&  	2.00289 {\modch (0.04830)} &	0.26412 {\modch (0.00156)} \\ 
            \hline 
            \end{tabular} 
            \caption{Monte Carlo estimates of $\theta_2$ and $\sigma^2$ from 1000 samples. We have here fixed $\beta = 0.3$.}
            \label{tab: metodo2 beta small}
    \hfill 
\end{table}    
\begin{table}[h]
\centering
        \begin{tabular}{|c|c|c|}
        \hline 
        & Mean {\modch (dev std)} for $\theta_2=2$ &  Mean  {\modch (dev std)} for $\sigma^2=0.25$ \\ 
        \hline 
        \rule{0pt}{1.3em}
        $\tilde{\theta}^{\text{Euler}}_n$ & 2.00368 {\modch (0.05195)} &	0.25150 {\modch (0.00161)} \\ 
        \hline 
        \rule{0pt}{1.3em} $\displaystyle \widetilde{\theta}_n$& 2.00302 {\modch (0.04992)} &	0.25067 {\modch (0.00159)} \\ 
        \hline 
        \end{tabular} 
\caption{Monte Carlo estimates of $\theta_2$ and $\sigma^2$ from 1 000 samples. We have here fixed $\beta = 0.49$.}
\label{tab:tabella totale}
\end{table}

We see that there is not a big difference in the estimation of $\theta_2$ with and without the correction term. 
The two estimators of $\theta_2$ differs in fact only for the contribution of the jumps  $\frac{J_1}{\Delta_n} = \frac{ c^2 \Delta_{n}^{2 \beta}}{\gamma} \int_\mathbb{R} u \varphi(u) F (\frac{u c\Delta_{n}^\beta}{\gamma}) du$ that is in this case close to zero, because of the natural choice of taking a truncated function which is symmetric. Indeed, even if the density function $F$ is not symmetric, asymptotically the only contribution it gives is due by its value in zero and, therefore, the symmetry of $\varphi$ is enough to ensure the limit of the integral is zero. \\
Regarding $\sigma^2$, as the correction term $\frac{J_2}{\Delta_n}$ is very small the two different estimators provide two means for $\sigma^2$ which are rather similar. However, the estimator we find through our approximated functions $\tilde{m}$ and $\tilde{m}_2$ performs a bit better than the estimator we got through Euler scheme.\\

\subsection{Kessler approximation}
In the previous case we have used the second order expansion from $m$ and $m_2$, available for any truncation function $\varphi$.
Now we still consider \eqref{eq:OU_in_simulation} in which $F_0$ is still the probability density of the law $\mathcal{N}(\mu_j, \sigma_J^2)$, but we use Kessler approximation to remove the bias.
To do that, we take an oscillating truncated function $\varphi$ for which the initial contributions of the discontinuous part of the generator disappear. \\
{\modch Indeed, according to Proposition \ref{prop: dl m2 intensita finita} (respectively Proposition 2 in \cite{Chapitre 1}) we know that, using sufficiently oscillating truncation functions, the expansion of $m_2$ (respectively $m$) is the same found by Kessler in the continuous case. The usefulness of Proposition \ref{prop: dl m2 intensita finita} is to illustrate it is possible to neglect the effect of the truncation function $\varphi$ as the expansions of $m$ and $m_2$ are identical to those of Kessler for the continuous part of the SDE.} Since the Kessler's expansion approximates the first conditional moments of $ \bar{X}_t=\bar{X}_0+ \int_0^t (\theta_1 \bar{X}_s + \theta_2 - \gamma \lambda \mu_J) ds + \sigma W_t$ (see \eqref{eq: m kessler}), which is the continuous part of \eqref{eq:OU_in_simulation} and which is explicit due to the linearity of the model, we decide to use directly the expression of the conditional moment and set
\begin{equation}
\widetilde{m}(\theta_1, \theta_2, x)= (x + \frac{\theta_2}{\theta_1} - \frac{\gamma \lambda \mu_J}{\theta_1}) e^{\theta_1 \Delta_{n,i}}+\frac{\gamma \lambda \mu_J-\theta_2}{\theta_1};
\label{eq:m_tilde_OU}
\end{equation} 
while the approximation of $m_2(\mu, \sigma, x)$ is
\begin{equation}
\widetilde{m}_2(\theta_1, \sigma, x)= \frac{\sigma^2}{2 \theta_1}(e^{2 \theta_1 \Delta_{n,i}} -1).
\label{eq:m2_tilde_OU}
\end{equation}
We want to compare the estimator $\tilde{\theta}_n$ we get by the minimization of the contrast function obtained by the Kessler exact correction of the bias in which we use the approximations \eqref{eq:m_tilde_OU} and \eqref{eq:m2_tilde_OU} for $m$ and $m_2$ with the estimator based on the Euler scheme approximation.
Following Nikolskii \cite{Nikolskii_Book}, we construct oscillating truncation functions in the following way. First, we choose $\varphi^{(0)}: \mathbb{R} \to [0,1]$ a $\mathcal{C}^\infty$ symmetric function with support on $[-2,2]$ such that $\varphi^{(0)}(x)=1$ for $|x| \leq 1$. We let, for $d>1$, $\varphi^{(1)}_d(x)= (d\varphi^{(0)}(x)-\varphi^{(0)}(x/d))/(d-1)$, which is a function equal to $1$ on $[-1,1]$,  vanishing on $[-d,d]^c$ and such that $\int_{\mathbb{R}} \varphi^{(1)}_d(x) dx=0$. For $l \in \mathbb{N}$, $l \ge 1$, and $d>1$, we set 
$\varphi_d^{(l)}(x)=c_d^{-1}\sum_{k=1}^l C_l^k (-1)^{k+1} \frac{1}{k} \varphi^{(1)}_d(x/k)$, where
$c_d=\sum_{k=1}^l C_l^k (-1)^{k+1} \frac{1}{k}$.  One can check that $\varphi_d^{(l)}$ is compactly supported, equal to $1$ on $[-1,1]$,  and that for all 
$k \in  \{0, \dots, l \}, \int_{\mathbb{R}} x^k  \varphi_d^{(l)}(x) dx =0$, for $l \ge 1$. 
With these notations, we estimate the parameter $(\theta_1,\theta_2, \sigma)$ by minimization of the contrast function \eqref{eq:contrast_OU}, implementing the just built truncation function $\varphi^{(l)}_{c \Delta_{n,i}^\beta}(X_{t_{i+1}}-X_{t_i})$,
where $l \in \mathbb{N}$ and $c>0$ will be specified latter.

For numerical simulations, we choose $T=1 \, 000$, $n$ will be chosen equal to $2 \, 000$, $10 \, 000$ and $50 \, 000$, $\theta_1=-1$, $\theta_2=2$, $\sigma=0.5$ and $X_0=x_0=0$. We estimate the bias and standard deviation of our estimators using a Monte Carlo method based on $2 \,000$ replications. 
As a start, we consider a situation without jumps $\lambda=0$ , in which we remove the truncation function $\varphi$ in the contrast, as it is useless in absence of jumps. 
 
\begin{table}[ht]
\begin{center}
\begin{tabular}{|c|c|c|c|}
\hline 
& Mean (std) for $\theta_1=-1$& Mean (std) for $\theta_2=2$ & Mean (std) for $\sigma=0.5$ \\ 
\hline 
$\Delta_n = 0.5$ & & & \\
\rule{0pt}{1.3em}
$\tilde{\theta}^{\text{Euler}}_n$ &-0.7922 (0.0348) &  1.5843 (0.0701) &  0.3980  (0.0062)\\ 
\rule{0pt}{1.3em} 
$\displaystyle \widetilde{\theta}_n$&-1.0063 (0.0578) &  2.0128 (0.1166) &  0.5017 (0.0099) \\ 
\hline 
$\Delta_n = 0.1$ & & &  \\
\rule{0pt}{1.3em}
$\tilde{\theta}^{\text{Euler}}_n$ &-0.9536 (0.0422) &  1.9068 (0.0857) &  0.4761 (0.0034)  \\ 
\rule{0pt}{1.3em} 
$\displaystyle \widetilde{\theta}_n$& -1.0053 (0.0465) &  2.0106 (0.0944) & 0.5003 (0.0037)  \\ 
\hline
$\Delta_n = 0.02$ & & &  \\
\rule{0pt}{1.3em}
$\tilde{\theta}^{\text{Euler}}_n$ &-0.9940 (0.0442) &  1.9888 (0.0895) & 0.4951 (0.0016)  \\ 
\rule{0pt}{1.3em} 
$\displaystyle \widetilde{\theta}_n$& -1.0039 (0.0437) &  2.0084 (0.0886) &  0.5001 (0.0016) \\ 
\hline
\end{tabular} 
\caption{Process without jump \label{T:no_jumps}} 
\end{center}
\end{table}
 
 In Table \ref{T:no_jumps}, we compare the estimator which uses the Kessler exact bias corrections given by \eqref{eq:m_tilde_OU} and \eqref{eq:m2_tilde_OU}, with an estimator based on the Euler scheme approximation.
From Table \ref{T:no_jumps} we see that the estimator $\tilde{\theta}^{\text{Euler}}_n$ based on Euler contrast exhibits some bias  which is completely removed using Kessler's correction. The improvement provided by Kessler approximation is particularly evident for $n$ small.

Next, we set a jump intensity $\lambda=0.1$, with jumps size whose common law is $\mathcal{N}(4,0.25)$ and set $\gamma=1$. 
  Also in this case, we compare the results obtained basing the estimation on Kessler or on Euler approximations for the quantities $m$ and $m_2$.
  Numerical simulations are provided for the truncation function $\varphi^{(0)}$, where $c=2$ and $\beta =0,49$. A plot of this function can be found in Figure \ref{Fig:phi}.

\begin{table}[ht]
\begin{center}
\begin{tabular}{|c|c|c|c |}
\hline 
& Mean (std) for $\theta_1=-1$& Mean (std) for $\theta_2=2$ & Mean (std) for $\sigma=0.5$ \\ 
\hline 
$\Delta_n = 0.5$ & & & \\
\rule{0pt}{1.3em}
$\tilde{\theta}^{\text{Euler}}_n$ &-0.7301 (0.0274) &  1.5796 (0.0512) &  0.4843 (0.0223)\\ 
\rule{0pt}{1.3em} 
$\displaystyle \widetilde{\theta}_n$&-0.9095 (0.0423) &  1.8697 (0.0763) &  0.5967 (0.0234) \\ 
\hline 
$\Delta_n = 0.1$ & & &  \\
\rule{0pt}{1.3em}
$\tilde{\theta}^{\text{Euler}}_n$ &-0.9509 (0.0153) &  1.9213 (0.0338) &  0.4761 (0.0034)  \\ 
\rule{0pt}{1.3em} 
$\displaystyle \widetilde{\theta}_n$& -0.9988 (0.0178) &  1.9974 (0.0390) & 0.5002 (0.0035)  \\ 
\hline
$\Delta_n = 0.02$ & & & \\
\rule{0pt}{1.3em}
$\tilde{\theta}^{\text{Euler}}_n$ &-0.9909 (0.0165) &  1.9858 (0.0361) & 0.4951 (0.0016)  \\ 
\rule{0pt}{1.3em} 
$\displaystyle \widetilde{\theta}_n$& -1.0013 (0.0166) &  2.0028 (0.0374) &  0.5000 (0.0016) \\ 
\hline
\end{tabular} 
\caption{Gaussian jumps with $\lambda=0.1$ \label{T:few_jumps}} 
\end{center}
\end{table}

   Results in Table \ref{T:few_jumps} show that the estimator deriving from Kessler approximation works well, better than the one deriving from Euler approximations. The bias is visibly reduced for all the choices of $n$, especially when $n$ is small. It matches with our theoretical results for which, implementing an oscillating truncation function, the discretization step can goes to zero arbitrarily slow.
We remark that by the choice of a symmetric truncation function one has $\int_{\mathbb{R}} u \varphi^{(0)}(u) d u=0$ and it can be seen that this conditions is sufficient, in the expansion of $m$ and $m_2$, to suppress the largest contribution of the discrete part of the generator.
 
  When the number of jumps is greater, e.g. for $\lambda=10$, we choose once again $c=2$ and $\beta =0,49$ but, unlike before, $T$ is fixed equal to $100 $ and $n$ is $10\, 000$, $50 \, 000$ and $500 \, 000$. Now the law of the jumps is $\mathcal{N}(0,1)$. We compare the results we get by considering different oscillating truncation functions $\varphi^{(0)}$, $\varphi^{(2)}_{1.4}$ and $\varphi^{(4)}_{1.2}$, whose plots are given in Figure \ref{Fig:phi}.
  
  \begin{table}[ht]
\begin{center}
\begin{tabular}{|c|c|c|c |}
\hline 
& Mean (std) for $\theta_1=-1$& Mean (std) for $\theta_2=2$ & Mean (std) for $\sigma=0.5$ \\ 
\hline 
$\Delta_n = 0.01$  & & & \\
\rule{0pt}{1.3em}
$\widetilde{\theta_n}$ using $\varphi^{(0)}$ &-0.9611 (0.0477) &  1.9229 (0.1440) &  1.1037 (0.0229)\\ 
\rule{0pt}{1.3em} 
$\widetilde{\theta_n}$ using $\varphi^{(2)}_{1.4}$&-0.9979 (0.0387) &  1.9967 (0.1157) &  0.3272 (0.0310) \\ 
\rule{0pt}{1.3em}
$\widetilde{\theta_n}$ using $\varphi^{(4)}_{1.2}$& -0.9979 (0.1020) & 1.9924 (0.3123) &  0.5376 (0.1013) \\ 
\hline 
$\Delta_n = 0.002$ & & &  \\
\rule{0pt}{1.3em}
$\widetilde{\theta_n}$ using $\varphi^{(0)}$  &-0.9948 (0.0307) &  1.9890 (0.0926) &  0.7121 (0.0110)  \\ 
\rule{0pt}{1.3em} 
$\widetilde{\theta_n}$ using $\varphi^{(2)}_{1.4}$& -1.0007 (0.0241) &  2.0007 (0.0707) & 0.4863 (0.0031)  \\ 
\rule{0pt}{1.3em}
$\widetilde{\theta_n}$ using $\varphi^{(4)}_{1.2}$& -0.9996 (0.0374) & 1.9994 (0.1116) &  0.5007 (0.0144) \\ 
\hline
$\Delta_n = 0.0002$  & & & \\
\rule{0pt}{1.3em}
$\widetilde{\theta_n}$ using $\varphi^{(0)}$ & -1.0008 (0.0237) &  1.9991 (0.0717) & 0.5331 (0.0027)  \\ 
\rule{0pt}{1.3em} 
$\widetilde{\theta_n}$ using $\varphi^{(2)}_{1.4}$& -1.0009 (0.0217) &  2.0021 (0.0643) &  0.4995 (0.0005) \\ 
\rule{0pt}{1.3em}
$\widetilde{\theta_n}$ using $\varphi^{(4)}_{1.2}$& -0.9996 (0.0226) & 1.9978 (0.0677) & 0.4999 (0.0010) \\ 
\hline
\end{tabular} 
\caption{Gaussian jumps with $\lambda=10$ \label{T:many_jumps}} 
\end{center}
\end{table}
  
  We see in Table \ref{T:many_jumps} that the use of the more oscillating kernels $\varphi^{(2)}_{1.4}$, $\varphi^{(4)}_{1.2}$ yields to a smaller bias than using $\varphi^{(0)}$. 
  
  The estimator we get using $\varphi^{(4)}_{1.2}$ performs particularly well in this situation, it has a negligible bias and a small standard deviation.

\begin{figure}[ht]
	\centering
	\begin{subfigure}[b]{0.3\linewidth}
		\includegraphics[width=\linewidth]{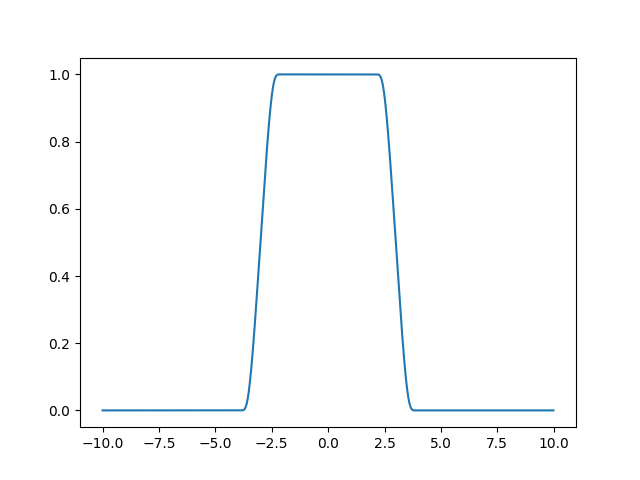}
		\caption{$\varphi^{(0)}$}
	\end{subfigure}
	\begin{subfigure}[b]{0.3\linewidth}
		\includegraphics[width=\linewidth]{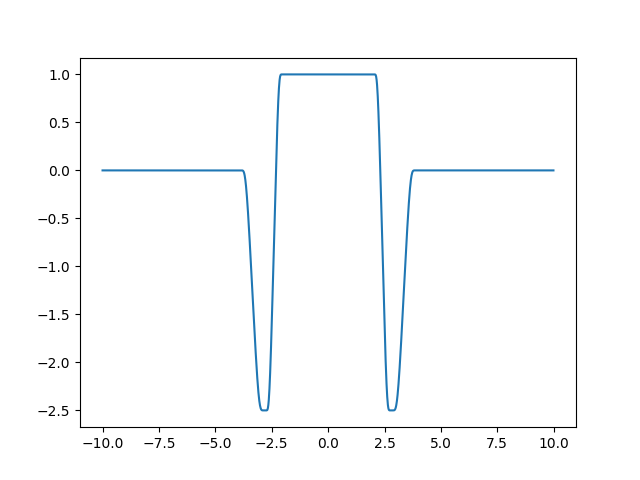}
		\caption{$\varphi^{(2)}_{1.4}$}
	\end{subfigure}
	\begin{subfigure}[b]{0.3\linewidth}
		\includegraphics[width=\linewidth]{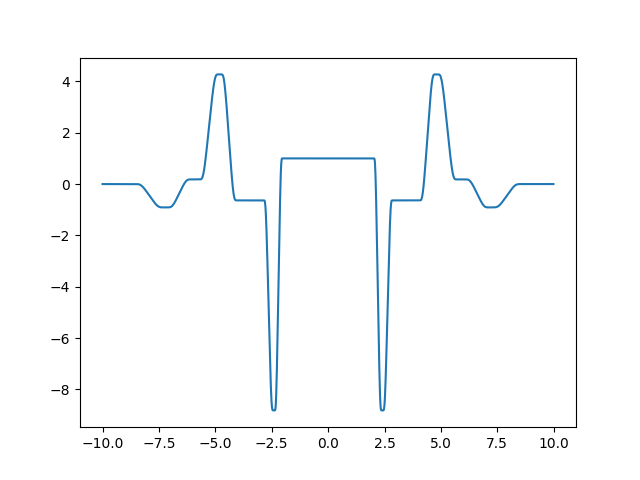}
		\caption{$\varphi^{(4)}_{1.2}$}
	\end{subfigure}
	\caption{Plot of the truncation functions}
\label{Fig:phi}
\end{figure}

}

{\modch
\section{Future perspectives}{\label{S: Perspectives}}

In future perspectives, we plan to generalize the obtained results. There are many remaining questions of interest, both from a theoretical and a numerical point of view. \\
As explained in the introduction, in this paper we extend previous works in which only the drift parameter was considered (see for example \cite{GLM} and \cite{Chapitre 1}). We may wonder what happens if the jump coefficient has another parameter, that we here denote as $\eta$. A first, simple answer is that, if the jumps are centered (which means $\int_\mathbb{R} z F(z) dz = 0$), then the knowledge of the jumps is no longer needed. Hence, as discussed below Proposition \ref{prop: dl alla capitolo 1 m2}, it is enough to use the approximation of $m$ and $m_2$ proposed in Proposition 2 of \cite{Chapitre 1} and in Proposition \ref{prop: dl m2 intensita finita} in order to make the contrast explicit. After that, we can still use our main results to estimate jointly the drift and the volatility parameters and the lack of knowledge of the jump parameter is not a problem. Estimating the three parameters jointly, that is a different story. In \cite{Shimizu}, Shimizu and Yoshida deal with it. They propose a contrast function which has two terms: the first corresponds to the contrast for an usual diffusion process, while the second concerns the discretization of
the likelihood function of a compound Poisson process with L\'evy density. In the abovementioned work they show the asymptotic normality of the estimators under some conditions on the sampling step which are more and more restrictive as the intensity of the jumps is high and, when the intensity is finite, they reduces to $n \Delta_n^2 \rightarrow 0$. \\
One may wonder how to modify the contrast function introduced in \eqref{eq: contrast function} in order to estimate jointly the three parameters under the general balance condition $n \Delta_n \rightarrow  \infty$ and $n \Delta_n^k \rightarrow 0$ for all $ k \ge 2$.
The natural idea is to add to \eqref{eq: contrast function} a term, corresponding to the one introduced by Shimizu and Yoshida, which regards the discretization of the likelihood function of the jumps. In order to weaken the condition on the sampling step we need to introduce in such a term a third unknown quantity that we would define on the specif purpose to make $\partial_\eta U_n$ a triangular array of martingale increments. However, it is not a simple issue. First of all, it is not trivial to understand how to correct the contrast function and, furthermore, it would become much more difficult to investigate the asymptotic behaviour of the new contrast function and of its derivatives.

For practical implementation, the question of approximation of $m$ and $m_2$ is crucial, and therefore one has to face the issue of choosing the threshold level, characterized here by $c$, $\beta$ and $\varphi$. Indeed, though the efficiency of the estimators has been established theoretically, it is known that in general their real performance strongly depends on a choice of tuning parameters; see for example Shimizu \cite{Shimizu threshold}, Iacus and Yoshida \cite{Iacus}. The filter is each time based on only one increment of the data and so, in this sense, this filter can be regarded as a local method. In \cite{Threshold}, Inatsugu and Yoshida introduce a global filtering method, which they call $\alpha$ threshold method. It uses all of the data to more accurately detect increments having jumps, based on the order statistics associated with all increments. Even if the $\alpha$ - threshold method involves the tuning parameter $\alpha$ to determine a selection rule for increments, it is robust against the choice of such a parameter. \\
In this work we use a local threshold method, having in mind a different objective. Here the idea, indeed, is mainly to remove the bias caused by a wrong choice of the threshold parameters rather than determine a good one. On contrary to more conventional threshold methods, the quantities $m$ and $m_2$ depend here by construction on the threshold level and may compensate for too large threshold. However, as seen in Section \ref{S: Applications}, the quantities $m$ and $m_2$ can be numerically very far from the approximations derived through the Euler scheme approximation and through the approximation scheme at higher order (as proposed in Proposition 3). A perspective would be to approximate numerically these two quantities, using for instance a Monte Carlo approach, and provide more accurate corrections than the ones here proposed. 

}

\section{Preliminary results}\label{S: Limit theorems}
Before proving the main statistical results of Section \ref{S:Construction_and_main}, we need to state several propositions which will be useful in the sequel. They will be proven in Section \ref{S: Proof_limit}.

\subsection{Limit theorems}
The asymptotic properties of estimators are deduced from the asymptotic behavior of the contrast function. We therefore state some propositions useful to get the asymptotic behavior of $U_n$.
\begin{proposition}
Suppose that Assumptions 1 to 4 and $A_{\mbox{Step}}$ hold, $\Delta_n \rightarrow 0$ and $T_n \rightarrow \infty$ and $f$ is a differentiable function $\mathbb{R} \times \Theta \rightarrow \mathbb{R}$ such that $|f(x, \theta)| \le c (1 + |x|^c)$, $|\partial_x f(x, \theta) | \le c (1 + |x|^c)$ and, for $\vartheta = \mu$ and $\vartheta = \sigma$, $|\partial_\vartheta f(x, \theta) | \le c (1 + |x|^c)$. Then $x \mapsto f(x, \theta)$ is a $\pi$-integrable function for any $\theta \in \Theta$ and the following convergences hold as $n \rightarrow \infty$ : \\
1.$ |\frac{1}{T_n} \sum_{i=0}^{n-1}\Delta_{n,i}f(X_{t_i}, \theta)1_{\left \{|X_{t_i}| \le \Delta_{n,i}^{-k} \right \} } - \int_\mathbb{R} f(x, \theta) \pi(dx)| \xrightarrow{\mathbb{P}} 0,$ \\
2. $|\frac{1}{T_n} \sum_{i=0}^{n-1}\Delta_{n,i}f(X_{t_i}, \theta)\varphi_{\Delta_{n,i}^\beta}(X_{t_{i+1}} - X_{t_i})1_{\left \{|X_{t_i}| \le \Delta_{n,i}^{-k} \right \} } - \int_\mathbb{R} f(x, \theta) \pi(dx)| \xrightarrow{\mathbb{P}} 0,$ \\
3. $| \frac{1}{n} \sum_{i = 0}^{n - 1} f(X_{t_i}, \theta) 1_{\left \{|X_{t_i}| \le \Delta_{n,i}^{-k} \right \} } - \int_\mathbb{R} f(x, \theta) \pi (dx)|\xrightarrow{\mathbb{P}} 0,$ \\
4. $|\frac{1}{n} \sum_{i = 0}^{n - 1} f(X_{t_i}, \theta)\varphi_{\Delta_{n,i}^\beta}(X_{t_{i+1}} - X_{t_i})1_{\left \{|X_{t_i}| \le \Delta_{n,i}^{-k} \right \} } - \int_\mathbb{R} f(x, \theta) \pi (dx)|\xrightarrow{\mathbb{P}} 0.$ 
\label{prop: ergodic}
\end{proposition}
Statements $1-2$ and $3-4$ of the proposition here above, as well as the first and the second point of Proposition \ref{prop: LT} below, turn out being similar if the sampling step $\Delta_{n,i} = \Delta_n$ considered is uniform. Otherwise, we need these two different convergences because, in order to estimate $\mu$ and $\sigma$ jointly, we have to deal with different scaling of the contrast function.

\begin{proposition}
Suppose that Assumptions 1 to 4 and $A_{\mbox{Step}}$ hold, $\Delta_n \rightarrow 0$ and $T_n \rightarrow \infty$ and $f$: $\mathbb{R} \times \Theta \rightarrow \mathbb{R}$. Moreover we suppose that $\exists c$: $|f(x, \theta)| \le c (1 + |x|^c)$ and that $\beta \in (\frac{1}{4}, \frac{1}{2})$. \\
Then, $\forall \theta \in \Theta$, 
$$1. \quad \frac{1}{T_n} \sum_{i= 0}^{n-1} f(X_{t_i}, \theta) \, (X_{t_{i+1}} - m(\mu, \sigma, X_{t_i}))^2 \, \varphi_{\Delta_{n,i}^\beta}(X_{t_{i+1}} - X_{t_i})1_{\left \{|X_{t_i}| \le \Delta_{n,i}^{-k} \right \} }\xrightarrow{\mathbb{P}} \int_\mathbb{R}f(x, \theta) a^2(x, \sigma_0) \pi(dx).$$
$$2. \quad \frac{1}{n} \sum_{i= 0}^{n-1} \frac{f(X_{t_i}, \theta)}{\Delta_{n,i}} \, (X_{t_{i+1}} - m(\mu, \sigma, X_{t_i}))^2 \, \varphi_{\Delta_{n,i}^\beta}(X_{t_{i+1}} - X_{t_i})1_{\left \{|X_{t_i}| \le \Delta_{n,i}^{-k} \right \} }\xrightarrow{\mathbb{P}} \int_\mathbb{R}f(x, \theta) a^2(x, \sigma_0) \pi(dx).$$
\label{prop: LT}
\end{proposition}
The proof relies on the following lemma. In the sequel we will denote $\mathbb{E}_i[.]$ for $\mathbb{E}[. | \mathcal{F}_{t_i}]$, where $(\mathcal{F}_s)_s$ is the filtration defined in Lemma \ref{lemma: Moment inequalities}. 
\begin{lemma}
Suppose that Assumptions 1 to 4 hold. Moreover we suppose that $\beta \in (\frac{1}{4}, \frac{1}{2}) $. Then
\begin{equation}
1. \mathbb{E}_i[(X_{t_{i+1}} - m(\mu, \sigma, X_{t_i}))^2 \, \varphi^2_{\Delta_{n,i}^\beta}(X_{t_{i+1}} - X_{t_i})] = \Delta_{n,i} a^2(X_{t_i}, \sigma_0) + R(\theta, \Delta_{n,i}^{1 + \beta}, X_{t_i} ),
\label{eq: Ei al quadrato}
\end{equation}
\begin{equation}
2. \mathbb{E}_i[(X_{t_{i+1}} - m(\mu, \sigma, X_{t_i}))^4 \, \varphi^4_{\Delta_{n,i}^\beta}(X_{t_{i+1}} - X_{t_i})] = 3\Delta^2_{n,i} a^4(X_{t_i}, \sigma_0) + R(\theta, \Delta_{n,i}^{ \frac{7}{4} + \beta}, X_{t_i} ),
\label{eq: Ei quarta}
\end{equation}
\begin{equation}
3.\mbox{For } k \ge 1, \quad |\mathbb{E}_i[(X_{t_{i+1}} - m(\mu, \sigma, X_{t_i})) \, \varphi^k_{\Delta_{n,i}^\beta}(X_{t_{i+1}} - X_{t_i})] |\le R(\theta, \Delta_{n,i}, X_{t_i} ),
\label{eq: stima Ei prima}
\end{equation}
\begin{equation}
4.\mbox{For } k \ge 2, \, \forall k'> 0, \quad  \mathbb{E}_i[|X_{t_{i+1}} - m(\mu, \sigma, X_{t_i})|^k \, |\varphi_{\Delta_{n,i}^\beta}(X_{t_{i+1}} - X_{t_i})|^{k'}] \le R(\theta, \Delta_{n,i}^{\frac{k}{2} \land (1 + \beta k)}, X_{t_i} ).
\label{eq: x - m prop 2}
\end{equation}
$$5. \forall k'> 0, \quad \mathbb{E}_i[(X_{t_{i + 1}} - m(\mu_0, \sigma_0, X_{t_i}))^3 |\varphi_{\Delta_{n,i}^\beta}(X_{t_{i+1}} - X_{t_i})|^{k'}] = R(\theta_0, \Delta_{n,i}^{\frac{4}{3} + \beta}, X_{t_i}).$$
\label{lemma: conditional expected value}
\end{lemma}

We observe that the first and the second points here above are particular cases of the the fourth one, in which we get some better estimation. In particular, we can identify in detail the main term. \\
Concerning the fifth point, instead, we remark that for $k = 3$ we don't have the main contribution of the Brownian part anymore, which gave us the rest function of size $\Delta_{n,i}^{\frac{k}{2}}$ in \eqref{eq: x - m prop 2}. In this case the main term of the development is given by the square of the Brownian integral times the jump part, which magnitude can be estimated by $\frac{4}{3} + \beta $.  \\ \\
In next lemma we consider the derivatives of $\varphi$, getting an improvement of the estimations here above. It relies on the fact that, from the definition we gave of such a function, we know its derivatives are different from zero only if the increments of our process are smaller than $2 \Delta_{n,i}^\beta$ (as it was for $\varphi$) and bigger than $\Delta_{n,i}^\beta$ (extra bound that we did not get using $\varphi$). Having therefore no longer only an upper bound but also a lower bound for $X_{t_{i + 1}} - X_{t_i}$, it is now possible to prove a better version of \eqref{eq: x - m prop 2}:

\begin{lemma}
Suppose that Assumptions A1-A5 A7 and Ad hold. Then $\forall p \ge 1$, $\forall k \ge 1$ and $\forall r > 0$,
$$\mathbb{E}_i[|X_{t_{i + 1}}^\theta - m(\mu, \sigma, X_{t_i})|^p |\varphi^{(k)}_{\Delta_{n,i}^\beta}(X_{t_{i + 1}}^\theta - X_{t_i}^\theta)|^{r}] \le R(\theta, h^{1 + \beta p}, X_{t_i}).$$
\label{lemma: x-m con varphi'}
\end{lemma}

Considering only the jump part, the following result holds:
\begin{lemma}
Suppose that Assumptions A1-A4 holds. Then, $\forall q \ge 1$ we have
$$\mathbb{E}_i[|\Delta X_i^J\varphi_{\Delta_{n,i}^\beta}(\Delta_i X)|^q] = R(\theta_0, \Delta_{n, i}^{(1 + \beta q) \land q}, X_{t_i}),$$
where $\Delta X_i^J := \int_{t_i}^{t_{i + 1}} \int_\mathbb{R} z \gamma(X_{s^-}) \tilde{\mu}(ds, dz)$.
\label{lemma: estim jumps}
\end{lemma}

Other estimation about the expected value of the jump part in the presence of an indicator function which is $0$ if the increments are bigger than $c \Delta_{n,i}^{\beta}$ are gathered in Lemma 4 of \cite{Chapitre 2}. \\ \\
Using the lemmas stated here above, it is possible to prove the following proposition, that will be proved in Section \ref{S: Proof_limit} and which is useful to show the tightness of the contrast function.
\begin{proposition}
Suppose that Assumptions 1 to 4 and $A_{\mbox{Step}}$ hold, $\Delta_n \rightarrow 0$ and $T_n \rightarrow \infty$ and $g_{i,n}$ is a differentiable function $\mathbb{R} \times \Theta \rightarrow \mathbb{R}$ such that $|g_{i,n}(x, \theta)| \le c (1 + |x|^c)$ and, for $\vartheta = \mu$ and $\vartheta = \sigma$, $|\partial_\vartheta g_{i,n}(x, \theta) | \le c (1 + |x|^c)$. We define
$$S_n(\theta) := \frac{1}{T_n} \sum_{i = 0}^{n - 1} (X_{t_{i + 1}} - m(\mu, \sigma, X_{t_i})) \varphi_{\Delta_{n,i}^\beta}(X_{t_{i + 1}} - X_{t_i})g_{i,n}(X_{t_i}, \theta).$$
Then $S_n(\theta)$ is tight in $(C(\Theta), \left \| . \right \|_\infty)$.
\label{prop: punto 2 conv deriv seconda}
\end{proposition}

\subsection{Derivatives of $m$ and $m_2$}
We now state some propositions which concern the derivatives of $m$ and $m_2$ that will be useful in the sequel. {\modch We introduce the following notation for the derivative operators $\partial_\vartheta := \frac{\partial}{\partial \vartheta}$, for $\vartheta = \mu$ and $\vartheta = \sigma $.}
\begin{proposition}
Suppose that Assumptions A1-A5 A7 and Ad hold. Then, for $|x| \le \Delta_{n,i}^{-k_0}$ and $\forall \epsilon >0$, we have
$$1. \, \partial_\mu m(\mu, \sigma, X_{t_i}) = \Delta_{n,i} \partial_\mu b(X_{t_i}, \mu) + R(\theta, \Delta_{n,i}^{\frac{5}{2} - \beta - \epsilon}, X_{t_i}),$$
$$2. \,|\partial_\sigma m(\mu, \sigma, X_{t_i})| \le R(\theta, \Delta_{n,i}, X_{t_i}),  $$
$$3. \,|\partial_\mu m_2(\mu, \sigma, X_{t_i})| \le R(\theta, \Delta_{n,i}^2, X_{t_i}),$$
$$4.\,\partial_\sigma m_2(\mu, \sigma, X_{t_i}) = 2\Delta_{n,i} \partial_\sigma a(X_{t_i}, \sigma) a(X_{t_i}, \sigma) + R(\theta, \Delta_{n,i}^{1 + \beta}, X_{t_i}).$$
\label{prop: dl derivate prime}
\end{proposition}
{\modch We denote $\partial^2_\vartheta := \frac{\partial^2}{\partial \vartheta^2}$ and $\partial^2_{\mu \sigma} := \frac{\partial^2}{\partial \mu \partial \sigma}$.} Estimation on the second derivatives are gathered in the following proposition:

\begin{proposition}
Suppose that Assumptions A1 - A5, A7 and Ad hold. Then
\begin{equation}
|\partial^2_{\mu\sigma} m(\mu, \sigma, X_{t_i})| \le R(\theta, \Delta_{n,i}^\frac{3}{2}, X_{t_i}), \qquad |\partial^2_{\sigma} m(\mu, \sigma, X_{t_i})| \le R(\theta, \Delta_{n,i}, X_{t_i}),
\label{eq: deriv seconde miste e sigma m}
\end{equation}
\begin{equation}
\partial^2_\mu m(\mu, \sigma, X_{t_i}) = \Delta_{n,i} \partial^2_\mu b(\mu, X_{t_i}) + R(\theta, \Delta_{n,i}^\frac{3}{2}, X_{t_i}),
\label{eq: derivata seconda mu m}
\end{equation}
\begin{equation}
|\partial^2_{\mu\sigma} m_2(\mu, \sigma, X_{t_i})| \le R(\theta, \Delta_{n,i}^2, X_{t_i}), \qquad
|\partial^2_{\mu} m_2(\mu, \sigma, X_{t_i})| \le R(\theta, \Delta_{n,i}^2, X_{t_i}),
\label{eq: deriv seconde miste e mu m2}
\end{equation}
\begin{equation}
\partial^2_{\sigma} m_2(\mu, \sigma, X_{t_i}) = 2 \Delta_{n,i} \partial_\sigma a(\sigma, X_{t_i}) a(\sigma, X_{t_i}) + R(\theta, \Delta_{n,i}^\frac{3}{2}, X_{t_i}).
\label{eq: deriv seconda sigma m2}
\end{equation}
\label{prop: derivate seconde}
\end{proposition}
Deriving once again, the orders do not get worse. Indeed, the following estimations hold.
\begin{proposition}
Suppose that Assumptions A1 - A5, A7 and Ad hold. Then
$$1. \quad |\partial^3_{\mu} m_2(\mu, \sigma, X_{t_i})| \le R(\theta, \Delta_{n,i}^2, X_{t_i}); \qquad 2. \quad |\partial^3_{\sigma \mu \sigma} m_2(\mu, \sigma, X_{t_i})| \le R(\theta, \Delta_{n,i}^2, X_{t_i}),$$
$$3. \quad |\partial^3_{\mu \mu \sigma} m_2(\mu, \sigma, X_{t_i})| \le R(\theta, \Delta_{n,i}^2, X_{t_i}); \qquad 4. \quad |\partial^3_{\sigma} m_2(\mu, \sigma, X_{t_i})| \le R(\theta, \Delta_{n,i}, X_{t_i}),$$
$$5. \quad |\partial^3_{\mu} m(\mu, \sigma, X_{t_i})| \le R(\theta, \Delta_{n,i}, X_{t_i}); \qquad 6. \quad |\partial^3_{\sigma \mu \sigma} m(\mu, \sigma, X_{t_i})| \le R(\theta, \Delta_{n,i}^\frac{3}{2}, X_{t_i}),$$
$$7. \quad |\partial^3_{\mu \mu \sigma} m(\mu, \sigma, X_{t_i})| \le R(\theta, \Delta_{n,i}^\frac{3}{2}, X_{t_i}); \qquad 8. \quad |\partial^3_{\sigma} m(\mu, \sigma, X_{t_i})| \le R(\theta, \Delta_{n,i}, X_{t_i}).$$
\label{prop: dervate terze}
\end{proposition}
{\modch The notation here above proposed for the third derivatives is the natural extension of the one introduced in Proposition \ref{prop: derivate seconde} for the second derivatives.} \\
Propositions \ref{prop: dl derivate prime}, \ref{prop: derivate seconde} and \ref{prop: dervate terze} will be proved in the appendix.

\section{Proof of main results}\label{S:Proof_Main}
We first of all study the asymptotic behaviour of the contrast, from which we find the consistency of our estimator. \\
We underline that, to get the consistency of the drift parameter, the normalization of the contrast function is different than the normalization we use to find the consistency of $\hat{\sigma}_n$. Even if it doesn't seem a natural choice, it works well on the basis of Proposition \ref{prop: LT}.
\subsection{Contrast's convergence}\label{S: contrast conv}
To prove the contrast convergences, the development \eqref{eq: hp dl m2} of $m_2$ will be useful. We have shown in \cite{Chapitre 1} (see \eqref{E:dev_m_chap1} also) that under Assumptions $(A1)$-$(A4)$ the following development of $m(\mu, \sigma, x)$ holds  :
\begin{equation}
m(\mu, \sigma, x) = x + \Delta_{n, i}b(x, \mu)  + R^J(\Delta_{n,i}, x) + r_1(\mu, \sigma, x),
\label{eq: dl m}
\end{equation}
where $r_1(\mu, \sigma, x)$ is a particular $R(\theta, \Delta_{n,i}^{1 + \delta}, X_{t_i})$ function (with $\delta > 0$) and $R^J(\Delta_{n,i}, x)= - \Delta_{n,i}\int_\mathbb{R}z \gamma(x)[1 - \varphi_{\Delta_{n,i}^\beta}(\gamma(x)z)]F(z)dz$; the $J$ underlines that it turns out from a jump term. It has the same properties of the function R defined in Section \ref{S: dl m2} but it does not depend on $\theta$.  \\
Let us now prove the consistency of $\hat{\theta}_n$. The first step are the following lemmas:
\begin{lemma}
Suppose that A1 - A5, $A_{\mbox{Step}}$ and Ad hold. Moreover we suppose that $\beta \in (\frac{1}{4}, \frac{1}{2})$. Then
\begin{equation}
\frac{1}{n} U_n(\mu, \sigma) \xrightarrow{\mathbb{P}} \int_\mathbb{R}[\frac{c(x, \sigma_0)}{c(x, \sigma)} +\log(c(x, \sigma))] \pi(dx),
\label{eq: conv Un consistenza sigma}
\end{equation}
where $c(x, \sigma) = a^2(x, \sigma)$ and $\pi$ is the invariant distribution defined in Lemma \ref{lemma: 2.1 GLM}.
\label{lemma: conv Un consistenza sigma}
\end{lemma}
Lemma \ref{lemma: conv Un consistenza sigma} is useful to prove the consistency of $\hat{\sigma}_n$, while we will use next lemma to show the consistency of $\hat{\mu}_n$
\begin{lemma}
Suppose that A1 - A5, $A_{\mbox{Step}}$ and Ad hold. Moreover we suppose that $\beta \in (\frac{1}{4}, \frac{1}{2})$ and that $2 \delta_1 > 1$. Then
\begin{equation}
\frac{1}{T_n} (U_n(\mu, \sigma) - U_n(\mu_0, \sigma))  \xrightarrow{\mathbb{P}} \int_\mathbb{R}\frac{(b(x, \mu_0)- b(x, \mu))^2}{c(x, \sigma)} \pi(dx) + \int_\mathbb{R} [r(\mu, \sigma, x) - r(\mu_0, \sigma, x)](1 - \frac{c(x, \sigma_0)}{c(x, \sigma)}) \pi(dx),
\label{eq: conv Un consistenza mu}
\end{equation}
where $r(\mu, \sigma, x)$ is the particular $R(\theta, 1, x)$ function who turns out from the development \eqref{eq: hp dl m2} of $m_2$.
\label{lemma: conv Un consistenza mu }
\end{lemma}
\subsubsection{Proof of Lemma \ref{lemma: conv Un consistenza sigma}}
\begin{proof}
We first of all observe that by the equation \eqref{eq: hp dl m2} we have, for $|X_{t_i}| \le \Delta_{n,i}^{- k}$,
$$\frac{1}{m_2(\mu, \sigma, X_{t_i})} = \frac{1}{\Delta_{n,i} c(X_{t_i}, \sigma)( 1 + \Delta_{n,i}^{\delta_1} \frac{r(X_{t_i})}{c(X_{t_i}, \sigma)} + \Delta_{n,i}r(\mu, \sigma, X_{t_i}) + R(\theta, \Delta_{n,i}^{1 + \delta_2}, X_{t_i}))} = $$
\begin{equation}
= \frac{1}{\Delta_{n,i} c(X_{t_i}, \sigma)}(1 - \Delta_{n,i}^{\delta_1} \frac{r(X_{t_i})}{c(X_{t_i}, \sigma)} - \Delta_{n,i}r(\mu, \sigma, X_{t_i}) + R(\theta, \Delta_{n,i}^{2 \delta_1 \land(1 + \delta_2) \land 2}, X_{t_i})).
\label{eq: 1/m2}
\end{equation}
In the sequel we will just write $\bar{r}$ for $2 \delta_1 \land(1 + \delta_2) \land 2$. We observe that, as a consequence of the Assumption $Ad$, we have for both $\vartheta = \mu$ and $\vartheta = \sigma$, $|\partial_\vartheta R(\theta, \Delta_{n,i}^{\bar{r}}, X_{t_i})| \le R(\theta, \Delta_{n,i}^{\bar{r}}, X_{t_i}) $, where the two rest functions are not necessarily the same.\\
Similarly, 
$$\log(\frac{m_2(\mu, \sigma, X_{t_i})}{\Delta_{n,i}}) = \log(c(X_{t_i}, \sigma)) + \log( 1 + \Delta_{n,i}^{\delta_1} \frac{r(X_{t_i})}{c(X_{t_i}, \sigma)} + \Delta_{n,i}r(\mu, \sigma, X_{t_i})+ R(\theta, \Delta_{n,i}^{1 + \delta_2}, X_{t_i})) =$$
\begin{equation}
= \log(c(X_{t_i}, \sigma)) + \Delta_{n,i}^{\delta_1} \frac{r(X_{t_i})}{c(X_{t_i}, \sigma)} +\Delta_{n,i}r(\mu, \sigma, X_{t_i}) + R(\theta, \Delta_{n,i}^{\bar{r}}, X_{t_i}).
\label{eq: log m2}
\end{equation}
Using both \eqref{eq: 1/m2} and \eqref{eq: log m2} in the definition of $U_n(\mu, \sigma)$, we have to show that
$$\frac{1}{n} \sum_{i = 0}^{n - 1} \frac{(X_{t_{i+1}} - m(\mu, \sigma, X_{t_i}))^2}{\Delta_{n,i}c(X_{t_i}, \sigma)}\varphi_{\Delta_{n,i}^\beta}(X_{t_{i+1}} - X_{t_i})1_{\left \{|X_{t_i}| \le \Delta_{n,i}^{-k} \right \} }(1 - \Delta_{n,i}^{\delta_1} \frac{r(X_{t_i})}{c(X_{t_i}, \sigma)} - \Delta_{n,i}r(\mu, \sigma, X_{t_i}) + R(\theta, \Delta_{n,i}^{\bar{r}}, X_{t_i})) +$$
$$ +\frac{1}{n} \sum_{i = 0}^{n - 1}(\log(c(X_{t_i}, \sigma)) + \Delta_{n,i}^{\delta_1} \frac{r(X_{t_i})}{c(X_{t_i}, \sigma)} + \Delta_{n,i}r(\mu, \sigma, X_{t_i}) + R(\theta, \Delta_{n,i}^{\bar{r}}, X_{t_i}))\varphi_{\Delta_{n,i}^\beta}(X_{t_{i+1}} - X_{t_i})1_{\left \{|X_{t_i}| \le \Delta_{n,i}^{-k} \right \} } = : \sum_{j = 1}^8 I_j^n$$
converges to the right hand side of \eqref{eq: conv Un consistenza sigma}.
We know that $I_1^n \xrightarrow{\mathbb{P}} \int_\mathbb{R} (\frac{c(x, \sigma_0)}{c(x, \sigma)}) \pi(dx)$ because of Proposition \ref{prop: LT}.
Using the third point of Proposition \ref{prop: ergodic}, $I_5^n \xrightarrow{\mathbb{P}} \int_\mathbb{R} \log(c(x, \sigma)) \pi(dx)$.
All the other terms converge to zero in norm $1$ and so in probability. Indeed, passing through the conditional expectation and using the first point of Lemma \ref{lemma: conditional expected value} we have
$$\mathbb{E}[|I_2^n|] \le \frac{1}{n}\sum_{i = 0}^{n - 1} \mathbb{E}[|\frac{\Delta_{n,i}^{\delta_1} r( X_{t_i})}{\Delta_{n,i}c^2(X_{t_i}, \sigma)}\mathbb{E}_i[(X_{t_{i+1}} - m(\mu, \sigma, X_{t_i}))^2\varphi_{\Delta_{n,i}^\beta}(X_{t_{i+1}} - X_{t_i})]1_{\left \{|X_{t_i}| \le \Delta_{n,i}^{-k} \right \} }|] \le $$
$$\le \frac{\Delta_n^{\delta_1}}{n} \sum_{i = 0}^{n - 1}\mathbb{E}[|R(\theta, 1, X_{t_i})|] \le c \Delta_n^{\delta_1},$$
reminding that $r(X_{t_i})$ is a function $R(\theta, 1, X_{t_i})$ by its definition and having used the property \eqref{propriety power R} on R, its polynomial growth and the third point of Lemma \ref{lemma: 2.1 GLM}. In the same way we obtain
$$\mathbb{E}[|I_3^n|] \le c \Delta_n \qquad \mbox{and} \qquad \mathbb{E}[|I_4^n|] \le c \Delta_n^{\bar{r}},$$
that goes to zero since $\bar{r} = 2 \delta_1 \land(1 + \delta_2) \land 2$ is always positive. \\
Concerning $I_6^n$, as a consequence of the definition of $r(x)$ and the fact that $\Delta_{n,i} \le \Delta_n$ we have again
$$\mathbb{E}[|I_6^n|] \le \frac{\Delta_n^{\delta_1}}{n} \sum_{i = 0}^{n - 1}\mathbb{E}[|R(\theta, 1, X_{t_i})|] \le c \Delta_n^{\delta_1}, $$
which converges to zero for $n \rightarrow \infty$. Again, acting in the same way we have $$\mathbb{E}[|I_7^n|] \le c \Delta_n \qquad \mbox{and} \qquad \mathbb{E}[|I_8^n|] \le c \Delta_n^{\bar{r}}.$$
Convergence \eqref{eq: conv Un consistenza sigma} follows.
\end{proof}
\subsubsection{Proof of Lemma \ref{lemma: conv Un consistenza mu }}
\begin{proof}
Using again \eqref{eq: 1/m2} and \eqref{eq: log m2} we have that
$$\frac{1}{T_n} (U_n(\mu, \sigma) - U_n(\mu_0, \sigma))= \frac{1}{T_n} \sum_{i = 0}^{n - 1} [\frac{(X_{t_{i+1}} - m(\mu, \sigma, X_{t_i}))^2}{\Delta_{n,i}c(X_{t_i}, \sigma)} - \frac{(X_{t_{i+1}} - m(\mu_0, \sigma, X_{t_i}))^2}{\Delta_{n,i}c(X_{t_i}, \sigma)} ]\varphi_{\Delta_{n,i}^\beta}(\Delta_i X)1_{\left \{|X_{t_i}| \le \Delta_{n,i}^{-k} \right \} }+$$
$$+ \frac{1}{T_n} \sum_{i = 0}^{n - 1} \Delta_{n,i}^{\delta_1} \frac{r(X_{t_i})}{c(X_{t_i}, \sigma)} [\frac{(X_{t_{i+1}} - m(\mu, \sigma, X_{t_i}))^2}{\Delta_{n,i}c(X_{t_i}, \sigma)} - \frac{(X_{t_{i+1}} - m(\mu_0, \sigma, X_{t_i}))^2}{\Delta_{n,i}c(X_{t_i}, \sigma)} ]\varphi_{\Delta_{n,i}^\beta}(\Delta_i X)1_{\left \{|X_{t_i}| \le \Delta_{n,i}^{-k} \right \} } +$$
$$+ \frac{1}{T_n} \sum_{i = 0}^{n - 1} \Delta_{n,i} r(\mu,\sigma, X_{t_i}) [1 - \frac{(X_{t_{i+1}} - m(\mu, \sigma, X_{t_i}))^2}{\Delta_{n,i}c(X_{t_i}, \sigma)}]\varphi_{\Delta_{n,i}^\beta}(\Delta_i X)1_{\left \{|X_{t_i}| \le \Delta_{n,i}^{-k} \right \} }+$$
$$- \frac{1}{T_n} \sum_{i = 0}^{n - 1} \Delta_{n,i} r(\mu_0,\sigma, X_{t_i}) [1 - \frac{(X_{t_{i+1}} - m(\mu_0, \sigma, X_{t_i}))^2}{\Delta_{n,i}c(X_{t_i}, \sigma)}]\varphi_{\Delta_{n,i}^\beta}(\Delta_i X)1_{\left \{|X_{t_i}| \le \Delta_{n,i}^{-k} \right \} } +$$
$$+ \frac{1}{T_n} \sum_{i = 0}^{n - 1} [R((\mu, \sigma), \Delta_{n,i}^{\bar{r}}, X_{t_i}) + R((\mu, \sigma), \Delta_{n,i}^{\bar{r}}, X_{t_i}) \frac{(X_{t_{i+1}} - m(\mu, \sigma, X_{t_i}))^2}{\Delta_{n,i}c(X_{t_i}, \sigma)}]\varphi_{\Delta_{n,i}^\beta}(\Delta_i X)1_{\left \{|X_{t_i}| \le \Delta_{n,i}^{-k} \right \} }+$$
$$+ \frac{1}{T_n} \sum_{i = 0}^{n - 1}[R((\mu_0, \sigma), \Delta_{n,i}^{\bar{r}}, X_{t_i}) + R((\mu_0, \sigma), \Delta_{n,i}^{\bar{r}}, X_{t_i}) \frac{(X_{t_{i+1}} - m(\mu_0, \sigma, X_{t_i}))^2}{\Delta_{n,i}c(X_{t_i}, \sigma)}]\varphi_{\Delta_{n,i}^\beta}(\Delta_i X)1_{\left \{|X_{t_i}| \le \Delta_{n,i}^{-k} \right \} }= : \sum_{j = 1}^6 I_j^n,$$
where we have introduced the notation $\Delta_i X : = X_{t_{i+1}} - X_{t_i} $ and we recall that $\bar{r} = 2 \delta_1 \land (1 + \delta_2) \land 2$. \\
We have already proved in Lemma 4 of \cite{Chapitre 1} that $I_1^n \xrightarrow{\mathbb{P}} \int_\mathbb{R}\frac{(b(x, \mu_0)- b(x, \mu))^2}{c(x, \sigma)} \pi(dx)$. We observe that $I_2^n$ differs from $I_1^n$ only from the presence of $\Delta_{n,i}^{\delta_1} \frac{r(X_{t_i})}{c(X_{t_i}, \sigma)}$ and so, since $\delta_1$ is positive, acting exactly as we did in order to prove the convergence of $I_1^n$ it is possible to show that the added $\Delta_{n,i}^{\delta_1}$ make $I_2^n$ converge to zero in probability. \\
Concerning $I_3^n$,
\begin{equation}
\frac{1}{T_n} \sum_{i = 0}^{n - 1} \Delta_{n,i} r(\mu,\sigma, X_{t_i})\varphi_{\Delta_{n,i}^\beta}(\Delta_i X)1_{\left \{|X_{t_i}| \le \Delta_{n,i}^{-k} \right \} } \xrightarrow{\mathbb{P}} \int_\mathbb{R} r(\mu,\sigma, x) \pi(dx)
\label{eq: conv r}
\end{equation}
as a consequence of the second point of Proposition \ref{prop: ergodic}. Moreover, using the third point of Proposition \ref{prop: LT}, we have that
\begin{equation}
\frac{1}{T_n} \sum_{i = 0}^{n - 1} \Delta_{n,i} r(\mu,\sigma, X_{t_i})\frac{(X_{t_{i+1}} - m(\mu, \sigma, X_{t_i}))^2}{\Delta_{n,i}c(X_{t_i}, \sigma)}\varphi_{\Delta_{n,i}^\beta}(\Delta_i X)1_{\left \{|X_{t_i}| \le \Delta_{n,i}^{-k} \right \} } \xrightarrow{\mathbb{P}} \int_\mathbb{R} r(\mu,\sigma, x) \frac{c(x, \sigma_0)}{c (x, \sigma)} \pi(dx).
\label{eq: conv r con x-m}
\end{equation}
From \eqref{eq: conv r} and \eqref{eq: conv r con x-m} it follows
$$I_3^n \xrightarrow{\mathbb{P}} \int_\mathbb{R} r(\mu,\sigma, x)[1 - \frac{c(x, \sigma_0)}{c (x, \sigma)}] \pi(dx).$$
Acting on $I_4^n$ exactly like we did on $I_3^n$ we get 
$$I_4^n \xrightarrow{\mathbb{P}} \int_\mathbb{R} r(\mu_0,\sigma, x)[1 - \frac{c(x, \sigma_0)}{c (x, \sigma)}] \pi(dx).$$
Concerning $I_5^n$, it is
$$\frac{1}{T_n} \sum_{i = 0}^{n - 1} R(\theta, \Delta_{n,i}^{\bar{r}}, X_{t_i})\varphi_{\Delta_{n,i}^\beta}(\Delta_i X)1_{\left \{|X_{t_i}| \le \Delta_{n,i}^{-k} \right \} }\le \Delta_{n}^{\bar{r} -1} \frac{1}{n}\sum_{i = 0}^{n - 1} R(\theta, 1, X_{t_i})\varphi_{\Delta_{n,i}^\beta}(\Delta_i X),$$
which converges to zero in norm $1$ and so in probability because of the boundedness of $\varphi$, the polynomial growth of $R$, the fact that $\frac{1}{T_n} = O(\frac{1}{n \Delta_n})$ and that $r-1$ is always positive since we have assumed $2 \delta_1 > 1$.
Moreover, passing through the conditional expectation and using the first point of Lemma \ref{lemma: conditional expected value} we have that
$$\frac{1}{T_n} \sum_{i = 0}^{n - 1} \mathbb{E}[R(\theta, \Delta_{n,i}^{\bar{r}}, X_{t_i}) \mathbb{E}_i[\frac{(X_{t_{i+1}} - m(\mu, \sigma, X_{t_i}))^2}{\Delta_{n,i}c(X_{t_i}, \sigma)}\varphi_{\Delta_{n,i}^\beta}(\Delta_i X)]1_{\left \{|X_{t_i}| \le \Delta_{n,i}^{-k} \right \} }] \le \Delta_n^{\bar{r} - 1} \frac{1}{n} \sum_{i = 0}^{n - 1} \mathbb{E}[R(\theta, 1, X_{t_i})] \le c \Delta_n^{ \bar{r} - 1}.$$
We have therefore proved that the second part of $I_5^n$ converges to $0$ in norm $1$ and therefore in probability. It follows $I_5^n \xrightarrow{\mathbb{P}} 0$ and, acting exactly in the same way, we have also $I_6^n \xrightarrow{\mathbb{P}} 0$.
It yields \eqref{eq: conv Un consistenza mu}.
\end{proof}

\subsection{Consistency of the estimator.}{\label{S: consistency}}
In order to prove the consistency of $\hat{\theta}_n$, we need that the convergences \eqref{eq: conv Un consistenza sigma} and \eqref{eq: conv Un consistenza mu} take place in probability uniformly in both the parameters, we want therefore to show the uniformity of the convergence in $\theta$. \\
We regard $\frac{U_n(\mu, \sigma)}{n}$ and $S_n(\mu, \sigma) : = \frac{1}{T_n}( U_n(\mu, \sigma) - U_n(\mu_0, \sigma))$ as random elements taking values in $(C(\Theta), \left \| . \right \|_\infty)$. It suffices to prove the tightness of these sequences; to do it we need some estimations for the derivatives of $m$ and $m_2$ with respect to both the parameters, which are stated in Proposition \ref{prop: dl derivate prime}, that will be proved in the Appendix. Such a proposition will be also useful to study the asymptotic behavior of the derivatives of the contrast function. We observe that, as a consequence of \eqref{eq: dl m} and Proposition \ref{prop: dl derivate prime}, for $\vartheta = \mu$ or $\vartheta = \sigma$ it is $|\partial_\vartheta r_1(\mu, \sigma, x) | \le R(\theta, \Delta_{n,i}, X_{t_i})$.\\
\begin{lemma}
Suppose that Assumption  A1-A5, Ad, $A_{\mbox{Step}}$ and A7 are satisfied. Then, the sequence $\frac{U_n(\mu, \sigma)}{n}$ is tight in $(C(\Theta), \left \| . \right \|_\infty)$.
\label{lemma: tightness per sigma}
\end{lemma}
\begin{proof}
The tightness is implied by $\sup_n \frac{1}{n} \mathbb{E}[\sup_{\mu, \sigma}|\partial_\vartheta U_n(\mu, \sigma)|] < \infty$ (see Corollary B.1 in \cite{Shimizu thesis}), for $\vartheta= \mu$ and $\vartheta= \sigma$. It is
\begin{equation}
\partial_\vartheta U_n(\mu, \sigma) = \sum_{i = 0}^{n - 1} [\frac{-2 \partial_\vartheta m (\mu, \sigma, X_{t_i})(X_{t_{i + 1}} - m(\mu, \sigma, X_{t_i}))}{m_2(\mu, \sigma, X_{t_i})} - \frac{\partial_\vartheta m_2 (\mu, \sigma, X_{t_i})(X_{t_{i + 1}} - m(\mu, \sigma, X_{t_i}))^2}{m^2_2(\mu, \sigma, X_{t_i})} + 
\label{eq: partial theta Un}
\end{equation}
$$ +\frac{\partial_\vartheta  m_2 (\mu, \sigma, X_{t_i})}{ m_2 (\mu, \sigma, X_{t_i})}]\varphi_{\Delta_{n,i}^\beta}(\Delta_i X)]1_{\left \{|X_{t_i}| \le \Delta_{n,i}^{-k} \right \} }.$$
Using the first and the third point of Proposition \ref{prop: dl derivate prime} and the development \eqref{eq: hp dl m2} of $m_2$ it follows
$$\mathbb{E}[\sup_{\mu, \sigma}|\partial_\mu U_n(\mu, \sigma)|] \le \sum_{i = 0}^{n - 1} \mathbb{E}[\sup_{\mu,\sigma} |R(\theta, 1, X_{t_i})(X_{t_{i + 1}} - m(\mu, \sigma, X_{t_i}))\varphi_{\Delta_{n,i}^\beta}(\Delta_i X)|1_{i,n} ] + $$
\begin{equation}
+ \sum_{i = 0}^{n - 1} \mathbb{E}[\sup_{\mu ,\sigma} |R(\theta, 1, X_{t_i})(X_{t_{i + 1}} - m(\mu, \sigma, X_{t_i}))^2\varphi_{\Delta_{n,i}^\beta}(\Delta_i X)| 1_{i,n}] + \sum_{i = 0}^{n - 1}\mathbb{E}[\sup_{\mu, \sigma} |R(\theta, \Delta_{n,i}, X_{t_i})|1_{i,n}], 
\label{eq: deriv Un mu}
\end{equation}
where we have used $1_{i,n}$ instead of $1_{\left \{|X_{t_i}| \le \Delta_{n,i}^{-k} \right \} }$ to shorten the notation. \\
We observe that 
$$\mathbb{E}[\sup_{\mu, \sigma}|R(\theta, 1, X_{t_i})(X_{t_{i + 1}} - m(\mu, \sigma, X_{t_i}))\varphi_{\Delta_{n,i}^\beta}(\Delta_i X)|1_{i,n}] \le $$
$$ \le \mathbb{E}[(\sup_{\mu, \sigma}|R(\theta, 1, X_{t_i})|)(\sup_{\mu, \sigma}|(X_{t_{i + 1}} - m(\mu, \sigma, X_{t_i}))\varphi_{\Delta_{n,i}^\beta}(\Delta_i X)|)1_{i,n}] \le $$
$$ \le \mathbb{E}[(\sup_{\mu, \sigma}|R(\theta, 1, X_{t_i})| 1_{i,n})(|(X_{t_{i + 1}} - m(\mu_0, \sigma, X_{t_i}))\varphi_{\Delta_{n,i}^\beta}(\Delta_i X)|)] +$$
\begin{equation}
+ c\mathbb{E}[(\sup_{\mu ,\sigma}|R(\theta, 1, X_{t_i})|)(\sup_{\mu, \sigma}| m(\mu, \sigma, X_{t_i}) - m(\mu_0, \sigma, X_{t_i}))|)1_{i,n}].
\label{eq: stima sup x-m}
\end{equation}
We can now use Cauchy-Schwartz inequality and \eqref{eq: Ei al quadrato} in Lemma \ref{lemma: conditional expected value} on the first, while on the second we use the development \eqref{eq: dl m} of $m$ getting that \eqref{eq: stima sup x-m} is upper bounded by 
$$c\mathbb{E}[R(\theta, \Delta_{n,i}, X_{t_i})]^\frac{1}{2} + c \mathbb{E}[(\sup_{\mu, \sigma}|R(\theta, 1, X_{t_i})|)(\sup_{\mu, \sigma}|\Delta_{n,i}(b(X_{t_i}, \mu) - b(X_{t_i}, \mu_0)) + r_1(\mu, \sigma, X_{t_i}) - r_1(\mu_0, \sigma, X_{t_i})|)] \le$$
\begin{equation}
\le c \Delta_n^\frac{1}{2} + c\mathbb{E}[\sup_{\mu, \sigma}|R(\theta, \Delta_{n,i}, X_{t_i})|] \le  c \Delta_n^\frac{1}{2} +  c \Delta_n \le  c \Delta_n^\frac{1}{2}, 
\label{eq: fine stima sup x-m}
\end{equation}
where we have also used the boundedness of $\varphi$, the fact that $R$ has polynomial growth uniformly in $\theta$ and the third point of Lemma \ref{lemma: 2.1 GLM} to say that our process has finite moments. \\
In the same way,
\begin{equation}
\mathbb{E}[\sup_{\mu ,\sigma}|R(\theta, 1, X_{t_i})(X_{t_{i + 1}} - m(\mu, \sigma, X_{t_i}))^2\varphi_{\Delta_{n,i}^\beta}(\Delta_i X)|1_{i,n}] \le \mathbb{E}[\sup_{\mu, \sigma}|R(\theta, \Delta_{n,i}, X_{t_i})|] + \mathbb{E}[\sup_{\mu ,\sigma}|R(\theta, \Delta_{n,i}^2, X_{t_i})|] \le  c \Delta_n. 
\label{eq: stima sup (x-m)^2}
\end{equation}
Replacing \eqref{eq: fine stima sup x-m} and \eqref{eq: stima sup (x-m)^2} in \eqref{eq: deriv Un mu} it follows 
$$\sup_n \frac{1}{n} \mathbb{E}[\sup_{\mu, \sigma}|\partial_\mu U_n(\mu, \sigma)|] \le c \Delta_n^\frac{1}{2} \le c < \infty.$$
We can act in the same way on $\partial_\sigma U_n(\mu, \sigma)$. Considering this time the second and the fourth point of Proposition \ref{prop: dl derivate prime} and still using the development \eqref{eq: hp dl m2} of $m_2$ and \eqref{eq: Ei al quadrato} in Lemma \ref{lemma: conditional expected value} it follows 
$$\sup_n \frac{1}{n} \mathbb{E}[\sup_{\mu,\sigma}|\partial_\sigma U_n(\mu, \sigma)|] \le \sup_n \frac{1}{n} \sum_{i = 0}^{n - 1} \mathbb{E}[\sup_{\mu ,\sigma}|R(\theta, \Delta_{n,i}^\frac{1}{2}, X_{t_i}) +R(\theta, 1, X_{t_i})|] \le c < \infty. $$
The tightness is therefore proved.
\end{proof}

\begin{lemma}
Suppose that Assumption  A1-A5, A7, $A_{\mbox{Step}}$ and Ad are satisfied.  We suppose moreover that $\delta_1$ in the Assumption Ad is such that $2 \delta_1 > 1$. Then, the sequence $S_n(\mu, \sigma) = \frac{1}{T_n}( U_n(\mu, \sigma) - U_n(\mu_0, \sigma))$ is tight in $(C(\Theta), \left \| . \right \|_\infty)$.
\label{lemma: tightness per mu}
\end{lemma}
\begin{proof}
We take again the notation used in the proof of Lemma \ref{lemma: conv Un consistenza mu }, for which $S_n(\mu, \sigma) = \sum_{j = 1}^6 I_j^n$. Since the sum of tight sequences is still tight, we will proceed showing that they are all tight. We start with $I_3^n$; acting as we did in Lemma \ref{lemma: tightness per sigma}, we prove that $\sup_n \mathbb{E}[\sup_{\mu ,\sigma}|\partial_\vartheta I_3^n|] < \infty$. 
We observe that, for $\vartheta = \mu$ and $\vartheta = \sigma$, 
$$\partial_\vartheta I_3^n = \frac{1}{T_n} \sum_{i = 0}^{n - 1} \Delta_{n,i}[\partial_\vartheta r(\mu, \sigma, X_{t_i})(1 - \frac{(X_{t_{i+1}} - m(\mu, \sigma, X_{t_i}))^2}{\Delta_{n,i}c(X_{t_i}, \sigma)})  + $$
$$ - r(\mu, \sigma, X_{t_i}) \partial_\vartheta (\frac{(X_{t_{i+1}} - m(\mu, \sigma, X_{t_i}))^2}{\Delta_{n,i}c(X_{t_i}, \sigma)})]\varphi_{\Delta_{n,i}^\beta}(\Delta_i X)1_{\left \{|X_{t_i}| \le \Delta_{n,i}^{-k} \right \} }= : I_{3,1}^n + I_{3,2}^n.$$
On $I_{3,1}^n$ we use the first point of Lemma \ref{lemma: conditional expected value} and that $|\partial_\vartheta r(\mu, \sigma, X_{t_i})| \le R(\theta, 1, X_{t_i})$ as stated in  Assumption Ad to get
\begin{equation}
\sup_n \mathbb{E}[\sup_{\mu ,\sigma}|I_{3,1}^n|]  \le c + \sup_n \frac{1}{n\Delta_n} \sum_{i =0}^{n -1} \mathbb{E}[\mathbb{E}_i[\sup_{\mu ,\sigma}|R(\theta,1, X_{t_i})(X_{t_{i+1}} - m(\mu, \sigma, X_{t_i}))^2\varphi_{\Delta_{n,i}^\beta}(\Delta_i X)|1_{i,n}]] \le c
\label{eq: stima I21 consistenza mu}
\end{equation}
where we have used the polynomial growth of $R$, the third point of Lemma \ref{lemma: 2.1 GLM} and \eqref{eq: stima sup (x-m)^2} and the notation $1_{i,n}$ instead of $1_{\left \{ |X_{t_i}| \le \Delta_{n,i}^{-k} \right \}}$. \\
Concerning $I_{3,2}^n$, the derivatives of $\frac{(X_{t_{i+1}} - m(\mu, \sigma, X_{t_i}))^2}{\Delta_{n,i}c(X_{t_i}, \sigma)}$ with respect to $\mu$ and $\sigma$ are different but in both cases they are upper bounded, using the first and the second point of Proposition \ref{prop: dl derivate prime}, by $|R(\theta, 1, X_{t_i})(X_{t_{i+1}} - m(\mu, \sigma, X_{t_i}))\varphi_{\Delta_{n,i}^\beta}(\Delta_i X)|$. We can therefore use \eqref{eq: stima sup x-m} and \eqref{eq: fine stima sup x-m}, getting
\begin{equation}
\sup_n \mathbb{E}[\sup_{\mu, \sigma}|I_{3,2}^n|]  \le c \sup_n \Delta_n^\frac{1}{2} \le c < \infty.
\label{eq: stima I22 consistenza mu}
\end{equation}
From \eqref{eq: stima I21 consistenza mu} and \eqref{eq: stima I22 consistenza mu} it follows the tightness of $I_3^n$. Acting exactly in the same way on $I_4^n$ it is clear it is tight too. Concerning $I_5^n$ and $I_6^n$, recalling that the function $R(\theta, \Delta_{n,i}^{\bar{r}}, X_{t_i})$ turns out from \eqref{eq: 1/m2} and it is such that its derivatives with respect to both the parameters remains of the same order, we observe it is possible to act like we did on $I_3^n$ getting
$$\sup_n \mathbb{E}[\sup_{\mu, \sigma}|I_{5}^n|] \le c\Delta_n^{ \bar{r} - 1} + c \Delta_n^{ \bar{r} - \frac{1}{2}} < \infty,$$
since we have chosen $2 \delta_1 > 1$ and so $\bar{r} - 1$ is positive.
Clearly the same estimation hold for $I_6^n$.\\
We now prove that $I_1^n$ is tight. To do it we observe that, using the development \eqref{eq: dl m} and the dynamic \eqref{eq: model} of the process $X$ we have
\begin{equation}
X_{t_{i + 1}} - m(\mu, \sigma, X_{t_i}) = \int_{t_i}^{t_{i + 1}} b(X_s, \mu_0) ds + \int_{t_i}^{t_{i + 1}} a(\sigma_0,X_s)dW_s + \int_{t_i}^{t_{i + 1}} \int_{\mathbb{R} \backslash \left \{0 \right \} }
\gamma(X_{s^-})z \tilde{\mu}(ds,dz)+
\label{eq: reformulation x -m}
\end{equation}
$$ - R^J(\Delta_{n,i}, X_{t_i})- \Delta_{n, i}b(X_{t_i}, \mu) - r_1(\mu, \sigma, X_{t_i}).$$
It is worth noting that only the last two terms here above depend on $\mu$ and so replacing \eqref{eq: reformulation x -m} in $I_1^n$ some terms are deleted by compensation. Therefore we can define
$$I_{1,1}^n : = \frac{1}{T_n} \sum_{i = 0}^{n - 1} \frac{\varphi_{\Delta_{n,i}^\beta}(\Delta_i X)1_{\left \{|X_{t_i}| \le \Delta_{n,i}^{-k} \right \} }}{\Delta_{n,i} c (\sigma, X_{t_i})} [\Delta_{n,i}^2 (b^2(X_{t_i}, \mu) - b^2(X_{t_i}, \mu_0)) + r_1^2(\mu, \sigma, X_{t_i}) - r_1^2(\mu_0, \sigma, X_{t_i}) +$$
$$ + 2\Delta_{n,i}b(X_{t_i}, \mu)r_1(\mu, \sigma, X_{t_i})- 2\Delta_{n,i}b(X_{t_i}, \mu_0)r_1(\mu_0, \sigma, X_{t_i})+ 2[\int_{t_i}^{t_{i + 1}} b(X_s, \mu_0) ds + \Delta X_i^J +$$
$$ - R^J(\Delta_{n,i}, X_{t_i})][\Delta_{n,i} (b(X_{t_i}, \mu) - b(X_{t_i}, \mu_0)) + r_1(\mu, \sigma, X_{t_i}) - r_1(\mu_0, \sigma, X_{t_i})],$$
where we have denoted by $\Delta X_i^J$ the jump part in $\Delta X_i$, that is $\int_{t_i}^{t_{i + 1}} \int_{\mathbb{R} \backslash \left \{0 \right \} }
\gamma(X_{s^-})z \tilde{\mu}(ds,dz)$. \\
Moreover we define
$$I_{1,2}^n: = \frac{1}{T_n} \sum_{i = 0}^{n - 1} \frac{2(b(X_{t_i}, \mu) - b(X_{t_i}, \mu_0))\int_{t_i}^{t_{i + 1}} a(\sigma_0,X_s)dW_s}{c (\sigma, X_{t_i})}\varphi_{\Delta_{n,i}^\beta}(\Delta_i X)1_{\left \{|X_{t_i}| \le \Delta_{n,i}^{-k} \right \} },$$
$$I_{1,3}^n: = \frac{1}{T_n} \sum_{i = 0}^{n - 1} \frac{2(r_1(\mu, \sigma, X_{t_i}) - r_1(\mu_0, \sigma, X_{t_i}))\int_{t_i}^{t_{i + 1}} a(\sigma_0,X_s)dW_s}{ \Delta_{n,i} c (\sigma, X_{t_i})}\varphi_{\Delta_{n,i}^\beta}(\Delta_i X)1_{\left \{|X_{t_i}| \le \Delta_{n,i}^{-k} \right \} }.$$
It is $I_1^n = I_{1,1}^n + I_{1,2}^n + I_{1,3}^n$. We are going to prove that $I_{1,1}^n$ is tight showing that the expected value of the derivatives is bounded, like we have already done. On $I_{1,2}^n$ and $I_{1,3}^n$ we will use instead the Kolmogorov criterion for which, if for some positive constant $H$ independent of $n$ and for $m \ge r >2$, $S_n$ is a sequence such that
\begin{equation}
\mathbb{E}[(S_n(\theta))^m] \le H \qquad \forall \theta \in \Theta ,
\label{eq: 1 criterion tightness}
\end{equation}
\begin{equation}
\mathbb{E}[(S_n(\theta_1)- S_n(\theta_2))^m] \le H|\mu_1 - \mu_2|^r + H|\sigma_1 - \sigma_2|^r \qquad \forall \theta_1, \theta_2 \in \Theta,
\label{eq: 2 criterion tightness}
\end{equation}
then $S_n$ is tight. \\
Let us start considering $I_{1,1}^n$: we want to show that $\sup_n \mathbb{E}[\sup_{\mu, \sigma}|\partial_\vartheta I_{1,1}^n|] < \infty$. 
We observe it is 
$$\partial_\mu I_{1,1}^n = \frac{1}{T_n} \sum_{i = 0}^{n - 1} \frac{\varphi_{\Delta_{n,i}^\beta}(\Delta_i X)1_{\left \{|X_{t_i}| \le \Delta_{n,i}^{-k} \right \} }}{\Delta_{n,i} c (\sigma, X_{t_i})} [\Delta_{n,i}^2 (2b \,\partial_\mu b)(X_{t_i}, \mu) + (2 r_1 \partial_\mu r_1)(\mu, \sigma, X_{t_i}) + 2\Delta_{n,i}((\partial_\mu b)(X_{t_i}, \mu) r_1(\mu, \sigma, X_{t_i}) + $$
$$+ b(X_{t_i}, \mu)(\partial_\mu r_1)(\mu, \sigma, X_{t_i})) + 2(\int_{t_i}^{t_{i + 1}} b(X_s, \mu_0) ds + \Delta X_i^J - R^J(\Delta_{n,i}, X_{t_i}))(\Delta_{n,i} \partial_\mu b(X_{t_i}, \mu) + \partial_ \mu r_1(\mu, \sigma, X_{t_i}))];$$
$$\partial_\sigma I_{1,1}^n = \frac{1}{T_n} \sum_{i = 0}^{n - 1} \frac{\varphi_{\Delta_{n,i}^\beta}(\Delta_i X)1_{\left \{|X_{t_i}| \le \Delta_{n,i}^{-k} \right \} }}{\Delta_{n,i} c (\sigma, X_{t_i})} [2 r_1 \partial_\sigma r_1(\mu, \sigma, X_{t_i}) -2 r_1 \partial_\sigma r_1(\mu_0, \sigma, X_{t_i}) + 2\Delta_{n,i}(b(X_{t_i}, \mu) \partial_\sigma r_1(\mu, \sigma, X_{t_i}) + $$
$$ -b(X_{t_i}, \mu_0) \partial_\sigma r_1(\mu_0, \sigma, X_{t_i})) + 2(\int_{t_i}^{t_{i + 1}} b(X_s, \mu_0) ds + \Delta X_i^J -
 R^J(\Delta_{n,i}, X_{t_i}))(\partial_ \sigma r_1(\mu_0, \sigma, X_{t_i}) - \partial_ \sigma r_1(\mu, \sigma, X_{t_i}))] + $$
 $$ - \frac{\partial_\sigma c(X_{t_i}, \sigma) \varphi_{\Delta_{n,i}^\beta}(\Delta_i X)1_{\left \{|X_{t_i}| \le \Delta_{n,i}^{-k} \right \} }}{\Delta_{n,i} c^2 (\sigma, X_{t_i})}[\Delta_{n,i}^2 (b^2(X_{t_i}, \mu) - b^2(X_{t_i}, \mu_0)) + r_1^2(\mu, \sigma, X_{t_i}) -  r_1^2(\mu_0, \sigma, X_{t_i}) +$$
 $$ + 2\Delta_{n,i}b(X_{t_i}, \mu)r_1(\mu, \sigma, X_{t_i}) - 2\Delta_{n,i}b(X_{t_i}, \mu_0)r_1(\mu_0, \sigma, X_{t_i}) + 2(\int_{t_i}^{t_{i + 1}} b(X_s, \mu_0) ds + \Delta X_i^J + $$
 $$ - R^J(\Delta_{n,i}, X_{t_i}))(\Delta_{n,i} (b(X_{t_i}, \mu) - b(X_{t_i}, \mu_0)) + r_1(\mu, \sigma, X_{t_i}) - r_1(\mu_0, \sigma, X_{t_i}))].$$
Using the polynomial growth of $b$ and recalling that $r_1$ is the particular $R(\theta, \Delta_{n,i}^{1 + \delta}, X_{t_i})$ function that turns out from the development \eqref{eq: dl m} of $m$ and it is such that $|\partial_\vartheta r_1(\mu, \sigma, X_{t_i})| \le R(\theta, \Delta_{n,i},X_{t_i})$ as a consequence of the first two points of Proposition \ref{prop: dl derivate prime}, we get
$$|\partial_\vartheta I_{1,1}^n| \le \frac{c}{n\Delta_n} \sum_{i = 0}^{n - 1} |\varphi_{\Delta_{n,i}^\beta}(\Delta_i X)|[R(\theta, \Delta_{n,i}, X_{t_i}) + R(\theta, \Delta_{n,i}^{1+ \delta}, X_{t_i}) + $$
$$ + 2(|\int_{t_i}^{t_{i + 1}} b(X_s, \mu_0) ds| + |\Delta X_i^J| + R^J(\Delta_{n,i}, X_{t_i}))(R(\theta, 1, X_{t_i}) + R(\theta, \Delta_{n,i}^{ \delta}, X_{t_i}))].$$
Using Lemma \ref{lemma: estim jumps}, the boundedness of $\varphi$, the fact that $\frac{1}{T_n} = O(\frac{1}{n \Delta_n})$ and that $2\mathbb{E}_i[|\int_{t_i}^{t_{i + 1}} b(X_s, \mu_0) ds|]$ is a $R(\theta_0, \Delta_{n,i}, X_{t_i})$, it follows 
$$\sup_n \mathbb{E}[\sup_{\mu, \sigma}|\partial_\vartheta I_{1,1}^n| ] \le \sup_n (\frac{1}{n\Delta_n} \sum_{i = 0}^{n - 1} \mathbb{E}[\sup_{\mu, \sigma}|R(\theta, \Delta_{n,i}, X_{t_i}) + (R(\theta_0, \Delta_{n,i}, X_{t_i}) + R^J(\Delta_{n,i}, X_{t_i}))R(\theta, 1, X_{t_i})|] +   $$
$$+ \frac{1}{n\Delta_n} \sum_{i = 0}^{n - 1} \mathbb{E}[(\sup_{\mu, \sigma}|R(\theta, 1, X_{t_i})|)\mathbb{E}_i[|\Delta X_i^J\varphi_{\Delta_{n,i}^\beta}(\Delta_i X)| ]]) \le c, $$
where in the last inequality we have used Lemma \ref{lemma: estim jumps} here above, the polynomial growth of $R$ uniform in $\theta$ and the third point of Lemma \ref{lemma: 2.1 GLM}. $I_{1,1}^n$ is therefore tight. \\
We now show that \eqref{eq: 1 criterion tightness} and \eqref{eq: 2 criterion tightness} hold on $I_{1,2}^n$. Indeed, using Burkholder and Jensen inequalities, we get $\mathbb{E}[|I_{1,2}^n(\theta_1) - I_{1,2}^n(\theta_2)|^m] \le$
\begin{equation}
\le \frac{c}{n^m \Delta_n^m} n^{\frac{m}{2} - 1} \sum_{i = 0}^{n - 1} \mathbb{E}[|\frac{b(X_{t_i}, \mu_1) - b(X_{t_i}, \mu_0)}{c(\sigma_1, X_{t_i})} - \frac{b(X_{t_i}, \mu_2) - b(X_{t_i}, \mu_0)}{c(\sigma_2, X_{t_i})}|^m |\int_{t_i}^{t_{i + 1}} a(\sigma_0,X_s)dW_s|^m |\varphi_{\Delta_{n,i}^\beta}(\Delta_i X)|^m ].
\label{eq: stima inizio I12 Kolmogorov}
\end{equation}
We observe that, as a consequence of the finite-increments theorem, we have \begin{equation}
|\frac{b(X_{t_i}, \mu_1) - b(X_{t_i}, \mu_0)}{c(\sigma_1, X_{t_i})} - \frac{b(X_{t_i}, \mu_2) - b(X_{t_i}, \mu_0)}{c(\sigma_2, X_{t_i})}|^m \le |\frac{\partial_\mu b(X_{t_i}, \tilde{\mu})}{c(X_{t_i}, \tilde{\sigma})}(\mu_1 - \mu_2) + 
\label{eq: Lagrange}
\end{equation}
$$ - \frac{(b(X_{t_i}, \tilde{\mu}) - b(X_{t_i}, \mu_0)) \partial_\sigma c(X_{t_i}, \tilde{\sigma})}{c(X_{t_i}, \tilde{\sigma})}(\sigma_1 - \sigma_2)|^m \le R(\theta, 1, X_{t_i})|\mu_1 - \mu_2|^m + R(\theta, 1, X_{t_i})|\sigma_1 - \sigma_2|^m,$$
where actually the functions $R$ are calculated in a point $\tilde{\theta}: = (\tilde{\mu}, \tilde{\sigma})$, with $\tilde{\mu} \in (\mu_1, \mu_2)$ and $ \tilde{\sigma} \in (\sigma_1, \sigma_2)$ but, since the property \eqref{eq: definition R} of $R$ is uniform in $\theta$, we have chosen to write it simply as $R(\theta, 1, X_{t_i})$. Using also the boundedness of $\varphi$, we get that \eqref{eq: stima inizio I12 Kolmogorov} is upper bounded by 
$$\frac{c n^{\frac{m}{2} - 1}}{n^m \Delta_n^m}  \sum_{i = 0}^{n - 1} \mathbb{E}[|\int_{t_i}^{t_{i + 1}} a(\sigma_0,X_s)dW_s|^m R(\theta, 1, X_{t_i})]|\mu_1 - \mu_2)
|^m + \mathbb{E}[|\int_{t_i}^{t_{i + 1}} a(\sigma_0,X_s)dW_s|^m R(\theta, 1, X_{t_i})]|\sigma_1 - \sigma_2|^m. $$
Using Burkholder-Davis-Gundy inequality we have, $\forall p \ge 2$, 
\begin{equation}
\mathbb{E}[(\int_{t_i}^{t_{i + 1}} a(\sigma,X_s)dW_s)^p] \le \mathbb{E}[(\int_{t_i}^{t_{i + 1}} a^2(\sigma,X_s)ds)^\frac{p}{2}] \le \mathbb{E}[R(\theta, \Delta_{n,i}, X_{t_i})^\frac{p}{2}] = c\Delta_{n,i}^\frac{p}{2},
\label{eq: BDG}
\end{equation}
where in the last inequality we have used the polynomial growth of $a$ and the third point of Lemma \ref{lemma: 2.1 GLM}. \\
From Holder inequality and \eqref{eq: BDG} it therefore follows
$$\mathbb{E}[|I_{1,2}^n(\theta_1) - I_{1,2}^n(\theta_2)|^m] \le \frac{c}{(n \Delta_n)^\frac{m}{2}}|\mu_1 - \mu_2|^m + \frac{c}{(n \Delta_n)^\frac{m}{2}}|\sigma_1 - \sigma_2|^m \le c|\mu_1 - \mu_2|^m + c|\sigma_1 - \sigma_2|^m,$$
where we have also used that $n\Delta_n \rightarrow\infty$ for $n\rightarrow \infty$. For $r: = m$ \eqref{eq: 2 criterion tightness} is proved. \\
Concerning \eqref{eq: 1 criterion tightness}, acting in the same way we get
$$\mathbb{E}[|I_{1,2}^n(\theta)|^m] \le \frac{cn^{\frac{m}{2} - 1}}{n^m \Delta_n^m}  \sum_{i = 0}^{n - 1} \mathbb{E}[R(\theta_1, 1, X_{t_i})^m |\int_{t_i}^{t_{i + 1}} a(\sigma_0,X_s)dW_s|^m |\varphi_{\Delta_{n,i}^\beta}(\Delta_i X)|^m ] \le \frac{c}{(n \Delta_n)^\frac{m}{2}} \le c.$$
$I_{1,2}^n$ is hence tight. The tightness of $I_{1,3}^n$ is obtained acting exactly in the same way, remarking that
$$|\frac{r_1(\mu_0,\sigma_1, X_{t_i}) - r_1(\mu_1,\sigma_1, X_{t_i})}{ \Delta_{n,i} c(\sigma_1, X_{t_i})} - \frac{r_1(\mu_0,\sigma_2, X_{t_i}) - r_1(\mu_2,\sigma_2, X_{t_i})}{ \Delta_{n,i} c(\sigma_2, X_{t_i})}|^m \le |\frac{\partial_\mu r_1(\tilde{\mu},\tilde{\sigma}, X_{t_i})}{\Delta_{n,i} c(X_{t_i}, \tilde{\sigma})}(\mu_1 - \mu_2) + $$
\begin{equation}
+ [\frac{\partial_\sigma r_1(\mu_0,\tilde{\sigma}, X_{t_i}) - \partial_\sigma r_1(\tilde{\mu},\tilde{\sigma}, X_{t_i})}{ \Delta_{n,i} c(\tilde{\sigma}, X_{t_i})} - \frac{\partial_\sigma c(\tilde{\sigma}, X_{t_i}) (r_1(\mu_0,\tilde{\sigma}, X_{t_i}) - r_1(\tilde{\mu}, \tilde{\sigma}, X_{t_i}))}{ \Delta_{n,i} c(\tilde{\sigma}, X_{t_i})}](\sigma_1 - \sigma_2)|^m \le
\label{eq: kolm estim i13}
\end{equation}
$$\le R(\theta, 1, X_{t_i})|\mu_1 - \mu_2|^m + R(\theta, 1, X_{t_i})|\sigma_1 - \sigma_2|^m,$$
as a consequence of the fact that $(r_1(\mu, \sigma, X_{t_i}))^m$  and $(\partial_\vartheta r_1(\mu, \sigma, X_{t_i}))^m$ are respectively upper bounded by $R(\theta, \Delta_{n,i}^{m(1 + \delta)}, X_{t_i})$ and $R(\theta, \Delta_{n,i}^{m}, X_{t_i})$. \\
Concerning $I_2^n$, we act like we did on $I_1^n$. We still use \eqref{eq: reformulation x -m} getting $I_{2,1}^n$, $I_{2,2}^n$ and $I_{2,3}^n$. We observe that, if we define $s_j^n$ as $I_{1,j}^n = : \sum_{i = 0}^{n - 1}s_j^n $, then $I_{2,j}^n = \sum_{i = 0}^{n - 1} \Delta_{n,i}^\delta \frac{r(X_{t_i})}{c(\sigma, X_{t_i})} s_j^n $. \\
By the computation of $\partial_\mu I_{2,1}^n$ and $\partial_\sigma I_{2,1}^n$ it follows that 
$$\sup_n \mathbb{E}[\sup_{\mu, \sigma}|\partial_\vartheta I_{2,1}^n| ] \le \sup_n (c \Delta_n^\delta + c \Delta_n^{\delta + \beta }) \le c. $$
In order to prove that also $I_{2,2}^n$ and $I_{2,3}^n$ are tight we still use Kolmogorov criterion. From \eqref{eq: Lagrange} and \eqref{eq: BDG} it follows 
$$\mathbb{E}[|I_{2,2}^n(\theta_1) - I_{2,2}^n(\theta_2)|^m] \le c \frac{\Delta_n^{\delta m}}{(n\Delta_n)^\frac{m}{2}}|\mu_1 - \mu_2|^m + c \frac{\Delta_n^{\delta m}}{(n\Delta_n)^\frac{m}{2}}|\sigma_1 - \sigma_2|^m \le c|\mu_1 - \mu_2|^m + c|\sigma_1 - \sigma_2|^m $$
and $\mathbb{E}[(I_{2,2}^n(\theta))^m] \le  c$. \\
The tightness of $I_{2,3}^n$ is obtained in the same way, through Kolmogorov criterion and \eqref{eq: kolm estim i13}. \\
The sequence $S_n$ is therefore tight in $(C(\Theta), \left \| . \right \|_\infty)$, as we wanted.
\end{proof}

\subsubsection{Proof of Theorem \ref{th: consistency}.}
\begin{proof}
Let us begin with the consistency of $\hat{\sigma_n}$. An application of lemmas
\ref{lemma: conv Un consistenza sigma} and \ref{lemma: tightness per sigma} yields
\begin{equation}
\frac{1}{n}U_n(\mu, \sigma) \xrightarrow{\mathbb{P}} U(\sigma, \sigma_0) := \int_\mathbb{R}[\frac{c(x, \sigma_0)}{c(x, \sigma)} + \log(c(x, \sigma))] \pi(dx)
\label{eq: consistenza sigma}
\end{equation}
uniformly in $\theta$. \\
In order to prove that the uniform convergence here above implies the consistency of $\hat{\sigma_n}$, since the convergence in probability is equivalent to the existence, for any subsequence, of a subsequence converging almost surely, we will consider that the convergence in \eqref{eq: consistenza sigma} is almost sure and prove that it implies that $\hat{\sigma}_n \rightarrow \sigma_0$ almost surely. For a fixed $\omega$, thanks to the compactness of $\Theta$, there exists a subsequence $n_k$ such that $(\hat{\mu}_{n_k}, \hat{\sigma}_{n_k})$ tends to a limit $\theta_\infty := (\mu_\infty, \sigma_\infty)$. Hence,
\eqref{eq: consistenza sigma} together with the continuity of $\sigma \mapsto U(\sigma, \sigma_0)$, implies
$$\frac{1}{n_k} U_{n_k}(\hat{\mu}_{n_k}, \hat{\sigma}_{n_k})(\omega) \rightarrow U(\sigma_\infty, \sigma_0).$$
But, by the definition of our estimator $\hat{\theta}_n$, 
$$\frac{1}{n_k} U_{n_k}(\hat{\mu}_{n_k}, \hat{\sigma}_{n_k}) \le \frac{1}{n_k} U_{n_k}(\hat{\mu}_{n_k}, \sigma_0).$$
So, using again the convergence \eqref{eq: consistenza sigma}, we get $U(\sigma_\infty, \sigma_0) \le U(\sigma_0, \sigma_0)$. On the other hand, since for all $y> 0$, $y_0 > 0$ it is $\frac{y_0}{y} + \log(y) \ge 1 + \log(y_0) $ we deduce, using also the identifiability stated in Assumption A6 and Proposition 8.1 in Supplemental materials of \cite{GLM}, that $\sigma_\infty = \sigma_0$. We have proved that any convergent subsequence of $\hat{\sigma}_n$ tends to $\sigma_0$, hence $\hat{\sigma}_n \xrightarrow{\mathbb{P}} \sigma_0$ and we are done. \\
Concerning the consistency of $\hat{\mu}_n$, we have from Lemmas \ref{lemma: conv Un consistenza mu } and \ref{lemma: tightness per mu} that the convergence \eqref{eq: conv Un consistenza mu} holds uniformly in $\theta$. In order to deduce the consistency of $\hat{\mu}_n$ the method is similar to the previous one. We know now that $(\hat{\mu}_{n_k}, \hat{\sigma}_{n_k})$ tends to $(\mu_\infty, \sigma_0)$, hence
$$\frac{1}{T_{n_k}} (U_{n_k}(\hat{\mu}_{n_k}, \hat{\sigma}_{n_k}) - U_{n_k}(\mu_0, \hat{\sigma}_{n_k})) \xrightarrow{\mathbb{P}} \int_\mathbb{R} \frac{(b(x, \mu_0) - b(x, \mu_\infty))^2}{c(x, \sigma_0)} \pi(dx) \ge 0.$$
But $U_{n_k}(\hat{\mu}_{n_k}, \hat{\sigma}_{n_k}) - U_{n_k}(\mu_0, \hat{\sigma}_{n_k}) \le 0$ and so we conclude by A6, getting $\mu_\infty = \mu_0$ and therefore the consistency of $\hat{\mu}_n$. 
\end{proof}

\subsection{Asymptotic normality of the estimator.}\label{S: As norm}
The proof of the asymptotic normality goes along a classical route (see for instance Section 5a of \cite{Genon Catalot}). We define the following notations:
$$M_n := \begin{pmatrix}
\frac{1}{\sqrt{T_n}} & 0 \\
0 & \frac{1}{\sqrt{n}}
\end{pmatrix}.$$
Let 
$$S_n := \begin{pmatrix}
\sqrt{T_n} (\hat{\mu}_n - \mu_0) \\
\sqrt{n} (\hat{\sigma}_n - \sigma_0)
\end{pmatrix}, \qquad
L_n(\theta_0) := \begin{pmatrix}
- \frac{1}{\sqrt{T_n}} \partial_\mu U_n (\mu_0, \sigma_0) \\
- \frac{1}{\sqrt{n}} \partial_\sigma U_n (\mu_0, \sigma_0)
\end{pmatrix}
$$
and 
$$C_n(\theta) = \begin{pmatrix}
\frac{1}{T_n} \frac{\partial^2}{\partial \mu^2 } U_n (\mu, \sigma) & \frac{1}{ \sqrt{n T_n}} \frac{\partial^2}{\partial \mu \sigma } U_n (\mu, \sigma) \\
\frac{1}{ \sqrt{n T_n}} \frac{\partial^2}{\partial \mu \sigma } U_n (\mu, \sigma) & \frac{1}{n } \frac{\partial^2}{\partial \sigma^2 } U_n (\mu, \sigma)
\end{pmatrix}.$$
Then 
\begin{equation}
M_n \nabla^2_\theta U_n(\mu, \sigma) M_n =  C_n(\theta).
\label{eq: matrici}
\end{equation}
Now, by Taylor's formula,
$$\int_0^1 \nabla^2_\theta U_n(\theta_0 + u (\hat{\theta}_n - \theta_0)) du \begin{pmatrix}
\hat{\mu}_n - \mu_0 \\
\hat{\sigma}_n - \sigma_0
\end{pmatrix} = - \nabla_\theta U_n (\theta_0),$$
since $\nabla_\theta U_n (\hat{\theta}_n) = 0$. Then, using \eqref{eq: matrici}, we have 
\begin{equation}
\int_0^1 C_n (\theta_0 + u (\hat{\theta}_n - \theta_0)) du S_n = L_n(\theta_0).
\label{eq: Taylor per normalita}
\end{equation}
We deduce from this equality that, in order to prove the asymptotic normality of $\hat{\theta}_n$ and hence to end the proof of Theorem \ref{th: normality}, it is enough to prove the following lemmas: 
\begin{lemma}
Suppose that Assumptions A1-A8 and Ad hold. Then, as $n \rightarrow \infty$,
$$L_n (\theta_0) \xrightarrow{d} L \sim N(0, K'),$$
where $K' = \begin{pmatrix} 
4 \int_\mathbb{R} (\frac{\partial_\mu b(x, \mu_0)}{a(x, \sigma_0)})^2 \pi(dx) & 0 \\
0 & 8 \int_\mathbb{R} (\frac{\partial_\sigma a(x, \sigma_0)}{a(x, \sigma_0)})^2 \pi(dx) 
\end{pmatrix}$.
\label{lemma: normality derivative contrast}
\end{lemma}
\begin{lemma}
Suppose that Assumptions A1-A8 and Ad hold. Then the following statements hold:
$$1. \, C_n (\theta_0) \xrightarrow{\mathbb{P}} B = \begin{pmatrix} 
-2 \int_\mathbb{R} (\frac{\partial_\mu b(x, \mu_0)}{a(x, \sigma_0)})^2 \pi(dx) & 0 \\
0 & 4 \int_\mathbb{R} (\frac{\partial_\sigma a(x, \sigma_0)}{a(x, \sigma_0)})^2 \pi(dx) 
\end{pmatrix},$$
$$2. \, \sup_{\left \{|\tilde{\theta}| \le \epsilon_n \right \}} |C_n (\theta_0 + \tilde{\theta}) - C_n(\theta_0)| \xrightarrow{\mathbb{P}} 0, \quad \mbox{where } \epsilon_n \rightarrow 0. $$
\label{lemma: derivate seconde contrasto}
\end{lemma}
\subsubsection{Proof of Lemma \ref{lemma: normality derivative contrast}.}
\begin{proof}
As a consequence of a combination of Theorem 3.2 and Theorem 3.4 in 
\cite{HH} (c.f. also Section A.2 in the Appendix of \cite{Shimizu thesis}) we get the result if we prove that $L_n(\theta_0)$ is a triangular array of martingale increments such that, for a constant $r> 0$, the following convergences hold. \\
We define $\zeta_i$ and $\tilde{\zeta}_i$ such that
$\partial_\mu U_n (\mu_0, \sigma_0) =: \sum_{i = 0}^{n -1} \zeta_i (\theta_0)$ and $\partial_\sigma U_n (\mu_0, \sigma_0) =: \sum_{i = 0}^{n -1} \tilde{\zeta}_i(\theta_0)$. Then it must be
\begin{equation}
\frac{1}{T_n} \sum_{i = 0}^{n -1} \mathbb{E}_i[\zeta_i^2(\theta_0)] \xrightarrow{\mathbb{P}} 4 \int_\mathbb{R} (\frac{\partial_\mu b(x, \mu_0)}{a(x, \sigma_0)})^2 \pi(dx) \qquad \frac{1}{(\sqrt{T_n})^{2 + r}} \sum_{i = 0}^{n -1}\mathbb{E}_i[\zeta_i^{2 + r}(\theta_0)] \xrightarrow{\mathbb{P}} 0,
\label{eq: normalita mu}
\end{equation}
\begin{equation}
\frac{1}{n} \sum_{i = 0}^{n -1} \mathbb{E}_i[\tilde{\zeta}_i^2(\theta_0)] \xrightarrow{\mathbb{P}} 8 \int_\mathbb{R} (\frac{\partial_\sigma a(x, \sigma_0)}{a(x, \sigma_0)})^2 \pi(dx) \qquad \frac{1}{(\sqrt{n})^{2 + r}} \sum_{i = 0}^{n -1} \mathbb{E}_i[\tilde{\zeta}_i^{2 + r}(\theta_0)] \xrightarrow{\mathbb{P}} 0,
\label{eq: normalita sigma}
\end{equation}
\begin{equation}
\frac{1}{\sqrt{n T_n}} \sum_{i = 0}^{n -1} |\mathbb{E}_i[ \zeta_i (\theta_0) \tilde{\zeta}_i (\theta_0)] |\xrightarrow{\mathbb{P}} 0.
\label{eq: derivate miste}
\end{equation}
First of all we observe that $L_n(\theta_0)$ is a triangular array of martingale increments as a consequence of the definitions of $m$ and $m_2$. Indeed, using \eqref{eq: partial theta Un}, we clearly have 
$$\mathbb{E}_i[\zeta_i(\theta_0)] = \frac{-2 \partial_\mu m (\mu_0, \sigma_0, X_{t_i})}{m_2(\mu_0, \sigma_0, X_{t_i})} \mathbb{E}_i[(X_{t_{i + 1}} - m(\mu_0, \sigma_0, X_{t_i}))\varphi_{\Delta_{n,i}^\beta}(\Delta_i X)]1_{\left \{|X_{t_i}| \le \Delta_{n,i}^{-k} \right \} } +$$
$$+ \frac{\partial_\mu  m_2 (\mu_0, \sigma_0, X_{t_i})}{ m_2 (\mu_0, \sigma_0, X_{t_i})}\mathbb{E}_i[(1 - \frac{(X_{t_{i + 1}} - m(\mu_0, \sigma_0, X_{t_i}))^2}{m_2(\mu_0, \sigma_0, X_{t_i})})\varphi_{\Delta_{n,i}^\beta}(\Delta_i X)]1_{\left \{|X_{t_i}| \le \Delta_{n,i}^{-k} \right \} } = 0.$$
In the same way, computing the derivative with respect to $\sigma$ we clearly have $\mathbb{E}_i[\tilde{\zeta}_i (\theta_0)] = 0$. \\
Concerning $\partial_\mu U_n$, using \eqref{eq: 1/m2} we can see $\zeta_i (\theta_0)$ as
$$ \frac{-2 \partial_\mu m (\mu_0, \sigma_0, X_{t_i})}{\Delta_{n,i} c(X_{t_i}, \sigma_0)}(X_{t_{i + 1}} - m(\mu_0, \sigma_0, X_{t_i}))\varphi_{\Delta_{n,i}^\beta}(\Delta_i X)1_{\left \{|X_{t_i}| \le \Delta_{n,i}^{-k} \right \} } + R_{i,n}(\theta_0) =: \hat{\zeta}_i(\theta_0) + R_{i,n}(\theta_0),$$
we have already proved in Lemma 6 of \cite{Chapitre 1} the asymptotic normality of $\frac{1}{\sqrt{T_n}} \sum_{i = 0}^{n-1} \hat{\zeta}_i(\theta_0)$ and, in particular, that convergences \eqref{eq: normalita mu} hold with $\hat{\zeta}_i(\theta_0)$ instead of $\zeta_i(\theta_0)$.\\
In order to conclude the proof of \eqref{eq: normalita mu}, it is enough to have $\frac{1}{T_n}\sum_{i = 0}^{n - 1}\mathbb{E}_i[R_{i,n}^2(\theta_0)] \xrightarrow{\mathbb{P}} 0$ and \\
$(\frac{1}{\sqrt{T_n}})^{2 + r} \sum_{i = 0}^{n - 1}\mathbb{E}_i[R_{i,n}^{2 + r}(\theta_0)] \xrightarrow{\mathbb{P}} 0$. 
It is
 $$\frac{1}{T_n}\sum_{i = 0}^{n - 1}\mathbb{E}_i[R_{i,n}^2(\theta_0)] \le \frac{c}{n \Delta_n}\sum_{i = 0}^{n - 1} (\frac{\partial_\mu m (\mu_0, \sigma_0, X_{t_i})R(\theta, \Delta_{n,i}^{\delta_1 \land 1} X_{t_i})}{\Delta_{n,i} c(X_{t_i}, \sigma_0)})^2 \mathbb{E}_i[(X_{t_{i + 1}} - m(\mu_0, \sigma_0, X_{t_i}))^2\varphi^2_{\Delta_{n,i}^\beta}(\Delta_i X)] +$$
 $$+ \frac{c}{n \Delta_n}\sum_{i = 0}^{n - 1}(\frac{\partial_\mu  m_2 (\mu_0, \sigma_0, X_{t_i})}{ m_2 (\mu_0, \sigma_0, X_{t_i})})^21_{\left \{|X_{t_i}| \le \Delta_{n,i}^{-k} \right \} } +$$
 $$ +\frac{c}{n \Delta_n}\sum_{i = 0}^{n - 1} (\frac{\partial_\mu  m_2 (\mu_0, \sigma_0, X_{t_i})}{ m_2 (\mu_0, \sigma_0, X_{t_i})})^2 \frac{\mathbb{E}_i[(X_{t_{i + 1}} - m(\mu_0, \sigma_0, X_{t_i}))^4\varphi^2_{\Delta_{n,i}^\beta}(\Delta_i X)]1_{\left \{|X_{t_i}| \le \Delta_{n,i}^{-k} \right \} }}{m^2_2(\mu_0, \sigma_0, X_{t_i})} $$
 As a consequence of the first and the third point of Proposition \ref{prop: dl derivate prime} and using first and second point of Lemma \ref{lemma: conditional expected value} it is upper bounded by
 $$\frac{c}{n}\sum_{i = 0}^{n - 1} R(\theta_0, \Delta_{n,i}^{2 \delta_1 \land 2}, X_{t_i} ) + \frac{c}{n \Delta_n}\sum_{i = 0}^{n - 1} R(\theta_0, \Delta_{n,i}^2, X_{t_i} ), $$
 which converges to zero in norm $1$ and so in probability. \\
 Acting in the same way, using this time the fourth point of Lemma \ref{lemma: conditional expected value} twice, for $k = (2 + r)$ and $k= 2(2 + r)$, it follows that also $(\frac{1}{\sqrt{T_n}})^{2 + r} \sum_{i = 0}^{n - 1}\mathbb{E}_i[R_{i,n}^{2 + r}(\theta_0)]$ goes to zero in probability. \\ 
 Concerning the derivative of the contrast with respect to $\sigma$, it is 
 {\modch 
 $$\frac{1}{n}\sum_{i = 0}^{n - 1} \mathbb{E}_i[\tilde{\zeta}^2_i(\theta_0)]= \frac{1}{n}\sum_{i = 0}^{n - 1} \mathbb{E}_i [(A_i^2 + 2 A_i B_i + B_i^2)\varphi^2_{\Delta_{n,i}^\beta}(\Delta_i X)]1_{\left \{|X_{t_i}| \le \Delta_{n,i}^{-k} \right \} },$$
 where we have defined
 $$A_i := \frac{-2 \partial_\sigma m (\mu_0, \sigma_0, X_{t_i})(X_{t_{i + 1}} - m(\mu_0, \sigma_0, X_{t_i}))}{m_2(\mu_0, \sigma_0, X_{t_i})}$$
 and 
 $$B_i := (\frac{\partial_\sigma  m_2 (\mu_0, \sigma_0, X_{t_i})}{ m_2 (\mu_0, \sigma_0, X_{t_i})}) (1 - \frac{(X_{t_{i + 1}} - m(\mu_0, \sigma_0, X_{t_i}))^2}{m_2(\mu_0, \sigma_0, X_{t_i})}).$$
 By the development \eqref{eq: hp dl m2} of $m_2$, the second point of Proposition \ref{prop: dl derivate prime} and equation \eqref{eq: Ei al quadrato} in Lemma \ref{lemma: conditional expected value} we have 
 $$\frac{1}{n}\sum_{i = 0}^{n - 1}\mathbb{E}_i [A_i^2 \varphi^2_{\Delta_{n,i}^\beta}(\Delta_i X)]1_{\left \{|X_{t_i}| \le \Delta_{n,i}^{-k} \right \} } \le \frac{1}{n} \sum_{i = 0}^{n - 1} R(\theta_0, \Delta_{n,i}, X_{t_i}),$$
  that goes to zero in norm $1$ because of the property \eqref{propriety power R} of $R$, its polynomial growth and the third point of Lemma \ref{lemma: 2.1 GLM}. The convergence to zero in probability follows. \\
  On the mixed term we use the development \eqref{eq: hp dl m2} of $m_2$, the second and the fourth point of Proposition \ref{prop: dl derivate prime} to get respectively an upper bound on the derivatives with respect to $\sigma$ of $m$ and $m_2$  and the first and the fifth point of Lemma \ref{lemma: conditional expected value} to obtain the following:
  $$\frac{1}{n}\sum_{i = 0}^{n - 1} | \mathbb{E}_i [ 2 A_i B_i \varphi^2_{\Delta_{n,i}^\beta}(\Delta_i X)]1_{\left \{|X_{t_i}| \le \Delta_{n,i}^{-k} \right \} }| \le \frac{1}{n}\sum_{i = 0}^{n - 1} \big[ \, |\mathbb{E}_i[(X_{t_{i + 1}} - m(\mu_0, \sigma_0, X_{t_i}))\varphi^2_{\Delta_{n,i}^\beta}(\Delta_i X)]| +$$
  $$ + R(\theta_0, \Delta_{n,i}^{-1}, X_{t_i})|\mathbb{E}_i[(X_{t_{i + 1}} - m(\mu_0, \sigma_0, X_{t_i}))^3 \varphi^2_{\Delta_{n,i}^\beta}(\Delta_i X)]| \, \big] \le \frac{1}{n}\sum_{i = 0}^{n - 1} [R(\theta_0, \Delta_{n,i}, X_{t_i}) + R(\theta_0, \Delta_{n,i}^{\frac{1}{3} + \beta}, X_{t_i})]. $$
  We obtain that the mixed term converges to zero in probability for the same argument we gave for the convergence of $A_i^2$: because of the property \eqref{propriety power R} of $R$, its polynomial growth and the third point of Lemma \ref{lemma: 2.1 GLM} we get the convergence in norm $1$ which implies the convergence in probability. \\
  We now study the convergence of the term 
  $$\frac{1}{n}\sum_{i = 0}^{n - 1} \mathbb{E}_i [ B_i^2 \varphi^2_{\Delta_{n,i}^\beta}(\Delta_i X)]1_{\left \{|X_{t_i}| \le \Delta_{n,i}^{-k} \right \} }.$$
 }
Using the fourth point of Proposition \ref{prop: dl derivate prime}, the development \eqref{eq: hp dl m2} of $m_2$ and \eqref{eq: 1/m2}, it is 
\begin{equation}
\frac{1}{n} \sum_{i = 0}^{n - 1} [(\frac{2 \partial_\sigma a (X_{t_i}, \sigma_0) a(X_{t_i}, \sigma_0)}{c(X_{t_i}, \sigma_0)})^2 + R(\theta_0, \Delta_{n,i}^{\beta \land \delta_1}, X_{t_i})]\mathbb{E}_i[(1 - \frac{(X_{t_{i + 1}} - m(\mu_0, \sigma_0, X_{t_i}))^2}{m_2(\mu_0, \sigma_0, X_{t_i})})^2\varphi^2_{\Delta_{n,i}^\beta}(\Delta_i X)]1_{\left \{|X_{t_i}| \le \Delta_{n,i}^{-k} \right \} }.
\label{eq: deriv sigma avanti}
\end{equation}
We now need the following lemma:
\begin{lemma}
Suppose that Assumptions A1-A4 hold. Then, $\forall q \ge 1$,
$$\mathbb{E}_i[|\varphi_{\Delta_{n,i}^\beta}(\Delta_i X)|^q] = 1 + R(\theta, \Delta_{n,i}, X_{t_i}).$$
\label{lemma: esperance varphi}
\end{lemma}
\begin{proof}\textit{Lemma \ref{lemma: esperance varphi}.} \\
We can see $\mathbb{E}_i[|\varphi_{\Delta_{n,i}^\beta}(\Delta_i X)|^q]$ as 
$1 + \mathbb{E}_i[|\varphi_{\Delta_{n,i}^\beta}(\Delta_i X)|^q-1]$. \\
Because of the definition of $\varphi$,the expected value here above is different from zero only if $|\Delta_i X| \ge \Delta_{n,i}^\beta$. Hence, using \eqref{eq: stima prob incrementi} it is 
$$\mathbb{E}_i[|\varphi_{\Delta_{n,i}^\beta}(\Delta_i X)|^q-1] \le c \mathbb{E}_i[1_{\left \{ |\Delta_i X| \ge \Delta_{n,i}^\beta \right \}}] \le R(\theta, \Delta_{n,i}, X_{t_i}).$$
\end{proof}
Using the lemma here above, still the development \eqref{eq: hp dl m2} of $m_2$ and \eqref{eq: Ei al quadrato} and \eqref{eq: Ei quarta} in Lemma \ref{lemma: conditional expected value} we have 
$$\mathbb{E}_i[(1 - \frac{(X_{t_{i + 1}} - m(\mu_0, \sigma_0, X_{t_i}))^2}{m_2(\mu_0, \sigma_0, X_{t_i})})^2\varphi^2_{\Delta_{n,i}^\beta}(\Delta_i X)] = 1 + R(\theta_0, \Delta_{n,i}^{1 \land \delta_1}, X_{t_i}) + 3 \frac{a^4(X_{t_i}, \sigma_0)}{c^2(X_{t_i}, \sigma_0)} - 2 \frac{a^2(X_{t_i}, \sigma_0)}{c(X_{t_i}, \sigma_0)} +$$
$$ + R(\theta_0, \Delta_{n,i}^{\delta_1 \land (\beta - \frac{1}{4})}, X_{t_i}) = 2 + R(\theta_0, \Delta_{n,i}^{\delta_1 \land (\beta - \frac{1}{4})}, X_{t_i}),$$ 
we remind that $c(x, \sigma) = a^2(x, \sigma)$. Replacing the last equation in \eqref{eq: deriv sigma avanti} we get
$$\frac{1}{n} \sum_{i = 0}^{n - 1} 2 \cdot \frac{4(\partial_\sigma a (X_{t_i}, \sigma_0))^2 }{c(X_{t_i}, \sigma_0)}1_{\left \{|X_{t_i}| \le \Delta_{n,i}^{-k} \right \} } + \frac{1}{n} \sum_{i = 0}^{n - 1}R(\theta_0, \Delta_{n,i}^{\beta \land \delta_1 \land (\beta - \frac{1}{4})}, X_{t_i}).$$
The first term here above converges to $8 \int_\mathbb{R} (\frac{\partial_\sigma a(x, \sigma_0)}{a(x, \sigma_0)})^2 \pi(dx) $ as a consequence of the third point of Proposition \ref{prop: ergodic} while the second one clearly goes to zero in norm 1 and so in probability thanks to the polynomial growth of $R$, its property \eqref{propriety power R} and the assumption we made $\beta > \frac{1}{4}$. It follows the first convergence of \eqref{eq: normalita sigma}. To obtain the second one we observe it is
$$\frac{1}{n^{1 + \frac{r}{2}}}\sum_{i = 0}^{n - 1} \mathbb{E}_i[\tilde{\zeta}^{2 + r}_i(\theta_0)] \le \frac{1}{n^{1 + \frac{r}{2}}}\sum_{i = 0}^{n - 1} (\frac{-2 \partial_\sigma m (\mu_0, \sigma_0, X_{t_i})}{m_2(\mu_0, \sigma_0, X_{t_i})})^{2 + r} 1_{\left \{|X_{t_i}| \le \Delta_{n,i}^{-k} \right \} }\mathbb{E}_i[(X_{t_{i + 1}} - m(\mu_0, \sigma_0, X_{t_i}))^{2+ r}\varphi^{2 + r}_{\Delta_{n,i}^\beta}(\Delta_i X)]+ $$
$$ + \frac{1}{n^{1 + \frac{r}{2}}} \sum_{i = 0}^{n - 1} (\frac{\partial_\sigma  m_2 (\mu_0, \sigma_0, X_{t_i})}{ m_2 (\mu_0, \sigma_0, X_{t_i})})^{2+ r} 1_{\left \{|X_{t_i}| \le \Delta_{n,i}^{-k} \right \} }(c + \frac{c}{m_2^{2 + r}(\mu_0, \sigma_0, X_{t_i})} \mathbb{E}_i[(X_{t_{i + 1}} - m(\mu_0, \sigma_0, X_{t_i}))^{2(2+ r)}\varphi^{2(2 + r)}_{\Delta_{n,i}^\beta}(\Delta_i X)]) \le $$
\begin{equation}
\le \Delta_n^{(1 + \frac{r}{2})\land (1 + \beta(2 + r))} \frac{1}{n^{1 + \frac{r}{2}}} \sum_{i = 0}^{n - 1} R(\theta_0, 1, X_{t_i}) + \frac{1}{n^{1 + \frac{r}{2}}} \sum_{i = 0}^{n - 1} R(\theta_0, 1, X_{t_i}) + \Delta_n^{0 \land (1 + 2\beta(2 + r) - (2 + r))}  \frac{1}{n^{1 + \frac{r}{2}}} \sum_{i = 0}^{n - 1}R(\theta_0, 1, X_{t_i}), 
\label{eq: deriv sigma 2+r}
\end{equation}
where we have acted like before using the development \eqref{eq: hp dl m2} on $m_2$ and the second and the fourth point of Proposition \ref{prop: dl derivate prime}. Besides, we have used the fourth point of Lemma \ref{lemma: conditional expected value} with $k= 2 + r$ and $k= 2(2 + r)$, respectively. It is now clear that the first two terms of \eqref{eq: deriv sigma 2+r} go to zero in norm $1$ and so in probability for $n\rightarrow \infty$. Concerning the third one, if the minimum between $0$ and $1 + 2\beta(2 + r) - (2 + r)$ is $0$ it is exactly like the second one and so we know it goes to zero in probability, otherwise it can be seen as $\frac{1}{(n\Delta_n)^\frac{r}{2}} \Delta_n^{1 + 2\beta(2 + r) - (2 + r) + \frac{r}{2}} \frac{1}{n}\sum_{i = 0}^{n - 1}R(\theta_0, 1, X_{t_i})$, that goes to zero since $n\Delta_n \rightarrow \infty$ for $n \rightarrow \infty$ and because of the fact that the exponent on $\Delta_n$ is always positive. Indeed $1 + 2\beta(2 + r) - (2 + r) + \frac{r}{2} > 0$ iff $2\beta(2 + r) > 1 + \frac{r}{2}$, that is $\beta > \frac{1 + \frac{r}{2}}{2(2 + r)} = \frac{1}{4}$. \\
To conclude, we prove the convergence \eqref{eq: derivate miste}. We have
$$\frac{1}{\sqrt{n T_n}} \sum_{i = 0}^{n - 1} |\mathbb{E}_i[\zeta_i(\theta_0) \tilde{\zeta}_i(\theta_0)]| =  \frac{1}{\sqrt{ n T_n}} \sum_{i = 0}^{n - 1}  |\mathbb{E}_i \big[ [(\frac{4 \partial_\mu m \, \partial_\sigma m}{m_2^2})(\mu_0, \sigma_0, X_{t_i})(X_{t_{i + 1}} - m(\mu_0, \sigma_0, X_{t_i}))^2 +$$
\begin{equation}
-2( \frac{\partial_\mu m \, \partial_\sigma m_2 + \partial_\mu m_2 \, \partial_\sigma m}{m_2^2})(\mu_0, \sigma_0, X_{t_i})(X_{t_{i + 1}} - m(\mu_0, \sigma_0, X_{t_i}))(1 - \frac{(X_{t_{i + 1}} - m(\mu_0, \sigma_0, X_{t_i}))^2}{m_2(\mu_0, \sigma_0, X_{t_i})}) + 
\label{eq: stima derivate miste}
\end{equation}
$$+ \frac{ \partial_\mu m_2 \, \partial_\sigma m_2}{m_2^2})(\mu_0, \sigma_0, X_{t_i})(1 - \frac{(X_{t_{i + 1}} - m(\mu_0, \sigma_0, X_{t_i}))^2}{m_2(\mu_0, \sigma_0, X_{t_i})})^2 ]\varphi^{2}_{\Delta_{n,i}^\beta}(\Delta_i X) \big]|1_{\left \{|X_{t_i}| \le \Delta_{n,i}^{-k} \right \} }. $$
Now using the four points of Proposition \ref{prop: dl derivate prime}, the first, second, third and fifth points of Lemma \ref{lemma: conditional expected value} and the lemma here above we get that \eqref{eq: stima derivate miste} is upper bounded by
$$\frac{1}{n\sqrt{\Delta_n}} \sum_{i = 0}^{n - 1} R(\theta_0, \Delta_{n,i}, X_{t_i}) + R(\theta_0, \Delta_{n,i}^{\frac{1}{3} + \beta}, X_{t_i}) + R(\theta_0, \Delta_{n,i}^{\frac{4}{3} + \beta}, X_{t_i}) \le \Delta_n^{ \beta - \frac{1}{6}} \frac{1}{n}\sum_{i = 0}^{n - 1} R(\theta_0, 1, X_{t_i}), $$
which converges to zero in norm $1$ and so in probability since we have chosen $\beta > \frac{1}{4} > \frac{1}{6}$.
\end{proof}

\subsubsection{Proof of Lemma \ref{lemma: derivate seconde contrasto}.}
\begin{proof}\textit{Point 1.} \\
We start showing the convergence of $C_n(\theta_0)$ to $B$. We observe that \begin{equation}
\partial^2_{\mu \sigma} U_n (\mu_0, \sigma_0) = \sum_{i = 0}^{n - 1} [\frac{-2(X_{t_{i + 1}} - m) \partial_{\mu \sigma}^2 m}{m_2} -2 \frac{\partial_\mu m \partial_\sigma m}{m_2} + \frac{2(X_{t_{i + 1}} - m)\partial_\mu m \partial_\sigma m_2}{m_2^2} + \frac{2(X_{t_{i + 1}} - m)\partial_\mu m_2 \partial_\sigma m}{m_2^2} +
\label{eq: derivate miste contrasto}
\end{equation}
$$  + \frac{2(X_{t_{i + 1}} - m)^2 \partial_\mu m_2 \partial_\sigma m_2}{m_2^3} + \frac{\partial_{\mu \sigma}^2 m_2}{m_2}(1 - \frac{(X_{t_{i + 1}} - m)^2}{m_2})- \frac{\partial_\mu m_2 \partial_\sigma m_2}{m_2^2}]  \varphi_{\Delta_{n,i}^\beta}(\Delta X_i) 1_{\left \{ |X_{t_i}| \le \Delta_{n,i}^{- k} \right \}}=: \sum_{j= 1}^7 \sum_{i = 0}^{n - 1} I_{i,j}^n$$
where, for shortness, we omit that $m$, $m_2$ and their derivatives are calculated in $(\mu_0, \sigma_0, X_{t_i})$. \\
In order to show that $\frac{1}{\sqrt{n T_n}}\partial^2_{\mu \sigma} U_n (\mu_0, \sigma_0) \xrightarrow{\mathbb{P}} 0$ we will use repeatedly Lemma 9 in \cite{Genon Catalot} and the estimation of the derivatives of $m$ and $m_2$ gathered in Propositions \ref{prop: dl derivate prime} and \ref{prop: derivate seconde}. \\
$\sum_{i = 0}^{n - 1} I_{i,1}^n$ goes to zero in probability for $n \rightarrow \infty$ because $I_{i,1}^n $ is centered and, using also the first point of Lemma \ref{lemma: conditional expected value}
$$\frac{1}{T_n} \frac{1}{n} \sum_{i = 0}^{n - 1} \mathbb{E}_i[( I_{i,1}^n)^2] \le \frac{1}{n \Delta_n} \frac{1}{n} \sum_{i = 0}^{n - 1} R(\theta, \Delta_{n,i}, X_{t_i})R(\theta, \Delta_{n,i}, X_{t_i}) \le \frac{1}{n \Delta_n} \frac{1}{n} \Delta_n^2  \sum_{i = 0}^{n - 1} R(\theta, 1, X_{t_i}), $$
which converges to zero in norm $1$ and so in probability as a consequence of the polynomial growth of $R$ and the third point of Lemma \ref{lemma: 2.1 GLM}. $\sum_{i = 0}^{n - 1} I_{i,2}^n$ goes to zero in norm $1$ and so in probability. Indeed, using the development \eqref{eq: hp dl m2} of $m_2$ and the first two point of Proposition \ref{prop: dl derivate prime}, it is
$$\mathbb{E}[\frac{1}{\sqrt{n T_n}}|\sum_{i = 0}^{n - 1} I_{i,2}^n|] \le \Delta_n^{- \frac{1}{2}} \frac{1}{n} \sum_{i = 0}^{n - 1} \mathbb{E}[|R(\theta, \Delta_n, X_{t_i})|] \le c \Delta_n^\frac{1}{2},$$
which goes to zero. \\
Concerning $I_{i,3}^n$, it is still centered and from the first and fourth points of Proposition \ref{prop: dl derivate prime} and the first point of Lemma \ref{lemma: conditional expected value} it follows
$$\frac{1}{T_n} \frac{1}{n} \sum_{i = 0}^{n - 1} \mathbb{E}_i[( I_{i,3}^n)^2] \le \frac{1}{n \Delta_n} \frac{1}{n} \sum_{i = 0}^{n - 1} R(\theta, 1, X_{t_i})R(\theta, \Delta_{n,i}, X_{t_i}) \le \frac{1}{n \Delta_n} \frac{1}{n} \Delta_n  \sum_{i = 0}^{n - 1} R(\theta, 1, X_{t_i}),$$
which converges to $0$ in norm $1$ and so in probability. Hence, $\frac{1}{ \sqrt{n T_n}} \sum_{i = 0}^{n - 1} I_{i,3}^n \xrightarrow{\mathbb{P}} 0$ from Lemma 9 in \cite{Genon Catalot}. \\
The same applies to $I_{i,4}^n$, which squared is upper bounded by
$\frac{1}{n \Delta_n} \frac{1}{n} \sum_{i = 0}^{n - 1} R(\theta, \Delta_{n,i}^2, X_{t_i})R(\theta, \Delta_{n,i}, X_{t_i}) \le \frac{1}{n \Delta_n} \frac{1}{n} \Delta^3_n  \sum_{i = 0}^{n - 1} R(\theta, 1, X_{t_i})$. \\
On $\sum_{i = 0}^{n - 1} I_{i,5}^n$ we prove the convergence in norm $1$ and so we have the convergence in probability: from the first point of Lemma \ref{lemma: conditional expected value}, the development \eqref{eq: hp dl m2} of $m_2$ and the third and the fourth points of Proposition \ref{prop: dl derivate prime}, it is
$$\frac{1}{\sqrt{n T_n}} \sum_{i = 0}^{n - 1} \mathbb{E}[|I_{i,5}^n| ]\le \frac{1}{n \sqrt{\Delta_n}} \sum_{i = 0}^{n - 1} R(\theta, \Delta_{n,i}, X_{t_i}) R(\theta, 1, X_{t_i}) \le \Delta_n^\frac{1}{2} \frac{1}{n} \sum_{i = 0}^{n - 1}R(\theta, 1, X_{t_i}),  $$
which goes to zero. \\
We observe that, as a consequence of the definition of $m_2$, $I_{i,6}^n$ is centered. In order to apply Lemma 9 in \cite{Genon Catalot} we evaluate its squared value, that is
$$\frac{1}{T_n} \frac{1}{n} \sum_{i = 0}^{n - 1} \mathbb{E}_i[(I_{i,6}^n)^2] \le \frac{1}{n \Delta_n} \frac{1}{n} \sum_{i = 0}^{n - 1}R(\theta, \Delta_{n,i}, X_{t_i}) R(\theta, 1, X_{t_i}) \le \Delta_n \frac{1}{n \Delta_n} \frac{1}{n} \sum_{i = 0}^{n - 1} R(\theta, 1, X_{t_i}),$$
where we have also used the estimation of the mixed second derivative of $m_2$ contained in Proposition \ref{prop: derivate seconde}. \\
To conclude the proof about the mixed derivative of the contrast function, we observe that $\sum_{i = 0}^{n - 1} I_{i,7}^n$ converges to $0$ in $L^1$ from the third and the fourth point of Proposition \ref{prop: dl derivate prime} and the boundedness of $\varphi$. We obtain
$$\frac{1}{ \sqrt{ n T_n}} \mathbb{E}[|\sum_{i = 0}^{n - 1} I_{i,7}^n|] \le \Delta_n^{- \frac{1}{2}} \frac{1}{n} \sum_{i = 0}^{n - 1} \mathbb{E}[R(\theta, \Delta_{n,i}, X_{t_i})] \le c \Delta_n^\frac{1}{2}, $$
that goes to $0$ as we wanted. \\
The next step is to prove that $\frac{1}{T_n} \partial^2_\mu U_n (\mu_0, \sigma_0) \xrightarrow{\mathbb{P}} -2 \int_\mathbb{R} (\frac{\partial_\mu b (x, \mu_0)}{a(x, \sigma_0)})^2 \pi(dx)$. 
In order to do it we compute $\partial^2_\mu U_n (\mu_0, \sigma_0)$ and we observe it is exactly like \eqref{eq: derivate miste contrasto} but all the derivatives are with respect to $\mu$. For such a reason we keep referring to \eqref{eq: derivate miste contrasto} and we write $\partial^2_\mu U_n (\mu_0, \sigma_0) =: \sum_{j= 1}^7 \sum_{i = 0}^{n - 1} \tilde{I}_{i,j}^n$.
We are going to show, in particular, that $\frac{1}{T_n}\sum_{i = 0}^{n - 1} \tilde{I}_{i,2}^n$ converges to the wanted integral, while $\frac{1}{T_n }(\sum_{i = 0}^{n - 1} \tilde{I}_{i,1}^n + \sum_{j= 3}^7 \sum_{i = 0}^{n - 1} \tilde{I}_{i,j}^n) \xrightarrow{\mathbb{P}} 0$. Indeed, we observe that $\tilde{I}_{i,1}^n$, $\tilde{I}_{i,3}^n$, $\tilde{I}_{i,4}^n$, $\tilde{I}_{i,6}^n$ are still centered and, using Lemma \ref{lemma: conditional expected value} and Propositions \ref{prop: dl derivate prime} and \ref{prop: derivate seconde} it is easy to show that their squared values are upper bounded in the following way:
$$\frac{1}{T_n^2} \sum_{i = 0}^{n - 1} \mathbb{E}_i[(\tilde{I}_{i,1}^n)^2] \le \frac{1}{n\Delta_n} \Delta_n^{-1} \frac{1}{n} \sum_{i = 0}^{n - 1} R(\theta, \Delta_{n,i}, X_{t_i})R(\theta,1, X_{t_i})\le \frac{1}{n\Delta_n} \frac{1}{n} \sum_{i = 0}^{n - 1}R(\theta,1, X_{t_i}),$$
that goes to zero in norm $1$ and so in probability since $n\Delta_n \rightarrow \infty$ for $n \rightarrow \infty$. 
$$\frac{1}{T_n^2}  \sum_{i = 0}^{n - 1} \mathbb{E}_i[(\tilde{I}_{i,3}^n)^2] \le \frac{1}{n\Delta_n} \Delta_n^{-1} \frac{1}{n} \sum_{i = 0}^{n - 1}R(\theta, \Delta_{n,i}, X_{t_i})R(\theta, \Delta_{n,i}^2, X_{t_i})\le \Delta_n^2 \frac{1}{n\Delta_n} \frac{1}{n} \sum_{i = 0}^{n - 1}R(\theta, 1, X_{t_i}) \xrightarrow{\mathbb{P}} 0.$$
Now $\tilde{I}_{i,4}^n$ and $\tilde{I}_{i,3}^n$ are exactly the same quantity and so the estimation here above clearly holds also for  $\tilde{I}_{i,4}^n$ instead of $\tilde{I}_{i,3}^n$. Concerning $\tilde{I}_{i,6}^n$, we have 
$$\frac{1}{T_n^2} \sum_{i = 0}^{n - 1} \mathbb{E}_i[(\tilde{I}_{i,6}^n)^2] \le \frac{1}{n\Delta_n} \Delta_n^{-1} \frac{1}{n} \sum_{i = 0}^{n - 1}R(\theta, \Delta_{n,i}^2, X_{t_i})R(\theta, 1, X_{t_i})\le \Delta_n \frac{1}{n\Delta_n} \frac{1}{n} \sum_{i = 0}^{n - 1}R(\theta, 1, X_{t_i}) \xrightarrow{\mathbb{P}} 0.$$
The application of Lemma 9 in \cite{Genon Catalot} gives us $\frac{1}{T_n} \sum_{i = 0}^{n - 1} (\tilde{I}_{i,1}^n + \tilde{I}_{i,3}^n + \tilde{I}_{i,4}^n + \tilde{I}_{i,6}^n )\xrightarrow{\mathbb{P}} 0.$
We now prove the convergence to $0$ in norm $1$ of $\frac{1}{T_n} \sum_{i = 0}^{n - 1} (\tilde{I}_{i,5}^n + \tilde{I}_{i,7}^n)$. Indeed, using again the first point of Lemma \ref{lemma: conditional expected value}, the development \eqref{eq: hp dl m2} of $m_2$ and the last two points of Proposition \ref{prop: derivate seconde} it is $\frac{1}{T_n} \sum_{i = 0}^{n - 1} \mathbb{E}[|\tilde{I}_{i,5}^n + \tilde{I}_{i,7}^n|] \le$
$$\le \frac{1}{n\Delta_n} \sum_{i = 0}^{n - 1} \mathbb{E}[R(\theta, \Delta_{n,i}, X_{t_i})R(\theta, \Delta_{n,i}, X_{t_i}) + R(\theta, \Delta_{n,i}^2, X_{t_i})]  \le \Delta_n \frac{1}{n} \sum_{i = 0}^{n - 1} \mathbb{E}[R(\theta,1, X_{t_i})] \le c \Delta_n,$$
which clearly goes to $0$. \\
Concerning the principal term $\frac{1}{T_n} \sum_{i = 0}^{n - 1} \tilde{I}_{i,2}^n$, we observe that using \eqref{eq: 1/m2} and the first point of Proposition \ref{prop: dl derivate prime}, it is 
$$\frac{(\partial_\mu m)^2}{m_2}(\mu_0, \sigma_0, X_{t_i}) = \frac{(\Delta_{n,i}\partial_\mu b(\mu_0, X_{t_i})+ R(\theta_0, \Delta_{n,i}^{\frac{5}{2} - \beta - \epsilon}, X_{t_i}))^2}{\Delta_{n, i} c(\sigma_0, X_{t_i})} (1 - \Delta_{n,i}^{\delta_1} \frac{r(X_{t_i})}{c(X_{t_i}, \sigma_0)} - \Delta_{n,i}r(\mu_0, \sigma_0, X_{t_i}) +$$
$$ +R(\theta_0, \Delta_{n,i}^{\bar{r}}, X_{t_i})) = \frac{\Delta_{n,i} (\partial_\mu b(\mu_0, X_{t_i}))^2}{c(\sigma_0, X_{t_i})} + R(\theta_0, \Delta_{n,i}^{(\frac{5}{2} - \beta - \epsilon) \land (1 + \delta_1)}, X_{t_i}),$$
with $\bar{r} = 2 \land (1 + \delta_2) \land 2 \delta_1$, as defined below \eqref{eq: 1/m2}. In the last equality we have used that the other terms are negligible.
Now we have that 
$$\frac{1}{T_n} \sum_{i = 0}^{n - 1}-2\frac{\Delta_{n,i} (\partial_\mu b(\mu_0, X_{t_i}))^2}{c(\sigma_0, X_{t_i})} \varphi_{\Delta_{n,i}^\beta}(\Delta X_i) 1_{\left \{ |X_{t_i}| \le \Delta_{n,i}^{- k} \right \}} \xrightarrow{\mathbb{P}}-2 \int_\mathbb{R} (\frac{\partial_\mu b (x, \mu_0)}{a(x, \sigma_0)})^2 \pi(dx), $$
as a consequence of the second point of Proposition \ref{prop: ergodic} while \\
$\frac{1}{n\Delta_n} \sum_{i = 0}^{n - 1}-2R(\theta_0, \Delta_{n,i}^{(\frac{5}{2} - \beta - \epsilon) \land (1 + \delta_1)}, X_{t_i})\varphi_{\Delta_{n,i}^\beta}(\Delta X_i) 1_{\left \{ |X_{t_i}| \le \Delta_{n,i}^{- k} \right \}} $ is upper bounded in norm $1$ by
$$\Delta_n^{(\frac{3}{2} - \beta - \epsilon) \land \delta_1} \frac{1}{n} \sum_{i = 0}^{n - 1} \mathbb{E}[R(\theta_0, 1, X_{t_i})] \le c \Delta_n^{(\frac{3}{2} - \beta - \epsilon) \land \delta_1}, $$
that converges to $0$ since the exponent on $\Delta_n$ is always positive. \\
To prove the first point of Lemma \ref{lemma: derivate seconde contrasto} we are left to show the convergence of $\frac{1}{n}\partial^2_\sigma U_n (\mu_0, \sigma_0)$. \\
Again, we still refer to \eqref{eq: derivate miste contrasto} observing that that the only difference is that all the derivatives are with respect to $\sigma$. We write $\partial^2_\sigma U_n (\mu_0, \sigma_0) =: \sum_{j= 1}^7 \sum_{i = 0}^{n - 1} \hat{I}_{i,j}^n$. \\
We keep using Lemma 9 in \cite{Genon Catalot} joint with the development \eqref{eq: hp dl m2} and Propositions \ref{prop: dl derivate prime} and \ref{prop: derivate seconde} to show that the centered terms go to zero in probability, that is 
$$\frac{1}{n} \sum_{i = 0}^{n -1}(\hat{I}_{i,1}^n + \hat{I}_{i,3}^n + \hat{I}_{i,4}^n +\hat{I}_{i,6}^n)\varphi_{\Delta_{n,i}^\beta}(\Delta X_i) 1_{\left \{ |X_{t_i}| \le \Delta_{n,i}^{- k} \right \}} \xrightarrow{\mathbb{P}} 0.$$
Moreover, 
$$\frac{1}{n} \sum_{i = 0}^{n -1}\mathbb{E}[|\hat{I}_{i,2}^n|] \le \frac{1}{n} \sum_{i = 0}^{n -1}\mathbb{E}[R(\theta, \Delta_{n,i}, X_{t_i})] \le c \Delta_n \rightarrow 0.$$
We are left to deal with the principal terms $\hat{I}_{i,5}^n$ and $\hat{I}_{i,7}^n$ and so we study the convergence of \\
$\frac{1}{n} \sum_{i = 0}^{n -1} \frac{(\partial_\sigma m_2)^2}{m_2^2}(\frac{2(X_{t_{i + 1}} - m)^2}{m_2} -1)\varphi_{\Delta_{n,i}^\beta}(\Delta X_i) 1_{\left \{ |X_{t_i}| \le \Delta_{n,i}^{- k} \right \}}$.
From the development \eqref{eq: hp dl m2} of $m_2$ and the development of $\partial_\sigma m_2$ stated in the fourth point of Proposition \ref{prop: dl derivate prime} it follows that the conditional expected value of the quantity here above is \\
$\frac{1}{n} \sum_{i = 0}^{n -1} \frac{(2 \Delta_{n,i} \partial_\sigma a(X_{t_i}, \sigma_0)a(X_{t_i}, \sigma_0))^2}{(\Delta_{n,i}a^2(X_{t_i}, \sigma_0))^2} (\frac{2(X_{t_{i + 1}} - m)^2}{\Delta_{n,i}a^2(X_{t_i}, \sigma_0)} - 1)1_{\left \{ |X_{t_i}| \le \Delta_{n,i}^{- k} \right \}}$  plus a negligible term that comes from the developments \eqref{eq: hp dl m2} and \eqref{eq: 1/m2} and it converges to zero in norm $1$ and so in probability. \\
The principal term is therefore such that, using {\modar the} second point of Proposition \ref{prop: LT} and the third of Proposition \ref{prop: ergodic}, we get 
$$\frac{1}{n} \sum_{i = 0}^{n -1}(\hat{I}_{i,5}^n + \hat{I}_{i,7}^n) \xrightarrow{\mathbb{P}} 4 \int_\mathbb{R} (\frac{\partial_\sigma a(x, \sigma_0)}{a(x, \sigma_0)})^2 \pi(dx).$$
\\
\textit{Point 2.} \\
 We start proving that $ \frac{1}{ \sqrt{n T_n}}\sup_{|\tilde{\theta}| \le \epsilon_n} |\partial^2_{\mu \sigma} U_n (\theta_0 +\tilde{\theta}) -  \partial^2_{\mu \sigma} U_n (\theta_0)|$ goes to $0$ in probability for $\epsilon_n$ that goes to $0$. \\
In order to do that, it is enough to show that the sequence $\frac{1}{\sqrt{n T_n}}\partial^2_{\mu \sigma} U_n (\theta)$ is tight, which is implied by $\sup_n \frac{1}{ \sqrt{n T_n}}\mathbb{E}[\sup_{\mu, \sigma}| \partial_\vartheta (\partial^2_{\mu \sigma} U_n (\mu, \sigma))|] < \infty$ (see Corollary B.1 in \cite{Shimizu thesis}), for $\vartheta = \mu$ or $\vartheta= \sigma$. \\
We observe it is 
$$\partial^3_{\mu \sigma \vartheta} U_n (\mu, \sigma) := \sum_{i = 0}^{n - 1} [ \frac{2 \partial_\vartheta m \partial^2_{\mu \sigma } m - 2(X_{t_{i + 1}} - m) \partial^3_{\mu \sigma \vartheta} m}{m_2} + \frac{2(X_{t_{i + 1}} - m) \partial^2_{\mu \sigma} m \partial_\vartheta m_2}{m_2^2} -2 \frac{(\partial^2_{\mu \vartheta} m \partial_\sigma m + \partial_\mu m \partial^2_{\sigma \vartheta} m)}{m_2} + $$
$$ + \frac{2 \partial_\mu m \partial_\sigma m \partial_\vartheta m_2}{m_2^2} - \frac{2 \partial_\vartheta m \partial_\mu m \partial_\sigma m_2}{m_2^2} + \frac{2 (X_{t_{i + 1}} - m)(\partial^2_{\mu \vartheta} m \partial_\sigma m_2 + \partial_\mu m \partial^2_{\sigma \vartheta} m_2)}{m_2^2} - \frac{4 (X_{t_{i + 1}} - m) \partial_\mu m \partial_\sigma m_2 \partial_\vartheta m_2}{m_2^3} +$$
$$- \frac{2 \partial_\vartheta m \partial_\mu m_2 \partial_\sigma m }{m_2^2} + \frac{2 (X_{t_{i + 1}} - m)(\partial^2_{\mu \vartheta} m_2 \partial_\sigma m + \partial_\mu m_2  \partial^2_{\sigma \vartheta} m)}{m_2^2} - \frac{4 (X_{t_{i + 1}} - m) \partial_\mu m_2 \partial_\sigma m \partial_\vartheta m_2}{m_2^3} +$$
$$- \frac{4(X_{t_{i + 1}} - m) \partial_\vartheta m \partial_\mu m_2 \partial_\sigma m_2 }{m_2^3 } + \frac{2 (X_{t_{i + 1}} - m)^2 (\partial^2_{\mu \vartheta} m_2 \partial_\sigma m_2 + \partial_\mu m_2 \partial^2_{\sigma \vartheta} m_2)}{m_2^3} - \frac{6 (X_{t_{i + 1}} - m)^2 \partial_\mu m_2 \partial_\sigma m_2 \partial_\vartheta m_2}{m_2^4} +$$
$$ + (\frac{\partial^3_{\mu \sigma \vartheta} m_2}{m_2} - \frac{\partial^2_{\mu \sigma}m_2 \partial_\vartheta m_2}{m_2^2})(1 - \frac{(X_{t_{i + 1}} - m)^2}{m_2}) + \frac{\partial^2_{\mu \sigma} m_2}{m_2}(\frac{2(X_{t_{i + 1}} - m)\partial_\vartheta m}{m_2} + \frac{(X_{t_{i + 1}} - m)^2 \partial_\vartheta m_2}{m_2^2}) + $$
$$- \frac{\partial^2_{\mu \vartheta} m_2 \partial_\sigma m_2 + \partial_\mu m_2 \partial^2_{\sigma \vartheta} m_2}{m_2^2} + \frac{2 \partial_\mu m_2 \partial_\sigma m_2 \partial_\vartheta m_2}{m_2^3}] \varphi_{\Delta_{n,i}^\beta}(\Delta X_i) 1_{\left \{ |X_{t_i}| \le \Delta_{n,i}^{- k} \right \}} = : \sum_{i = 0 }^{n - 1} \sum_{j = 1}^{17} I_{i, j}(\mu, \sigma, \vartheta).$$
Now using Assumption $Ad$, the estimation on the derivatives of $m$ and $m_2$ gathered in Propositions \ref{prop: dl derivate prime}, \ref{prop: derivate seconde} and \ref{prop: dervate terze} and the inequalities \eqref{eq: stima sup x-m}, \eqref{eq: fine stima sup x-m} and \eqref{eq: stima sup (x-m)^2} it follows 
$$\mathbb{E}[\sup_{\mu, \sigma}| \partial_\mu (\partial^2_{\mu \sigma} U_n (\mu, \sigma))|]  = \mathbb{E}[\sup_{\mu, \sigma}|  \sum_{i = 0 }^{n - 1} \sum_{j = 1}^{17} I_{i, j}(\mu, \sigma, \mu)|  ] \le c \, n \Delta_{n}^\frac{1}{2}.$$
In the same way we get
$$\mathbb{E}[\sup_{\mu, \sigma}| \partial_\sigma (\partial^2_{\mu \sigma} U_n (\mu, \sigma)|]  = \mathbb{E}[\sup_{\mu, \sigma}|  \sum_{i = 0 }^{n - 1} \sum_{j = 1}^{17} I_{i, j}(\mu, \sigma, \sigma)|  ] \le c \, n \Delta_{n,i}^\frac{1}{2}.$$
Hence, for both $\vartheta = \mu$ and $\vartheta = \sigma$ we can say it is $\sup_n \frac{1}{ \sqrt{n T_n}}\mathbb{E}[\sup_{\mu, \sigma}| \partial_\vartheta (\partial^2_{\mu \sigma} U_n (\mu, \sigma)|] \le c < \infty$; that implies the tightness of our sequence and so that $ \frac{1}{\sqrt{n T_n}}\sup_{|\tilde{\theta}| \le \epsilon_n} |\partial^2_{\mu \sigma} U_n (\theta_0 +\tilde{\theta}) -  \partial^2_{\mu \sigma} U_n (\theta_0)|$ goes to $0$ in probability for $\epsilon_n$ that goes to $0$. \\
To prove the convergence to $0$ in probability of $ \frac{1}{n}\sup_{|\tilde{\theta}| \le \epsilon_n} |\partial^2_{\sigma} U_n (\theta_0 +\tilde{\theta}) -  \partial^2_{\sigma} U_n (\theta_0)|$ for $\epsilon_n$ that goes to $0$ we act in the same way: we show that the sequence $\frac{1}{n} \partial^2_{\sigma} U_n (\theta)$ is tight through the criterion $\sup_n \frac{1}{n}\mathbb{E}[\sup_{\mu, \sigma}| \partial_\vartheta (\partial^2_{ \sigma} U_n (\mu, \sigma))|] < \infty$. \\
We observe that, computing the derivative with respect to $\vartheta$ of $\partial^2_{ \sigma} U_n (\mu, \sigma)$, we obtain $17$ terms analogous to the case just studied, with the only difference that also the derivatives that were with respect to $\mu$ are now with respect to $\sigma$. In particular, it is $\partial_\vartheta (\partial^2_{ \sigma} U_n (\mu, \sigma)) = \sum_{i = 0 }^{n - 1} \sum_{j = 1}^{17} I_{i, j}(\sigma, \sigma, \vartheta)$. \\
We still use Assumption $Ad$, the estimation on the derivatives of $m$ and $m_2$ gathered in Propositions \ref{prop: dl derivate prime}, \ref{prop: derivate seconde} and \ref{prop: dervate terze} and the inequalities \eqref{eq: stima sup x-m}, \eqref{eq: fine stima sup x-m} and \eqref{eq: stima sup (x-m)^2} to prove that $\mathbb{E}[\sup_{\mu, \sigma}| \partial_\sigma (\partial^2_{\sigma} U_n (\mu, \sigma)|]  < \infty$. Indeed, it is
$$\mathbb{E}[\sup_{\mu, \sigma}| \partial_\sigma (\partial^2_{\sigma} U_n (\mu, \sigma)|]  = \mathbb{E}[\sup_{\mu, \sigma}|  \sum_{i = 0 }^{n - 1} \sum_{j = 1}^{17} I_{i, j}(\sigma, \sigma, \sigma)|  ] \le c.$$
Moreover, since the order in which we compute the derivatives of the contrast function commute, we have
$$\mathbb{E}[\sup_{\mu, \sigma}| \partial_\mu (\partial^2_{\sigma} U_n (\mu, \sigma)|]  = \mathbb{E}[\sup_{\mu, \sigma}|  \sum_{i = 0 }^{n - 1} \sum_{j = 1}^{17} I_{i, j}(\sigma, \sigma, \mu)|  ] =  \mathbb{E}[\sup_{\mu, \sigma}|  \sum_{i = 0 }^{n - 1} \sum_{j = 1}^{17} I_{i, j}(\mu, \sigma, \sigma)|  ] \le  \sum_{i = 0 }^{n - 1} c \Delta_{n,i}^\frac{1}{2}.$$
We can therefore say that, for both $\vartheta = \mu$ and $\vartheta = \sigma$, it is $\sup_n \frac{1}{n}\mathbb{E}[\sup_{\mu, \sigma}| \partial_\vartheta (\partial^2_{\sigma} U_n (\mu, \sigma)|] \le c < \infty$; that implies the tightness. \\
We are now left to show that $\frac{1}{T_n} \sup_{|\tilde{\theta}| \le \epsilon_n}|\partial^2_\mu U_n(\theta_0 + \tilde{\theta}) - \partial^2_\mu U_n(\theta_0) | \xrightarrow{\mathbb{P}} 0$ for $\epsilon_n \rightarrow 0$. We still consider the notation introduced in the first point for which $\partial^2_\mu U_n (\theta) =: \sum_{j = 1}^7 \sum_{i = 0}^{n - 1} \tilde{I}_{i,j}^n (\theta)$ with $\partial^2_\mu U_n (\theta)$ that is as in \eqref{eq: derivate miste contrasto} but both the derivatives are calculated with respect to $\mu$. \\
From Proposition \ref{prop: punto 2 conv deriv seconda} we already know that $\frac{1}{T_n} \sum_{i = 0}^{n - 1} \tilde{I}_{i,j}^n (\theta)$ are tight sequences for $j \in \left \{ 1, 3, 4 \right \}$; having taken as $g_{i,n}(\theta, X_{t_i})$ respectively $\frac{-2 \partial^2_\mu m(\mu, \sigma, X_{t_i})}{m_2(\mu, \sigma, X_{t_i})}$ and $\frac{2 \partial_\mu m(\mu, \sigma, X_{t_i}) \partial_\mu m_2(\mu, \sigma, X_{t_i})}{m_2^2(\mu, \sigma, X_{t_i})}$, twice. We see that the assumptions required on $g_{i,n}$ hold as a consequence of the estimation on the derivatives of $m$ and $m_2$ gathered in Propositions \ref{prop: dl derivate prime} and \ref{prop: derivate seconde} and the development Ad of $m_2$. \\
We also show the tightness of the other terms proving that, for both $\vartheta= \mu$ and $\vartheta = \sigma$, \\
$\sup_n \frac{1}{T_n} \mathbb{E}[\sup_{\mu, \sigma} | \sum_{i = 0}^{n - 1} \partial_\vartheta ( \tilde{I}_{i,2}^n + \tilde{I}_{i,5}^n + \tilde{I}_{i,6}^n + \tilde{I}_{i,7}^n) |] \le c$. Indeed, using the estimation on the first derivatives of $m$ and $m_2$ gathered in Proposition \ref{prop: dl derivate prime} and the development Ad of $m_2$ it is
$$\sup_n \frac{1}{T_n} \mathbb{E}[\sup_{\mu, \sigma} | \sum_{i = 0}^{n - 1} \partial_\mu \tilde{I}_{i,2}^n|] \le  \sup_n \frac{c}{n \Delta_n} \sum_{i = 0}^{n - 1}(\Delta_{n,i} + \Delta_{n,i}^2) \le c,$$
$$\sup_n \frac{1}{T_n} \mathbb{E}[\sup_{\mu, \sigma} | \sum_{i = 0}^{n - 1} \partial_\sigma \tilde{I}_{i,2}^n|] \le  \sup_n \frac{c}{n \Delta_n} \sum_{i = 0}^{n - 1} (\Delta_{n,i} + \Delta_{n,i}) \le c.$$
In the same way, using Assumption $Ad$, the estimation on the derivatives of $m$ and $m_2$ gathered in Propositions \ref{prop: dl derivate prime}, \ref{prop: derivate seconde} and \ref{prop: dervate terze} and the inequalities \eqref{eq: stima sup x-m}, \eqref{eq: fine stima sup x-m} and \eqref{eq: stima sup (x-m)^2} it follows 
$$\sup_n \frac{1}{T_n} \mathbb{E}[\sup_{\mu, \sigma} | \sum_{i = 0}^{n - 1} \partial_\mu (\tilde{I}_{i,5}^n+ \tilde{I}_{i,6}^n + \tilde{I}_{i,7}^n)|] \le \sup_n \frac{c}{n \Delta_n}(\Delta_n^{\frac{5}{2} \land 2 \land 3} + \Delta_n^{1 \land 2 \land \frac{3}{2} \land 2} + \Delta_n^{2 \land 3}) \le c;  $$
$$\sup_n \frac{1}{n \Delta_n} \mathbb{E}[\sup_{\mu, \sigma} | \sum_{i = 0}^{n - 1} \partial_\sigma (\tilde{I}_{i,5}^n+ \tilde{I}_{i,6}^n + \tilde{I}_{i,7}^n)|] \le \sup_n \frac{c}{n \Delta_n}(\Delta_n^{\frac{5}{2} \land 2 \land 2} + \Delta_n^{1 \land 1 \land \frac{3}{2} \land 2} + \Delta_n^{1 \land 2}) \le c.  $$
We have therefore proved that the sequence $\frac{1}{T_n} \partial^2_\mu U_n (\theta)$ is tight, which implies the convergence to zero in probability of $\frac{1}{T_n}\sup_{|\tilde{\theta}| \le \epsilon_n} |\partial^2_{\mu} U_n (\theta_0 +\tilde{\theta}) -  \partial^2_{\mu} U_n (\theta_0)|$.  
\end{proof}

\subsubsection{Proof of Theorem \ref{th: normality}. }
\begin{proof}
By \eqref{eq: Taylor per normalita} we get
$$(\int_0^1 [C_n (\theta_0 + u (\hat{\theta}_n - \theta_0)) - C_n (\theta_0)] du + C_n (\theta_0))S_n = L_n(\theta_0).$$
We find that the matrix 
\begin{equation}
\int_0^1 [C_n (\theta_0 + u (\hat{\theta}_n - \theta_0)) - C_n (\theta_0)] du + C_n (\theta_0)
\label{eq: conv C}
\end{equation}
converges in probability to the nonsingular matrix B. Hence, taking the limit on both sides after multiplying by the inverse of \eqref{eq: conv C}, we see by the continuous mapping theorem that $S_n \xrightarrow{d} B^{-1}L \sim N(0, K^{-1}).$ \\
The asymptotic normality of $S_n$ is therefore proved.
\end{proof}

\subsection{Proof of Proposition \ref{Prop:normality_approximate_contrast}}\label{S: proof prop approx}
{\modch 
We observe that, having assumed $Ad$ on $\tilde{m}_2$ and as a consequence of the first point of $A_\rho$, the developments \eqref{eq: hp dl m2} and \eqref{eq: dl m} keep holding true with $\tilde{m}$ and $\tilde{m}_2$ instead of $m$ and $m_2$ with the only difference that the function $r_1(\mu, \sigma, x)$ in the development of $\tilde{m}$ contains also the rest function $R(\theta, \Delta_{n,i}^{\rho_1}, x)$.
Now it is possible to prove on $\tilde{m}$ and $\tilde{m}_2$ every result stated in Section \ref{S: contrast conv} following the proof we give with $\tilde{m}$ and $\tilde{m}_2$ replacing $m$ and $m_2$, since the only tools we use are the here above discussed developments. Moreover, the results stated in Section \ref{S: consistency} hold true also on $\tilde{m}$ and $\tilde{m}_2$ for the third point of $A_\rho$. \\
The substantial difference between $m$ and $m_2$ and their approximation is that we have defined $m$ and $m_2$ as in \eqref{eq: definition m} and \eqref{eq: definition m2} on purpose to make $L_n(\theta_0)$ a triangular array of martingale increments without requiring any constraint on the rate at which the step discretization has to go to zero. Defining $\tilde{U}_n (\mu, \sigma)$ the contrast function in which we have replaced $m$ and $m_2$ with their approximation $\tilde{m}$ and $\tilde{m}_2$, we have that
$$\tilde{L}_n(\theta_0) := \begin{pmatrix}
- \frac{1}{\sqrt{T_n}} \partial_\mu \tilde{U}_n (\mu_0, \sigma_0) \\
- \frac{1}{\sqrt{n}} \partial_\sigma \tilde{U}_n (\mu_0, \sigma_0)
\end{pmatrix}$$
is no longer a triangular array of martingale increments regardless and so we have to provide an alternative to Lemma \ref{lemma: normality derivative contrast}, which is gathered in the following lemma.
\begin{lemma}
Suppose that Assumptions A1-A8, Ad, $A_{\mbox{Step}}$ and $A \rho$ hold, with  $0 < k < k_0$, and that $\sqrt{n}\Delta_n^{\rho_1 -1/2} \rightarrow  0$ and $\sqrt{n}\Delta_n^{\rho_2 -1} \rightarrow  0$ as $n \rightarrow \infty$. Then, as $n \rightarrow \infty$,
$$\tilde{L}_n (\theta_0) \xrightarrow{d} L \sim N(0, K'),$$
where $K' = \begin{pmatrix} 
4 \int_\mathbb{R} (\frac{\partial_\mu b(x, \mu_0)}{a(x, \sigma_0)})^2 \pi(dx) & 0 \\
0 & 8 \int_\mathbb{R} (\frac{\partial_\sigma a(x, \sigma_0)}{a(x, \sigma_0)})^2 \pi(dx) 
\end{pmatrix}$.
\label{lemma: approx normality derivative contrast}
\end{lemma}
\begin{proof}
The result follows from a combination of Theorem 3.2 and Theorem 3.4 in
\cite{HH} (c.f. also Section A.2 in the Appendix of \cite{Shimizu thesis}).
We get the lemma proven if, for a constant $r> 0$, the following convergences hold. \\
We define $\zeta_i$ and $\tilde{\zeta}_i$ such that
$\partial_\mu \tilde{U}_n (\mu_0, \sigma_0) =: \sum_{i = 0}^{n -1} \zeta_i (\theta_0)$ and $\partial_\sigma \tilde{U}_n (\mu_0, \sigma_0) =: \sum_{i = 0}^{n -1} \tilde{\zeta}_i(\theta_0)$. Then it must be
\begin{equation}
\frac{1}{\sqrt{T_n}}\sum_{i = 0}^{n -1} |\mathbb{E}_i[\zeta_i(\theta_0)] | \xrightarrow{\mathbb{P}} 0 \qquad \frac{1}{\sqrt{n}}\sum_{i = 0}^{n -1} |\mathbb{E}_i[\tilde{\zeta}_i(\theta_0)] | \xrightarrow{\mathbb{P}} 0,
\label{eq: triangular array}    
\end{equation}
\begin{equation}
\frac{1}{T_n} \sum_{i = 0}^{n -1} \mathbb{E}_i[\zeta_i^2(\theta_0)] \xrightarrow{\mathbb{P}} 4 \int_\mathbb{R} (\frac{\partial_\mu b(x, \mu_0)}{a(x, \sigma_0)})^2 \pi(dx) \qquad \frac{1}{(\sqrt{T_n})^{2 + r}} \sum_{i = 0}^{n -1}\mathbb{E}_i[\zeta_i^{2 + r}(\theta_0)] \xrightarrow{\mathbb{P}} 0,
\label{eq: approx normalita mu}
\end{equation}
\begin{equation}
\frac{1}{n} \sum_{i = 0}^{n -1} \mathbb{E}_i[\tilde{\zeta}_i^2(\theta_0)] \xrightarrow{\mathbb{P}} 8 \int_\mathbb{R} (\frac{\partial_\sigma a(x, \sigma_0)}{a(x, \sigma_0)})^2 \pi(dx) \qquad \frac{1}{(\sqrt{n})^{2 + r}} \sum_{i = 0}^{n -1} \mathbb{E}_i[\tilde{\zeta}_i^{2 + r}(\theta_0)] \xrightarrow{\mathbb{P}} 0,
\label{eq: approx normalita sigma}
\end{equation}
\begin{equation}
\frac{1}{\sqrt{n T_n}} \sum_{i = 0}^{n -1} |\mathbb{E}_i[ \zeta_i (\theta_0) \tilde{\zeta}_i (\theta_0)] |\xrightarrow{\mathbb{P}} 0.
\label{eq: approx derivate miste}
\end{equation}
It is enough to follow the proof of \eqref{eq: normalita mu}, \eqref{eq: normalita sigma} and \eqref{eq: derivate miste} with $\tilde{m}$ and $\tilde{m}_2$ instead of $m$ and $m_2$ to get \eqref{eq: approx normalita mu}, \eqref{eq: approx normalita sigma} and \eqref{eq: approx derivate miste}. \\
We are left to show \eqref{eq: triangular array}.
We observe that, by the definition of $m$ and the first point of $A_\rho$, we have 
$$|\mathbb{E}_i[(X_{t_{i + 1}} - \tilde{m} (\mu_0, \sigma_0 , X_{t_i}))\varphi_{\Delta_{n,i}^\beta}(\Delta_i X)] | =  | \mathbb{E}_i[(X_{t_{i + 1}} - m (\mu_0, \sigma_0 , X_{t_i}))\varphi_{\Delta_{n,i}^\beta}(\Delta_i X)] + $$
\begin{equation}
 + \mathbb{E}_i[(m (\mu_0, \sigma_0 , X_{t_i}) - \tilde{m} (\mu_0, \sigma_0 , X_{t_i}))\varphi_{\Delta_{n,i}^\beta}(\Delta_i X)] | \le R(\theta_0, \Delta_{n,i}^{\rho_1}, X_{t_i}).
\label{eq: approx deriv 1}
\end{equation}
In the same way, using also the definition of $m_2$ and the estimation which assesses the quality of the approximation of $m_2$ through $\tilde{m}_2$, still in the first point of $A_\rho$, we get
$$\mathbb{E}_i[(1 - \frac{(X_{t_{i + 1}} - \tilde{m} (\mu_0, \sigma_0 , X_{t_i}))^2}{\tilde{m}_2 (\mu_0, \sigma_0 , X_{t_i})})\varphi_{\Delta_{n,i}^\beta}(\Delta_i X)] =$$
$$ = \mathbb{E}_i[(1 - \frac{(X_{t_{i + 1}} - m(\mu_0, \sigma_0 , X_{t_i}))^2}{m_2 (\mu_0, \sigma_0 , X_{t_i})})\varphi_{\Delta_{n,i}^\beta}(\Delta_i X)] + $$
$$ + \mathbb{E}_i[( \frac{(X_{t_{i + 1}} - \tilde{m} (\mu_0, \sigma_0 , X_{t_i}))^2}{\tilde{m}_2 (\mu_0, \sigma_0 , X_{t_i})} - \frac{(X_{t_{i + 1}} - m(\mu_0, \sigma_0 , X_{t_i}))^2}{m_2 (\mu_0, \sigma_0 , X_{t_i})} )\varphi_{\Delta_{n,i}^\beta}(\Delta_i X)]=$$
$$= 0 + (\frac{1}{\tilde{m}_2 (\mu_0, \sigma_0 , X_{t_i})} - \frac{1}{{m}_2 (\mu_0, \sigma_0 , X_{t_i})})\mathbb{E}_i[(X_{t_{i + 1}} - m (\mu_0, \sigma_0 , X_{t_i}))^2\varphi_{\Delta_{n,i}^\beta}(\Delta_i X)] +$$
$$ + \mathbb{E}_i[\frac{(m (\mu_0, \sigma_0 , X_{t_i}) - \tilde{m} (\mu_0, \sigma_0 , X_{t_i}))^2 + 2 (m (\mu_0, \sigma_0 , X_{t_i}) - \tilde{m} (\mu_0, \sigma_0 , X_{t_i}))(X_{t_{i + 1}} - m(\mu_0, \sigma_0 , X_{t_i})) }{\tilde{m}_2 (\mu_0, \sigma_0 , X_{t_i})}\varphi_{\Delta_{n,i}^\beta}(\Delta_i X)].$$
Now, using also the development $A_d$ for $m_2$ and $\tilde{m}_2$ and the first and the third point of Lemma \ref{lemma: conditional expected value}, the absolute value of the equation here above is upper bounded by 
\begin{equation}
R(\theta_0, \Delta_{n,i}^{\rho_2 - 1}, X_{t_i}) + R(\theta_0, \Delta_{n,i}^{2 \rho_1 - 1}, X_{t_i}) + R(\theta_0, \Delta_{n,i}^{\rho_1}, X_{t_i})= R(\theta_0, \Delta_{n,i}^{(\rho_2 - 1) \land \rho_1}, X_{t_i}).
\label{eq: approx deriv 2}
\end{equation}
We compute hereafter the derivatives of $\tilde{U}_n$ with respect to $\mu$, obtaining
\begin{equation}
\mathbb{E}_i[\zeta_i(\theta_0)] = \frac{-2 \partial_\mu \tilde{m} (\mu_0, \sigma_0, X_{t_i})}{\tilde{m}_2(\mu_0, \sigma_0, X_{t_i})} \mathbb{E}_i[(X_{t_{i + 1}} - \tilde{m}(\mu_0, \sigma_0, X_{t_i}))\varphi_{\Delta_{n,i}^\beta}(\Delta_i X)]1_{\left \{|X_{t_i}| \le \Delta_{n,i}^{-k} \right \} } +
\label{eq: zeta}
\end{equation}
$$+ \frac{\partial_\mu  \tilde{m}_2 (\mu_0, \sigma_0, X_{t_i})}{ \tilde{m}_2 (\mu_0, \sigma_0, X_{t_i})}\mathbb{E}_i[(1 - \frac{(X_{t_{i + 1}} - \tilde{m}(\mu_0, \sigma_0, X_{t_i}))^2}{m_2(\mu_0, \sigma_0, X_{t_i})})\varphi_{\Delta_{n,i}^\beta}(\Delta_i X)]1_{\left \{|X_{t_i}| \le \Delta_{n,i}^{-k} \right \} }.$$
As a consequence of the third point of $A_\rho$, Proposition \ref{prop: dl derivate prime} still holds replacing the derivatives of $m$ and $m_2$ with the derivatives of $\tilde{m}$ and $\tilde{m}_2$. Therefore, using also $Ad$ and the equations \eqref{eq: approx deriv 1} and \eqref{eq: approx deriv 2}, we have 
$$\frac{1}{\sqrt{T_n}} \sum_{i = 0}^{n-1} |\mathbb{E}_i[\zeta_i(\theta_0)] |\le \frac{c}{\sqrt{n \Delta_n}}  \sum_{i = 0}^{n-1}(R(\theta_0, \Delta_{n,i}^{\rho_1}, X_{t_i}) + R(\theta_0, \Delta_{n,i}^{\rho_2 \land \rho_1 + 1}, X_{t_i})) \le \frac{\Delta_n^{(\rho_1 - \frac{1}{2}) \land (\rho_2 - \frac{1}{2})}}{\sqrt{n}} \sum_{i = 0}^{n-1}R(\theta_0, 1, X_{t_i}), $$
which converges to $0$ in norm $1$ and so in probability for $\sqrt{n} \Delta_n^{(\rho_1 - \frac{1}{2}) \land (\rho_2 - \frac{1}{2})} \rightarrow 0$ for $n \rightarrow \infty$, as we have required. \\
To show Lemma \ref{lemma: approx normality derivative contrast} holds true we have to prove finally that $\frac{1}{\sqrt{n}} \sum_{i = 0}^{n-1} |\mathbb{E}_i[\tilde{\zeta}_i(\theta_0)] | \xrightarrow{\mathbb{P}} 0$ for $n \rightarrow \infty$, recalling that $\mathbb{E}_i[\tilde{\zeta}_i(\theta_0)]$ is defined as \eqref{eq: zeta} but with the derivatives with respect to $\sigma$ which take now the place of the derivatives with respect to $\mu$. From Proposition \ref{prop: dl derivate prime}, $Ad$ and the equations \eqref{eq: approx deriv 1} and \eqref{eq: approx deriv 2}, it follows
$$\frac{1}{\sqrt{n}} \sum_{i = 0}^{n-1} \mathbb{E}_i[\tilde{\zeta}_i(\theta_0)] \le \frac{c}{\sqrt{n}}  \sum_{i = 0}^{n-1}(R(\theta_0, \Delta_{n,i}^{\rho_1}, X_{t_i}) + R(\theta_0, \Delta_{n,i}^{\rho_2 - 1 \land \rho_1}, X_{t_i})) \le \frac{\Delta_n^{\rho_1 \land (\rho_2 - 1)}}{\sqrt{n}} \sum_{i = 0}^{n-1}R(\theta_0, 1, X_{t_i}),$$
which converges to $0$ in norm $1$ and so in probability for $\sqrt{n} \Delta_n^{\rho_1 \land (\rho_2 - 1)} \rightarrow 0$ for $n \rightarrow \infty$, as we have required. \\
We get $\tilde{L}_n (\theta_0)$ is asymptotically normal, as we wanted.
\end{proof}
After having replaced Lemma \ref{lemma: normality derivative contrast} with the Lemma \ref{lemma: approx normality derivative contrast} just showed, it is enough to follow the proof of the asymptotic normality of the estimator $\hat{\theta_n}$
given in Sections \ref{S: contrast conv}$-$\ref{S: As norm} to get the asymptotic normality of $\tilde{\theta}_n$.

}

\appendix
\section{Appendix}
In this section we prove all the technical results we have introduced, starting from the preliminary results stated in Section \ref{S: Limit theorems}.
\subsection{Proof of limit theorems}\label{S: Proof_limit}
We first show Proposition \ref{prop: ergodic}, observing that its last two points are the discretized version of the first point of Lemma \ref{lemma: 2.1 GLM}.
\subsubsection{Proof of Proposition \ref{prop: ergodic}}
\begin{proof}
The first two points have already been proved in Proposition 3 of \cite{Chapitre 1}. \\
We want to show that $ \frac{1}{n} \sum_{i =0}^{n -1} f(X_{t_i}, \theta)$ converges in $L^2$ to $ \int_\mathbb{R} f(x, \theta) \pi (dx)$.\\
Since
$Var(\frac{1}{n} \sum_{i =0}^{n -1} f(X_{t_i}, \theta)) \le \frac{1}{n^2} \sum_{i =0}^{n -1}  \sum_{j =0}^{n -1} Cov(f(X_{t_i}, \theta), f(X_{t_j}, \theta)),$ we need to estimate the covariance. \\
We know that, under our assumptions, the process $X$ is $\beta$- mixing with exponential decay (see \cite{18 GLM}) that is $\exists \gamma >0$ such that $\beta_X (k) = O ( e^{- \gamma k} )$; with $\beta_X (k)$ as defined in Section 1.3.2 of \cite{Doukhan}. If a process is $\beta$- mixing, then it is also $\alpha$-mixing and so the following estimation holds (see Theorem 3 in Section 1.2.2 of \cite{Doukhan})
$$ |Cov(X_{t_i}, X_{t_j})| \le c \left \| X_{t_i} \right \|_p \left \| X_{t_j} \right \|_q \alpha^\frac{1}{r} (X_{t_i}, X_{t_j}) $$
with $p$, $q$ and $r$ such that $ \frac{1}{p} + \frac{1}{q} + \frac{1}{r} =1$.
Using that $\alpha (X_{t_i}, X_{t_j}) \le \beta_X (|{t_i}-{t_j}|) = O ( e^{- \gamma |{t_i}-{t_j}|} )$, in our case the inequality here above becomes 
$|Cov(f(X_{t_i}, \theta), f(X_{t_j}, \theta))| \le c e^{- \frac{1}{r} \gamma |t_i -t_j|}$,
where we have also used the polynomial growth of $f$ and the third point of Lemma \ref{lemma: 2.1 GLM} to include the two norms in the constant $c$. \\
We introduce a partition of $(0, T_n]$ based on the sets $A_k:= ( k \frac{T_n}{n}, (k+1) \frac{T_n}{n}]$, for which $(0, T_n] = \cup_{k=0}^{n-1} A_k$. 
Now each point $t_i$ in $(0, T_n]$ can be seen as $t_{k, h}$, where $k$ identifies the particular set $A_k$ to which the point belongs while, defining $M_k$ as $|A_k|$, $h$ is a number in $\left \{ 1, ... , M_k \right \}$ which enumerates the points in each set. It follows
$$\frac{c}{n^2} \sum_{i =0}^{n -1}  \sum_{j =0}^{n -1} e^{- \frac{1}{r} \gamma |t_i -t_j|} \le \frac{c}{n^2} \sum_{k_1 =0}^{n -1}  \sum_{k_2 =0}^{n -1} \sum_{h_1 =1}^{M_{k_1}}  \sum_{h_2 =1}^{M_{k_2}} e^{- \frac{1}{r} \gamma |t_{k_1, h_1} -t_{k_2, h_2}|} \le \frac{c e^{\frac{1}{r} \frac{T_n}{n}}}{n^2} \sum_{k_1 =0}^{n -1}  \sum_{k_2 =0}^{n -1} \sum_{h_1 =1}^{M_{k_1}}  \sum_{h_2 =1}^{M_{k_2}} e^{- \frac{1}{r} \gamma |k_1 - k_2| \frac{T_n}{n}}, $$
where the last inequality is a consequence of the following estimation: for each $k_1, k_2 \in \left \{ 0, ... , n-1 \right \}$ it is $|t_{k_1, h_1} -t_{k_2, h_2}| \ge |k_1 - k_2|\frac{T_n}{n} - \frac{T_n}{n}$. \\
Now we observe that the exponent does not depend on $h$ anymore, hence the last term here above can be upper bounded by $\frac{c e^{\frac{1}{r}  \frac{T_n}{n}}}{n^2} \sum_{k_1 =0}^{n -1}  \sum_{k_2 =0}^{n -1} M_{k_1} M_{k_2}e^{- \frac{1}{r} \gamma |k_1 - k_2|\frac{T_n}{n}} $. \\
Moreover, remarking that the length of each interval $A_k$ is $\frac{T_n}{n}$, it is easy to show that we can always upper bound $M_k$ with $\frac{T_n}{n} \frac{1}{\Delta_{min}}$, with $T_n = \sum_{i = 0}^{n - 1} \Delta_{n,i} \le n \Delta_{n}$ and so $M_k \le \frac{\Delta_{n}}{\Delta_{min}}$, that we have assumed bounded by a constant $c_1$. \\
Furthermore, still using that $T_n \le n \Delta_{n}$, we have $e^{\frac{1}{r}  \frac{T_n}{n}} \le e^{\frac{1}{r} \Delta_{n} } \le c$ because, by our hypothesis, $\Delta_{max}$ goes to $0$ for $n \rightarrow \infty$. To conclude, we have to show that 
$\frac{c }{n^2} \sum_{k_1 =0}^{n -1}  \sum_{k_2 =0}^{n -1} e^{- \frac{1}{r} \gamma |k_1 - k_2| \frac{T_n}{n}} \rightarrow 0$ for $n \rightarrow \infty$. We define $j := k_1 - k_2$ and we apply a change of variable, getting 
$$\frac{c }{n^2} \sum_{k_1 =0}^{n -1}  \sum_{k_2 =0}^{n -1} e^{- \frac{1}{r} \gamma |k_1 - k_2|\frac{T_n}{n}} \le \frac{c }{n^2} \sum_{j = -(n -1)}^{n -1} e^{- \frac{1}{r} \gamma |j|\frac{T_n}{n}} |n - j| \le \frac{c }{n}  \sum_{j = -(n -1)}^{n -1} e^{- \frac{1}{r} \gamma |j| \Delta_{min}} \le  \frac{c}{n (1 - e^{- \frac{1}{r} \gamma \Delta_{min}})} \le \frac{c}{T_n} , $$
that goes to $0$ for $n$ that goes to $\infty$. We therefore get $|\frac{1}{n} \sum_{i = 0}^{n - 1}f(X_{t_i}, \theta) - \int_\mathbb{R} f(x, \theta) \pi(dx)| \xrightarrow{\mathbb{P}} 0$. \\
In order to show the third point we observe that
$$|\frac{1}{n} \sum_{i = 0}^{n - 1}f(X_{t_i}, \theta)1_{ \left \{ |X_{t_i}| \le \Delta_{n,i}^{- k} \right \}} - \int_\mathbb{R} f(x, \theta) \pi(dx)| \le |\frac{1}{n} \sum_{i = 0}^{n - 1}f(X_{t_i}, \theta)1_{ \left \{ |X_{t_i}| \le \Delta_{n,i}^{- k} \right \}} - \frac{1}{n} \sum_{i = 0}^{n - 1}f(X_{t_i}, \theta)| + $$
$$ + |\frac{1}{n} \sum_{i = 0}^{n - 1}f(X_{t_i}, \theta) - \int_\mathbb{R} f(x, \theta) \pi(dx)|.$$
We have already proved that the second goes to $0$ in probability, while the first is \\ $|\frac{1}{n} \sum_{i = 0}^{n - 1}f(X_{t_i}, \theta)1_{ \left \{ |X_{t_i}| > \Delta_{n,i}^{- k} \right \}}|$, that converges to $0$ in $L^1$ as a consequence of the polynomial growth of $f$ and the third point of Lemma \ref{lemma: 2.1 GLM} and so in probability. \\
We act in the same way in order to show the fourth point, observing that, by the definition of $\varphi$, it is
\begin{equation}
|\frac{1}{n} \sum_{i = 0}^{n - 1}f(X_{t_i}, \theta)1_{ \left \{ |X_{t_i}| \le \Delta_{n,i}^{- k} \right \}}(\varphi_{\Delta_{n,i}^\beta}(\Delta X_i)- 1)| \le \frac{c}{n} \sum_{i = 0}^{n - 1}|f(X_{t_i}, \theta)|1_{ \left \{ |X_{t_i}| \le \Delta_{n,i}^{- k} \right \}} 1_{ \left \{ |X_{t_i}| \ge \Delta_{n,i}^{\beta} \right \}}. 
\label{eq: ora}
\end{equation}
We observe that, since $\Delta X_i^c = \Delta X_i - \Delta X_i^J$, if $|\Delta X_i| \ge \Delta_{n,i}^\beta$ and $|\Delta X_i^J| < \frac{\Delta_{n,i}^\beta}{2}$, then $|\Delta X_i^c|$ must be more than $\frac{\Delta_{n,i}^\beta}{2}$. Hence 
$$\mathbb{E}[1_{\left \{|\Delta X_i| \ge \Delta_{n,i}^\beta \right \}}] = \mathbb{E}[1_{\left \{|\Delta X_i| \ge \Delta_{n,i}^\beta, |\Delta X_i^J| < \frac{\Delta_{n,i}^\beta}{2} \right \}}] + \mathbb{E}[1_{\left \{|\Delta X_i| \ge \Delta_{n,i}^\beta, |\Delta X_i^J| \ge \frac{\Delta_{n,i}^\beta}{2} \right \}}] \le$$
\begin{equation}
\le \mathbb{P}(|\Delta X_i^c| \ge \frac{\Delta_{n,i}^\beta}{2}) + \mathbb{P}(|\Delta X_i^J| \ge \frac{\Delta_{n,i}^\beta}{2}) \le c \frac{\mathbb{E}[|\Delta X_i^c|^r]}{\Delta_{n,i}^{\beta r}} + R(\theta, \Delta_{n,i}, X_{t_i}) \le R(\theta, \Delta_{n,i}^{(\frac{1}{2} - \beta)r \land 1}, X_{t_i}) = R(\theta, \Delta_{n,i}, X_{t_i}).
\label{eq: stima prob incrementi}
\end{equation}
On the first probability here above we have used Tchebychev inequality and the fourth point of Lemma \ref{lemma: Moment inequalities}, for $|\Delta X_i^J| \ge  \frac{\Delta_{n,i}^\beta}{2}$ the fact that the intensity of jumps is finite and therefore the probability to have at least one jump bigger than $\frac{\Delta_{n,i}^\beta}{2}$ can be computed and it is of order $\Delta_n$. Moreover, by the arbitrariness of $r > 1$, we get that $R(\theta, \Delta_{n,i}^{(\frac{1}{2} - \beta)r }, X_{t_i})$ is negligible compared to $R(\theta, \Delta_{n,i}, X_{t_i})$. \\
From Holder inequality, the polynomial growth of $f$, the third point of Lemma \ref{lemma: 2.1 GLM} and \eqref{eq: stima prob incrementi} it follows that the right hand side of \eqref{eq: ora} goes to $0$ in norm $1$ and so in probability. The proposition is therefore proved. 
\end{proof}

We now prove Proposition \ref{prop: LT}, that is a consequence of Lemma \ref{lemma: conditional expected value}.
\subsubsection{Proof of Proposition \ref{prop: LT}}
\begin{proof}
In order to show that the first convergence holds, we define \\ 
$s_i^n := \frac{1}{T_n} f(X_{t_i}, \theta) \, (X_{t_{i+1}} - m(\mu, \sigma, X_{t_i}))^2 \, \varphi_{\Delta_{n,i}^\beta}(X_{t_{i+1}} - X_{t_i})1_{\left \{ |X_{t_i}| \le \Delta_{n,i}^{- k} \right \}}$.
From Lemma 9 in \cite{Genon Catalot}, if we show that 
$$\sum_{i = 0}^{n - 1}\mathbb{E}_i[s_i^n]\xrightarrow{\mathbb{P}} \int_\mathbb{R}f(x, \theta) a^2(x, \sigma_0) \pi(dx) \quad \mbox{and  }\sum_{i = 0}^{n - 1} \mathbb{E}_i[(s_i^n)^2]\rightarrow 0,$$
then the proposition is proved.
We observe that \eqref{eq: Ei al quadrato} yields
$$\sum_{i = 0}^{n - 1}\mathbb{E}_i[s_i^n] = \frac{1}{T_n} \sum_{i=0}^{n - 1}\Delta_{n,i} f(X_{t_i}, \theta)a^2(X_{t_i}, \sigma_0)1_{\left \{ |X_{t_i}| \le \Delta_{n,i}^{- k} \right \}} + \frac{1}{T_n} \sum_{i=0}^{n - 1}\Delta_{n,i} f(X_{t_i}, \theta) R(\theta_0, \Delta_{n,i}^{\beta}, X_{t_i}).  $$
The first term here above converges in probability to $\int_\mathbb{R}f(x, \theta) a^2(x, \sigma_0) \pi(dx)$ as a consequence of the first point of Proposition \ref{prop: ergodic}, while the second is upper bounded by \\ $\Delta_n^{\beta} \frac{1}{n} \sum_{i=0}^{n - 1}f(X_{t_i}, \theta) R(\theta_0, 1, X_{t_i}) $, which converges to zero in norm 1 (and so in probability) by the polynomial growth of both $R$ and $f$ and the third point of Lemma \ref{lemma: 2.1 GLM}. Moreover, using \eqref{eq: Ei quarta} and the fact that $\frac{1}{T_n} = O(\frac{1}{n \Delta_n})$, we have 
$$ \sum_{i=0}^{n - 1}|\mathbb{E}_i[(s_i^n)^2] | \le \frac{1}{(n\Delta_n)^2} \sum_{i=0}^{n - 1} (f(X_{t_i}, \theta))^2 R(\theta_0, \Delta_{n,i}^2, X_{t_i}) \le \frac{1}{n^2} \sum_{i=0}^{n - 1} f^2(X_{t_i}, \theta) R(\theta_0, 1, X_{t_i}), $$
which goes to zero in norm 1 and so in probability for $n \rightarrow \infty$ as a consequence of the polynomial growth of both $R$ and $f$ and the third point of Lemma \ref{lemma: 2.1 GLM}. The first point is therefore proved. In order to show the second point of Proposition \ref{prop: LT} is enough to act on the sequence
$\tilde{s}_i^n := \frac{1}{n} \frac{f(X_{t_i}, \theta)}{\Delta_{n,i}} \, (X_{t_{i+1}} - m(\mu, \sigma, X_{t_i}))^2 \, \varphi_{\Delta_{n,i}^\beta}(X_{t_{i+1}} - X_{t_i})1_{\left \{ |X_{t_i}| \le \Delta_{n,i}^{- k} \right \}}$ exactly like we have just did here above, with the only difference that the third point of Proposition \ref{prop: ergodic} has to be applied instead of the first one.
\end{proof}

\subsubsection{Proof of Lemma \ref{lemma: conditional expected value}}
\begin{proof}
Replacing $m(\mu, \sigma, X_{t_i})$ with its development \eqref{eq: dl m} and using the dynamic \eqref{eq: model} of $X$ we have
$$X_{t_{i + 1}} - m(\mu, \sigma, X_{t_i}) = \int_{t_i}^{t_{i + 1}} b(X_s, \mu_0) ds + \int_{t_i}^{t_{i + 1}} a(X_s, \sigma_0) dW_s + \int_{t_i}^{t_{i + 1}} \int_\mathbb{R} z \gamma(X_{s^-}) \tilde{\mu}(ds, dz) + R(\theta, \Delta_{n,i}, X_{t_i}) =$$
\begin{equation}
 =: \int_{t_i}^{t_{i + 1}} a(X_s, \sigma_0) dW_s + B_{i,n}.
\label{eq: reformulation X - m}
\end{equation}
In order to prove \eqref{eq: Ei al quadrato} we start considering 
$$\mathbb{E}_i[(X_{t_{i + 1}} - m(\mu, \sigma, X_{t_i}))^2 \varphi_{\Delta_{n,i}^\beta}(X_{t_{i+1}} - X_{t_i}) ] =\mathbb{E}_i[(\int_{t_i}^{t_{i + 1}} a(X_s, \sigma_0) dW_s)^2 \varphi_{\Delta_{n,i}^\beta}(X_{t_{i+1}} - X_{t_i}) ] + $$
$$ +\mathbb{E}_i[(B_{i,n})^2 \varphi_{\Delta_{n,i}^\beta}(X_{t_{i+1}} - X_{t_i}) ] +2 \mathbb{E}_i[B_{i,n}(\int_{t_i}^{t_{i + 1}} a(X_s, \sigma_0) dW_s) \varphi_{\Delta_{n,i}^\beta}(X_{t_{i+1}} - X_{t_i}) ]. $$
Now the first term on the right hand side here above is
$$\mathbb{E}_i[(\int_{t_i}^{t_{i + 1}} a(X_s, \sigma_0) dW_s)^2] + \mathbb{E}_i[(\int_{t_i}^{t_{i + 1}} a(X_s, \sigma_0) dW_s)^2 (\varphi_{\Delta_{n,i}^\beta}(X_{t_{i+1}} - X_{t_i}) -1) ] = $$
\begin{equation}
= \Delta_{n,i} a^2(X_{t_i}, \sigma_0) + \mathbb{E}_i[\int_{t_i}^{t_{i + 1}} [a^2(X_s, \sigma_0) - a^2(X_{t_i}, \sigma_0)] ds] + \mathbb{E}_i[(\int_{t_i}^{t_{i + 1}} a(X_s, \sigma_0) dW_s)^2 (\varphi_{\Delta_{n,i}^\beta}(X_{t_{i+1}} - X_{t_i}) -1) ].
\label{eq: a carre}
\end{equation}
Moreover, 
$$|\mathbb{E}_i[\int_{t_i}^{t_{i + 1}} [a^2(X_s, \sigma_0) - a^2(X_{t_i}, \sigma_0)] ds] + \mathbb{E}_i[(\int_{t_i}^{t_{i + 1}} a(X_s, \sigma_0) dW_s)^2 (\varphi_{\Delta_{n,i}^\beta}(X_{t_{i+1}} - X_{t_i}) -1) ]| \le $$
\begin{equation}
\le \int_{t_i}^{t_{i+ 1}} \mathbb{E}_i[|2a \partial_x a (X_u, \sigma_0)| |X_s - X_{t_i}|] ds + \mathbb{E}_i[(\int_{t_i}^{t_{i + 1}} a(X_s, \sigma_0) dW_s)^{2p}]^\frac{1}{p} \mathbb{E}_i[1_{\left \{|\Delta X_i| \ge \Delta_{n,i}^\beta \right \}}]^\frac{1}{q},
\label{eq: prima di stima varphi 1}
\end{equation}
where $X_u \in (X_s, X_{t_i})$ and we have used Holder inequality and the definition of $\varphi$, that is equal to 1 for $|\Delta X_i| < \Delta_{n,i}^\beta$.  \\
Using Cauchy-Schwartz inequality and the second point of Lemma \ref{lemma: Moment inequalities} on the first term of \eqref{eq: prima di stima varphi 1} and Burkholder-Davis-Gundy inequality and \eqref{eq: stima prob incrementi} on the second we have that the right hand side is upper bounded by
$$\int_{t_i}^{t_{i + 1}} R(\theta_0, \Delta_{n,i}^\frac{1}{2}, X_{t_i} ) ds + R(\theta_0, \Delta_{n,i}, X_{t_i} ) R(\theta_0, \Delta_{n,i}^{\frac{1}{q}}, X_{t_i} ) \le R(\theta_0, \Delta_{n,i}^\frac{3}{2}, X_{t_i} ) + R(\theta_0, \Delta_{n,i}^{2 - \epsilon}, X_{t_i} ),$$
where we have taken $q$ next to $1$. \\
Replacing in \eqref{eq: a carre} we get
\begin{equation}
\mathbb{E}_i[(\int_{t_i}^{t_{i + 1}} a(X_s, \sigma_0) dW_s)^2 \varphi_{\Delta_{n,i}^\beta}(X_{t_{i+1}} - X_{t_i}) ] = \Delta_{n,i} a^2(X_{t_i}, \sigma_0) + R(\theta_0, \Delta_{n,i}^{\frac{3}{2}}, X_{t_i} ).
\label{eq: finale a carre}
\end{equation}
Now we evaluate the contribution of $(B_{i,n})^2$:
$$\mathbb{E}_i[(B_{i,n})^2 \varphi_{\Delta_{n,i}^\beta}(X_{t_{i+1}} - X_{t_i})] \le c \mathbb{E}_i[(\int_{t_i}^{t_{i + 1}} b(X_s, \mu_0) ds)^2 \varphi_{\Delta_{n,i}^\beta}(X_{t_{i+1}} - X_{t_i}) ] + $$
$$ +c \mathbb{E}_i[(\int_{t_i}^{t_{i + 1}} \int_\mathbb{R} \gamma(X_{s^-}) z \tilde{\mu}(ds, dz))^2 \varphi_{\Delta_{n,i}^\beta}(X_{t_{i+1}} - X_{t_i}) ] + R(\theta, \Delta_{n,i}^2, X_{t_i}) \le c \Delta_{n,i} \int_{t_i}^{t_{i + 1}} \mathbb{E}_i[b^2(X_s, \mu)] ds + $$
\begin{equation}
+R(\theta_0, \Delta_{n,i}^{1 + 2 \beta}, X_{t_i}) = R(\theta_0, \Delta_{n,i}^2, X_{t_i}) + R(\theta_0, \Delta_{n,i}^{1 + 2\beta}, X_{t_i}) = R(\theta_0, \Delta_{n,i}^{1 + 2\beta}, X_{t_i}),
\label{eq: Bin carre}
\end{equation}
where we have used Jensen inequality, Lemma \ref{lemma: estim jumps} and the fact that $R(\theta_0, \Delta_{n,i}^2, X_{t_i})$ is always negligible compared to $ R(\theta_0, \Delta_{n,i}^{1 + 2\beta}, X_{t_i})$ since $2 > 1 + 2\beta $. \\
We observe that \eqref{eq: finale a carre} still holds with $1$ instead of $\varphi$ (see \eqref{eq: a carre} ). Using it, Cauchy-Schwartz inequality and \eqref{eq: Bin carre} it follows
$$\mathbb{E}_i[B_{i,n}(\int_{t_i}^{t_{i + 1}} a(X_s, \sigma_0) dW_s) \varphi_{\Delta_{n,i}^\beta}(X_{t_{i+1}} - X_{t_i}) ] \le c \mathbb{E}_i[(\int_{t_i}^{t_{i + 1}} a(X_s, \sigma_0) dW_s)^2]^\frac{1}{2} \mathbb{E}_i[(B_{i,n})^2\varphi^2_{\Delta_{n,i}^\beta}(X_{t_{i+1}} - X_{t_i})]^\frac{1}{2} \le$$
\begin{equation}
\le R(\theta_0, \Delta_{n,i}, X_{t_i})^\frac{1}{2} R(\theta_0, \Delta_{n,i}^{1 + 2\beta}, X_{t_i})^\frac{1}{2} = R(\theta_0, \Delta_{n,i}^{1 + \beta}, X_{t_i}).
\label{eq: double produit}
\end{equation}
From \eqref{eq: reformulation X - m}, \eqref{eq: finale a carre} - \eqref{eq: double produit} it follows \eqref{eq: Ei al quadrato}. \\
\\
Concerning \eqref{eq: Ei quarta}, we have 
\begin{equation}
\mathbb{E}_i[(\int_{t_i}^{t_{i + 1}} a(X_s, \sigma_0) dW_s)^4 \varphi_{\Delta_{n,i}^\beta}(\Delta X_i) ] = \mathbb{E}_i[(\int_{t_i}^{t_{i + 1}} a(X_s, \sigma_0) dW_s)^4 ] + \mathbb{E}_i[(\int_{t_i}^{t_{i + 1}} a(X_s, \sigma_0) dW_s)^4 (\varphi_{\Delta_{n,i}^\beta}(\Delta X_i) - 1) ].
\label{eq: scomposizione a quarta}
\end{equation}
Using Holder inequality and the definition of $\varphi$ we have that the second term here above is upper bounded by 
\begin{equation}
\mathbb{E}_i[(\int_{t_i}^{t_{i + 1}} a(X_s, \sigma_0) dW_s)^{4p} ]^\frac{1}{p} \mathbb{P}_i(|\Delta X_i| \ge \Delta_{n,i}^\beta)^\frac{1}{q} \le R(\theta_0, \Delta_{n,i}^2, X_{t_i})R(\theta_0, \Delta_{n,i}^{\frac{1}{q}}, X_{t_i}) = R(\theta_0, \Delta_{n,i}^{3 - \epsilon}, X_{t_i}),
\label{eq: stima varphi -1}
\end{equation}
where we have used BDG inequality, \eqref{eq: stima prob incrementi} and we have taken $q$ next to $1$. Moreover,
\begin{equation}
\mathbb{E}_i[(\int_{t_i}^{t_{i + 1}} a(X_s, \sigma_0) dW_s)^4 ] = \mathbb{E}_i[(\int_{t_i}^{t_{i + 1}} a(X_{t_i}, \sigma_0) dW_s)^4 ] + \mathbb{E}_i[(\int_{t_i}^{t_{i + 1}} [a(X_s, \sigma_0) - a(X_{t_i}, \sigma_0)] dW_s)^4 ] +
\label{eq: a quarta}
\end{equation}
$$+ \sum_{j = 1}^3 \binom{4}{j} \mathbb{E}_i[(\int_{t_i}^{t_{i + 1}} a(X_{t_i}, \sigma_0) dW_s)^j (\int_{t_i}^{t_{i + 1}} [a(X_s, \sigma_0) - a(X_{t_i}, \sigma_0)] dW_s)^{4-j} ].$$
Since the expected value of the fourth moment of the gaussian law is known we have 
\begin{equation}
\mathbb{E}_i[(\int_{t_i}^{t_{i + 1}} a(X_{t_i}, \sigma_0) dW_s)^4 ] = 3 \Delta_{n,i}^2 a^4(X_{t_i}, \sigma_0).
\label{eq: varianza stoc}
\end{equation}
On the second term of the right hand side of \eqref{eq: a quarta} we use again BDG inequality to get
$$\mathbb{E}_i[(\int_{t_i}^{t_{i + 1}} [a(X_s, \sigma_0) - a(X_{t_i}, \sigma_0)] dW_s)^4 ]\le {\modch c} \mathbb{E}_i[(\int_{t_i}^{t_{i + 1}} [a(X_s, \sigma_0) - a(X_{t_i}, \sigma_0)]^2 ds)^2 ] \le $$
\begin{equation}
\le {\modch c} \Delta_{n,i} \int_{t_i}^{t_{i + 1}} \left \| \partial_x a \right \|^4_\infty \mathbb{E}_i[|X_s - X_{t_i}|^4] ds \le c\Delta_{n,i} \int_{t_i}^{t_{i + 1}}|s - t_i|(1 + |X_{t_i}|^4) ds \le R(\theta_0, \Delta_{n,i}^3, X_{t_i}),  
\label{eq: stima incremento a}
\end{equation}
where we have also used Jensen inequality and the second point of Lemma \ref{lemma: Moment inequalities}. \\
Concerning the last term in the right hand side of \eqref{eq: a quarta}, from Holder inequality it is upper bounded by 
$$\sum_{j = 1}^3 \binom{4}{j} \mathbb{E}_i[(\int_{t_i}^{t_{i + 1}} a(X_{t_i}, \sigma_0) dW_s)^{j p_1}]^\frac{1}{p_1} \mathbb{E}_i[(\int_{t_i}^{t_{i + 1}} [a(X_s, \sigma_0) - a(X_{t_i}, \sigma_0)] dW_s)^{(4-j) p_2} ]^\frac{1}{p_2}.$$
Now we take $p_1 = \frac{4}{j}$ and so $p_2 = \frac{4}{4-j}$. Therefore, using also \eqref{eq: varianza stoc} and \eqref{eq: stima incremento a}, the expression here above is
$$\sum_{j = 1}^3 \binom{4}{j} \mathbb{E}_i[(\int_{t_i}^{t_{i + 1}} a(X_{t_i}, \sigma_0) dW_s)^{4}]^\frac{j}{4} \mathbb{E}_i[(\int_{t_i}^{t_{i + 1}} [a(X_s, \sigma_0) - a(X_{t_i}, \sigma_0)] dW_s)^{4} ]^\frac{4-j}{4} \le   $$
\begin{equation}
\le \sum_{j = 1}^3 \binom{4}{j}  R(\theta_0, \Delta_{n,i}^2, X_{t_i})^\frac{j}{4}  R(\theta_0, \Delta_{n,i}^3, X_{t_i})^\frac{4 - j}{4} \le \sum_{j = 1}^3 \binom{4}{j}  R(\theta_0, \Delta_{n,i}^{3 - \frac{j}{4}}, X_{t_i}) = R(\theta_0, \Delta_{n,i}^\frac{9}{4}, X_{t_i}),
\label{eq: stima termini incrociati a quarta}
\end{equation}
since when $j= 1$ and $j=2$ we get terms that are negligible if compared to $R(\theta_0, \Delta_{n,i}^\frac{9}{4}, X_{t_i})$. \\
Replacing \eqref{eq: stima varphi -1}, \eqref{eq: varianza stoc} - \eqref{eq: stima termini incrociati a quarta} in \eqref{eq: scomposizione a quarta} it follows
\begin{equation}
\mathbb{E}_i[(\int_{t_i}^{t_{i + 1}} a(X_s, \sigma_0) dW_s)^4 \varphi_{\Delta_{n,i}^\beta}(\Delta X_i) ] = 3 \Delta_{n,i}^2 a^4(X_{t_i}, \sigma_0) + R(\theta_0, \Delta_{n,i}^{ \frac{9}{4}}, X_{t_i}).
\label{eq: finale a quarta}
\end{equation}
We now study the contribution of $B_{i,n}$. First, 
$$\mathbb{E}_i[(B_{i,n})^4 |\varphi_{\Delta_{n,i}^\beta}(\Delta X_i)| ] \le c \mathbb{E}_i[(\int_{t_i}^{t_{i + 1}} b(X_s, \mu_0) ds)^4 |\varphi_{\Delta_{n,i}^\beta}(\Delta X_i)| ] + $$
$$ +c \mathbb{E}_i[(\int_{t_i}^{t_{i + 1}} \int_\mathbb{R} \gamma(X_{s^-}) z \tilde{\mu}(ds, dz))^4 |\varphi_{\Delta_{n,i}^\beta}(\Delta X_i)| ] + R(\theta, \Delta_{n,i}^4, X_{t_i}) \le c \Delta_{n,i}^3 \int_{t_i}^{t_{i + 1}} \mathbb{E}_i[b^4(X_s, \mu_0)] ds + $$
\begin{equation}
+R(\theta_0, \Delta_{n,i}^{1 + 4\beta}, X_{t_i}) = R(\theta_0, \Delta_{n,i}^4, X_{t_i}) + R(\theta_0, \Delta_{n,i}^{1 + 4\beta}, X_{t_i}) = R(\theta_0, \Delta_{n,i}^{1 + 4\beta}, X_{t_i}),
\label{eq: Bin quarta}
\end{equation}
where we have used Jensen inequality, Lemma \ref{lemma: estim jumps} and the fact that $R(\theta_0, \Delta_{n,i}^4, X_{t_i})$ is always negligible compared to $R(\theta_0, \Delta_{n,i}^{1 + 4 \beta}, X_{t_i})$ since $4 > 1 + 4\beta $. \\
Using \eqref{eq: reformulation X - m} we have that
\begin{equation}
\mathbb{E}_i[(X_{t_{i + 1}} - m(\mu, \sigma, X_{t_i}))^4 |\varphi_{\Delta_{n,i}^\beta}(\Delta X_i)| ] = \mathbb{E}_i[(\int_{t_i}^{t_{i + 1}} a(X_s, \sigma_0) dW_s)^4 |\varphi_{\Delta_{n,i}^\beta}(\Delta X_i)| ] + \mathbb{E}_i[(B_{i,n})^4 |\varphi_{\Delta_{n,i}^\beta}(\Delta X_i)| ] +
\label{eq: insieme quarta}
\end{equation}
$$+ \sum_{j = 1}^3 \binom{4}{j} \mathbb{E}_i[(\int_{t_i}^{t_{i + 1}} a(X_s, \sigma_0) dW_s)^j (B_{i, n})^{4 - j}|\varphi_{\Delta_{n,i}^\beta}(\Delta X_i)|]. $$
On the last term here above we act like we did in \eqref{eq: double produit}, using holder inequality and taking $p_1 = \frac{4}{j}$. It follows, using also \eqref{eq: a quarta}, \eqref{eq: varianza stoc}, \eqref{eq: stima incremento a}, \eqref{eq: stima termini incrociati a quarta} and \eqref{eq: Bin quarta}, that it is upper bounded by 
\begin{equation}
\sum_{j = 1}^3 \binom{4}{j} R(\theta_0, \Delta_{n,i}^{2}, X_{t_i})^\frac{j}{4} R(\theta_0, \Delta_{n,i}^{1 + 4 \beta}, X_{t_i})^{\frac{4 - j}{4}}= \sum_{j = 1}^3 \binom{4}{j} R(\theta_0, \Delta_{n,i}^{1 + 4 \beta + \frac{j}{4}(1 - 4 \beta)}, X_{t_i}).
\label{eq: quarta Bin e a }
\end{equation}
Since we have chosen $\beta > \frac{1}{4}$, the terms in which $j=1,2$ are negligible compared to the one in which $j=3$ and so we get $R(\theta_0, \Delta_{n,i}^{\frac{7}{4} + \beta}, X_{t_i})$. From \eqref{eq: finale a quarta}- \eqref{eq: quarta Bin e a } it follows \eqref{eq: Ei quarta}. \\
\\
In order to show \eqref{eq: stima Ei prima} we start considering $B_{i,n}$:
$$|\mathbb{E}_i[B_{i,n}\varphi^k_{\Delta_{n,i}^\beta}(\Delta X_i)]| \le R(\theta_0, \Delta_{n,i}, X_{t_i}) + c \mathbb{E}_i[|\int_{t_i}^{t_{i + 1}}b(X_s, \mu)ds|] + c \mathbb{E}_i[|\int_{t_i}^{t_{i + 1}} \int_\mathbb{R} z \gamma(X_{s^-})\tilde{\mu}(ds, dz)|] \le$$
\begin{equation}
\le R(\theta_0, \Delta_{n,i}, X_{t_i}) + c\mathbb{E}_i[\int_{t_i}^{t_{i + 1}}|b(X_s, \mu)|ds] + c \mathbb{E}_i[\int_{t_i}^{t_{i + 1}} (\int_\mathbb{R}|z| F(z) dz) |\gamma(X_{s^-})| ds] \le R(\theta_0, \Delta_{n,i}, X_{t_i}),  
\label{eq: stima Bin prima}
\end{equation}
having used the definition of $B_{i,n}$ given in \eqref{eq: reformulation X - m}, the boundedness of $\varphi^k$, the polynomial growth of both $b$ and $\gamma$ and the third point of Lemma \ref{lemma: Moment inequalities}. \\
Moreover,
$$|\mathbb{E}_i[(\int_{t_i}^{t_{i + 1}} a(X_s, \sigma_0) dW_s) \varphi^k_{\Delta_{n,i}^\beta}(\Delta X_i) ]| = |\mathbb{E}_i[\int_{t_i}^{t_{i + 1}} a(X_s, \sigma_0) dW_s] + \mathbb{E}_i[(\int_{t_i}^{t_{i + 1}} a(X_s, \sigma_0) dW_s)( \varphi^k_{\Delta_{n,i}^\beta}(\Delta X_i) -1) ]| \le $$
\begin{equation}
 \le R(\theta_0, \Delta_{n,i}^\frac{1}{2}, X_{t_i})\mathbb{E}_i[1_{\left \{|\Delta X_i| \ge \Delta_{n,i}^\beta \right \}}]^\frac{1}{q} \le  R(\theta_0, \Delta_{n,i}^{\frac{3}{2} - \epsilon}, X_{t_i}),
\label{eq: stima a prima}
\end{equation}
where we have used \eqref{eq: stima prob incrementi} and taken $q$ next to $1$.
From the inequality here above and \eqref{eq: stima Bin prima} it follows \eqref{eq: stima Ei prima}. \\
\\
Concerning \eqref{eq: x - m prop 2}; we have 
$$\mathbb{E}_i[|X_{t_{i + 1}} - m(\mu, \sigma, X_{t_i})|^k |\varphi_{\Delta_{n,i}^\beta}(\Delta X_i)|^{k'} ] \le c \mathbb{E}_i[|\int_{t_i}^{t_{i + 1}} a(X_s, \sigma_0) dW_s|^k |\varphi_{\Delta_{n,i}^\beta}(\Delta X_i)|^{k'} ] + c\mathbb{E}_i[|B_{i,n}|^k |\varphi_{\Delta_{n,i}^\beta}(\Delta X_i)|^{k'} ] \le $$
$$\le R(\theta_0, \Delta_{n,i}^\frac{k}{2}, X_{t_i}) + R(\theta_0, \Delta_{n,i}^k, X_{t_i}) + R(\theta_0, \Delta_{n,i}^{1 + k \beta }, X_{t_i}) = R(\theta_0, \Delta_{n,i}^{\frac{k}{2} \land (1 + k \beta)}, X_{t_i}),$$
where we have used on the first term here above the fact that $\varphi$ is bounded and BDG inequality while on the second we have acted like we did in \eqref{eq: Bin carre} or \eqref{eq: Bin quarta}, with $q$ that this time is equal to $k$. \\ \\
We now want to show the fifth and last point of the lemma. Using \eqref{eq: reformulation X - m} we have
$$(X_{t_{i + 1}} - m(\mu_0, \sigma_0, X_{t_i}))^3 = \sum_{j = 0}^3 \binom{3}{j} (\int_{t_i}^{t_{i + 1}} a (X_s, \sigma_0) dW_s)^j B_{i,n}^{3 - j}.$$
Therefore 
$$\mathbb{E}_i[(X_{t_{i + 1}} - m(\mu_0, \sigma_0, X_{t_i}))^3\varphi^{k'}_{\Delta_{n,i}^\beta}(\Delta X_i)] = \sum_{j = 0}^3 \binom{3}{j} \mathbb{E}_i[(\int_{t_i}^{t_{i + 1}} a (X_s, \sigma_0) dW_s)^j B_{i,n}^{3 - j}\varphi^{k'}_{\Delta_{n,i}^\beta}(\Delta X_i)].$$
We observe that, for $j=3$, it is 
$$\mathbb{E}_i[(\int_{t_i}^{t_{i + 1}} a (X_s, \sigma_0) dW_s)^3\varphi^{k'}_{\Delta_{n,i}^\beta}(\Delta X_i)] = \mathbb{E}_i[(\int_{t_i}^{t_{i + 1}} a (X_s, \sigma_0) dW_s)^3] + \mathbb{E}_i[(\int_{t_i}^{t_{i + 1}} a (X_s, \sigma_0) dW_s)^3(\varphi^{k'}_{\Delta_{n,i}^\beta}(\Delta X_i)-1)] \le$$
$$\le c \mathbb{E}_i[(\int_{t_i}^{t_{i + 1}} a (X_{t_i}, \sigma_0) dW_s)^3] +c \mathbb{E}_i[(\int_{t_i}^{t_{i + 1}} [a (X_s, \sigma_0)- a(X_{t_i}, \sigma_0)] dW_s)^3]+ R(\theta_0, \Delta_{n,i}^{\frac{5}{2} - \epsilon}, X_{t_i}). $$
We remark that the first term here above is centered while on the second we can act like we did on \eqref{eq: stima incremento a} (still with an exponent that is 3 instead of 4), obtaining 
$$\mathbb{E}_i[(\int_{t_i}^{t_{i + 1}} [a (X_s, \sigma_0)- a(X_{t_i}, \sigma_0)] dW_s)^3] \le R(\theta_0, \Delta_{n,i}^{\frac{5}{2}}, X_{t_i}).$$
For $j = 0$, instead, we get a term on which we act like we did in \eqref{eq: Bin carre} or \eqref{eq: Bin quarta}, with $q$ that this time is equal to $3$ and so we can upper bound it with $R(\theta_0, \Delta_{n,i}^{(1 + 3 \beta) \land 3 }, X_{t_i})$. \\
For $j = 1$ and $j = 2$ we use Holder inequality, getting
$$\mathbb{E}_i[(\int_{t_i}^{t_{i + 1}} a (X_s, \sigma_0) dW_s)^j B_{i,n}^{3 - j}\varphi^{k'}_{\Delta_{n,i}^\beta}(\Delta X_i)] \le \mathbb{E}_i[(\int_{t_i}^{t_{i + 1}} a (X_s, \sigma_0) dW_s)^{jp}]^\frac{1}{p} \mathbb{E}_i[B_{i,n}^{(3 - j) q}\varphi^{k'q}_{\Delta_{n,i}^\beta}(\Delta X_i)]^\frac{1}{q} \le$$
$$\le \mathbb{E}_i[(\int_{t_i}^{t_{i + 1}} a (X_s, \sigma_0) dW_s)^{3}]^\frac{j}{3} \mathbb{E}[B_{i,n}^{3}\varphi^{k'\frac{3}{3 - j}}_{\Delta_{n,i}^\beta}(\Delta X_i)]^{1 - \frac{j}{3}} = R(\theta_0, \Delta_{n,i}^{\frac{3}{2} \frac{j}{3}}, X_{t_i}  ) R(\theta_0, \Delta_{n,i}^{(1 + 3 \beta) (1 - \frac{j}{3})}, X_{t_i}  ) = $$
$$ = R(\theta_0, \Delta_{n,i}^{1 + 3 \beta + j (\frac{1}{6}- \beta)}, X_{t_i}  ).$$
Now, since $\beta > \frac{1}{4} > \frac{1}{6}$, the term we get for $j=1$ is negligible compared to the one we get for $j =2$, which is $R(\theta_0, \Delta_{n,i}^{ \frac{4}{3} + \beta}, X_{t_i}  )$. In conclusion we have 
$$\mathbb{E}_i[(X_{t_{i + 1}} - m(\mu_0, \sigma_0, X_{t_i}))^3\varphi^{k'}_{\Delta_{n,i}^\beta}(\Delta X_i)] = R(\theta_0, \Delta_{n,i}^{\frac{5}{2} - \epsilon}, X_{t_i}) + R(\theta_0, \Delta_{n,i}^{(1 + 3 \beta) \land 3 }, X_{t_i}) + $$
$$ + R(\theta_0, \Delta_{n,i}^{ \frac{4}{3} + \beta}, X_{t_i}  ) = R(\theta_0, \Delta_{n,i}^{ \frac{4}{3} + \beta}, X_{t_i}  ),$$
since we can always find an $\epsilon > 0$ for which $\frac{5}{2} - \epsilon > 1 + 3 \beta > \frac{4}{3} + \beta$. The result follows.
\end{proof}

\subsubsection{Proof of Lemma \ref{lemma: x-m con varphi'}}
\begin{proof}
We first of all observe that, for all $k\ge 1$, $|\varphi'_{\Delta_{n,i}^\beta}(X_{t_{i + 1}}^\theta - X_{t_i}^\theta)|^{k}$ is different from $0$ only if $|X_{t_{i + 1}}^\theta - X_{t_i}^\theta|\in [\Delta_{n,i}^\beta, 2\Delta_{n,i}^\beta]$. Recalling that from its definition \eqref{eq: reformulation X - m} $B_{i,n} = \int_{t_i}^{t_{i + 1}} b(X_s, \mu) ds + \int_{t_i}^{t_{i + 1}} \int_\mathbb{R} z \gamma(X_{s^-}) \tilde{\mu}(ds, dz) + R(\theta, \Delta_{n,i},X_{t_i})$,
it follows
\begin{equation}
\mathbb{E}[|X_{t_{i + 1}}^\theta - m(\mu, \sigma, X_{t_{i}})|^p |\varphi'_{\Delta_{n,i}^\beta}(X_{t_{i + 1}}^\theta - X_{t_i}^\theta)|^{k}] \le c \mathbb{E}[|\int_{t_i}^{t_{i + 1}} a(X_s^\theta, \sigma)dW_s|^p 1_{\left \{ |X_{t_{i + 1}}^\theta - X_{t_i}^\theta|\in [\Delta_{n,i}^\beta, 2\Delta_{n,i}^\beta]\right \}}] +  \label{eq: ref con varphi'}
\end{equation}
$$+ c\mathbb{E}[|B_{i,n}|^p 1_{\left \{ |X_{t_{i + 1}}^\theta - X_{t_i}^\theta|\in [\Delta_{n,i}^\beta, 2\Delta_{n,i}^\beta] \right \}}].$$
On the first term here above we use Holder inequality, \eqref{eq: BDG} and \eqref{eq: stima prob incrementi}, remarking that $\left \{ |X_{t_{i + 1}}^\theta - X_{t_i}^\theta| \in [\Delta_{n,i}^\beta, 2\Delta_{n,i}^\beta] \right \} \subset \left \{ |X_{t_{i + 1}}^\theta - X_{t_i}^\theta| \ge \Delta_{n,i}^\beta \right \} $. We get it is upper bounded by
$$\mathbb{E}[|\int_{t_i}^{t_{i + 1}} a(X_s^\theta, \sigma)dW_s|^{pp_1}]^\frac{1}{p_1}\mathbb{E}[1_{\left \{ |X_{t_{i + 1}}^\theta - X_{t_i}^\theta|\in [\Delta_{n,i}^\beta, 2\Delta_{n,i}^\beta] \right \}}]^\frac{1}{p_2} \le R(\theta, \Delta_{n,i}^{\frac{p}{2}},X_{t_i})R(\theta, \Delta_{n,i},X_{t_i})^\frac{1}{p_2} = R(\theta, \Delta_{n,i}^{\frac{p}{2} + 1 - \epsilon},X_{t_i}),$$
for all $\epsilon > 0$. In the last inequality we have taken $p_1$ big and $p_2$ next to $1$. \\
We now study the second term of \eqref{eq: ref con varphi'}. From the definition of $B_{i,n}$ given here above, Holder inequality, the polynomial growth of $b$ and still \eqref{eq: stima prob incrementi} we get that the second term of \eqref{eq: ref con varphi'} is upper bounded by $R(\theta,\Delta_{n,i}^{p + 1 - \epsilon},X_{t_i}) + \mathbb{E}[|\Delta X_i^J|^p1_{\left \{ |X_{t_{i + 1}}^\theta - X_{t_i}^\theta| \in [\Delta_{n,i}^\beta, 2\Delta_{n,i}^\beta] \right \}}]$.\\
We now recall that $\Delta X_i^c = (X_{t_{i + 1}}^\theta - X_{t_i}^\theta) - \Delta X_i^J$ and so when $|X_{t_{i + 1}}^\theta - X_{t_i}^\theta| \le 2 \Delta_{n,i}^{\beta}$ and $|\Delta X_i^J| \ge 4\Delta_{n,i}^{\beta}$, then $|\Delta X_i^c|$ must be more than 
$2 \Delta_{n,i}^{\beta}$. Hence
$$\mathbb{E}[|\Delta X_i^J|^p1_{\left \{ |X_{t_{i + 1}}^\theta - X_{t_i}^\theta|\in [\Delta_{n,i}^\beta, 2\Delta_{n,i}^\beta] \right \}}] \le \mathbb{E}[|\Delta X_i^J|^p1_{\left \{ |X_{t_{i + 1}}^\theta - X_{t_i}^\theta|\in [\Delta_{n,i}^\beta, 2\Delta_{n,i}^\beta], |\Delta X_i^J| \ge 4\Delta_{n,i}^{\beta} \right \}}] + $$
$$ + \mathbb{E}[|\Delta X_i^J|^p1_{\left \{ |X_{t_{i + 1}}^\theta - X_{t_i}^\theta|\in [\Delta_{n,i}^\beta, 2\Delta_{n,i}^\beta], |\Delta X_i^J| \le 4\Delta_{n,i}^{\beta} \right \}}] \le c\mathbb{E}[|\Delta X_i^J|^{p p_1}]^\frac{1}{p_1} \mathbb{P}(|\Delta X_i^c| \ge 2 \Delta_{n,i}^\beta)^\frac{1}{p_2} +$$
\begin{equation}
+ c \Delta_{n,i}^{\beta p} \mathbb{P}(|X_{t_{i + 1}}^\theta - X_{t_i}^\theta|\in [\Delta_{n,i}^\beta, 2\Delta_{n,i}^\beta], |\Delta X_i^J| \le 4\Delta_{n,i}^{\beta}) \le
\label{eq: per lemma jumps}
\end{equation}
$$\le R(\theta, \Delta_{n,i}^\frac{1}{p_1}, X_{t_i})R(\theta, \Delta_{n,i}^{\frac{(\frac{1}{2} - \beta)r}{p_2}}, X_{t_i}) + R(\theta, \Delta_{n,i}^{1 + \beta p}, X_{t_i}) = R(\theta, \Delta_{n,i}^{(\frac{1}{2} - \beta)r - \epsilon }, X_{t_i})+ R(\theta, \Delta_{n,i}^{1 + \beta p}, X_{t_i}) =  R(\theta, \Delta_{n,i}^{1 + \beta p}, X_{t_i}),$$
where we have used Kunita inequality (for $p p_1 \ge 2$, that holds since we take $p_1$ big and $p_2$ next to $1$), Tchebyschev inequality as we did in \eqref{eq: stima prob incrementi} on the first term and still \eqref{eq: stima prob incrementi} on the second. Moreover we have used that, by the arbitrariness of $r > 0$, we can always find $r$ and $\epsilon$ such that $(\frac{1}{2} - \beta)r - \epsilon > 1 + \beta p$. The result follows.
\end{proof}

\subsubsection{Proof of Lemma \ref{lemma: estim jumps}}
\begin{proof}
The case $q \ge 2$ has already been proved in Lemma 10 of \cite{Chapitre 1} so, we are going to focus on the case $q \in [1,2)$. \\
For all $n \in \mathbb{N}$ and $i \in \mathbb{N}$ we define the set on which all the jumps of $L$ on the interval $(t_i, t_{i+1}]$ are small:
$$N_n^i : = \left \{ |\Delta L_s| \le \frac{4 \Delta_{n,i}^\beta}{\gamma_{min}}; \quad \forall s \in (t_i, t_{i+1}]  \right \}.$$
We split the jumps on $N_{i,n}$ and its complementary, getting
\begin{equation}
\mathbb{E}_i[|\Delta {X}_i^J\varphi_{\Delta_{n,i}^\beta}(X_{t_{i+1}} - X_{t_i})|^q 1_{N^i_n} ] + \mathbb{E}_i[|\Delta {X}_i^J\varphi_{\Delta_{n,i}^\beta}(X_{t_{i+1}} - X_{t_i})|^q 1_{(N^i_n)^c} ]. 
\label{eq: salti con partizione}
\end{equation}
We now observe that, by the definition of $N^i_n$, the first term here above is upper bounded by
$$\mathbb{E}_i[|\int_{t_i}^{t_{i+1}} \int_{|z| \le \frac{4 \Delta_{n,i}^\beta}{\gamma_{min}}} z \, \gamma(X_{s^-}) \tilde{\mu}(ds,dz)|^q+ |\int_{t_i}^{t_{i+1}} \int_{|z| \ge \frac{4 \Delta_{n,i}^\beta}{\gamma_{min}}} |z| \, |\gamma(X_{s^-})| \bar{\mu}(ds,dz) |^q ].$$
From our assumptions on the jump density the second term here above is upper bounded by a $R(\theta, \Delta_{n,i}^q, X_{t_i})$ function while on the first one we use Lemma 2.1.5 of \cite{Protter}. We can therefore upper bound it with $$c \mathbb{E}_i[\int_{t_i}^{t_{i+1}} \int_{|z| \le \frac{4 \Delta_{n,i}^\beta}{\gamma_{min}}}|z|^q |\gamma(X_{s^-})|^q \bar{\mu}(ds,dz) ] \le R(\theta, \Delta_{n,i}^{1 + \beta q}, X_{t_i}),$$ 
having used again that $\bar{\mu}(ds,dz) = F(dz) ds$ and Assumption 4 on $F$. It follows
$$\mathbb{E}_i[|\Delta {X}_i^J\varphi_{\Delta_{n,i}^\beta}(\Delta X_i)|^q 1_{N^i_n} ] \le R(\theta, \Delta_{n,i}^q, X_{t_i}) + R(\theta, \Delta_{n,i}^{1 + \beta q}, X_{t_i}) = R(\theta, \Delta_{n,i}^q, X_{t_i}).$$
Regarding the second term of \eqref{eq: salti con partizione}, we have that $|\Delta X_i^J| \le |\Delta X_i|+ |\Delta X_i^c|$ and, as we have already remarked several times, by the definition of $\varphi$ it is different from zero only if $|\Delta X_i|^q \le c \Delta_{n,i}^{\beta q}$. It follows 
$$\mathbb{E}_i[|\Delta {X}_i|^q |\varphi_{\Delta_{n,i}^\beta}(\Delta X_i)|^q 1_{(N^i_n)^c} ] \le c \Delta_{n,i}^{\beta q}\mathbb{P}_i((N^i_n)^c) \le R(\theta, \Delta_{n,i}^{1 + \beta q}, X_{t_i}),$$
where the last inequality is a consequence of the following:
$$\mathbb{P}_i((N^i_n)^c) = \mathbb{P}_i(\exists s \in (t_i, t_{i+1}]\, : \, |\Delta L_s| > \frac{4 \Delta_{n,i}^\beta}{\gamma_{min}} ) \le c \int_{t_i}^{t_{i+1}} \int_{\frac{4 \Delta_{n,i}^\beta}{\gamma_{min}} }^{\infty} F(z) dz ds \le c \Delta_{n,i}.$$
In the same way
$$\mathbb{E}_i[|\Delta {X}_i^c|^q |\varphi_{\Delta_{n,i}^\beta}(\Delta X_i)|^q 1_{(N^i_n)^c} ] \le c \Delta_{n,i}^{\frac{1}{2} q - \epsilon}\Delta_{n,i}^{1 - \epsilon}(1 + |X_{t_i}|^c ) = R(\theta, \Delta_{n,i}^{1 + \frac{1}{2} q}, X_{t_i}).$$
Putting all pieces together we have 
$$\mathbb{E}_i[|\Delta {X}_i^J|^q |\varphi_{\Delta_{n,i}^\beta}(\Delta X_i)|^q  ] \le R(\theta, \Delta_{n,i}^q, X_{t_i}),$$
that is the result we wanted remarking that, for $q \in [1,2)$, $1 + \beta q  > q$ and so $\Delta_{n,i}^q= \Delta_{n,i}^{q \land (1 + \beta q)}$. 
\end{proof}

\subsubsection{Proof of Proposition \ref{prop: punto 2 conv deriv seconda}}
\begin{proof}
We want to prove the tightness of the sequence $S_n(\theta)$. Since the sum of tight sequences is still tight, we show the tightness of the sequence $S_{n1} (\theta)$ and $S_{n2} (\theta)$ which are such that $S_n (\theta) = S_{n1} (\theta) + S_{n2} (\theta)$:
$$S_{n1} (\theta) := \frac{1}{T_n} \sum_{i = 0}^{n - 1} (X_{t_{i + 1}} - m(\mu_0, \sigma_0, X_{t_i})) \varphi_{\Delta_{n,i}^\beta}(X_{t_{i + 1}} - X_{t_i}) g_{i,n}(\theta, X_{t_i}) \quad \mbox{and}$$
$$S_{n2} (\theta) := \frac{1}{T_n} \sum_{i = 0}^{n - 1} (m(\mu_0, \sigma_0, X_{t_i}) - m(\mu, \sigma, X_{t_i})) \varphi_{\Delta_{n,i}^\beta}(X_{t_{i + 1}} - X_{t_i}) g_{i,n}(\theta, X_{t_i}).$$
We prove that $S_{n1} (\theta)$ is tight using Kolmogorov criterion, we therefore want to show that inequalities analogous to \eqref{eq: 1 criterion tightness} and \eqref{eq: 2 criterion tightness} hold. Starting with the proof of \eqref{eq: 2 criterion tightness} we have that, using Burkholder and Jensen inequality,
\begin{equation}
\mathbb{E}[|S_{n1} (\theta_1) - S_{n1} (\theta_2)|^m] \le \frac{c n^{\frac{m}{2} - 1}}{(n \Delta_{n})^m}\sum_{i = 0}^{n - 1}\mathbb{E}[|(X_{t_{i + 1}} - m(\mu_0, \sigma_0, X_{t_i}))|^m |\varphi^m_{\Delta_{n,i}^\beta}(X_{t_{i + 1}} - X_{t_i})|g_{i,n}(\theta_1, X_{t_i}) - g_{i,n}(\theta_2, X_{t_i})|^m ].
\label{eq: kolm prop}
\end{equation}
We observe that, using finite-increments theorem and the assumption on the derivatives of $g_{i,n}$ with respect to the parameters, it is
\begin{equation}
|g_{i,n}(\theta_1, X_{t_i}) - g_{i,n}(\theta_2, X_{t_i})|^m \le |R(\theta, 1, X_{t_i}) |^m |\mu_1- \mu_2|^m + |R(\theta, 1, X_{t_i}) |^m |\sigma_1- \sigma_2|^m
\label{eq: estim g}
\end{equation}
where actually the function $R$ is computed in a point $\tilde{\theta}: = (\tilde{\mu}, \tilde{\sigma})$, with $\tilde{\mu} \in (\mu_1, \mu_2)$ and $\tilde{\sigma} \in (\sigma_1, \sigma_2)$ but, since the property \eqref{eq: definition R} of $R$ is uniform in $\theta$, we have chosen to write it simply as $R(\theta, 1, X_{t_i})$.
Replacing \eqref{eq: estim g} in \eqref{eq: kolm prop} and using the fourth point of Lemma \ref{lemma: conditional expected value} it follows
$$\mathbb{E}[|S_{n1} (\theta_1) - S_{n1} (\theta_2)|^m] \le \frac{c n^{\frac{m}{2} - 1}}{(n \Delta_{n})^m} n \Delta_{n}^{\frac{m}{2} \land (1 + m \beta )}(|\mu_1- \mu_2|^m +|\sigma_1- \sigma_2|^m) \le \frac{c }{(n \Delta_n)^\frac{m}{2}} \Delta_{n}^{0 \land (1 + m \beta - \frac{m}{2} )}(|\mu_1- \mu_2|^m +|\sigma_1- \sigma_2|^m), $$
with $\frac{1}{(n \Delta_n)^\frac{m}{2}} \Delta_{n}^{0 \land (1 + m \beta - \frac{m}{2} )} < c$ because $n \Delta_n$ is lower bounded by a constant and we can always find an $m \ge 2$ for which $1 + m \beta - \frac{m}{2}  > 0$ since $\beta \in (\frac{1}{4}, \frac{1}{2})$. Hence, \eqref{eq: 2 criterion tightness} is proved.  \\
Acting exactly in the same way but using this time the control on $g_{i,n}$ instead of on its derivatives we have also an estimation for $S_{n1}$ analogous to \eqref{eq: 1 criterion tightness}; the tightness of $S_{n1}$ follows. \\
Concerning $S_{n2}$ we observe that, for $\vartheta = \mu$ and $\vartheta = \sigma$, it is
$$|\partial_\vartheta S_{n2}(\theta)| \le c|\partial_\vartheta m(\mu, \sigma, X_{t_i})|| g_{i,n}(\theta, X_{t_i})| + c|m(\mu_0, \sigma_0, X_{t_i}) - m(\mu, \sigma, X_{t_i})| |\partial_\vartheta g_{i,n}(\theta, X_{t_i})|.$$
From the controls we have assumed on $g_{i,n}$ and its derivatives, the finite-increments theorem and the first and the second point of Proposition \ref{prop: dl derivate prime} we have $|\partial_\vartheta S_{n2}(\theta)| \le R(\theta, \Delta_{n,i}, X_{t_i})$. Therefore for both $\vartheta = \mu$ and $\vartheta = \sigma$, using also that $\frac{1}{T_n} = O(\frac{1}{n \Delta_n})$, we get
$$\sup_n \mathbb{E}[\sup_{\mu, \sigma}|\partial_\vartheta S_{n2}(\theta)|] \le \sup_n \frac{c}{T_n} \sum_{i = 0}^{n - 1} \mathbb{E}[\sup_{\mu, \sigma}|R(\theta, \Delta_{n,i}, X_{t_i} )] \le c.$$
The tightness of $S_{n2}$ (and therefore of $S_n$) follows.

\end{proof}

\subsection{Proof of derivatives of $m$ and $m_2$}\label{Ss:proof_der_m_m_2}
In order to prove the developments and the bounds on the derivatives of $m$ and $m_2$, the following lemmas will be useful. We point out that $X_t^\theta$ is $X_t^{\theta, x}$ and so the process starts in $0$: $X_0^{\theta,x} = x$.
\begin{lemma}
Suppose that Assumptions from 1 to 4 and A7 hold. Then, $\forall p \ge 2$ $\exists c > 0$: $\forall h \le \Delta_n$ $\forall x$ we have
\begin{equation}
\mathbb{E}[|\frac{\partial_\mu X^{\theta,x}_h}{h}|^p ] \le c(1 + |x|^c ), \qquad \mathbb{E}[|\frac{\partial^2_\mu X^{\theta,x}_h}{h}|^p ] \le c(1 + |x|^c ), 
\label{eq: derivate x rispetto mu}
\end{equation}
\begin{equation}
\mathbb{E}[|\frac{\partial_\sigma X^{\theta,x}_h}{h^\frac{1}{2}}|^p ] \le c(1 + |x|^c ), \qquad \mathbb{E}[|\frac{\partial^2_{\sigma \mu} X^{\theta,x}_h}{h^\frac{3}{2}}|^p ] \le c(1 + |x|^c ), 
\label{eq: derivata x rispetto sigma e mista}
\end{equation}
\begin{equation}
\mathbb{E}[|\frac{\partial^2_{\sigma} X^{\theta,x}_h}{h^\frac{1}{2}}|^p ] \le c(1 + |x|^c ), 
\label{eq: derivate seconde x rispetto a sigma}
\end{equation}
\begin{equation}
\mathbb{E}[|\frac{\partial^3_{\sigma} X^{\theta,x}_h}{h^\frac{1}{2}}|^p ] \le c(1 + |x|^c ), \qquad \mathbb{E}[|\frac{\partial^3_{\mu} X^{\theta,x}_h}{h}|^p ] \le c(1 + |x|^c )
\label{eq: derivate terze x rispetto a sigma e mu}
\end{equation}
\begin{equation}
\mathbb{E}[|\frac{\partial^3_{\sigma\mu \sigma} X^{\theta,x}_h}{h^\frac{3}{2}}|^p ] \le c(1 + |x|^c ), \qquad \mathbb{E}[|\frac{\partial^3_{\mu \mu \sigma} X^{\theta,x}_h}{h^\frac{3}{2}}|^p ] \le c(1 + |x|^c ), 
\label{eq: derivate terze x miste}
\end{equation}

\label{lemma: derivate x}
\end{lemma}
\begin{proof}\textit{Lemma \ref{lemma: derivate x}.} \\
Inequalities \eqref{eq: derivate x rispetto mu} have already been proved in Lemma 9 of \cite{Chapitre 1}. To show the first inequality of \eqref{eq: derivata x rispetto sigma e mista}, we observe that the dynamic of the process $\partial_\sigma X^{\theta, x}$ is known (cf. \cite{Derivative EDS}, section 5):  
\begin{equation}
\partial_\sigma X^{\theta, x}_h = \int_0^h \partial_x b(\mu, X^{\theta,x}_s)\partial_\sigma X^{\theta,x}_s ds + \int_0^h (\partial_x a(\sigma, X^{\theta,x}_s)\partial_\sigma X^{\theta,x}_s + \partial_\sigma a(\sigma, X^{\theta,x}_s)) dW_s + \int_0^h \int_\mathbb{R} \partial_x \gamma(X^{\theta,x}_{s^-}) \partial_\sigma X^{\theta,x}_s z \tilde{\mu}(dz, ds).
\label{eq: dynamic partial sigma x}
\end{equation}
From now on, we will drop the dependence of the starting point in order to make the notation easier. Taking the $L^p$ norm of \eqref{eq: dynamic partial sigma x} we get it is upper bounded by the sum of three terms. On the first one we use Jensen inequality and the fact that the derivatives of $b$ with rapport to $x$ are supposed bounded to obtain
\begin{equation}
\mathbb{E}[|\int_0^h (\partial_x b(\mu, X^{\theta}_s)\partial_\sigma X^{\theta}_s ds|^p] \le h^{p-1 } \int_0^h \mathbb{E}[|\partial_x b(\mu, X^{\theta}_s)|^p|\partial_\sigma X^{\theta}_s|^p] ds \le c h^{p - 1} \int_0^h \mathbb{E}[|\partial_\sigma X^{\theta}_s|^p] ds.
\label{eq: b in partial sigma}
\end{equation}
Let us now consider the stochastic integral. Using Burkholder-Davis-Gundy inequality we have 
$$\mathbb{E}[|\int_0^h (\partial_x a(\sigma, X^{\theta}_s)\partial_\sigma X^{\theta}_s + \partial_\sigma a(\sigma, X^{\theta}_s)) dW_s|^p] \le c\mathbb{E}[|\int_0^h (\partial_x a(\sigma,X^{\theta}_s))^2(\partial_\sigma X^{\theta}_s)^2 ds|^\frac{p}{2}] +$$
\begin{equation}
+ c\mathbb{E}[|\int_0^h (\partial_\sigma a(\sigma, X^{\theta}_s))^2 ds|^\frac{p}{2}] \le ch^{\frac{p}{2}- 1} \int_0^h \mathbb{E}[|\partial_\sigma X^{\theta}_s|^p] ds + ch^{\frac{p}{2}- 1} \int_0^h \mathbb{E}[(1 + |X_s^\theta|^c)] ds \le
\label{eq: a in partial sigma}
\end{equation}
$$\le ch^{\frac{p}{2}- 1} \int_0^h \mathbb{E}[|\partial_\sigma X^{\theta}_s|^p] ds + ch^{\frac{p}{2}}(1 + |x|^c), $$
where we have used Jensen inequality, the fact that, by A7, the derivatives of $a$ with rapport to $x$ are supposed bounded and those with rapport to $\sigma$ have polynomial growth and the second point of Lemma \ref{lemma: Moment inequalities}. \\
We now consider the third term on the right hand side of \eqref{eq: dynamic partial sigma x}, it can be estimated using Kunita inequality (cf. the Appendix of \cite{Protter}):
$$\mathbb{E}[|\int_0^h \int_\mathbb{R} \partial_x \gamma(X^\theta_{s^-})\partial_\sigma X^{\theta}_s z \tilde{\mu}(dz, ds)|^p] \le {\modch c} \mathbb{E}[\int_0^h \int_\mathbb{R} |\partial_x \gamma(X^\theta_{s^-})\partial_\sigma X^{\theta}_s|^p |z|^p \bar{\mu}(dz, ds)] +$$
$$ + {\modch c} \mathbb{E}[|\int_0^h \int_\mathbb{R} (\partial_x \gamma(X^\theta_{s^-})\partial_\sigma X^{\theta}_s)^2 z^2 \bar{\mu}(dz, ds)|^\frac{p}{2}] \le  {\modch c} \int_0^h\mathbb{E}[|\partial_x \gamma(X^\theta_{s^-})|^p|\partial_\sigma X^{\theta}_s|^p] ( \int_\mathbb{R} |z|^p F(z)dz) ds + $$
$$+ 
{\modch c} \mathbb{E}[|\int_0^h (\partial_x \gamma(X^\theta_{s^-})\partial_\sigma X^{\theta}_s)^2 (\int_\mathbb{R}z^2 F(z)dz) ds|^\frac{p}{2}]\le c\int_0^h\mathbb{E}[|\partial_\sigma X^{\theta}_s|^p] ds + 
c \mathbb{E}[|\int_0^h (\partial_\sigma X^{\theta}_s)^2 ds|^\frac{p}{2}], $$
where in the last two inequalities we have just used the definition of the compensated measure $\bar{\mu}$,the third point of Assumption 4 and the fact that the derivatives of $\gamma$ are supposed bounded. By the Jensen inequality we get
\begin{equation}
\mathbb{E}[|\int_0^h \int_\mathbb{R} \partial_x \gamma(X^\theta_{s^-})\partial_\sigma X^{\theta}_s z \tilde{\mu}(dz, ds)|^p] \le c(1 + h^{\frac{p}{2} - 1}) \int_0^h\mathbb{E}[|\partial_\sigma X^{\theta}_s|^p] ds.
\label{eq: estimation gamma expected norme p dotx }
\end{equation}
From \eqref{eq: b in partial sigma}, \eqref{eq: a in partial sigma} and \eqref{eq: estimation gamma expected norme p dotx }, it follows
$$\mathbb{E}[|\partial_\sigma X_h^\theta|^p] \le c(h^{p - 1} + h^{\frac{p}{2} - 1} + 1)\int_0^h\mathbb{E}[|\partial_\sigma X^{\theta}_s|^p] ds + ch^{\frac{p}{2}}(1 + |x|^c).$$
Gronwall Lemma gives us
$$\mathbb{E}[|\partial_\sigma X_h^\theta|^p] \le ch^{\frac{p}{2}}(1 + |x|^c) e^{c(h^{p - 1} + h^{\frac{p}{2} - 1} + 1)}, $$
we therefore obtain the first inequality in \eqref{eq: derivata x rispetto sigma e mista}. Concerning the second, we observe we can deduce the dynamic of the process $\partial_{\mu \sigma}X^\theta$ from \eqref{eq: dynamic partial sigma x}. It is
$$\partial^2_{\mu \sigma}X^\theta_h = \int_0^h (\partial^2_x b(\mu, X_s^\theta) \partial_\sigma X_s^\theta \partial_\mu X_s^\theta + \partial^2_{\mu x}b (\mu, X_s^\theta) \partial_\sigma X_s^\theta + \partial_x b(\mu, X_s^\theta) \partial^2_{\sigma \mu}X_s^\theta ) ds + \int_0^h (\partial_x^2 a (\sigma, X_s^\theta)\partial_\sigma X_s^\theta \partial_\mu X_s^\theta + $$
$$+ \partial_{\sigma x}^2 a (\sigma, X_s^\theta) \partial_\mu X_s^\theta + \partial_x  a (\sigma, X_s^\theta) \partial^2_{\sigma \mu}X_s^\theta) dW_s + \int_0^h \int_\mathbb{R} (\partial_x^2 \gamma(X^{\theta}_{s^-}) \partial_\sigma X_s^\theta \partial_\mu X_s^\theta + \partial_x \gamma(X^{\theta}_{s^-}) \partial^2_{\sigma \mu}X_s^\theta) z \tilde{\mu}(ds, dz) $$
On the $p$-norm of the first integral we use Jensen inequality, the fact that the derivatives with respect to $x$ are bounded and the estimation we have already proved on the $L^p$ norm of the derivatives of our process with respect to $\mu$ and $\sigma$. We get it is upper bounded by
$$ch^{p -1 } \int_0^h (\mathbb{E}[| \partial_\sigma X_s^\theta \partial_\mu X_s^\theta|^p] + \mathbb{E}[| \partial_\sigma X_s^\theta|^p] + \mathbb{E}[|\partial^2_{\sigma \mu}X_s^\theta|^p] ) ds \le c(h^{\frac{5}{2}p} + h^{\frac{3}{2}p})(1 +|x|^c) + ch^{p - 1}\int_0^h \mathbb{E}[|\partial^2_{\sigma \mu}X_s^\theta|^p] ds,$$
having also used Holder inequality. Acting in the same way on the stochastic integral we get it is upper bounded by
$$ch^{\frac{p}{2}}(h^{\frac{3}{2}p} + h^p)(1 +|x|^c) + ch^{\frac{p}{2} - 1}\int_0^h \mathbb{E}[|\partial^2_{\sigma \mu}X_s^\theta|^p] ds,$$
while we upper bound the third term in the dynamic of $\partial_{\sigma \mu}^2X^\theta$, acting as we did in order to show \eqref{eq: estimation gamma expected norme p dotx }, with
$$ch^{\frac{3}{2}p + 1}(1 +|x|^c) + c \int_0^h \mathbb{E}[|\partial^2_{\sigma \mu}X_s^\theta|^p] ds + ch^{2 p}(1 +|x|^c) + ch^{\frac{p}{2} - 1}\int_0^h \mathbb{E}[|\partial^2_{\sigma \mu}X_s^\theta|^p] ds.$$
In total we have, not considering the negligible terms, 
$$\mathbb{E}[|\partial^2_{\sigma \mu}X_h^\theta|^p] \le ch^{\frac{3}{2}p}(1 +|x|^c) + c(h^{p - 1} + h^{\frac{p}{2} - 1} + 1)\int_0^h \mathbb{E}[|\partial^2_{\sigma \mu}X_s^\theta|^p] ds.$$
From Gronwall Lemma it follows the second inequality of \eqref{eq: derivata x rispetto sigma e mista}, as we wanted. \\
We are left to show \eqref{eq: derivate seconde x rispetto a sigma}. Again, the dynamic of $\partial_\sigma^2 X^\theta$ is known:
$$\partial^2_{\sigma}X^\theta_h = \int_0^h (\partial^2_x b(\mu, X_s^\theta) (\partial_\sigma X_s^\theta)^2 + \partial_x b(\mu, X_s^\theta) \partial^2_{\sigma}X_s^\theta ) ds + \int_0^h (\partial_x^2 a (\sigma, X_s^\theta)(\partial_\sigma X_s^\theta)^2  + 2 \partial_{\sigma x}^2 a (\sigma, X_s^\theta) \partial_\sigma X_s^\theta + $$
\begin{equation}
 + \partial_x  a (\sigma, X_s^\theta) \partial^2_{\sigma}X_s^\theta + \partial^2_\sigma a(\sigma, X_s^\theta)) dW_s + \int_0^h \int_\mathbb{R} (\partial_x^2 \gamma(X^{\theta}_{s^-}) (\partial_\sigma X_s^\theta)^2 + \partial_x \gamma(X^{\theta}_{s^-}) \partial^2_{\sigma}X_s^\theta) z \tilde{\mu}(ds, dz).
\label{eq: dinamica partial2 sigma}
\end{equation}
Acting exactly like we did for the estimation of the $p$-moments of the processes $\partial_\sigma X^\theta$ and $\partial^2_{\sigma \mu}X^\theta$ we get
\begin{equation}
\mathbb{E}[|\partial^2_{\sigma}X^\theta_h|^p] \le c(1 +|x|^c)(h^{2p} + h^{\frac{3}{2}p}+ h^{p} + h^\frac{p}{2} + h^{p + 1}) + c(h^{p-1}+ h^{\frac{p}{2}-1} +1)\int_0^h \mathbb{E}[|\partial^2_{\sigma}X_s^\theta|^p] ds.
\label{eq: stima din partial 2 sigma X}
\end{equation}
Using Gronwall Lemma and remarking that the other terms are negligible compared to $ch^\frac{p}{2}(1 +|x|^c)$, we obtain the result wanted. \\
Concerning the third derivatives, it is easy to see that, writing the dynamics of $\partial^3_\sigma X^\theta_h$, $\partial^3_\mu X^\theta_h$, $\partial^3_{\sigma \mu \sigma} X^\theta_h$ and $\partial^3_{ \mu \mu \sigma} X^\theta_h$ the principal terms are such that their order are, respectively, $h^\frac{1}{2}$, $h$ and twice $h^\frac{3}{2}$. Acting exactly like before, \eqref{eq: derivate terze x rispetto a sigma e mu} and \eqref{eq: derivate terze x miste} follow.
\end{proof}

We are left to show one last proposition, before showing Propositions \ref{prop: dl derivate prime}, \ref{prop: derivate seconde} and \ref{prop: dervate terze}:
\begin{proposition}
Suppose that Assumptions 1 to 4 hold.
Moreover suppose that $(Z_h)_h$ is a family of random variables such that $\mathbb{E}[|Z_h|^p | X_0^\theta = x] \le c (1 + |x|^c)$. Then $\forall k \ge 1$ $\forall \epsilon > 0$, we have
$$ \sup_{h \in [0, \Delta_n]}\mathbb{E}[|Z_h||\varphi_{h^\beta}^{(k)}(X_{h}^\theta - x)||X_0^\theta = x] = R(\theta,h^{1 - \epsilon}, x).$$
We have used $\varphi_{h^\beta}^{(k)}(y)$ in order to denote the k-th derivative $\varphi^{(k)}(\frac{y}{h^\beta})$.
\label{prop: truc moche h}
\end{proposition}
\begin{proof}\textit{Proposition \ref{prop: truc moche h}.} \\
Once again, $|\varphi_{h^\beta}^{(k)}(X_{h}^\theta - x)|$ is different from $0$ only if $|X_{h}^\theta - x| \in [h^\beta, 2 h^\beta]$. We can therefore use Holder inequality (with $p$ big and $q$ next to $1$) and \eqref{eq: stima prob incrementi} to get, $\forall h \in [0, \Delta_n]$, 
$$\mathbb{E}[|Z_h||\varphi_{h^\beta}^{(k)}(X_{h}^\theta - x)| |X_0^\theta = x]  \le \mathbb{E}[|Z_h|^p|X_0^\theta = x]^\frac{1}{p} \mathbb{E}[1_{\left \{ |X_{h}^\theta - x| \in [h^\beta, 2 h^\beta]\right \}} |X_0^\theta = x]^\frac{1}{q} \le R(\theta, h, x)^\frac{1}{q} = R(\theta, h^{1 - \epsilon},x).$$
\end{proof}
\subsubsection{Proof of Proposition \ref{prop: dl derivate prime}}
\begin{proof}
The first point has already been showed in Proposition 8 of \cite{Chapitre 1}, we start proving the second. Since by the homogeneity of the equation $m$ and $m_2$ depend only on the difference $t_{i+1} - t_i$ we can consider WLOG $\forall h \le \Delta_n$
\begin{equation}
m(\mu, \sigma, x): =\frac{\mathbb{E}[X_h^\theta \varphi_{h^\beta}(X_{h}^\theta - X_0^\theta)|X_0^\theta = x]}{\mathbb{E}[\varphi_{h^\beta}(X_h^\theta - X_0^\theta)|X_0^\theta = x]} = \frac{\mathbb{E}[X_h^\theta \varphi_{h^\beta}(X_{h}^\theta - x)]}{\mathbb{E}[\varphi_{h^\beta}(X_h^\theta -x)]}.
\label{eq: def m per prop derivate}
\end{equation}
Hence,
$$\partial_\sigma m(\mu, \sigma, x) = \frac{\mathbb{E}[(\partial_\sigma X_h^\theta) \varphi_{h^\beta}(X_{h}^\theta - x)] + \mathbb{E}[X_h^\theta h^{- \beta} (\partial_\sigma X_{h}^\theta) \varphi'_{h^\beta}(X_{h}^\theta - x)]}{\mathbb{E}[\varphi_{h^\beta}(X_h^\theta - x)]} - m(\mu, \sigma, x)\frac{\mathbb{E}[h^{- \beta} (\partial_\sigma X_{h}^\theta) \varphi'_{h^\beta}(X_{h}^\theta - x)]}{\mathbb{E}[\varphi_{h^\beta}(X_h^\theta - x)]}.$$
On the numerator of the second and third term we use Proposition \ref{prop: truc moche h} taking as ${\modch Z_h}$, respectively, $X_h^\theta \frac{\partial_\sigma X_{h}^\theta}{h^\frac{1}{2}}$ and $\frac{\partial_\sigma X_{h}^\theta}{h^\frac{1}{2}}$. We get, remarking moreover that from \eqref{eq: dl m} $m(\mu, \sigma, x)$ is $R(\theta, 1, x)$ and from Theorem  1 in \cite{Chapitre 1} also the denominator is lower bounded for $|x| < |h|^{- k}$,
\begin{equation}
|\partial_\sigma m(\mu, \sigma, x)| \le R(\theta, 1, x)|\mathbb{E}[(\partial_\sigma X_h^\theta) \varphi_{h^\beta}(X_{h}^\theta - x)] | +  R(\theta, h^{\frac{3}{2}- \beta - \epsilon}, x). 
\label{eq: estim partial sigma m}
\end{equation}
To estimate $|\mathbb{E}[(\partial_\sigma X_h^\theta) \varphi_{h^\beta}(X_{h}^\theta - x)] |$ we replace the dynamic \eqref{eq: dynamic partial sigma x} of $\partial_\sigma X_h^\theta$. On the first integral we use Holder inequality and \eqref{eq: b in partial sigma} to get
$$|\mathbb{E}[\int_0^h (\partial_x b(X^{\theta,x}_s, \mu)\partial_\sigma X^{\theta,x}_s ds \varphi_{h^\beta}(X_{h}^\theta - x)]| \le (c h^{p - 1}\int_0^h \mathbb{E}[|\partial_\sigma X^{\theta,x}_s|^p])^\frac{1}{p} \le c h^{\frac{3}{2}}(1 + |x|^c)= R(\theta, h^\frac{3}{2},x),$$
where in the last inequality we have also used the first inequality of \eqref{eq: derivata x rispetto sigma e mista}. On $\int_0^h \partial_x a(X^{\theta,x}_s, \sigma)\partial_\sigma X^{\theta,x}_s dW_s$ we use again Holder inequality, \eqref{eq: a in partial sigma} (considering only the estimation on its first term) and the first inequality of \eqref{eq: derivata x rispetto sigma e mista} to obtain
$$|\mathbb{E}[\int_0^h \partial_x a(X^{\theta,x}_s, \sigma)\partial_\sigma X^{\theta,x}_s dW_s \varphi_{h^\beta}(X_{h}^\theta - x)]|\le R(\theta, h, x).$$
Concerning $|\mathbb{E}[\int_0^h \partial_\sigma a(X^{\theta,x}_s, \sigma) dW_s \varphi_{h^\beta}(X_{h}^\theta - x)]|$, we act on it like we did in \eqref{eq: stima a prima}, with $\partial_\sigma a$ instead of $a$. We therefore get $|\mathbb{E}[\int_0^h \partial_\sigma a(X^{\theta,x}_s, \sigma) dW_s \varphi_{h^\beta}(X_{h}^\theta - x)]| \le R(\theta, h, x)$. To conclude the proof of this point we use on the jump part Holder inequality, \eqref{eq: estimation gamma expected norme p dotx } and the first inequality of \eqref{eq: derivata x rispetto sigma e mista}. We get
$$|\mathbb{E}[\int_0^h \int_\mathbb{R} \partial_x \gamma(X^\theta_{s^-})\partial_\sigma X^{\theta}_s z \tilde{\mu}(dz, ds)\varphi_{h^\beta}(X_{h}^\theta - x)]| \le (c(1 + h^{\frac{p}{2} - 1}) \int_0^h\mathbb{E}[|\partial_\sigma X^{\theta}_s|^p] ds)^\frac{1}{p} \le R(\theta, h^{(\frac{1}{p}+ \frac{1}{2}) \land 1}, x).$$
We can take $p=2$, finding $|\mathbb{E}[(\partial_\sigma X_h^\theta) \varphi_{h^\beta}(X_{h}^\theta - x)] | = R(\theta, h, x)$. Replacing it in \eqref{eq: estim partial sigma m} and observing that $\frac{3}{2} - \beta $ is always more than $1$, it follows the second point of Proposition \ref{prop: dl derivate prime}. \\
In order to prove the third and the fourth point we first of all need to compute the derivative of $m_2$ with respect to both the parameters. We can just write
$$\partial_\vartheta m_2(\mu, \sigma, x) = 2 \frac{\mathbb{E}[(X_h^\theta - m(\mu, \sigma, x))(\partial_\vartheta X_h^\theta - \partial_\vartheta m(\mu, \sigma,x)) \varphi_{h^\beta}(X_{h}^\theta - x) ]}{\mathbb{E}[\varphi_{h^\beta}(X_{h}^\theta - x)]} + $$
$$ + \frac{\mathbb{E}[(X_h^\theta - m(\mu, \sigma, x))^2 h^{- \beta}(\partial_\vartheta X_h^\theta) \varphi'_{h^\beta}(X_{h}^\theta - x)]}{\mathbb{E}[\varphi_{h^\beta}(X_{h}^\theta - x)]} - m_2(\mu, \sigma, x) h^{- \beta} \frac{\mathbb{E}[(\partial_\vartheta X_h^\theta) \varphi'_{h^\beta}(X_{h}^\theta - x)]}{\mathbb{E}[\varphi_{h^\beta}(X_{h}^\theta - x)]}= : I_{1,\theta} + I_{2,\theta} + I_{3,\theta}.$$
We are going to show that, considering the derivatives with respect to both $\mu$ and $\sigma$, $ I_{2,\theta}$ and $ I_{3,\theta}$ are negligible compared to $ I_{1,\theta}$. In order to prove it we use Theorem 1 of \cite{Chapitre 1} on the denominator of $I_{2,\theta}$ and $I_{3,\theta}$, while on the numerator of $I_{2,\theta}$ we use Holder inequality, \eqref{eq: derivate x rispetto mu} if we consider the derivative with respect to $\mu$ or the first equation of \eqref{eq: derivata x rispetto sigma e mista} if we consider the derivative with respect to $\sigma$ and Lemma \ref{lemma: x-m con varphi'}.
On the numerator of $ I_{3,\theta}$ we use Proposition \ref{prop: truc moche h} and we remind that, as a consequence of Ad, $m_2$ is a $R(\theta, h,x)$ function. It follows
\begin{equation}
|I_{2, \mu}+ I_{3, \mu}| \le R(\theta, h^{2 + \beta - \epsilon}, x) + R(\theta, h^{3 - \beta - \epsilon}, x) = R(\theta, h^{2 + \beta - \epsilon}, x);
\label{eq: I2 mu and I3 mu}
\end{equation}
\begin{equation}
|I_{2, \sigma}+ I_{3, \sigma}| \le  R(\theta, h^{\frac{3}{2} + \beta - \epsilon}, x) + R(\theta, h^{\frac{5}{2} - \beta - \epsilon}, x) = R(\theta, h^{\frac{3}{2} + \beta - \epsilon}, x).
\label{eq: I2 sigma and I3 sigma}
\end{equation}
Concerning $I_{1, \theta}$, its numerator is
$$2\mathbb{E}[(X_h^\theta - m(\mu, \sigma, x))\partial_\theta X_h^\theta \varphi_{h^\beta}(X_{h}^\theta - x) ] - 2\partial_\theta m(\mu, \sigma,x))\mathbb{E}[(X_h^\theta -  m(\mu, \sigma, x)) \varphi_{h^\beta}(X_{h}^\theta - x) ]= : I_{1,1}^\theta + I_{1,2}^\theta.$$
From the first two points we have already proved of Proposition \ref{prop: dl derivate prime} and the third point of Lemma \ref{lemma: conditional expected value} we get 
\begin{equation}
| I_{1,2}^\mu| \le R(\theta, h^2, x), \qquad | I_{1,2}^\sigma| \le R(\theta, h^2, x).
\label{eq: estim I12}
\end{equation}
Now we consider $I_{1,1}^\theta$ and we act differently depending on if we are dealing with the derivative with rapport to $\mu$ or those with rapport to $\sigma$. We start studying $I_{1,1}^\sigma$. Using a notation analogous to the one used in the proof of Lemma \ref{lemma: conditional expected value}, we set 
$$X_h^\theta -  m(\mu, \sigma, x) =: \int_0^h a(X_s^\theta, \sigma) dW_s + B_h$$
and so it turns we have
\begin{equation}
I_{1,1}^\sigma= 2\mathbb{E}[(\int_0^h a(X_s^\theta, \sigma) dW_s)\partial_\sigma X_h^\theta \varphi_{h^\beta}(X_{h}^\theta - x) ] + 2\mathbb{E}[(B_h)\partial_\sigma X_h^\theta \varphi_{h^\beta}(X_{h}^\theta - x) ].
\label{eq: I11 sigma after ref}
\end{equation}
On the second term here above we use Cauchy-Schwartz inequality, control analogous to \eqref{eq: Bin carre} and the first estimation in \eqref{eq: derivata x rispetto sigma e mista}, getting
\begin{equation}
\mathbb{E}[(B_h)\partial_\sigma X_h^\theta \varphi_{h^\beta}(X_{h}^\theta - x) ] \le c\mathbb{E}[(B_h)^2 \varphi^2_{h^\beta}(X_{h}^\theta - x) ]^\frac{1}{2} \mathbb{E}[(\partial_\sigma X_h^\theta)^2 ]^\frac{1}{2} \le R(\theta, h^{1 + 2 \beta},x)^\frac{1}{2}R(\theta,h^\frac{1}{2},x) = R(\theta, h^{1 + \beta}, x).
\label{eq: I11 sigma con Bin}
\end{equation}
To evaluate the first term on the right hand side of \eqref{eq: I11 sigma after ref} we replace the dynamic \eqref{eq: dynamic partial sigma x} of $\partial_\sigma X_h^\theta$, isolating the principal term: $\partial_\sigma X_h^\theta := \int_0^h \partial_\sigma a(X_s^\theta, \sigma) dW_s + G_\sigma.$
We get
\begin{equation}
\mathbb{E}[(\int_0^h a(X_s^\theta, \sigma) dW_s)(\int_0^h \partial_\sigma a(X_s^\theta, \sigma) dW_s)] + \mathbb{E}[(\int_0^h a(X_s^\theta, \sigma) dW_s)(\int_0^h \partial_\sigma a(X_s^\theta, \sigma) dW_s) ( \varphi_{h^\beta}(X_{h}^\theta - x)-1) ] +
\label{eq: main term I11 sigma}
\end{equation}
$$+ \mathbb{E}[(\int_0^h a(X_s^\theta, \sigma) dW_s)G_\sigma \varphi_{h^\beta}(X_{h}^\theta - x) ].$$
On the first term here above we add and substract both $a(x, \sigma)$ and $\partial_\sigma a(x, \sigma)$ getting a main term and three terms of increments. We observe it holds the following estimation, using Cauchy-Schwartz inequality, \eqref{eq: BDG} and the first point of Lemma \ref{lemma: Moment inequalities}
\begin{equation}
\mathbb{E}[(\int_0^h [a(X_s^\theta, \sigma) - a(x, \sigma)] dW_s)(\int_0^h \partial_\sigma a(X_s^\theta, \sigma) dW_s)]\le \mathbb{E}[(\int_0^h [a(X_s^\theta, \sigma) - a(x, \sigma)] dW_s)^2]^\frac{1}{2} \mathbb{E}[(\int_0^h \partial_\sigma a(X_s^\theta, \sigma) dW_s)^2]^\frac{1}{2} \le
\label{eq: estim increments I11 sigma}
\end{equation}
$$ \le c\mathbb{E}[\int_0^h (X_s - x)^2 ds ]^\frac{1}{2} R(\theta, h^\frac{1}{2}, x) \le { \modch R(\theta, h,x)} R(\theta, h^\frac{1}{2}, x)= { \modch R(\theta, h^{\frac{3}{2}},x)}. $$
We can act in the same way considering the increments of $\partial_\sigma a$ or the term on which we have the increments of both $a$ and $\partial_\sigma a$. It follows that the first term of \eqref{eq: main term I11 sigma} is
\begin{equation}
\mathbb{E}[(\int_0^h a(x, \sigma) dW_s)(\int_0^h \partial_\sigma a(x, \sigma) dW_s)] + { \modch R(\theta, h^{\frac{3}{2}},x)}= h \, a(x, \sigma) \partial_\sigma a(x, \sigma) + { \modch R(\theta, h^{\frac{3}{2}},x)}.
\label{eq: main term I11 sigma fine}
\end{equation}
On the second term of \eqref{eq: main term I11 sigma} we use Holder inequality twice (with $p$ big and $q$ next to $1$), \eqref{eq: BDG} twice and \eqref{eq: stima prob incrementi}. We get it is upper bounded by
$$\mathbb{E}[|(\int_0^h a(X_s^\theta, \sigma) dW_s)(\int_0^h \partial_\sigma a(X_s^\theta, \sigma) dWs)|^p]^\frac{1}{p} \mathbb{E}[(\varphi_{h^\beta}(X_{h}^\theta - x)-1)^q ]^\frac{1}{q} \le R(\theta, h, x)R(\theta, h, x)^\frac{1}{q}= R(\theta, h^{2 - \epsilon}, x) $$
Concerning the third term of \eqref{eq: main term I11 sigma}, we first of all use Cauchy-Schwartz inequality, the fact that $\varphi$ is bounded in absolute value and \eqref{eq: BDG} in order to estimate the stochastic integral while, to estimate the $2$-norm of $G_\sigma$, we use \eqref{eq: b in partial sigma}, the estimation \eqref{eq: a in partial sigma} about the negligible part of the stochastic integral and \eqref{eq: estimation gamma expected norme p dotx }. It follows it is upper bounded by
\begin{equation}
R(\theta, h^\frac{1}{2}, x)c(h^\frac{3}{2} + h ) = R(\theta, h^\frac{3}{2},x),  
\label{eq: stima I11 sigma con G}
\end{equation}
where we have also used the first estimation \eqref{eq: derivata x rispetto sigma e mista} to estimate the expected value. From \eqref{eq: I11 sigma con Bin}, \eqref{eq: main term I11 sigma fine} - \eqref{eq: stima I11 sigma con G} it follows $$I_{1,1}^\sigma = h \, a(x, \sigma) \partial_\sigma a(x, \sigma) + R(\theta, h^{1 + \beta},x).$$
Using also \eqref{eq: I2 sigma and I3 sigma} and \eqref{eq: estim I12} we get the development of $\partial_\sigma m_2$ we wanted. \\
We are left to prove the third point of Proposition \ref{prop: dl derivate prime}. It means, comparing it with \eqref{eq: I2 mu and I3 mu} and \eqref{eq: estim I12}, to prove that $|I_{1,1}^\mu|\le R(\theta, h^2,x)$. We observe that \eqref{eq: I11 sigma after ref} still holds with the derivative with respect to $\mu$ instead of those with respect to $\sigma$. We now recall that $B_h = \int_0^h b(X_s, \mu) ds + R(\theta,h,x) + \Delta X_h^J$ and we replace it in $\mathbb{E}[B_h(\partial_\mu X_s^\theta)\varphi_{h^\beta}(X_{h}^\theta - x)]$, getting
\begin{equation}
|\mathbb{E}[(\int_0^h b(X_s, \mu) ds)(\partial_\mu X_s^\theta)\varphi_{h^\beta}(X_{h}^\theta - x)] + R(\theta,h,x)\mathbb{E}[(\partial_\mu X_s^\theta)\varphi_{h^\beta}(X_{h}^\theta - x)] + \mathbb{E}[(\Delta X_h^J)(\partial_\mu X_s^\theta)\varphi_{h^\beta}(X_{h}^\theta - x)]|\le
\label{eq: below I11 mu}
\end{equation}
$$\le R(\theta, h^2, x) + R(\theta, h^{1 + \beta q},x)^\frac{1}{q} R(\theta, h, x) = R(\theta, h^2, x) + R(\theta, h^{2 + \beta - \epsilon}, x) = R(\theta, h^2, x),$$
where we have used Holder inequality (having taken $p$ big and $q$ next to $1$), the fact that $\varphi$ is bounded, the polynomial growth of $b$, the first estimation in \eqref{eq: derivate x rispetto mu} and Lemma \ref{lemma: estim jumps} (in a non-conditional form). Concerning the first term of \eqref{eq: I11 sigma after ref}, we still replace the dynamic of $\partial_\mu X_s^\theta := \int_0^h \partial_\mu b(X_s^\theta, \mu) ds + G_\mu$, where $G_\mu$ is the negligible part in the dynamic of $\partial_\mu X_s^\theta$ and it is such that 
$$\mathbb{E}[|G_\mu|^p]\le c(1 + h^{\frac{p}{2} - 1} + h^{p - 1})\int_0^h \mathbb{E}[|\partial_\mu X_s^\theta|^p] ds \le c(h^{p + 1} + h^{\frac{3}{2}p} + h^{2p})(1 + |x|^c) $$ (see Lemma 9 in \cite{Chapitre 1}). We have also used the first inequality of \eqref{eq: derivate x rispetto mu}. It follows
$\mathbb{E}[|G_\mu|^p]^\frac{1}{p} \le R(\theta, h^{1 + \frac{1}{p}},x),$ for $p \ge 2$.
The first term of \eqref{eq: I11 sigma after ref} is therefore
\begin{equation}
\mathbb{E}[(\int_0^h a(X_s^\theta, \sigma) dW_s)(\int_0^h \partial_\mu b(X_s^\theta, \mu) ds)] + \mathbb{E}[(\int_0^h a(X_s^\theta, \sigma) dW_s)(\int_0^h \partial_\mu b(X_s^\theta, \mu) ds) (1 - \varphi_{h^\beta}(X_{h}^\theta - x)) ] +
\label{eq: main term I11 mu}
\end{equation}
$$+ \mathbb{E}[(\int_0^h a(X_s^\theta, \sigma) dW_s)G_\mu \varphi_{h^\beta}(X_{h}^\theta - x) ].$$
Now on the first term here above we act like we did for the estimation of the derivative with respect to $\sigma$: we get
$$h \, \partial_\mu b(x, \mu)\mathbb{E}[\int_0^h a(X_s^\theta, \sigma) dW_s] + \mathbb{E}[(\int_0^h a(X_s^\theta, \sigma) dW_s)(\int_0^h [b(X_s^\theta, \mu) - b(x, \mu)]ds)]$$
Now we observe that the first expected value is $0$, while on the others we can use Cauchy-Schwartz inequality, \eqref{eq: BDG} and we estimate the increments like in \eqref{eq: estim increments I11 sigma}:
{ \modch $$  \mathbb{E}[(\int_0^h a(X_s^\theta, \sigma) dW_s)(\int_0^h [b(X_s^\theta, \mu) - b(x, \mu)]ds)] \le  \mathbb{E}[(\int_0^h a(X_s^\theta, \sigma) dW_s)^2]^\frac{1}{2} \mathbb{E}[(\int_0^h [b(X_s^\theta, \mu) - b(x, \mu)]ds)^2]^\frac{1}{2} \le $$
$$\le R(\theta, h^\frac{1}{2},x) \mathbb{E}[h \int_0^h (X_s^\theta - x)^2 ds ]^\frac{1}{2} \le R(\theta, h^\frac{1}{2},x) R(\theta, h^\frac{3}{2},x) =  R(\theta, h^2,x), $$
where we have also used the first point of Lemma \ref{lemma: Moment inequalities}.} \\
On the second term of \eqref{eq: main term I11 mu} we act like we did on the second term of \eqref{eq: main term I11 sigma}, using Holder inequality twice (with $p$ big and $q$ next to $1$), \eqref{eq: BDG}, the polynomial growth of both $a$ and $b$ and \eqref{eq: stima prob incrementi}. We get
$$\mathbb{E}[(\int_0^h a(X_s^\theta, \sigma) dW_s)(\int_0^h \partial_\mu b(X_s^\theta, \mu) ds) (1 - \varphi_{h^\beta}(X_{h}^\theta - x)) ] \le R(\theta, h^\frac{1}{2},x)R(\theta, h,x)R(\theta, h^{1 - \epsilon},x) = R(\theta, h^{\frac{5}{2}- \epsilon},x).$$
Using on the third term of \eqref{eq: main term I11 mu} Cauchy-Schwartz inequality, \eqref{eq: BDG}, the fact that $|\varphi|$ is bounded and the estimation of the $L^p$ norm of $G_\mu$ given above \eqref{eq: main term I11 mu} for $p= 2$, we get it is upper bounded in absolute value by $R(\theta,h^2,x)$. \\
It follows $|I_{1,1}^\mu| \le R(\theta, h^2,x)$, as we wanted.
\end{proof}
\subsubsection{Proof of Proposition \ref{prop: derivate seconde}}
\begin{proof}
We first of all write $\partial^2_{\mu \sigma} m$. Again, since by the homogeneity of the equation $m$ and $m_2$ depend only on the difference $t_{i + 1} - t_i$ we can consider, for all $h \le \Delta_n$, $m(\mu, \sigma, x)$ as in 
\eqref{eq: def m per prop derivate}. Hence, writing $\varphi$ for $\varphi_{h^\beta}(X_h^\theta - x)$ (and $\varphi^{(k)}$ for $\varphi_{h^\beta}^{(k)}(X_h^\theta - x)$), we have $\partial^2_{\mu \sigma} m (\mu, \sigma,x) =$
$$ \frac{\mathbb{E}[(\partial^2_{\mu \sigma} X_h^\theta) \varphi]}{\mathbb{E}[\varphi]} + \frac{2 h^{- \beta} \mathbb{E}[(\partial_\mu X_h^\theta)(\partial_\sigma X_h^\theta) \varphi']}{\mathbb{E}[\varphi]}- \frac{h^{- \beta} \mathbb{E}[(\partial_\mu X_h^\theta)\varphi']\mathbb{E}[(\partial_\sigma X_h^\theta)\varphi]}{(\mathbb{E}[\varphi])^2} +\frac{h^{- \beta} \mathbb{E}[X_h^\theta(\partial^2_{\mu \sigma} X_h^\theta)\varphi']}{\mathbb{E}[\varphi]} +$$
\begin{equation}
+ \frac{h^{- 2 \beta} \mathbb{E}[X_h^\theta(\partial_{\sigma} X_h^\theta)(\partial_{\mu} X_h^\theta)\varphi'']}{\mathbb{E}[\varphi]} -\partial_\mu m \frac{h^{- \beta} \mathbb{E}[(\partial_{\sigma} X_h^\theta)\varphi']}{\mathbb{E}[\varphi]} - \frac{h^{- 2 \beta} \mathbb{E}[X_h^\theta(\partial_{\sigma} X_h^\theta)\varphi'] \mathbb{E}[(\partial_{\mu} X_h^\theta)\varphi']}{(\mathbb{E}[\varphi])^2}  - \frac{h^{- \beta} m \mathbb{E}[(\partial^2_{\mu \sigma} X^\theta_h)\varphi']}{\mathbb{E}[\varphi]} +
\label{eq: derivata seconda mu sigma m}
\end{equation}
$$- \frac{h^{- 2\beta} m \mathbb{E}[(\partial_{\mu} X^\theta_h)(\partial_{\sigma} X^\theta_h)\varphi'']}{\mathbb{E}[\varphi]} + \frac{h^{- 2\beta} m \mathbb{E}[(\partial_{\sigma} X^\theta_h)\varphi']\mathbb{E}[(\partial_{\mu} X^\theta_h)\varphi']}{(\mathbb{E}[\varphi])^2} = : \sum_{j = 1}^{10} I_j^n.$$
Now on $I_2^n$ and $\sum_{j=4}^{10} I_j^n$ we use Proposition \ref{prop: truc moche h}, where $Z_h$ are respectively $\frac{\partial_\mu X_h^\theta}{h}\frac{\partial_\sigma X_h^\theta}{h^\frac{1}{2}}$,  $X_h^\theta \frac{\partial^2_{\mu \sigma}X_h^\theta}{h^\frac{3}{2}}$,  $X_h^\theta\frac{\partial_\mu X_h^\theta}{h}\frac{\partial_\sigma X_h^\theta}{h^\frac{1}{2}}$,  $\frac{\partial_\sigma X_h^\theta}{h^\frac{1}{2}}$, $X_h^\theta\frac{\partial_\sigma X_h^\theta}{h^\frac{1}{2}}$ in the first and $\frac{\partial_\mu X_h^\theta}{h}$ in the second expected value of $I_7^n$,  $\frac{\partial^2_{\mu \sigma}X_h^\theta}{h^\frac{3}{2}}$,  $\frac{\partial_\mu X_h^\theta}{h}\frac{\partial_\sigma X_h^\theta}{h^\frac{1}{2}}$,  $\frac{\partial_\sigma X_h^\theta}{h^\frac{1}{2}}$ in the first expected value of $I_{10}^n$ and $\frac{\partial_\mu X_h^\theta}{h^\frac{1}{2}}$ in the second one. Recalling moreover that $|m(\mu, \sigma,x)| \le R(\theta, 1,x)$, $|\partial_\mu m(\mu, \sigma, x)| \le R(\theta,h, x)$ and the denominator gives always $R(\theta,1,x)$ it follows
\begin{equation}
|I_2^n + \sum_{j=4}^{10} I_j^n | \le R(\theta, h^{\frac{5}{2}- \beta- \epsilon},x)+R(\theta, h^{\frac{5}{2}- 2\beta - \epsilon},x) +R(\theta, h^{3- \beta- \epsilon},x)+ R(\theta, h^{\frac{7}{2}- 2\beta- \epsilon},x) = R(\theta, h^{\frac{5}{2}- 2\beta- \epsilon},x). 
\label{eq: stima I2 - I10 partial mu sigma m}
\end{equation}
On $I_{1}^n$ we use Holder inequality, the fact that $|\varphi|$ is bounded, second inequality in \eqref{eq: derivata x rispetto sigma e mista} and still the fact that the denominator is $R(\theta, 1,x)$ to get
\begin{equation}
|I_1^n| \le R(\theta,h^\frac{3}{2},x).
\label{eq: estim I1 in partial mu sigma m}
\end{equation}
We now deal with the numerator of $I_3^n$, the denominator is still $R(\theta, 1,x)$ as a consequence of Lemma \ref{lemma: esperance varphi}. In the second expected value we apply Holder inequality, the boundedness of $|\varphi|$ and first inequality of \eqref{eq: derivata x rispetto sigma e mista} while on the first we use Proposition \ref{prop: truc moche h} with $Z_h$ that this time is $\frac{\partial_\mu X_h^\theta}{h}$. We get
\begin{equation}
|I_3^n| \le R(\theta,h^{\frac{5}{2}- \beta - \epsilon},x).
\label{eq: I3 in partial mu sigma m}
\end{equation}
From \eqref{eq: stima I2 - I10 partial mu sigma m}, \eqref{eq: estim I1 in partial mu sigma m} and \eqref{eq: I3 in partial mu sigma m} it follows the first inequality in \eqref{eq: deriv seconde miste e sigma m}. \\
In order to prove the second inequality in \eqref{eq: deriv seconde miste e sigma m} we compute $\partial^2_\sigma m(\mu, \sigma,x)$, getting $10$ terms as in \eqref{eq: derivata seconda mu sigma m} in which the derivatives are always with respect to $\sigma$. We therefore say that $\partial^2_\sigma m(\mu, \sigma, x) : = \sum_{j = 1}^{10} \tilde{I}_j^n$. On $\tilde{I}_2^n$ and  $\sum_{j = 4}^{10}  \tilde{I}_j^n$ we still use Proposition \ref{prop: truc moche h} taking as $Z$ respectively $\frac{(\partial_\sigma X_h^\theta)^2}{h}$, $X_h^\theta \frac{\partial^2_{\sigma} X_h^\theta}{h^\frac{1}{2}}$,  $X_h^\theta\frac{(\partial_\sigma X_h^\theta)^2}{h}$,  $\frac{\partial_\sigma X_h^\theta}{h^\frac{1}{2}}$, $X_h^\theta\frac{\partial_\sigma X_h^\theta}{h^\frac{1}{2}}$ in the first and $\frac{\partial_\sigma X_h^\theta}{h^\frac{1}{2}}$ in the second expected value of $\tilde{I}_7^n$,  $\frac{\partial^2_{ \sigma}X_h^\theta}{h^\frac{1}{2}}$,  $\frac{(\partial_\sigma X_h^\theta)^2}{h}$,  $\frac{\partial_\sigma X_h^\theta}{h^\frac{1}{2}}$ in both the first and the second expected values of $\tilde{I}_{10}^n$. Recalling also that $|\partial_\sigma m(\mu, \sigma, x)| \le R(\theta, \Delta_{n,i},x)$ it follows
\begin{equation}
|\tilde{I}_2^n + \sum_{j=4}^{10} \tilde{I}_j^n | \le R(\theta, h^{2- \beta- \epsilon},x)+R(\theta, h^{\frac{3}{2} - \beta - \epsilon},x) +R(\theta, h^{2- 2\beta- \epsilon},x) +R(\theta, h^{\frac{5}{2}- \beta - \epsilon},x)+ R(\theta, h^{3- 2\beta- \epsilon},x) = R(\theta, h^{\frac{3}{2} - \beta - \epsilon},x). 
\label{eq: stima I2 - I10 partial 2 sigma m}
\end{equation}
Concerning the numerator of $\tilde{I}_3^n$, using Holder inequality, first inequality in \eqref{eq: derivata x rispetto sigma e mista}, the boundedness of $\varphi$ and Proposition \ref{prop: truc moche h} for $Z = \frac{\partial_\sigma X_h^\theta}{h^\frac{1}{2}}$ it follows
\begin{equation}
|\tilde{I}_3^n| \le R(\theta, h^{2 - \beta - \epsilon}, x).
\label{eq: estim tilde I3}
\end{equation}
We now have to study $\tilde{I}_1^n$. We replace the dynamic \eqref{eq: dinamica partial2 sigma} isolating the principal term: $\partial^2_\sigma X_h^\theta =: \int_0^h \partial^2_\sigma a(\sigma, X_s^\theta) dW_s + G_{\sigma \sigma}$. \\
$G_{\sigma \sigma}$ is the negligible part and, as a consequence of \eqref{eq: stima din partial 2 sigma X}, (in which we recall that $ch^{\frac{p}{2}}(1 + |x|^c )$ comes from the principal term), we already know that
\begin{equation}
\mathbb{E}[|G_{\sigma \sigma}|^p] \le c(1 +|x|^c)(h^{2p} + h^{\frac{3}{2}p}+ h^{p} + h^{p + 1}) + c(h^{p-1}+ h^{\frac{p}{2}-1} +1)\int_0^h \mathbb{E}[|\partial^2_{\sigma}X_s^\theta|^p] ds \le c(1 +|x|^c)h^{\frac{p}{2} +1}.
\label{eq: estim Gsigma sigma}
\end{equation}
Therefore, $\forall p \ge 2$, $\mathbb{E}[|G_{\sigma \sigma}|^p]^\frac{1}{p} \le R(\theta, h^{\frac{1}{2} + \frac{1}{p}}, x)$. \\
We can see $\tilde{I}_1^n$ in the following way:
$$\tilde{I}_1^n = \mathbb{E}[\int_0^h \partial^2_\sigma a(\sigma, X_s^\theta) dW_s] + \mathbb{E}[(\int_0^h \partial^2_\sigma a(\sigma, X_s^\theta) dW_s)(\varphi_{h^\beta}(X_h^\theta - x)-1)] + \mathbb{E}[G_{\sigma \sigma}\varphi_{h^\beta}(X_h^\theta - x)].$$
The first expected value is zero, on the second we use Holder inequality, \eqref{eq: BDG} and \eqref{eq: stima prob incrementi} to get it is upper bounded by $R(\theta, h^{\frac{3}{2} - \epsilon},x)$. On the third term here above we use Cauchy-Schwartz inequality, the boundedness of $\varphi$ and \eqref{eq: estim Gsigma sigma} for $p=2$ getting it is $R(\theta, h, x)$. \\
It follows second inequality in \eqref{eq: deriv seconde miste e sigma m}. Equation \eqref{eq: derivata seconda mu m} has already been proved in Proposition 8 in \cite{Chapitre 1}. \\
Concerning the second derivatives of $m_2$, it is $\partial^2_{\sigma \mu} m_2 =$

$$ = \frac{2\mathbb{E}[(\partial_\mu X_h^\theta - \partial_\mu m )(\partial_\sigma X_h^\theta - \partial_\sigma m) \varphi]}{\mathbb{E}[\varphi]} + \frac{2 \mathbb{E}[(X_h^\theta - m)(\partial^2_{\mu \sigma} X_h^\theta - \partial^2_{\mu \sigma} m)\varphi]}{\mathbb{E}[\varphi]} + \frac{2 \mathbb{E}[(X_h^\theta - m )(\partial_{\mu} X_h^\theta - \partial_{\mu} m)h^{- \beta} \, \partial_\sigma X_h^\theta \,\varphi']}{\mathbb{E}[\varphi]} +$$
$$ - \frac{2 \mathbb{E}[(X_h^\theta - m)(\partial_{\mu} X_h^\theta - \partial_{\mu} m)h^{- \beta} \,\varphi] \mathbb{E}[\partial_\sigma X_h^\theta \,\varphi']}{(\mathbb{E}[\varphi])^2} + \frac{2 \mathbb{E}[(X_h^\theta - m )(\partial_{\sigma} X_h^\theta - \partial_{\sigma} m)h^{- \beta} \, \partial_\mu X_h^\theta \,\varphi']}{\mathbb{E}[\varphi]} + $$
\begin{equation}
+ \frac{\mathbb{E}[(X_h^\theta - m )^2(\partial_{\mu} X_h^\theta) h^{- 2 \beta} \, (\partial_\sigma X_h^\theta) \,\varphi'']}{\mathbb{E}[\varphi]} + \frac{ \mathbb{E}[(X_h^\theta - m)^2 (\partial^2_{\mu \sigma} X_h^\theta) h^{- \beta} \varphi']}{\mathbb{E}[\varphi]} - \frac{ \mathbb{E}[(X_h^\theta - m)^2(\partial_{\mu} X_h^\theta) h^{- \beta} \,\varphi'] \mathbb{E}[h^{- \beta }\partial_\sigma X_h^\theta \,\varphi']}{(\mathbb{E}[\varphi])^2} + 
\label{eq: partial2 mu sigma m2}
\end{equation}
$$ - \partial_\sigma m_2 \frac{\mathbb{E}[h^{- \beta }\partial_\mu X_h^\theta \,\varphi']}{\mathbb{E}[\varphi]} - m_2 \frac{\mathbb{E}[h^{- 2\beta }(\partial_\mu X_h^\theta) (\partial_\sigma X_h^\theta) \,\varphi'']}{\mathbb{E}[\varphi]}  - m_2 \frac{\mathbb{E}[h^{- \beta }(\partial^2_{\mu \sigma}X_h^\theta)\,\varphi']}{\mathbb{E}[\varphi]} + m_2\frac{\mathbb{E}[h^{- 2\beta }(\partial_\mu X_h^\theta) \,\varphi'] \mathbb{E}[(\partial_\sigma X_h^\theta)\varphi']}{(\mathbb{E}[\varphi])^2} =$$
$$=: \sum_{j = 1}^{12} I_j^n. $$ 
On $I_1^n$ we use the first and the second point of Proposition \ref{prop: dl derivate prime} to say that both the derivatives of $m$ are $R(\theta, h, x)$. From the boundedness of $\varphi$, Lemma \ref{lemma: esperance varphi}  and the estimation of the derivatives of $X$ gathered in Lemma \ref{lemma: derivate x} it follows
$$|-2 \mathbb{E}[(\partial_\mu X_h^\theta - \partial_\mu m ) \partial_\sigma m \varphi] | \le R(\theta, h^2,x).$$
We now have to study 
$$2 \mathbb{E}[\partial_\mu X_h^\theta \partial_\sigma X_h^\theta \varphi] - 2 \partial_\mu m \mathbb{E}[\partial_\sigma X_h^\theta \varphi] = : I_{1,1}^n + I_{1,2}^n. $$
We start considering $I_{1,2}^n$. As we have already done after \eqref{eq: I11 sigma con Bin}, we see $\partial_\sigma X_h^\theta$ as $\int_0^h \partial_\sigma a(X_s^\theta, \sigma) dW_s + G_\sigma$, hence
$$|I_{1,2}^n| \le R(\theta, h,x)|\mathbb{E}[(\int_0^h \partial_\sigma a(X_s^\theta, \sigma) dW_s)(\varphi - 1)] + \mathbb{E}[G_\sigma \varphi]| \le R(\theta, h,x)[R(\theta, h^{\frac{3}{2} - \epsilon},x) + R(\theta, h,x)] = R(\theta, h^2,x),$$
where in the last inequality we have used on the first term Holder inequality, Burkholder - Davis - Gundy inequality and \eqref{eq: stima prob incrementi} while on the second we have used Cauchy - Schwartz inequality and the fact that the $2$- norm of $G_\sigma$ is upper bounded by a $R(\theta, h,x)$ function as a consequence of \eqref{eq: b in partial sigma}, \eqref{eq: a in partial sigma} and \eqref{eq: estimation gamma expected norme p dotx }. \\
Concerning $I_{1,1}^n$, replacing again $\partial_\sigma X_h^\theta$ and using the estimation on the $2$-norm of $G_\sigma$ as we have already done it follows 
$$|I_{1,1}^n| \le c|\mathbb{E}[(\int_0^h \partial_\sigma a(X_s^\theta, \sigma) dW_s)\partial_\mu X_h^\theta \varphi]| + R(\theta, h^2,x).$$
Now we observe we have already proved in the conclusion of Proposition \ref{prop: dl derivate prime}, starting below \eqref{eq: below I11 mu}, that 
$I_{1,1}^\mu := 2 \mathbb{E}[[(\int_0^h a(X_s^\theta, \sigma) dW_s)\partial_\mu X_h^\theta \varphi]$ is such that $|I_{1,1}^\mu| \le R(\theta, h^2,x) $. 
Acting exactly in the same way, considering now $\partial_\sigma a(X_s^\theta, \sigma)$ in the stochastic integral instead of $a(X_s^\theta, \sigma)$, we get that $|I_{1,1}^n| \le R(\theta, h^2,x)$ and so, using also Lemma \ref{lemma: esperance varphi}, $|I_{1}^n| \le R(\theta, h^2,x)$. \\
Considering $I_2^n$, it is $I_2^n =: I_{2,1}^n + I_{2,2}^n$, where
$I_{2,1}^n := \frac{2 \mathbb{E}[(X_h^\theta - m)(\partial^2_{\mu \sigma} X_h^\theta)\varphi]}{\mathbb{E}[\varphi]}$ and \\ $I_{2,2}^n = : (- \partial^2_{\mu \sigma} m)\frac{2 \mathbb{E}[(X_h^\theta - m )\varphi]}{\mathbb{E}[\varphi]}$. \\
Now on $I_{2,1}^n$ we use Cauchy-Schwartz inequality, first point of Lemma \ref{lemma: conditional expected value} and \eqref{eq: derivata x rispetto sigma e mista} getting
$|I_{2,1}^n| \le R(\theta, h^2,x)$, while on $I_{2,2}^n$ we use the third point of Lemma \ref{lemma: conditional expected value} and \eqref{eq: derivata seconda mu sigma m} we have just proved in order to obtain $|I_{2,2}^n| \le R(\theta, h^\frac{5}{2},x)$.
The application of Holder inequality, Lemma \ref{lemma: x-m con varphi'}, Lemma \ref{lemma: derivate x} and the first point of Proposition \ref{prop: dl derivate prime} on the numerator of $I_3^n$ gives us 
$$|I_3^n| \le h^{- \beta }R(\theta, h^{1 + \beta p},x)^\frac{1}{p} R(\theta, h^\frac{3}{2},x).$$
It is enough to take $p$ next to $1$ to get it is negligible compared to $R(\theta, h^2,x)$. We act on the first expected value of $I_4^n $ like we did on $I_2^n $ while on the second we use Proposition \ref{prop: truc moche h}. It yields
$$|I_4^n| \le h^{- \beta } R(\theta, h^{\frac{3}{2}- \epsilon}, x) R(\theta, h^{\frac{3}{2}}, x ) = R(\theta, h^{3 - \beta- \epsilon},x).$$
On $I_5^n$ we act like we did on $I_3^n$, hence
$$|I_5^n| \le h^{- \beta }R(\theta, h^{1 + \beta p},x)^\frac{1}{p} R(\theta, h^{\frac{3}{2} - \epsilon},x),$$
that is negligible. In the same way
$$|I_6^n| \le h^{- 2\beta }R(\theta, h^{1 + 2\beta p},x)^\frac{1}{p} R(\theta, h^{\frac{3}{2}- \epsilon},x),$$
$$|I_7^n| \le h^{- \beta }R(\theta, h^{1 + 2\beta p},x)^\frac{1}{p} R(\theta, h^{\frac{3}{2}- \epsilon},x) \quad \mbox{ and}$$
$$|I_8^n| \le h^{- 2\beta }R(\theta, h^{1 + 2\beta p},x)^\frac{1}{p} R(\theta, h^{1- \epsilon},x) R(\theta, h^{\frac{3}{2}- \epsilon}, x).$$
We recall that $|\partial_\sigma m_2 (\mu, \sigma,x)| \le R(\theta, h, x)$ and $| m_2 (\mu, \sigma,x)| \le R(\theta, h, x)$, therefore
$$|I_9^n| \le h^{- \beta} R(\theta, h, x) R(\theta, h^{2 - \epsilon}, x) = R(\theta, h^{3 - \beta - \epsilon},x),$$
$$|I_{10}^n| \le h^{- 2\beta} R(\theta, h, x) R(\theta, h^{\frac{5}{2} - \epsilon}, x) = R(\theta, h^{\frac{7}{2} - 2\beta - \epsilon},x),$$
$$|I_{11}^n| \le h^{- \beta} R(\theta, h, x) R(\theta, h^{\frac{5}{2}- \epsilon}, x) = R(\theta, h^{\frac{7}{2} - \beta - \epsilon},x),$$
$$|I_{12}^n| \le h^{- 2\beta} R(\theta, h, x)R(\theta, h^{2 - \epsilon}, x) R(\theta, h^{\frac{3}{2} - \epsilon}, x) = R(\theta, h^{\frac{9}{2} - 2\beta - \epsilon},x).$$
First inequality in \eqref{eq: deriv seconde miste e mu m2} follows. In order to show the second one we should compute $\partial_\mu^2 m_2(\mu, \sigma,x)$. Since it is exactly like $\partial_{\mu \sigma}^2 m_2(\mu, \sigma,x)$ but with all the derivatives with respect to $\mu$, we still refer to \eqref{eq: partial2 mu sigma m2} and we will write
$\partial^2_\mu m_2(\mu, \sigma, x) =: \sum_{j = 1}^{12} \tilde{I}_j^n$. \\
On $\tilde{I}_1^n$, $\sum_{j = 3}^{12} \tilde{I}_j^n$ we act exactly like we did here above, getting
$$|\tilde{I}_{1}^n + \sum_{j = 3}^{12} \tilde{I}_j^n| \le R(\theta, h^2,x) + h^{- \beta } R(\theta, h^{1 + \beta p},x)^\frac{1}{p}R(\theta, h^2,x) + h^{-\beta} R(\theta, h^{\frac{3}{2}},x) R(\theta, h^{2  - \epsilon},x) + h^{- \beta } R(\theta, h^{1 + \beta p},x)^\frac{1}{p}R(\theta, h^{2},x) + $$
$$ + h^{- 2 \beta}R(\theta, h^{1 + 2\beta p},x)^\frac{1}{p}R(\theta, h^{2},x) + h^{- \beta } R(\theta, h^{1 + 2\beta p},x)^\frac{1}{p}R(\theta, h,x) + h^{- 2\beta } R(\theta, h^{1 + 2\beta p},x)^\frac{1}{p}R(\theta, h,x) R(\theta, h^{2  - \epsilon},x) +$$
$$  + h^{- \beta }R(\theta, h^2,x)R(\theta, h^{2  - \epsilon},x) + h^{- 2\beta }R(\theta, h,x)R(\theta, h^{3  - \epsilon},x) +h^{- \beta }R(\theta, h,x)R(\theta, h^{2  - \epsilon},x) + h^{- 2\beta }R(\theta, h,x)R(\theta, h^{4  - \epsilon},x),$$
that is a $R(\theta, h^2, x)$ function as a consequence of the fact that, choosing $p$ next to $1$, all terms are negligible compared to the first one. \\
We now deal with $\tilde{I}_2^n$. We still need to split it in $\tilde{I}_{2,1}^n$ and $\tilde{I}_{2,2}^n$, where
$\tilde{I}_{2,1}^n := \frac{2 \mathbb{E}[(X_h^\theta - m )(\partial^2_{\mu} X_h^\theta)\varphi]}{\mathbb{E}[\varphi]}$ and \\ $\tilde{I}_{2,2}^n = : (- \partial^2_{\mu} m)\frac{2 \mathbb{E}[(X_h^\theta - m )\varphi]}{\mathbb{E}[\varphi]}$. \\
Using on $\tilde{I}_{2,2}^n$ Lemma \ref{lemma: conditional expected value} and \eqref{eq: derivata seconda mu m} we obtain $|\tilde{I}_{2,2}^n| \le R(\theta, h^2,x)$. \\
In order to estimate $\tilde{I}_{2,1}^n$ we isolate the principal term in the dynamic of $\partial^2_\mu X_h^\theta$, getting 
$\partial^2_\mu X_h^\theta = \int_0^h \partial^2_\mu b(X_s^\theta, \mu) ds + G_{\mu \mu}$, where $G_{\mu \mu}$ is the negligible part and it is such that, for $p \ge 2$,
\begin{equation}
\mathbb{E}[|G_{\mu \mu}|^p] \le c(1 + |x|^c )(h^{2p}+ h^{p +1}) = R(\theta, h^{p + 1}, x),
\label{eq: stima G mu mu}
\end{equation}
as showed in the proof of Lemma 9 in \cite{Chapitre 1}. \\
Replacing in the definition of $\tilde{I}_{2,1}^n$ we have 
$$|\tilde{I}_{2,1}^n| \le |\frac{2 \mathbb{E}[ G_{\mu \mu}(X_h^\theta - m ) \varphi]}{\mathbb{E}[\varphi]} | + |\frac{2 \mathbb{E}[\int_0^h \partial^2_\mu b(X_s^\theta, \mu) ds (X_h^\theta - m ) \varphi]}{\mathbb{E}[\varphi]} |.$$
Now on the first term here above we use Cauchy-Schwartz inequality, the first point of Lemma \ref{lemma: conditional expected value} and \eqref{eq: stima G mu mu} obtaining it is upper bounded by $R(\theta, h^\frac{3}{2},x)R(\theta, h^\frac{1}{2},x) = R(\theta, h^2,x)$, while the second one is upper bounded by
$$ ch \partial^2_\mu b(x, \mu) | \mathbb{E}[ (X_h^\theta - m ) \varphi] |  + c | \mathbb{E}[\int_0^h[ \partial^2_\mu b(X_s^\theta, \mu) - \partial^2_\mu b(x, \mu)] ds (X_h^\theta - m ) \varphi] | \le$$
$$ \le h R(\theta, h, x) + R(\theta, h^\frac{3}{2},x)R(\theta, h^\frac{1}{2},x) = R(\theta, h^2,x), $$
where we have also used Lemma \ref{lemma: esperance varphi} in order to say that the denominator is lowed bounded by $1 + R(\theta, h, x)$, Cauchy-Schwartz inequality, the first point of Lemma \ref{lemma: conditional expected value} and an estimation analogous to \eqref{eq: stima incremento a} for the increments of the derivative of $b$. It follows $|\tilde{I}_{2,1}^n| \le R(\theta, h^2, x)$. \\
Second inequality in \eqref{eq: deriv seconde miste e mu m2} follows. \\ \\
We now want to show \eqref{eq: deriv seconda sigma m2}. To do it, we need to compute $\partial^2_\sigma m_2(\mu, \sigma, x)$. Again, we still refer to \eqref{eq: partial2 mu sigma m2} considering all the derivatives with respect to $\sigma$, writing
$\partial^2_\sigma m_2(\mu, \sigma,x) = : \sum_{j = 1}^{12} \hat{I}_j^n$. From $\hat{I}_3^n$ to $\hat{I}_{12}^n$ all the terms are negligible. Indeed, acting like we did on both $\sum_{j = 3}^{12} I_j^n$ and $\sum_{j = 3}^{12} \tilde{I}_j^n$ we have 
$$|\sum_{j = 3}^{12} \hat{I}_j^n| \le h^{- \beta} R(\theta, h^{1 + \beta p},x)^\frac{1}{p} R(\theta, h,x) + h^{- \beta} R(\theta, h^{1 + \beta p},x)^\frac{1}{p} R(\theta, h^{2 - \epsilon},x) +  h^{- \beta} R(\theta, h^{1 + \beta p},x)^\frac{1}{p} R(\theta, h,x) +$$
$$+  h^{- 2\beta} R(\theta, h^{1 + 2\beta p},x)^\frac{1}{p} R(\theta, h,x) +  h^{- \beta} R(\theta, h^{1 + 2\beta p},x)^\frac{1}{p} R(\theta, h^\frac{1}{2},x) +  h^{-2 \beta} R(\theta, h^{1 + 2\beta p},x)^\frac{1}{p} R(\theta, h^\frac{1}{2},x)R(\theta, h^{\frac{3}{2}  - \epsilon},x) +$$
$$+ h^{- \beta} R(\theta, h,x)R(\theta, h^{\frac{3}{2}  - \epsilon},x) + h^{- 2\beta} R(\theta, h,x)R(\theta, h^{2  - \epsilon},x) +  h^{- \beta} R(\theta, h,x)R(\theta, h^{\frac{3}{2}  - \epsilon},x) +  h^{- 2\beta} R(\theta, h,x)R(\theta, h^{3  - 2\epsilon},x),$$
that is always negligible compared to $R(\theta, h^{\frac{3}{2}, x})$ taking $p $ next to $1$. We now consider $\hat{I}_1^n$. 
Its numerator is 
$$\mathbb{E}[(\partial_\sigma X_h^\theta)^2 \varphi ] + (\partial_\sigma m)^2\mathbb{E}[ \varphi ] -2(\partial_\sigma m)\mathbb{E}[(\partial_\sigma X_h^\theta) \varphi ] = : \hat{I}_{1,1}^n + \hat{I}_{1,2}^n + \hat{I}_{1,3}^n.$$
We start considering $\hat{I}_{1,1}^n$ , in which we replace the dynamic of $\partial_\sigma X_h^\theta: = \int_0^h \partial_\sigma a(\sigma, X_s^\theta) dW_s + G_\sigma$, where we have already seen that, for $p \ge 2$, $\mathbb{E}[|G_\sigma|^p] \le c (h^\frac{3p}{2} + h^p + h^{\frac{p}{2} + 1})(1 + |x|^c) = R(\theta, h^{\frac{p}{2} + 1}, x)$.
We therefore can say that, taking $p= 2$ and using \eqref{eq: BDG}, we have 
$$\hat{I}_{1,1}^n = \mathbb{E}[(\int_0^h \partial_\sigma a(\sigma, X_s^\theta) dW_s)^2 \varphi] + R(\theta, h^\frac{3}{2},x) + R(\theta, h^2,x).$$
The first term here above can be seen as
$$\mathbb{E}[(\int_0^h \partial_\sigma a(\sigma, X_s^\theta) dW_s)^2 (\varphi-1)] + \mathbb{E}[\int_0^h [(\partial_\sigma a)^2(\sigma, X_s^\theta) - (\partial_\sigma a)^2(\sigma, x)] ds] + h(\partial_\sigma a)^2(\sigma, x).$$
Acting like we did in \eqref{eq: a carre} - \eqref{eq: finale a carre}, we have 
$$|\mathbb{E}[(\int_0^h \partial_\sigma a(\sigma, X_s^\theta) dW_s)^2 (\varphi-1)] + \mathbb{E}[\int_0^h [(\partial_\sigma a(\sigma, X_s^\theta))^2 - (\partial_\sigma a(\sigma, x))^2] ds]| \le R(\theta, h^\frac{3}{2},x)$$
and so we get
$$\hat{I}_{1,1}^n = h (\partial_\sigma a(\sigma, x))^2 + R(\theta, h^\frac{3}{2},x).$$
On $\hat{I}_{1,2}^n$ we use the second point of Proposition \ref{prop: dl derivate prime} and Lemma \ref{lemma: esperance varphi} to get
$|\hat{I}_{1,2}^n| \le R(\theta, h^2, x).$
In the same way, from second point of Proposition \ref{prop: dl derivate prime}, the boundedness of $\varphi$ and \eqref{eq: derivata x rispetto sigma e mista} we get $|\hat{I}_{1,3}^n| \le R(\theta, h^\frac{3}{2}, x)$.
It follows $\hat{I}_{1}^n = h (\partial_\sigma a(\sigma, x))^2 + R(\theta, h^\frac{3}{2},x).$ \\
We now study $\hat{I}_{2}^n $, that the denominator is still $1+ R(\theta, h,x)$ as a consequence of Lemma \ref{lemma: esperance varphi} while the numerator can be seen as 
$$2\mathbb{E}[(X_h^\theta - m)\partial^2_\sigma X^\theta_h \varphi] - 2\partial^2_\sigma m \mathbb{E}[(X_h^\theta - m)\varphi] = : \hat{I}_{2,1}^n + \hat{I}_{2,2}^n.$$
In order to deal with $\hat{I}_{2,1}^n$ we consider the reformulation \eqref{eq: reformulation X - m}, so we have 
$2 \mathbb{E}[(\int_0^h a(X_s^\theta, \sigma) dW_s)(\partial^2_{\sigma} X_h^\theta)\varphi] + 2 \mathbb{E}[(B_h)(\partial^2_{\sigma} X_h^\theta)\varphi]$. \\
The second, as a consequence of Holder inequality, the definition of $B_h$, Lemma \ref{lemma: estim jumps} used to estimate the jumps and the first inequality in \eqref{eq: derivata x rispetto sigma e mista} is upper bounded by $R(\theta, h^{(1 + \beta p) \land p },x)^\frac{1}{p} R(\theta, h^\frac{1}{2}, x)$. Taking $p$ next to $1$ it follows its order is $h^{\frac{3}{2}}$. \\
In order to study $\mathbb{E}[(\int_0^h a(X_s^\theta, \sigma) dW_s)(\partial^2_{\sigma} X_h^\theta)\varphi]$ we introduce once again the notation used above \eqref{eq: estim Gsigma sigma}, for which $\partial^2_{\sigma} X_h^\theta = \int_0^h \partial^2_\sigma a(\sigma, X_s^\theta) dW_s + G_{\sigma \sigma}$. 
Hence, we have 
$$\mathbb{E}[(\int_0^h a(X_s^\theta, \sigma) dW_s)(\int_0^h \partial^2_\sigma a(\sigma, X_s^\theta) dW_s)\varphi] + \mathbb{E}[(\int_0^h a(X_s^\theta, \sigma) dW_s)(G_{\sigma \sigma})\varphi].$$
The second term here above is, using Cauchy-Schwartz inequality, \eqref{eq: BDG}, the boundedness of $\varphi$ and \eqref{eq: estim Gsigma sigma}, just $R(\theta, h^\frac{3}{2}, x)$ while the first one is 
$$\mathbb{E}[(\int_0^h a(X_s^\theta, \sigma) dW_s)(\int_0^h \partial^2_\sigma a(\sigma, X_s^\theta) dW_s)(\varphi-1)] + \mathbb{E}[\int_0^h a(X_s^\theta, \sigma) \partial^2_\sigma a(\sigma, X_s^\theta) ds] = R(\theta, h^2,x) +$$
$$ + \mathbb{E}[\int_0^h [a(X_s^\theta, \sigma)- a(x, \sigma)] \partial^2_\sigma a(\sigma, X_s^\theta) ds] + a(x, \sigma)\mathbb{E}[\int_0^h [\partial^2_\sigma a(\sigma, X_s^\theta)- \partial^2_\sigma a(\sigma,x)] ds)] + h a(x, \sigma)\partial^2_\sigma a(\sigma, x), $$
where we have also used Holder inequality, \eqref{eq: BDG} and \eqref{eq: stima prob incrementi}. Now we observe that
$$|\mathbb{E}[\int_0^h [a(X_s^\theta, \sigma)- a(x, \sigma)] \partial^2_\sigma a(\sigma, X_s^\theta) ds]| \le c\int_0^h \mathbb{E}[|X_s^\theta - x|^2]^\frac{1}{2}\mathbb{E}[|\partial^2_\sigma a(\sigma, X_s^\theta)|^2]^\frac{1}{2} ds \le \int_0^h R(\theta, h^\frac{1}{2},x) = R(\theta, h^\frac{3}{2}, x).  $$
Acting in the same way on $a(x, \sigma)\mathbb{E}[\int_0^h [\partial^2_\sigma a(\sigma, X_s^\theta)- \partial^2_\sigma a(\sigma,x)] ds)]$ it follows that 
$$\hat{I}_{2,1}^n = h a(x, \sigma)\partial^2_\sigma a(\sigma, x) + R(\theta, h^\frac{3}{2},x).$$
On $\hat{I}_{2,2}^n$ we use the second inequality in \eqref{eq: deriv seconde miste e sigma m} and the third point of Lemma \ref{lemma: conditional expected value} and so we obtain 
$$|\hat{I}_{2,2}^n| \le R(\theta, h^2, x),$$
the result follows.
\end{proof}

\subsubsection{Proof of Proposition \ref{prop: dervate terze}}
\begin{proof}
We define $\partial^2_\mu m_2(\mu, \sigma, x)$ as $\sum_{j = 1}^{12} \tilde{I}_j^n = : \sum_{j = 1}^{12} \frac{N_{\tilde{I}_j^n}}{D_{\tilde{I}_j^n}}$ and we recall that, as stated in Proposition \ref{prop: derivate seconde}, $|\tilde{I}_j^n| \le R(\theta, \Delta_{n,i}^2,x)$  . We can therefore see its derivative with respect to $\mu$ in the following way:
\begin{equation}
\partial^3_\mu m_2(\mu, \sigma, x) = \sum_{j = 1}^{12}( \frac{\partial_\mu N_{\tilde{I}_j^n}}{D_{\tilde{I}_j^n}} - \tilde{I}_j^n \frac{ \partial_\mu D_{\tilde{I}_j^n}}{D_{\tilde{I}_j^n}}).
\label{eq: proof deriv terze m2}
\end{equation}
We start considering the second term here above: we remind that $D_{\tilde{I}_j^n}$ can be $\mathbb{E}[\varphi]$ or $(\mathbb{E}[\varphi])^2$. In both cases its derivative is, using Proposition \ref{prop: truc moche h}, a $R(\theta, h^{2 - \beta - \epsilon}, x)$ function; which makes the second term of \eqref{eq: proof deriv terze m2} upper bounded by $\sum_{j = 1}^{12}|\tilde{I}_j^n| R(\theta, h^{2 - \beta - \epsilon}, x) \le R(\theta, h^{4 - \beta - \epsilon},x)$, as a consequence of the second inequality of \eqref{eq: deriv seconde miste e mu m2}. \\
Concerning the first term of \eqref{eq: proof deriv terze m2}, we know that $D_{\tilde{I}_j^n}$ is a $R(\theta,1, x)$ function because of Lemma \ref{lemma: esperance varphi}, while the magnitude of the derivative of the numerator does not get bigger since the order of $X_h^\theta$ and $m$ remains the same by deriving them once more (as gathered in the second inequality of \eqref{eq: derivate terze x rispetto a sigma e mu} and in Remark 10 of \cite{Chapitre 1}) and the fact that, by deriving $\varphi$, there comes out $h^{- \beta} \varphi' \partial_\mu X_h^\theta$. From the first inequality in \eqref{eq: derivate x rispetto mu} such terms are negligible compared to the ones we have already studied. \\
The first term of \eqref{eq: proof deriv terze m2} is therefore still a $R(\theta, h^2, x)$ function. \\
The same reasoning applies to the study of the other seven third derivatives.
\end{proof}

\subsection{Development of $m_2 (\mu, \sigma,x)$}\label{Ss:dev_m_2}
In order to prove our main results we need a development of $m_2$, that we stated in Ad and we find, under the choice of a particular oscillating function $\varphi$, in Propositions \ref{prop: dl m2 intensita finita}, \ref{prop: dl m2 particolare per appl} and \ref{prop: dl alla capitolo 1 m2}. \\
The main tools is the iteration of the Dynkin's formula that provides us the following expansion for every function $f$: $\mathbb{R} \rightarrow \mathbb{R}$ such that $f$ is in $C^{2(k+1)}$:
\begin{equation}
\mathbb{E}[f(X^\theta_{t_{i+1}})| X^\theta_{t_i} = x] = \sum_{j=0}^k \frac{\Delta_{n,i}^j}{j!}A^jf(x) + \int_{t_i}^{t_{i+1}} \int_{t_i}^{u_1} ... \int_{t_i}^{u_k} \mathbb{E}[A^{k+1}f(X^\theta_{u_{k+1}})| X^\theta_{t_i} = x]\, du_{k+1} ... du_2\, du_1 
\label{eq: Dynkin formula beginning}
\end{equation} 
where $A$ denotes the generator of the diffusion, setting $A^0 = Id$. $A$ is the sum of the continuous and discrete part: $A: = {\modch \bar{A}}_c + A_d$, with
$${\modch \bar{A}}_cf(x)= \frac{1}{2}a^2(x, \sigma)f''(x) + \bar{b}(x, \mu)f'(x);$$
$\bar{b}(x, \mu) = b(x, \mu) - \int_{\mathbb{R}} z\gamma(x)  F(z) dz$
and
$$A_df(x)= \int_{\mathbb{R}} (f(x + \gamma(x)z) - f(x))\, F(z) dz. $$
\subsubsection{Proof of Proposition \ref{prop: dl m2 intensita finita}.}
\begin{proof}
We first of all observe that, by the definition of $m_2$ and $m$ it is
$$m_2(\mu, \sigma, x) = \frac{\mathbb{E}[(X_{t_{i + 1}}^\theta)^2 \varphi_{\Delta_{n,i}^\beta}(X_{t_{i + 1}}^\theta - X^\theta_{t_i}) | X^\theta_{t_i}= x ]}{\mathbb{E}[\varphi_{\Delta_{n,i}^\beta}(X_{t_{i + 1}}^\theta - X^\theta_{t_i}) | X^\theta_{t_i}= x ]} - (m(\mu, \sigma,x))^2.$$
It has already been showed in Proposition 2 in \cite{Chapitre 1} that
$$m(\mu, \sigma, x) = x + \sum_{k = 1}^{\lfloor \beta(M + 2) \rfloor} A_{K_1}^{(k)}(x) \frac{\Delta_{n,i}^k}{k!} + R(\theta, \Delta_{n,i}^{\beta(M + 2)}, x),$$
where $A_{K_1}^{(k)}(x) = \bar{A}_c (h_1)(x)$, with $h_1(y) = (y - x)$. \\
Acting like we did in Proposition 2 of \cite{Chapitre 1}, we want to find a development for 
$$\frac{\mathbb{E}[(X_{t_{i + 1}}^\theta)^2 \varphi_{\Delta_{n,i}^\beta}(X_{t_{i + 1}}^\theta - X^\theta_{t_i}) | X^\theta_{t_i}= x ]}{\mathbb{E}[\varphi_{\Delta_{n,i}^\beta}(X_{t_{i + 1}}^\theta - X^\theta_{t_i}) | X^\theta_{t_i}= x ]} =: \frac{n_{\Delta_{n,i}}(x)}{d_{\Delta_{n,i}}(x)}. $$
First we focus on the expression of $n$. We define the set of rest functions $\mathcal{R}^p$ in the following way:
$$\mathcal{R}^p := \left \{ r(x, y, \Delta_{n,i}^p, \theta) \mbox{ s. t. } \forall l \ge 0, \forall l' \in \left \{ 0,1 \right \} \, , \exists c ,  |\frac{\partial^l \, \partial^{l'}}{\partial y^l \, \partial \vartheta^{l'}} r(x, y, \Delta_{n,i}^p, \theta) |\le c (1 + |x|^{c} + |y|^{c}) \Delta_{n,i}^p, \mbox{ for } \vartheta= \mu, \, \vartheta = \sigma  \right \}.$$
It is worth noting that, if $r \in \mathcal{R}^p$, then both ${\modch \bar{A}}_c r$ and $A_d r$ are in $\mathcal{R}^p$, where the integral-differential operators ${\modch \bar{A}}_c$ and $A_d$ are applied with respect to the second variable $y$ (see details in proof of Proposition 2 in \cite{Chapitre 1}). \\
We also introduce the set $\tilde{\mathcal{F}}^p$:
$$\tilde{\mathcal{F}}^p : = \left \{ \tilde{g}(y, \theta) \, \, s.t. \, \tilde{g}(y, \theta) = \sum_{k = 0}^p \varphi^{(k)}((y - x) \Delta_{n,i}^{- \beta}) \Delta_{n,i}^{- k \beta} (\sum_{j = 0}^k h_{k,j}(x,y, \theta) \Delta_{n,i}^{\beta j})  \right \}$$
where, $\forall k, j$, $\forall l \ge 0$, $\forall l' \in \left \{ 0,1 \right \}$, $\exists c$ such that, for $\vartheta = \mu$ and $\vartheta= \sigma$, $\sup_{\theta \in \Theta}| \frac{\partial^l \, \partial^{l'}}{\partial y^l \, \partial \vartheta^{l'}} h_{k,j}(x,y, \theta)| \le c (1 + |x|^{c} + |y|^{c})$, with $c$ that depends on $k,j,l$ and $l'$. By the same proof as in Proposition 2 of \cite{Chapitre 1} it is possible to prove that, if $\tilde{g} \in \tilde{\mathcal{F}}^p$ then ${\modch \bar{A}}_c\tilde{g} \in \tilde{\mathcal{F}}^{p + 2}$ and, for all $\tilde{g} \in \tilde{\mathcal{F}}^p$, 
\begin{equation}
A_d \tilde{g}(y, \theta) = -\lambda \tilde{g} (y, \theta) + r(x,y, \Delta_{n,i}^{\beta(M + 2 - p)}, \theta).
\label{eq: adg in tildeF}
\end{equation}
It has also been proved that, as a consequence of the equation here above, the following relation holds:
$$A_{i_N} \circ ... \circ A_{i_1} \tilde{g}(y) = {\modch \bar{A}}_c^{l(i_1, ... , i_N)} \tilde{g} (y) (- \lambda)^{N - l(i_1, ... , i_N) } + r(x, y, \Delta_{n,i}^{\beta(M +2) - 2 \beta l(i_1, ... , i_N)}),$$
{\modch for $\tilde{g} \in \tilde{\mathcal{F}}_0$ and} with $l(i_1, ... , i_N)$ the number of $c$ in $\left \{ i_1, ... , i_N \right \}$, the iterations considered. \\ 
We observe that $\bar{g}(y): = y^2 \varphi ((y - x) \Delta_{n,i}^{ -\beta})$ belongs to $\tilde{\mathcal{F}}_0$. To find its development through Dynkin's formula  we can act exactly like we did in Proposition 2 of \cite{Chapitre 1}.\\ 
We get that the principal term in the development of the numerator is 
\begin{equation}
\sum_{l = 0}^N \frac{\Delta_{n,i}^l}{l!} A^{(l)}_{K_2} (x) \sum_{k' = 0 }^{N - l}\frac{ \Delta_{n,i}^{k'} (- \lambda)^{k'}}{k'!} + r(x,x, \Delta_{n,i}^{\beta(M + 2 )}, \theta),
\label{eq: acg risolto}
\end{equation}
with $A_{K_2}(x) = \bar{A}_c(h_2)(x)$ for $h_2(y) = y^2$ and $r(x,y, \Delta_{n,i}^{\beta(M + 2 )}, \theta) \in \mathcal{R}^{\beta (M + 2)}$. The integral rest term in the Dynkin formula is
\begin{equation}
|\int_{t_i}^{t_{i+1}} \int_{t_i}^{u_1} ... \int_{t_i}^{u_N}\mathbb{E}[A^{N + 1} h_2 (X_{u_{N + 1}}) | X_{t _i} = x ] du_{N + 1}... du_2 du_1| \le R(\theta, \Delta_{n,i}^{(1 - 2 \beta ) (N + 1)}, x). 
\label{eq: rest in g}
\end{equation}
Using \eqref{eq: acg risolto} and \eqref{eq: rest in g} we have the following development:
\begin{equation}
n_{\Delta_{n,i}}(x) = \sum_{l = 0}^N \frac{\Delta_{n,i}^l}{l!} A^{(l)}_{K_2}(x) \sum_{k' = 0 }^{N - l}\frac{ \Delta_{n,i}^{k'} (- \lambda)^{k'}}{k'!} + R(\theta, \Delta_{n,i}^{\beta (M + 2)}, x) + R(\theta, \Delta_{n,i}^{(1 - 2 \beta ) (N + 1)}, x).
\label{eq: dl numeratore intenita finita}
\end{equation}
If $(N + 1)(1 - 2 \beta) \ge \beta (M + 2)$, it entails
\begin{equation}
n_{\Delta_{n,i}}(x) = \sum_{l = 0}^{\lfloor \beta(M + 2) \rfloor} \frac{\Delta_{n,i}^l}{l!} A^{(l)}_{K_2}(x) \sum_{k' = 0 }^{\lfloor \beta(M + 2) \rfloor} \frac{ \Delta_{n,i}^{k'} (- \lambda)^{k'}}{k'!} + R(\theta, \Delta_{n,i}^{\beta (M + 2)}, x).
\label{eq: fine num}
\end{equation}
To get the control \eqref{eq: tesi derivate int finita} on the derivatives of \eqref{eq: fine num}, we show that one can differentiate with respect to the parameters the remainder term. \\
We remark that, as $r(x,y, \Delta_{n,i}^{\beta(M + 2)}, \theta)$ is in $\mathcal{R}^{\beta(M + 2)}$ and because of the definition of such a set, we clearly have that for $\vartheta = \mu$ and $\vartheta= \sigma$, the function $\partial_\vartheta r(x,y, \Delta_{n,i}^{\beta(M + 2)}, \theta)$ is still a $R(\theta, \Delta_{n,i}^{\beta (M + 2)}, x)$ function. \\
Concerning the derivatives of the integral rest, we have that, since $A^{N + 1} \bar{g}  \in \tilde{\mathcal{F}}^{2(N + 1)}$,
$$A^{N + 1} \bar{g} (y) = \sum_{k = 0}^{2N} \varphi^{(k)} ((y -x) \Delta_{n,i}^{- \beta}) \Delta_{n,i}^{- \beta k} (\sum_{j = 0}^k \tilde{h}_{k,j} (x, y, \theta)\Delta_{n,i}^{\beta j});$$
where $\tilde{h}_{k,j} $ are polynomial functions of $a$, $b$ and their derivatives. Then, for $\vartheta = \mu$ and $\vartheta = \sigma$ it is 
$$\partial_\vartheta \mathbb{E}[A^{N + 1} \bar{g} (X_{u_{N + 1}}) | X_{t _i} = x ] = \mathbb{E}[ \partial_\vartheta A^{N + 1} \bar{g} (X_{u_{N + 1}}) + \partial_X A^{N + 1} \bar{g} (X_{u_{N + 1}}) \partial_\vartheta X_{u_{N + 1}}| X_{t _i} = x ].$$
We use an upper bound on the conditional moments of the derivative of $X$ with respect to both the parameters (see Lemma \ref{lemma: derivate x} in the Appendix) and we act as for \eqref{eq: rest in g} to get a control of the integral rest.
From the computation of $\partial_X A^{N + 1} \bar{g} (X_{u_{N + 1}}) $ it shows up an extra $\Delta_{n,i}^{- \beta}$ but we can always choose $N$ such that the derivatives of the integral rest remain negligible compared to $R(\theta, \Delta_{n,i}^{\beta(M + 2)},x)$. \\
We now consider the denominator $d_{\Delta_{n,i}}(x)$: since it is exactly like it was in the development of $m$, we have already proved in Proposition 2 of \cite{Chapitre 1} that its expansion is
\begin{equation}
d_{\Delta_{n,i}}(x) = \sum_{k = 0}^{\lfloor \beta(M + 2) \rfloor} \frac{\Delta_{n,i}^k}{k!}(- \lambda)^k + r(x, x, \Delta_{n,i}^{\beta(M+2)}), 
\label{eq: d int finita}
\end{equation}
where $\lambda$ is the intensity of jumps defined in the fourth point of Assumption 4 and $r$ is a rest function which belongs to $\mathcal{R}^{\beta(M + 2)}$. \\
Acting exactly as we have done on the numerator we get that the derivatives of $r(x, x, \Delta_{n,i}^{\beta(M+2)})$ are still $R(\theta, \Delta_{n,i}^{\beta(M + 2)},x)$ functions and that the derivatives of the integral rest remain negligible compared to $R(\theta, \Delta_{n,i}^{\beta(M + 2)},x)$. \\
From the expansion of $d_{\Delta_{n,i}}(x)$ we can say that there exists $k_0 > 0$ such that for $|x| \le \Delta_{n,i}^{-k_0}$, $d_{\Delta_{n,i}}(x) \ge \frac{1}{2}$ $\forall n, i \le n$: we are avoiding the possibility that the denominator is in the neighborhood of $0$. Hence, for $|x| \le \Delta_{n,i}^{k_0} $, the development of $m_2$ is 
\begin{equation}
\frac{n_{\Delta_{n,i}}(x)}{d_{\Delta_{n,i}}(x)} - (m(\mu, \sigma,x))^2 = \sum_{l = 0}^{\lfloor \beta(M + 2) \rfloor} \frac{\Delta_{n,i}^l}{l!} A^{(l)}_{K_2} (x) - (x + \sum_{k = 1}^{\lfloor \beta(M + 2) \rfloor} A_{K_1}^{(k)}(x) \frac{\Delta_{n,i}^k}{k!})^2 + r(x,x, \Delta_{n,i}^{\beta (M + 2)}, \theta).
\label{eq: dl tout ordre m}
\end{equation}
The expansion \eqref{eq: dl m2 intensita finita} follows after remarking that $A_{K_2}^{(0)}(x) = 0 $. \\
Moreover, since the rest term in the development of $m$ comes from the same place as the rest in the fraction just studied, also its derivatives with respect to $\mu$ and $\sigma$ remain $R(\theta, \Delta_{n,i}^{\beta (M + 2)}, x)$ functions. The result follows.
\end{proof}

We now prove Proposition \ref{prop: dl alla capitolo 1 m2}, Proposition \ref{prop: dl m2 particolare per appl} is a consequence of it.

\subsubsection{Proof of Proposition \ref{prop: dl alla capitolo 1 m2}}
We recall that 
$$m_2 (\mu, \sigma, x) = \frac{\mathbb{E}[(X_{t_{i+1}}^\theta - m (\mu, \sigma,X_{t_i}))^2 \varphi_{\Delta_{n,i}^\beta}(X_{t_{i+1}}^\theta - X_{t_i}^\theta)|X_{t_i}^\theta = x]}{\mathbb{E}
	[\varphi_{\Delta_{n,i}^\beta}(X_{t_{i+1}}^\theta - X_{t_i}^\theta)|X_{t_i}^\theta = x]}
  := 
\frac{\tilde{n}_{\Delta_{n,i}}(x)}{\tilde{d}_{\Delta_{n,i}}(x)}.$$
The expansion \eqref{eq: dl alla capitolo 1 m2} in Proposition \ref{prop: dl alla capitolo 1 m2} is a consequence of the following two expansions for $\tilde{n}_{\Delta_{n,i}}(x)$ and $\tilde{d}_{\Delta_{n,i}}(x)$,
	\begin{multline}\label{E:dev_n_tilde}
\tilde{n}_{\Delta_{n,i}}(x)	=  \Delta_{n, i} a^2(x,\sigma) + \frac{\Delta_{n,i}^{1 + 3 \beta}}{\gamma(x)} \int_\mathbb{R} v^2 \varphi(v) F(\frac{v \Delta_{n,i}^\beta}{\gamma(x)}) dv +  \Delta_{n,i}^2[ 3\bar{b}^2(x, \mu) +  h_2 (x, \theta) -\lambda a^2(x,\sigma)] \\
 + \frac{\Delta_{n,i}^{2 + \beta}a^2(x, \sigma)}{2 \gamma(x)}\int_{u : |u| \le 2} [2\varphi(u)+u \varphi'(u)+
 u^2 \varphi''(u)] F(\frac{u\Delta_{n,i}^\beta}{\gamma(x)}) du + R (\theta, \Delta_{n,i}^{(3 - 2 \beta) \land (2 + \beta)},x),
	\end{multline}
\begin{equation}\label{E:dev_d_tilde}
\tilde{d}_{\Delta_{n,i}}(x)= 1 - \Delta_{n, i} \lambda + \frac{\Delta_{n, i}^{1+\beta}}{\gamma(x)}
\int_{u : |u| \le 2} \varphi(u) F(\frac{u\Delta_{n,i}^\beta}{\gamma(x)}) du + R(\theta,\Delta_{n, i}^{2-2\beta},x).
\end{equation}
The fact that the order of the remainder term  in \eqref{eq: dl alla capitolo 1 m2} is unchanged by derivation with respect to the parameters  $\mu$ and $\sigma$, will be a consequence of a similar property for the remainder terms appearing in 
\eqref{E:dev_n_tilde}--\eqref{E:dev_d_tilde}.

\begin{proof}[ Proof of \eqref{E:dev_n_tilde}]
In order to develop the numerator  $\tilde{n}_{\Delta_{n,i}}(x)$, we define the function $h_{i,n}(y) := (y - m(\mu, \sigma,x))^2 \varphi_{\Delta_{n,i}^\beta}(y - x)$ and we use Dynkin formula \eqref{eq: Dynkin formula beginning} on it, up to third order. It becomes
\begin{equation}
\mathbb{E}[h_{i,n}(X^\theta_{t_{i+1}})| X^\theta_{t_i} = x] = h_{i,n}(x) + \Delta_{n,i} A h_{i,n}(x) +\frac{1}{2} \Delta_{n,i}^2 A^2h_{i,n}(x) + \frac{1}{6} \int_{t_i}^{t_{i + 1}} \int_{t_i}^{u_1} \int_{t_i}^{u_2} \mathbb{E}[A^3 h_{i,n}(X^\theta_{u_3})| X^\theta_{t_i} = x] du_3 du_2 du_1. 
\label{eq: dynkin per prop applicazioni}
\end{equation}
 We now successively study the contribution of each term in the Dynkin's development.

By the definition of $h_{i,n}$, we have
$h_{i,n}(x) = ( x - m(\mu, \sigma, x))^2$. 

We recall that $A h_{i,n}(x) = {\modch \bar{A}}_c h_{i,n}(x) + A_d h_{i,n}(x)  $, where ${\modch \bar{A}}_c h_{i,n}(x) = \frac{1}{2} a^2(x, \sigma)h''_{i,n}(x) + \bar{b}(x, \mu) h'_{i,n}(x) $ and 
$$A_d h_{i,n}(x) = \int_\mathbb{R} [h_{i,n}(x + z \gamma(x)) - h_{i,n}(x)] F(z) dz = \int_\mathbb{R} (x + z \gamma(x) - m(\mu, \sigma,x))^2 \varphi_{\Delta_{n,i}^\beta}(z \gamma(x)) F(z) dz + $$
\begin{equation}
- (x  - m(\mu, \sigma,x))^2 \int_\mathbb{R} \varphi_{\Delta_{n,i}^\beta}(0) F(z) dz = (x  - m(\mu, \sigma,x))^2 \int_\mathbb{R} [\varphi_{\Delta_{n,i}^\beta}(z \gamma(x)) - \varphi_{\Delta_{n,i}^\beta}(0)] F(z) dz  + 
\label{eq: bar A d}
\end{equation}
$$+ \gamma^2(x) \int_\mathbb{R} z^2 \varphi_{\Delta_{n,i}^\beta}(z \gamma(x)) F(z) dz + 2 (x  - m(\mu, \sigma,x)) \gamma(x) \int_\mathbb{R} z \varphi_{\Delta_{n,i}^\beta}(z \gamma(x)) F(z) dz.$$
Using now the boundedness of $\varphi$, the fact that $\int_\mathbb{R} F(z) dz = \lambda $ and the development of $m$ \eqref{eq: dl m}, we have that the first term here above is a $R(\theta, \Delta_{n,i}^2, x)$ function and, from Proposition \ref{prop: dl derivate prime}, also its derivatives with respect to both the parameters are $R(\theta, \Delta_{n,i}^2, x)$ functions. \\
On the second term of \eqref{eq: bar A d} we apply the change of variable $v := \frac{\gamma(x) z}{\Delta_{n,i}^\beta}$, getting $\frac{\Delta_{n,i}^{3 \beta}}{\gamma(x)} \int_\mathbb{R} v^2  \varphi(v) F(\frac{v \Delta_{n,i}^\beta}{\gamma (x)}) dv$. On the third we use the definition of $\varphi$ for which $\varphi(\zeta) = 0$ for $\zeta$ such that $| \zeta| \ge 2$ and again the development of $m$ to get it is upper bounded by
\begin{equation}
|c (x - m (\mu, \sigma,x)) \gamma(x) \int_{z : |z| \le \frac{2 \Delta_{n,i}^\beta}{\gamma (x)}} z F(z) dz |\le R(\theta, \Delta_{n,i}^{1 + \beta},x).
\label{eq: stima supporto varphi z}
\end{equation}
Again, we can calculate the derivatives with respect to $\vartheta$ for both $\vartheta = \mu$ and $\vartheta = \sigma$, getting a term that is still a $R(\theta, \Delta_{n,i}^{1 + \beta},x)$ function by Proposition \ref{prop: dl derivate prime}. \\
It follows
$$A_d h_{i,n}(x) = R(\theta, \Delta_{n,i}^2, x) + \Delta_{n,i}^{2 \beta} \int_\mathbb{R} v^2  \varphi(v) F(\frac{v \Delta_{n,i}^\beta}{\gamma (x)}) dv + R(\theta, \Delta_{n,i}^{1 + \beta},x).$$
In order to compute ${\modch \bar{A}}_c h_{i,n}(x)$ we need the derivatives of $h_{i,n}$, they are
$$h'_{i,n}(y) = 2(y - m(\mu, \sigma,x))\varphi_{\Delta_{n,i}^\beta}(y - x) + (y - m(\mu, \sigma,x))^2 \Delta_{n,i}^{- \beta}\varphi'_{\Delta_{n,i}^\beta}(y - x), $$
\begin{equation}
h''_{i,n}(y) = 2 \varphi_{\Delta_{n,i}^\beta}(y - x) + 4(y - m(\mu, \sigma,x))\Delta_{n,i}^{- \beta}\varphi'_{\Delta_{n,i}^\beta}(y - x) + (y - m(\mu, \sigma,x))^2 \Delta_{n,i}^{- 2 \beta}\varphi''_{\Delta_{n,i}^\beta}(y - x), 
\label{eq: derivata seconda hin}
\end{equation}
$$h^{(3)}_{i,n}(y) = 6 \Delta_{n,i}^{- \beta}\varphi'_{\Delta_{n,i}^\beta}(y - x) + 6(y - m(\mu, \sigma,x))\Delta_{n,i}^{-2 \beta}\varphi''_{\Delta_{n,i}^\beta}(y - x) + (y - m(\mu, \sigma,x))^2 \Delta_{n,i}^{- 3 \beta}\varphi^{(3)}_{\Delta_{n,i}^\beta}(y - x), $$
$$h^{(4)}_{i,n}(y) = 12 \Delta_{n,i}^{-2 \beta}\varphi''_{\Delta_{n,i}^\beta}(y - x) + 8(y - m(\mu, \sigma,x))\Delta_{n,i}^{-3 \beta}\varphi^{(3)}_{\Delta_{n,i}^\beta}(y - x) + (y - m(\mu, \sigma,x))^2 \Delta_{n,i}^{- 4 \beta}\varphi^{(4)}_{\Delta_{n,i}^\beta}(y - x); $$
we have calculated the derivatives up to the fourth because they will be useful in the sequel. \\
Replacing the first two derivatives, calculated in $x$, it follows
\begin{equation}
{\modch \bar{A}}_c h_{i,n}(x) = a^2(x, \sigma) + 2(x - m(\mu, \sigma,x)) \bar{b}(x, \mu).
\label{eq: Achin(x)}
\end{equation}
Therefore we have 
\begin{equation}
\mathbb{E}_i[h_{i,n}(X^\theta_{t_{i + 1}})] = (x - m(\mu, \sigma,x))^2 + \Delta_{n,i}a^2(x, \sigma) +  \frac{\Delta_{n,i}^{1 + 3\beta}}{\gamma(x)} \int_\mathbb{R} v^2 \varphi(v) F(\frac{v \Delta_{n,i}^\beta}{\gamma(x)}) dv + 2 \Delta_{n,i}(x - m(\mu, \sigma,x)) \bar{b}(x, \mu) +
\label{eq: hin primo ordine riassunto}
\end{equation}
$$+ R(\theta, \Delta_{n,i}^{2 + \beta}, x) +\frac{1}{2} \Delta_{n,i}^2 A^2h_{i,n}(x) + \frac{1}{6} \int_{t_i}^{t_{i + 1}} \int_{t_i}^{u_1} \int_{t_i}^{u_2} \mathbb{E}[A^3 h_{i,n}(X^\theta_{u_3})| X^\theta_{t_i} = x] du_3 du_2 du_1.$$

We now study $A^2h_{i,n}(x)$. We recall it is 
$$A^2h_{i,n}(x) = {\modch \bar{A}}^2_c h_{i,n}(x) + {\modch \bar{A}}_c A_d h_{i,n}(x) + A_d {\modch \bar{A}}_c h_{i,n}(x) + A^2_d h_{i,n}(x).$$
We observe that we can write ${\modch \bar{A}}_c^2 h_{i,n}(x)$  as
 $\sum_{j=1}^4 h_j(x, \theta) h^{(j)}_{i,n}(x),$
where, for each $j \in \left \{ 1, 2, 3, 4 \right \}$, $h_j$ is a function of $a$, $\bar{b}$ and their derivatives up to second order: $h_1= \frac{1}{2} a^2\bar{b}^{''} + \bar{b}\bar{b}'$, $h_2 = \frac{1}{2} a^2(a')^2 + \frac{1}{2} a^3a'' + a^2\bar{b}' + aa'\bar{b} + \bar{b}^2 $, $h_3= a^3a' + a^2\bar{b}$ and $h_4 = \frac{1}{4} a^4$.
Moreover, recalling that by the definition of $\varphi$ we have $\varphi_{\Delta_{n,i}^\beta}(0) = 1$ and $\varphi^{(k)}_{\Delta_{n,i}^\beta}(0) = 0$ $\forall k \ge 1$, it follows $h_{i,n}' (x) = 2 (x - m (\mu, \sigma,x))$, $h_{i,n}''(x) = 2$ and $h_{i,n}^{(l)} = 0$ $\forall l \ge 3$. We obtain
\begin{equation}
{\modch \bar{A}}^2_c h_{i,n}(x) = 2 h_1(x, \theta) (x - m (\mu, \sigma, x)) + 2 h_2 (x, \theta) = R(\theta, \Delta_{n,i},x) + 2 h_2 (x, \theta),
\label{eq: A2c}
\end{equation}
where the last equality is a consequence of the development \eqref{eq: dl m} of $m$. \\
Concerning ${\modch \bar{A}}_c A_d h_{i,n}(x)$, it is $\frac{1}{2} a^2(x, \sigma)(A_d h_{i,n}(x))'' + \bar{b}(x, \mu)(A_d h_{i,n}(x))'$. We start considering 
$$(A_d h_{i,n}(x))' = \int_\mathbb{R} [h_{i,n}' (x + z \gamma(x)) (1 + z \gamma'(x)) - h_{i,n}'(x)] F(z) dz.$$
We now observe that, $\forall k \ge 1$ and $\forall y \in \mathbb{R}$, $|y - m(\mu, \sigma, x)|^k |\varphi_{\Delta_{n,i}^\beta}(y - x )| \le (|y - x|^k + |x - m(\mu, \sigma, x)|^k)|\varphi_{\Delta_{n,i}^\beta}(y - x )|$.
We have that $|y - x|^k |\varphi_{\Delta_{n,i}^\beta}(y - x )| \le c \Delta_{n,i}^{k \beta}|\varphi_{\Delta_{n,i}^\beta}(y - x )|$ as a consequence of the definition of $\varphi$ while, using the development \eqref{eq: dl m} of $m$ , it follows $|x - m(\mu, \sigma, x)|^k|\varphi_{\Delta_{n,i}^\beta}(y - x )| \le R(\theta, \Delta_{n,i}^k,x)$. Putting the pieces together it is 
\begin{equation}
|y - m(\mu, \sigma, x)|^k |\varphi_{\Delta_{n,i}^\beta}(y - x )| \le R(\theta, \Delta_{n,i}^{\beta k},x) +  R(\theta, \Delta_{n,i}^k,x) = R(\theta, \Delta_{n,i}^{\beta k},x);
\label{eq: stima diff times varphi}
\end{equation}
it is easy to remark that the same reasoning holds with the derivatives of $\varphi$ instead of $\varphi$. We underline that from the estimation \eqref{eq: stima diff times varphi} here above and the computation of the derivatives of $h_{i,n}$ we get that, $\forall l \ge 0$, each term of the $l$-derivative of $h_{i,n}$ has the same size. Indeed, each time we derive $\varphi$ an extra $\Delta_{n,i}^{- \beta}$ turns out but we can recover it on the basis of \eqref{eq: stima diff times varphi}. In particular, $\forall l \ge 0$ it follows
\begin{equation}
\left \| h_{i,n}^{(l)} \right \|_\infty \le R( \theta, \Delta_{n,i}^{\beta (2 - l)},x).
\label{eq: estim norma inf deriv hin}
\end{equation}
Therefore we obtain 
\begin{equation}
|(A_d h_{i,n}(x))'| \le c\Delta_{n,i}^\beta (2 \lambda + c \left \| \gamma' \right \|_\infty ).
\label{eq: estim Adhin'}
\end{equation}
Concerning the second derivative of $A_d h_{i,n}(x)$, it is 
$$(A_d h_{i,n}(x))'' = \int_\mathbb{R}[h''_{i,n}(x + z \gamma(x))(1 + z \gamma'(x))^2 + h'_{i,n}(x + z \gamma(x))z \gamma''(x) - h''_{i,n}(x)] F(z) dz =$$
$$= \int_\mathbb{R}[h''_{i,n}(x + z \gamma(x))- h''_{i,n}(x) + 2 z \gamma'(x) h''_{i,n}(x + z \gamma(x)) + z^2 (\gamma'(x))^2 h''_{i,n}(x + z \gamma(x)) + h'_{i,n}(x + z \gamma(x))z \gamma''(x)] F(z) dz = \sum_{j = 1}^5 I_j.$$
On $I_3$ and $I_4$ we act like we did on the integral in \eqref{eq: stima supporto varphi z}, observing that in the computation of $h_{i,n}''$ we always have $\varphi$ or its derivatives which make the integrals different from $0$ only for $|z| \le c \Delta_{n,i}^\beta$. On $I_5$ we use \eqref{eq: estim norma inf deriv hin} for $l = 1$, getting
$|I_5| \le R(\theta, \Delta_{n,i}^\beta, x)$. We observe that, by the computation of $h_{i,n}''(x)$ we obtain $I_2 = - 2 \lambda$. \\
To conclude the study of ${\modch \bar{A}}_c A_d h_{i,n}(x)$ we have to deal with  $I_1$:
\begin{equation}
\int_\mathbb{R} h_{i,n}''(x + z \gamma(x)) F(z) dz = \int_\mathbb{R} 2 \varphi_{\Delta_{n,i}^\beta} (z \gamma(x)) F (z) dz + 4 \Delta_{n,i}^{- \beta }(x - m (\mu, \sigma, x)) \int_\mathbb{R} \varphi'_{\Delta_{n,i}^\beta} (z \gamma(x)) F (z) dz + 
\label{eq: int hin''}
\end{equation}
$$+ {\modch 4} \Delta_{n,i}^{- \beta } \int_\mathbb{R} z \gamma(x) \varphi'_{\Delta_{n,i}^\beta} (z \gamma(x)) F (z) dz + \Delta_{n,i}^{-2 \beta }\int_\mathbb{R}((x - m (\mu, \sigma, x)) + z \gamma(x))^2  \varphi''_{\Delta_{n,i}^\beta} (z \gamma(x)) F (z) dz.$$
Applying the change of variable $u := \frac{z \gamma(x)}{\Delta_{n,i}^\beta}$ and recalling that from the development of $m$ it follows $|x - m (\mu, \sigma, x)|^k \le R(\theta, \Delta_{n,i}^k, x)$ for each $k\ge 1$, we obtain
$$\int_\mathbb{R} h_{i,n}''(x + z \gamma(x)) F(z) dz = \frac{\Delta_{n,i}^\beta}{\gamma(x)} \int_\mathbb{R} (2 \varphi(u) + u \varphi'(u) + u^2 \varphi''(u)) F(\frac{u \Delta_{n,i}^\beta}{\gamma(x)}) du + R(\theta, \Delta_{n,i}^{(1 - \beta) \land (2 - 2 \beta) \land (1 - 2 \beta)},x).$$
It is worth noting that the magnitude of the first term in the left hand side of the equation here above depends on the density $F$. \\
Remarking also that $\varphi(u) = 0$ for each $u$ such that $|u| \ge 2$, it follows
\begin{equation}
(A_d h_{i,n}(x))'' = \frac{\Delta_{n,i}^\beta}{\gamma(x)} \int_{u : |u| \le 2} (2 \varphi(u) + u \varphi'(u) + u^2 \varphi''(u)) F(\frac{u \Delta_{n,i}^\beta}{\gamma(x)}) du - 2 \lambda + R(\theta, \Delta_{n,i}^{(1 - 2 \beta) \land \beta}, x).
\label{eq: stima adhin''}
\end{equation}
From the definition of ${\modch \bar{A}}_c A_d h_{i,n}(x)$, \eqref{eq: estim Adhin'} and \eqref{eq: stima adhin''} we get
\begin{equation}
{\modch \bar{A}}_c A_d h_{i,n}(x) = \frac{\Delta_{n,i}^\beta a^2(x, \sigma)}{2 \gamma(x)} \int_\mathbb{R} (2 \varphi(u) + u \varphi'(u) + u^2 \varphi''(u)) F(\frac{u \Delta_{n,i}^\beta}{\gamma(x)}) du - a^2(x, \sigma) \lambda + R(\theta, \Delta_{n,i}^{(1 - 2 \beta) \land \beta}, x).
\label{eq: AcAd}
\end{equation}
Now we deal with
\begin{equation}
A_d {\modch \bar{A}}_c h_{i,n}(x) = \int_\mathbb{R} [{\modch \bar{A}}_c h_{i,n}(x + z \gamma (x)) - {\modch \bar{A}}_c h_{i,n}(x)] F(z) dz.
\label{eq: AdAc start}
\end{equation}
From \eqref{eq: Achin(x)} it follows
\begin{equation}
\int_\mathbb{R} {\modch \bar{A}}_c h_{i,n}(x) F(z) dz = \lambda a^2(x, \sigma) + 2 \lambda (x - m (\mu, \sigma,x))\bar{b}(x, \mu) = \lambda a^2(x, \sigma) + R(\theta, \Delta_{n,i},x),
\label{eq: int achin}
\end{equation}
where we have also used the development \eqref{eq: dl m} of $m$. Moreover,
\begin{equation}
\int_\mathbb{R} {\modch \bar{A}}_c h_{i,n}(x + z \gamma(x)) F(z) dz = \int_\mathbb{R} [\frac{1}{2} a^2(x + z \gamma(x), \sigma) h_{i,n}''(x + z \gamma(x)) + \bar{b}(x + z \gamma(x), \mu)h_{i,n}'(x + z \gamma(x))] F(z) dz.
\label{eq: int achin spostato}
\end{equation}
Acting on the first term of the right hand side of the equation here above exactly like we did in \eqref{eq: int hin''} we get it is equal to
\begin{equation}
\frac{\Delta_{n,i}^\beta}{2\gamma(x)} \int_{u : |u| \le 2}a^2(x + u \Delta_{n,i}^\beta, \sigma) (2 \varphi(u) + u \varphi'(u) + u^2 \varphi''(u)) F(\frac{u \Delta_{n,i}^\beta}{\gamma(x)}) du  + R(\theta, \Delta_{n,i}^{1 - 2 \beta}, x).
\label{int hin'' per adac}
\end{equation}
We upper bound the second term of the right hand side of \eqref{eq: int achin spostato}, instead, using \eqref{eq: estim norma inf deriv hin} for $l=1$. It yields
\begin{equation}
|\int_\mathbb{R} \bar{b}(x + z \gamma(x), \mu)h_{i,n}'(x + z \gamma(x)) F(z) dz | \le R (\theta, \Delta_{n,i}^\beta,x),
\label{eq: int negl adac}
\end{equation}
where we have also used the polynomial growth of $b$. Replacing in \eqref{eq: AdAc start} the equations \eqref{eq: int achin} - \eqref{eq: int negl adac} we obtain 
\begin{equation}
A_d {\modch \bar{A}}_c h_{i,n}(x) = \frac{\Delta_{n,i}^\beta}{2 \gamma(x)} \int_{u : |u| \le 2}a^2(x + u \Delta_{n,i}^\beta, \sigma) (2 \varphi(u) + u \varphi'(u) + u^2 \varphi''(u)) F(\frac{u \Delta_{n,i}^\beta}{\gamma(x)}) du - \lambda a^2(x, \sigma) + R(\theta, \Delta_{n,i}^{(1 - 2 \beta) \land \beta}, x).
\label{eq: AdAc}
\end{equation}
We have to study $A_d A_d h_{i,n}(x)$. It is
\begin{equation}
| A_d^2 h_{i,n}(x)| = |\int_\mathbb{R} [A_d h_{i,n}(x + z \gamma(x)) - A_d h_{i,n}(x)] F(z) dz| \le \int_\mathbb{R}\left  \| (A_d h_{i,n})' \right \|_ \infty |z \gamma(x)| F(z) dz \le R(\theta, \Delta_{n,i}^\beta,x),
\label{eq: AdAd}
\end{equation}
where we have used the definition $A_d h_{i,n}'$ and \eqref{eq: estim norma inf deriv hin}, remarking that the estimation \eqref{eq: estim Adhin'} holds also in no matter which $y \in \mathbb{R}$. \\
From \eqref{eq: A2c}, \eqref{eq: AcAd}, \eqref{eq: AdAc} and \eqref{eq: AdAd} it follows
\begin{equation}
\frac{1}{2} \Delta_{n,i}^2 A^2 h_{i,n}(x) = \Delta_{n,i}^2 h_2 (x, \theta) - \lambda \Delta_{n,i}^2 a^2 (x, \sigma) + 
\label{eq: riassunto second order}
\end{equation}
$$+ \frac{\Delta_{n,i}^{2 + \beta}}{4 \gamma(x)}\int_{u : |u| \le 2}[a^2(x, \sigma) + a^2(x + u \Delta_{n,i}^\beta, \sigma)] (2 \varphi(u) + u \varphi'(u) + u^2 \varphi''(u)) F(\frac{u\Delta_{n,i}^\beta}{\gamma(x)}) du + R (\theta, \Delta_{n,i}^{(3 - 2 \beta) \land (2 + \beta)},x).$$

To complete the study of the numerator of $m_2$ we need to estimate $\frac{1}{6} \int_{t_i}^{t_{i + 1}} \int_{t_i}^{u_1} \int_{t_i}^{u_2} \mathbb{E}[A^3 h_{i,n}(X^\theta_{u_3})| X^\theta_{t_i} = x] du_3 du_2 du_1$. We introduce the following norm
 on $\mathcal{C}^p$ functions, with $p \ge 0$, $c>0$:
$$\left \| f \right \|_{\infty, c, p} := \sum_{k = 0}^p \sup_{y \in \mathbb{R}
} | \frac{f^{(k)}(y)}{(1 + |y| )^c} |.$$
We observe it is 
$$\int_\mathbb{R} | f(y + z \gamma(y))| F(z) dz = \int_\mathbb{R} \frac{| f(y + z \gamma(y))| (1 + | y + z \gamma(y)| )^{\tilde{c}}}{(1 + | y + z \gamma(y)| )^{\tilde{c}}} 
 F(z) dz \le \left  \| f \right \|_{ \infty , \tilde{c}, 0} \int_\mathbb{R}(1 + | y + z \gamma(y)| )^{\tilde{c}} F(z) dz.$$
 We can therefore evaluate the norm of $A_d f$, getting
$$\left  \| A_d f \right \|_{ \infty , \tilde{c}, 0} =  \left  \| \frac{A_d f}{(1 +| y| )^{\tilde{c}}} \right \|_{ \infty} \le 
 c
\left  \| f \right \|_{ \infty , \tilde{c}, 0} (\frac{ \int_\mathbb{R} (1 +| y|^{\tilde{c}} +   (1+|y|)^{\tilde{c}} |z|^{\tilde{c}
		} ) F(z) dz}{(1 +| y| )^{\tilde{c}}} + \lambda) 
\le c \left  \| f \right \|_{ \infty , \tilde{c}, 0} .$$
By similar computations on the derivatives of $A_d f$ we obtain, $\forall p \ge 0$,  with  $p \le 4$, $\left  \| A_d f \right \|_{ \infty , \tilde{c}, p} \le c \left  \| f \right \|_{ \infty , \tilde{c}, p} $ (similar calculations are in Theorem 2.3 of \cite{Mies}). \\
In order to find an upper bound for the norm of ${\modch \bar{A}}_c f$ we observe that, from the polynomial growth of both the coefficients $a$ and $b$, it follows
$$|\frac{{\modch \bar{A}}_c f}{(1 +| y| )^{\tilde{c} + 2}}| = |\frac{\frac{1}{2} a^2(y, \sigma) f''(y)}{(1 +| y| )^{\tilde{c} + 2}} + \frac{b(y, \mu) f'(y)}{(1 +| y| )^{\tilde{c} + 2}}| \le \frac{c}{(1 +| y| )^{ \tilde{c}}}|f''(y)| +  \frac{c}{(1 +| y| )^{ \tilde{c}}}|f'(y)|.$$
Hence, we deduce that $\left  \| {\modch \bar{A}}_c f \right \|_{ \infty , \tilde{c} + 2, 0}  \le c \left  \| f \right \|_{ \infty , \tilde{c}, 2}$. Acting again with similar computation on the following derivatives, as done detailing in the proof of Theorem 2.3 in \cite{Mies}, we get $\left  \| {\modch \bar{A}}_c f \right \|_{ \infty , \tilde{c} + 2, p}  \le c \left  \| f \right \|_{ \infty , \tilde{c}, p + 2}$. \\
We want to use the estimations on $\left  \| A_d f \right \|_{ \infty , \tilde{c}, p}$ and $\left  \| {\modch \bar{A}}_c f \right \|_{ \infty , \tilde{c}, p}$ and equation \eqref{eq: estim norma inf deriv hin} to evaluate each term of \\
$\mathbb{E}[A^3 h_{i_n} (X_{u_3}^\theta)| X_{t_i}^\theta = x]$ but ${\modch \bar{A}}_c^3 h_{i,n}$. We observe it is, for $\tilde{c} \ge 4$,
$$\mathbb{E}_i[A_d^3 h_{i,n} (X_{u_3}^\theta)] \le \left  \| A_d^3 h_{i,n} \right \|_{ \infty , \tilde{c}, 0} \mathbb{E}_i[(1 + |X_{u_3}^\theta|)^{\tilde{c}}] \le c\left  \| h_{i,n} \right \|_{ \infty , \tilde{c}, 0} R(\theta, 1, x) \le R(\theta, \Delta_{n,i}^{2 \beta},x); $$
$$\mathbb{E}_i[A_d A_d {\modch \bar{A}}_c h_{i,n} (X_{u_3}^\theta)] \le \left  \| A_d A_d {\modch \bar{A}}_c h_{i,n} \right \|_{ \infty , \tilde{c}, 0} \mathbb{E}_i[(1 + |X_{u_3}^\theta|)^{\tilde{c}}] \le c\left  \| h_{i,n} \right \|_{ \infty ,\tilde{c} - 2, 2} R(\theta, 1, x) \le R(\theta, 1,x); $$
remarking that the same estimation holds for $A_d {\modch \bar{A}}_c A_d h_{i_n}$ and ${\modch \bar{A}}_c A_d A_d h_{i_n}$.
Moreover we have 
$$\mathbb{E}_i[{\modch \bar{A}}_c {\modch \bar{A}}_c A_d h_{i,n} (X_{u_3}^\theta)] \le \left  \| {\modch \bar{A}}_c {\modch \bar{A}}_c A_d h_{i,n} \right \|_{ \infty , \tilde{c}, 0} \mathbb{E}_i[(1 + |X_{u_3}^\theta|)^{\tilde{c}}] \le c\left  \| h_{i,n} \right \|_{ \infty , \tilde{c}- 4, 4} R(\theta, 1, x) \le R(\theta, \Delta_{n,i}^{- 2 \beta},x) $$
and we can upper bound ${\modch \bar{A}}_c A_d {\modch \bar{A}}_c h_{i,n}$ and $A_d {\modch \bar{A}}_c {\modch \bar{A}}_c h_{i,n}$ with the same quantity. \\
We are now left to study $\mathbb{E}_i[{\modch \bar{A}}_c^3 h_{i,n} (X_{u_3}^\theta)]$.
As we have already done considering ${\modch \bar{A}}_c^2 h_{i,n} (x)$, we see ${\modch \bar{A}}_c^3 h_{i,n}$ as linear combination of the derivatives of $h_{i,n}$: ${\modch \bar{A}}_c^3 h_{i,n}(y) := \sum_{j = 1}^6 h_j(y, \theta) h_{i,n}^{(j)}(y)$ where for each $j \in \left \{ 1, ... , 6 \right \}$ $h_j$ is a function of $a$, $b$ and their derivatives up to fourth order. \\
Using a conditional version of Proposition \ref{prop: truc moche h} we get $\forall \epsilon > 0$, $\forall j \ge 3$
$\mathbb{E}_i[|h_j(X_{u_3}^\theta, \theta) h_{i,n}^{(j)}(X_{u_3}^\theta)|] \le R(\theta, \Delta_{n,i}^{1 + (2 - j)\beta - \epsilon},x)$. \\
Concerning the first two terms of the sum, we have
$$|\mathbb{E}_i[h_1(X_{u_3}^\theta, \theta) h_{i,n}'(X_{u_3}^\theta) + h_2(X_{u_3}^\theta, \theta) h_{i,n}''(X_{u_3}^\theta)]| \le \left \| h_{i,n}' \right \|_\infty \mathbb{E}_i[|h_1(X_{u_3}^\theta, \theta)|] +\left \| h_{i,n}'' \right \|_\infty \mathbb{E}_i[|h_2(X_{u_3}^\theta, \theta)|] \le R(\theta, 1,x), $$
which follows from \eqref{eq: estim norma inf deriv hin} and from the polynomial growth of $a$, $b$ and their derivatives, which constitute the functions $h_1$ and $h_2$. Putting all the pieces together we get
$$\frac{1}{6} \int_{t_i}^{t_{i + 1}} \int_{t_i}^{u_1} \int_{t_i}^{u_2} \mathbb{E}[A^3 h_{i,n}(X^\theta_{u_3})| X^\theta_{t_i} = x] du_3 du_2 du_1 \le R(\theta, \Delta_{n,i}^{(3 - 2 \beta) \land (4 - 4 \beta - \epsilon)},x) = R(\theta, \Delta_{n,i}^{3 - 2 \beta},x).$$
From \eqref{eq: hin primo ordine riassunto}, \eqref{eq: riassunto second order} and the equation here above it follows $\mathbb{E}_i[h_{i,n}(X^\theta_{t_{i + 1}})] =$
\begin{equation}
 =  (x - m(\mu, \sigma,x))^2 + \Delta_{n,i}a^2(x, \sigma) +  \frac{\Delta_{n,i}^{1 + 3 \beta}}{\gamma(x)} \int_\mathbb{R} v^2 \varphi(v) F(\frac{v \Delta_{n,i}^\beta}{\gamma(x)}) dv + 2 \Delta_{n,i}(x - m(\mu, \sigma,x)) \bar{b}(x, \mu) +  \Delta_{n,i}^2 h_2 (x, \theta) +
\label{eq: num quasi finito}
\end{equation}
$$- \lambda \Delta_{n,i}^2 a^2 (x, \sigma) + \frac{\Delta_{n,i}^{2 + \beta}}{4 \gamma(x)}\int_{u : |u| \le 2}[a^2(x, \sigma) + a^2(x + u \Delta_{n,i}^\beta, \sigma)] (2 \varphi(u) + u \varphi'(u) + u^2 \varphi''(u)) F(\frac{u\Delta_{n,i}^\beta}{\gamma(x)}) du + R (\theta, \Delta_{n,i}^{(3 - 2 \beta) \land (2 + \beta)},x).$$
The expansion \eqref{E:dev_n_tilde} follows from \eqref{eq: num quasi finito} and \eqref{E:dev_m_chap1}, with the smoothness of $z\mapsto a^2(z,\sigma)$, and
	$\int_{\mathbb{R}} F(z) dz <\infty$.

\end{proof}
\begin{proof}[ Proof of \eqref{E:dev_d_tilde}]
Concerning the denominator of $m_2$, we still use Dynkin formula up to third order, this time on $f_{i,n}(y) : = \varphi_{\Delta_{n,i}^\beta}(y -x)$. We observe that, by the building, $f_{i,n}(x) = 1$ and $f_{i,n}^{(k)}(x) = 0$ for each $k \ge 1$. Hence, ${\modch \bar{A}}_c f_{i,n}(x) = 0$ and 
$$A_d f_{i,n}(x) = \int_\mathbb{R}[f_{i,n}(x + \gamma(x) z) -1] F(z) dz = \int_{ z : |z | \le \frac{2 \Delta_{n,i}^\beta}{|\gamma(x)|} } \varphi_{\Delta_{n,i}^\beta}(z \gamma(x)) F(z) dz - \lambda.$$
As we have already done we can see the first term here above, after the change of variable $u := \frac{z \gamma(x)}{\Delta_{n,i}^\beta}$, as $\frac{\Delta_{n,i}^\beta}{\gamma(x)} \int_{u : |u| \le 2 } \varphi(u) F(\frac{u \Delta_{n,i}^\beta}{\gamma(x)}) du$, which order depends on the density $F$. \\ 
Concerning the study of $A^2 f_{i,n}(x)$, we first of all remark that ${\modch \bar{A}}^2_c f_{i,n}(x) = 0$. \\
Moreover, we observe it is $f_{i,n}^{(k)}(y) = \Delta_{n,i}^{- \beta k} \varphi_{\Delta_{n,i}^\beta}^{(k)} (y - x)$ and so by the boundedness of $\varphi$ and its derivatives we get that, for each $k \ge 1$,
\begin{equation}
\left \| f_{i,n}^{(k)} \right \|_\infty \le R(\theta, \Delta_{n,i}^{- \beta k},x).
\label{eq: stima deriv fin}
\end{equation}
We therefore have, $\forall y \in \mathbb{R}$, 
$$|(A_d f_{i,n} (y))'| = |\int_\mathbb{R} f_{i,n}' (y + z \gamma(y))(1 + z \gamma'(y)) F(z)dz - \lambda f_{i,n}' (y) | \le R(\theta, \Delta_{n,i}^{- \beta},y) $$
and, in the same way, $|(A_d f_{i,n} (y))''| \le R(\theta, \Delta_{n,i}^{- 2 \beta},y)$. It follows 
$$|{\modch \bar{A}}_c A_d f_{i,n} (x)| = |\frac{1}{2} a^2(x, \sigma)(A_d f_{i,n} (x))'' + b(x,\mu )(A_d f_{i,n} (x))'| \le R(\theta, \Delta_{n,i}^{- 2 \beta},x);$$
$$|A_d {\modch \bar{A}}_c f_{i,n} (x)| = |\int_\mathbb{R} {\modch \bar{A}}_c f_{i,n} (x + z \gamma(x)) F(z) dz | \le R(\theta, \Delta_{n,i}^{- 2 \beta},x)$$
and, using also finite-increments theorem, $|A_d A_d f_{i,n} (x)| \le R(\theta, \Delta_{n,i}^{- \beta},x).$
Putting pieces together we have $|\frac{1}{2} \Delta_{n,i}^2 A^2 f_{i,n} (x)| \le R(\theta, \Delta_{n,i}^{2 - 2 \beta},x)$. \\
Considering the integral rest of the Dynkin formula , we act like we did in the study of the numerator, passing through the use of the norm $\left \| . \right \|_{\infty, c, p}$ and the estimation of the derivatives of $f_{i,n}$ gathered in \eqref{eq: stima deriv fin}. It yields
$$\mathbb{E}_i[A_d^3 f_{i_n} (X_{u_3}^\theta)] \le R(\theta, 1,x), \qquad \mathbb{E}_i[A_d A_d {\modch \bar{A}}_c f_{i_n} (X_{u_3}^\theta)] \le R(\theta, \Delta_{n,i}^{- 2 \beta},x),$$
$$\mathbb{E}_i[A_d {\modch \bar{A}}_c {\modch \bar{A}}_c f_{i_n} (X_{u_3}^\theta)] \le R(\theta, \Delta_{n,i}^{- 4 \beta},x), \qquad \mathbb{E}_i[{\modch \bar{A}}_c^3 f_{i_n} (X_{u_3}^\theta)] \le R(\theta, \Delta_{n,i}^{1 - 6 \beta - \epsilon},x),$$
having also used Proposition \ref{prop: truc moche h} to get the last one estimation here above. It turns out the denominator is
\begin{equation}
1 - \Delta_{n,i} \lambda + \frac{\Delta_{n,i}^{1 +\beta}}{\gamma(x)} \int_{u : |u| \le 2 } \varphi(u) F(\frac{u \Delta_{n,i}^\beta}{\gamma(x)}) du+ R(\theta, \Delta_{n,i}^{2 - 2 \beta}, x);
\label{eq: esp den}
\end{equation}
since we can always find an $\epsilon > 0$ for which $4 - 6 \beta - \epsilon > 2 - 2 \beta$ and we have $3 - 4 \beta > 2 - 2 \beta$.
 This concludes the proof of the expansion \eqref{E:dev_d_tilde}.
\end{proof}

To conclude the proof of the Proposition \ref{prop: dl alla capitolo 1 m2} we are left to show that the derivatives with respect to both the parameters of the rest terms
 in the expansions \eqref{E:dev_n_tilde}--\eqref{E:dev_d_tilde} are still rest functions and their order remains the same. \\ 
We observe that up to the development of second order in Dynkin formula, the rest functions R are totally explicit in our computation and so it is possible to calculate its derivatives with respect to both $\mu$ and $\sigma$. As we have already seen during the proof, we can use the estimations on $\partial_\vartheta m(\mu, \sigma, x)$ for $\vartheta = \mu$ and $\vartheta = \sigma$ gathered in Proposition \ref{prop: dl derivate prime} and the fact that $(x - m(\mu, \sigma,x))$ is a $R(\theta, \Delta_{n,i},x)$ function to get that size of $h_{i,n}$ and of the rest functions does not change after having derived with respect to the parameters. \\
Concerning the integral rest coming from the third order of the Dynkin formula, we have that
$$\partial_\vartheta \mathbb{E}_i[A^3 h_{i,n}(X^\theta_{u_3})] = \mathbb{E}_i[\partial_\vartheta A^3 h_{i,n}(X^\theta_{u_3})] + \mathbb{E}_i[\partial_X A^3 h_{i,n}(X^\theta_{u_3}) \partial_\vartheta X].$$
On the first term of the right hand side here above we can act exactly like we did on $\mathbb{E}_i[ A^3 h_{i,n}(X^\theta_{u_3})]$ getting a rest function whose order does not change, while from the computation of $\partial_X A^3 h_{i,n}(X^\theta_{u_3})$ an extra $\Delta_{n,i}^{- \beta}$ appears but, since from Lemma \ref{lemma: derivate x}  the norm $1$ of $\partial_\vartheta X$ is $R(\theta, \Delta_{n,i},x)$ for $\vartheta = \mu$ and $R(\theta, \Delta_{n,i}^\frac{1}{2},x)$ for $\vartheta = \sigma$, it is enough to use previously Holder inequality and observe that both $\frac{1}{2} - \beta$ and $1 - \beta$ are positive to get that the second term here above is negligible compared to the first.

\subsubsection{Proof of Proposition \ref{prop: dl m2 particolare per appl}}
Proposition  \ref{prop: dl m2 particolare per appl} is a particular case in which Proposition \ref{prop: dl alla capitolo 1 m2} holds. The proof relies on the fact that the intensity $F$ is supposed to be $\mathcal{C}^1$ and so we can move from $F(\frac{u \Delta_{n,i}^\beta}{\gamma(x)})$ to $F(0)$ through finite-increments theorem.
\begin{proof}
From \eqref{eq: dl alla capitolo 1 m2} we get the proposition proved remarking that
$$\Delta_{n,i}^{1 + 2 \beta} \int_\mathbb{R} u^2 \varphi(u) F(\frac{u \Delta_{n,i}^\beta}{\gamma(x)}) du = \Delta_{n,i}^{1 + 2 \beta} \int_\mathbb{R} u^2 \varphi(u) F(0) du + \Delta_{n,i}^{1 + 2 \beta} \int_\mathbb{R} u^2 \varphi(u)[ F(\frac{u \Delta_{n,i}^\beta}{\gamma(x)})- F(0)] du$$
and, from finite-increments theorem, the last term is in absolute value upper bounded by \\
$c \Delta_{n,i}^{1 + 3 \beta} \frac{1}{|\gamma(x)|}\int_\mathbb{R} |u|^3 |\varphi(u)| |F'(\tilde{u})| du= R(\theta, \Delta_{n,i}^{1 + 3 \beta},x)$, where $\tilde{u} \in (0, \frac{u \Delta_{n,i}^\beta}{\gamma(x)})$. \\
Moreover, by the smoothness on $F$ we have required, it follows that terms in \eqref{eq: dl alla capitolo 1 m2} whose size depends on the density $F$ are now upper bounded by a  $R(\theta, \Delta_{n,i}^{2 + \beta}, x)$ function. It yields \eqref{eq: dl particolare m2}.
\end{proof}
\vspace{1 cm}
On behalf of all authors, the corresponding author states that there is no conflict of interest.

\end{document}